\documentclass[microtype]{gtpart}     
\usepackage{amsthm,amsfonts,amsmath,amssymb,latexsym}
\usepackage{epsfig,graphics,color}
\usepackage[matrix,arrow,curve]{xy}
\usepackage{hyperref}
\usepackage{verbatim}
\usepackage{bm}

\title{Symplectic homology and the Eilenberg--Steenrod axioms}

\author{Kai Cieliebak}
\givenname{Kai}
\surname{Cieliebak}
\address{Universit\"at Augsburg, Universit\"atsstrasse 14, D-86159 Augsburg, Germany}
\email{kai.cieliebak@math.uni-augsburg.de}
\urladdr{}

\author{Alexandru Oancea}
\givenname{Alexandru}
\surname{Oancea}
\address{Sorbonne Universit\'e, Institut de Math\'ematiques de Jussieu-Paris Rive Gauche, Case 247, 4 place Jussieu, F-75005, Paris, France}
\email{alexandru.oancea@imj-prg.fr}
\urladdr{}



\arxivreference{1511.00485}

\newtheorem{PARA}{}[section]
\newtheorem{theorem}[PARA]{Theorem}
\newtheorem{corollary}[PARA]{Corollary}
\newtheorem{lemma}[PARA]{Lemma}
\newtheorem{proposition}[PARA]{Proposition}
\newtheorem{conjecture}[PARA]{Conjecture}
\theoremstyle{definition}
\newtheorem{definition}[PARA]{Definition}
\newtheorem{remark}[PARA]{Remark}
\newtheorem{example}[PARA]{Example}

\newcommand{\cC}{\mathcal{C}}

\newcommand{\cH}{\mathcal{H}}

\newcommand{\cL}{\mathcal{L}}
\newcommand{\cM}{\mathcal{M}}

\newcommand{\cP}{\mathcal{P}}

\newcommand{\cX}{\mathcal{X}}

\renewcommand{\C}{{\mathbb{C}}}

\renewcommand{\H}{{\mathbb{H}}}
\newcommand{\K}{{\mathbb{K}}}

\newcommand{\N}{{\mathbb{N}}}
\renewcommand{\Q}{{\mathbb{Q}}}
\renewcommand{\R}{{\mathbb{R}}}

\renewcommand{\Z}{{\mathbb{Z}}}
\newcommand{\im}{\mbox{im}\,}        
\newcommand{\id}{\mbox{ id}}         

\newcommand{\ind}{\mbox{ind}}

\newcommand{\Hom}{\mbox{Hom}}
\newcommand{\eps}{{\varepsilon}}
\newcommand{\om}{{\omega}}

\newcommand{\CZ}{\mbox{CZ}}
\def\NABLA#1{{\mathop{\nabla\kern-.5ex\lower1ex\hbox{$#1$}}}}
\def\Nabla#1{\nabla\kern-.5ex{}_{#1}}
\def\Tabla#1{\Tilde\nabla\kern-.5ex{}_{#1}}
\renewcommand{\Tilde}{\widetilde}

\newcommand{\p}{{\partial}}

\newcommand{\wh}{\widehat}
\newcommand{\wt}{\widetilde}
\newcommand{\ol}{\overline}
\newcommand{\MM}{\mathcal{M}}
\newcommand{\NN}{\mathcal{N}}

\begin{document}



\begin{abstract}
We give a definition of symplectic homology for pairs of filled Liouville
cobordisms, and show that it satisfies analogues of the Eilenberg-Steenrod axioms except for the dimension axiom. The resulting long exact sequence of
a pair generalizes various earlier long exact sequences such as the handle
attaching sequence, the Legendrian duality sequence, and the exact sequence
relating symplectic homology and Rabinowitz Floer homology. New consequences
of this framework include a Mayer-Vietoris exact sequence for symplectic homology, invariance of Rabinowitz Floer homology under subcritical handle attachment, and a new product on Rabinowitz Floer homology
unifying the pair-of-pants product on symplectic homology with a secondary
coproduct on positive symplectic homology. 

In the appendix, joint with
Peter Albers, we discuss obstructions to the existence of certain Liouville cobordisms.
\end{abstract}

\maketitle

\tableofcontents


\section{Introduction} 

To begin with, a story. At the Workshop on Conservative Dynamics and Symplectic Geometry held at IMPA, Rio de Janeiro in August 2009, the participants had seen in the course of a single day at least four kinds of Floer homologies for non-compact objects, among which wrapped Floer homology, symplectic homology, Rabinowitz-Floer homology, and linearized contact homology. The second author was seated in the audience next to Albert Fathi, who at some point suddenly turned to him and exclaimed: \emph{``There are too many such homologies!"}. Hopefully this paper can serve as a structuring answer: although there are indeed several versions of symplectic homology (non-equivariant, $S^1$-equivariant, Lagrangian, each coming in several flavors determined by suitable action truncations), we show that they all obey the same axiomatic pattern, very much similar to that of the Eilenberg-Steenrod axioms for singular homology. In order to exhibit such a structured behaviour we need to extend the definition of symplectic homology to pairs of cobordisms endowed with an exact filling. 

We find it useful to explain immediately our definition, although there is a price to pay regarding the length of this Introduction. 

We need to first recall the main version of symplectic homology that is currently in use, which can be interpreted as dealing with cobordisms with empty negative end. This construction associates to a Liouville domain, meaning an exact symplectic manifold $(W^{2n},\omega,\lambda)$, $\omega=d\lambda$ such that $\alpha=\lambda|_{\p W}$ is a positive contact form (see~\S\ref{sec:Liouville-cobordisms}), a \emph{symplectic homology group $SH_*(W)$} which is an invariant of the symplectic completion $(\wh W,\wh\omega)=(W,\omega)\, \cup \, \big([1,\infty)\times\p W, d(r\alpha)\big)$. The generators of the underlying chain complex can be thought of as being either the critical points of a Morse function on $W$ which is increasing towards the boundary, or the positively parameterized closed orbits of the Reeb vector field $R_\alpha$ on $\p W$ defined by $d\alpha(R_\alpha,\cdot)=0$, $\alpha(R_\alpha)=1$. Since the generators of the underlying complex are closed Hamiltonian orbits, we also refer to symplectic homology as being a \emph{theory of closed strings} (compare with the discussion of Lagrangian symplectic homology, or wrapped Floer homology, further below).
We interpret a Liouville domain $(W,\omega,\lambda)$ as an exact symplectic filling of its contact boundary $(M,\xi=\ker\alpha)$, or as an exact cobordism from the empty set to $M$, which we call the \emph{positive boundary of $W$}, also denoted $M=\p^+W$. 

The implementation of this setup is the following. One considers on $\wh W$ (smooth time-dependent $1$-periodic approximations of) Hamiltonians $H_\tau$ which are identically zero on $W$ and equal to the linear function $\tau r-\tau$, $r\in[1,\infty)$ on the symplectization part $[1,\infty)\times M$, where $\tau>0$ is different from the period of a closed Reeb orbit on $M$. 
One then sets 
$$
SH_*(W)=\lim\limits^{\longrightarrow}_{\tau\to\infty} FH_*(H_\tau)
$$ 
where $FH_*(H_\tau)$ stands for Hamiltonian Floer homology of $H_\tau$ which is generated by closed Hamiltonian orbits of period $1$, 
and the direct limit is considered with respect to continuation maps induced by increasing homotopies of Hamiltonians. The dynamical interpretation of these homology groups reflects the fact that the Hamiltonian vector field of a function $h(r)$ defined on the symplectization part $[1,\infty)\times M$ is equal to $X_h(r,x)=h'(r)R_\alpha(x)$. A schematic picture for the Hamiltonians underlying symplectic homology of such cobordisms with empty negative end is given in Figure~\ref{fig:SHdomains}, in which the arrows indicate the location of the two kinds of generators for the underlying complex, constant orbits in the interior of the cobordism and nonconstant orbits located in the ``bending" region near the positive boundary. The vertical thick dotted arrow in Figure~\ref{fig:SHcobordisms} indicates that we consider a limit over $\tau\to\infty$. 

\begin{figure} [ht]
\centering
\input{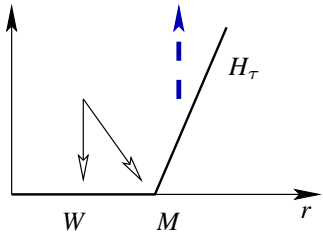}
\caption{Symplectic homology of a domain}
\label{fig:SHdomains}
\end{figure}

Key to our construction is the notion of \emph{Liouville cobordism with filling}. The definition of a Liouville cobordism $W^{2n}$ is similar to that of a Liouville domain, with the notable difference that we allow the volume form $\alpha\wedge(d\alpha)^{n-1}$ determined by $\alpha$ on $\p W$ to define the opposite of the boundary orientation on some of the components of $\p W$, the collection of which is called \emph{the negative boundary of $W$} and is denoted $\p^-W$, while \emph{the positive boundary of $W$} is $\p^+ W=\p W\setminus \p^- W$. In addition, we assume that one is given a Liouville domain $F$ whose positive boundary is isomorphic to the contact negative boundary of $W$, so that the concatenation $F\circ W$ is a Liouville domain with positive boundary $\p^+W$. 

Given a Liouville cobordism $W$ with filling $F$, the output of the closed theory is a symplectic homology group $SH_*(W)$. Although we drop the filling $F$ from the notation for the sake of readability, this homology group does depend on $F$. The dependence is  well understood in terms of the geometric augmentation of the contact homology algebra of $\p^-W$ induced by the filling, see~\cite{BOcont}. Symplectic homology $SH_*(W)$ is an invariant of the Liouville homotopy class of $W$ with filling, and the generators of the underlying chain complex can be thought of as being of one of the following three types: negatively parameterized closed Reeb orbits on $\p^-W$, constants in $W$, and positively parameterized closed Reeb orbits on $\p^+W$. 

To implement this setup one considers (smooth time-dependent $1$-periodic approximations of) Hamiltonians $H_{\mu,\tau}$ described as follows: they are equal to the linear function $\tau r-\tau$ on the symplectization part $[1,\infty)\times\p^+W$, they are identically equal to $0$ on $W$, they are equal to the linear function $-\mu r+\mu$ on some finite but large part of the negative symplectization $(\delta,1]\times\p^-W\subset F$ with $\delta>0$, and finally they are constant on the remaining part of $F$. Here $\tau>0$ is required not to be equal to the period of a closed Reeb orbit on $\p^+W$, 
and $\mu>0$ is required not to be equal to the period of a closed Reeb orbit on $\p^-W$. 
Finally, one sets 
$$
SH_*(W)=\lim\limits^{\longrightarrow}_{b\to\infty} \lim\limits^{\longleftarrow}_{a\to-\infty} \lim\limits^{\longrightarrow}_{\mu,\tau\to\infty}  FH_*^{(a,b)}(H_{\mu,\tau}),
$$
where $FH_*^{(a,b)}$ denotes Floer homology \emph{truncated in the finite action window $(a,b)$}. 

Though the definition may seem frightening when compared to the one for Liouville domains, it is actually motivated analogously by the dynamical interpretation of the groups that we wish to construct. Let us consider the corresponding shape of Hamiltonians depicted in Figure~\ref{fig:SHcobordisms}. (The vertical thick dotted arrows in Figure~\ref{fig:SHcobordisms} indicate that we consider limits over $\mu\to\infty$ and $\tau\to\infty$.) A Hamiltonian $H_{\mu,\tau}$ has $1$-periodic orbits either in the regions where it is constant, or in the small ``bending" regions near $\{\delta\}\times\p^-W$ and $\p^\pm W$ where it acquires some derivative with respect to the symplectization coordinate $r$. The role of the finite action window $(a,b)$ in the definition is to take into account only the orbits located in the areas indicated by arrows in Figure~\ref{fig:SHcobordisms}: as $\mu$ and $\tau$ increase, the orbits located deep inside the filling $F$ have very negative action and naturally fall outside the action window. The order of the limits on the extremities of the action window, first an inverse limit on $a\to-\infty$ and then a direct limit on $b\to\infty$, is important. It has two motivations: (i) the inverse limit functor is not exact except when applied to an inverse system consisting of finite dimensional vector spaces. Should one wish to exchange the order of the limits on $a$ and $b$, such a finite dimensionality property would typically not hold on the inverse system indexed by $a\to-\infty$, and this would have implications on the various exact sequences that we construct in the paper. (ii) With this definition, symplectic homology of a cobordism is a ring with unit (see~\S\ref{sec:products}). Should one wish to reverse the order of the limits on $a$ and $b$, this would not be true anymore. 

\begin{figure} [ht]
\centering
\input{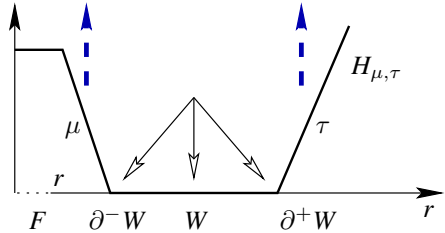}
\caption{Symplectic homology of a cobordism}
\label{fig:SHcobordisms}
\end{figure}

It turns out that the full structure of symplectic homology involves in a crucial way a definition that is yet more involved, namely that of \emph{symplectic homology groups of a pair of filled Liouville cobordisms}. To give the definition of such a pair it is important to single out the operation of \emph{composition of cobordisms} which we already implicitly used above. Given cobordisms $W$ and $W'$ such that $\p^+W=\p^-W'$ as contact manifolds, one forms the Liouville cobordism $W\circ W' = W \ _{\p^+ W}\!\cup_{\p^-W'} \ W'$ by gluing the two cobordisms along the corresponding boundary. The resulting Liouville structure is well-defined up to homotopy. A \emph{pair of Liouville cobordisms} $(W,V)$ then consists of a Liouville cobordism $(W,\omega,\lambda)$ together with a codimension $0$ submanifold with boundary $V\subset W$ such that $(V,\omega|_V,\lambda|_V)$ is a Liouville cobordism and $(W\setminus V,\omega|,\lambda|)$ is the disjoint union of two Liouville cobordisms $W^{bottom}$ and $W^{top}$ such that $W=W^{bottom}\circ V\circ W^{top}$. We allow any of the cobordisms $W^{bottom}$, $W^{top}$, or $V$ to be empty. If $V=\varnothing$ we think of the pair $(W,\varnothing)$ as being the cobordism $W$ itself. A convenient abuse of notation is to allow $V=\p^+W$ or $V=\p^-W$, in which case we think of $V$ as being a trivial collar cobordism on $\p^\pm W$. This setup does not allow for $V=\p W$ in case the latter has both negative and positive components, 
{\color{black}but one can extend it in this direction without much
  difficulty at the price of somewhat burdening the notation, see
  Remark~\ref{rmk:multilevel} and Section~\ref{sec:multilevel}.} 
A \emph{pair of Liouville cobordisms with filling} is a pair $(W,V)$ as above, together with an exact filling $F$ of $\p^-W$. In this case the cobordism $V$ inherits a natural filling $F\circ W^{bottom}$. See Figure~\ref{fig:cobordism_pair}. 

\begin{figure} [ht]
\centering
\input{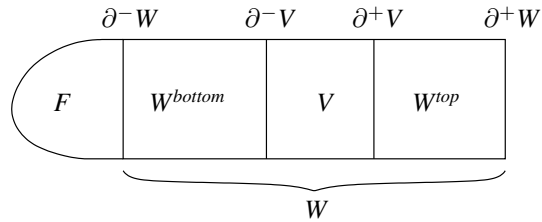}
\caption{Cobordism pair $(W,V)$ with filling $F$}
\label{fig:cobordism_pair}
\end{figure}

Given a cobordism pair $(W,V)$ with filling $F$ we define a symplectic homology group $SH_*(W,V)$ by a procedure similar to the above, involving suitable direct and inverse limits and based on Hamiltonians that have the more complicated shape depicted in Figure~\ref{fig:SHcobordism_pair}. The Hamiltonians depend now on three parameters $\mu,\nu,\tau>0$ and the vertical thick dotted arrows in Figure~\ref{fig:SHcobordism_pair} indicate that we consider limits over $\mu,\nu,\tau\to\infty$. One sets
$$
SH_*(W,V)=\lim\limits^{\longrightarrow}_{b\to\infty} \lim\limits^{\longleftarrow}_{a\to-\infty} \lim\limits^{\longrightarrow}_{\mu,\tau\to\infty}  \lim\limits^{\longleftarrow}_{\nu\to\infty} FH_*^{(a,b)}(H_{\mu,\nu,\tau}).
$$
This is as complicated as it gets.  The definition is again motivated by the dynamical interpretation of the groups that we wish to construct. For a given finite action window and for suitable choices of the parameters the orbits that are taken into account in $FH_*(H_{\mu,\nu,\tau})$ are located in the regions indicated by arrows in Figure~\ref{fig:SHcobordism_pair}. They correspond (from left to right in the picture) to negatively parameterized closed Reeb orbits on $\p^-W$, to constants in $W^{bottom}$, to negatively parameterized closed Reeb orbits on $\p^- V$, to positively parameterized closed Reeb orbits on $\p^+V$, to constants in $W^{top}$, and finally to positively parameterized closed Reeb orbits on $\p^+W$ (see~\S\ref{sec:excision}). 

We wish to emphasise at this point the fact that the above groups of periodic orbits \emph{cannot be singled out solely from action considerations}. Filtering by the action and considering suitable subcomplexes or quotient complexes is the easiest way to extract useful information from some large chain complex, but this is not enough for our purposes here. Indeed, getting hold of enough tools in order to single out the desired groups of orbits was one of the major difficulties that we encountered. We gathered these tools in~\S\ref{sec:confinement}, and there are no less than {\color{black}four} of them: a robust maximum principle due to Abouzaid and Seidel~\cite{Abouzaid-Seidel} (Lemma~\ref{lem:no-escape}), an asymptotic behaviour lemma which appeared for the first time in~\cite{BOcont} (Lemma~\ref{lem:asy}), a new stretch-of-the-neck argument tailored to the situation at hand (Lemma~\ref{lem:neck}), {\color{black}and a new mechanism to exclude certain Floer trajectories asymptotic to constant orbits (Lemma~\ref{lem:constant})}. The simultaneous use of these tools is illustrated by the proof of the Excision Theorem~\ref{thm:excision}.

\begin{figure} [ht]
\centering
\input{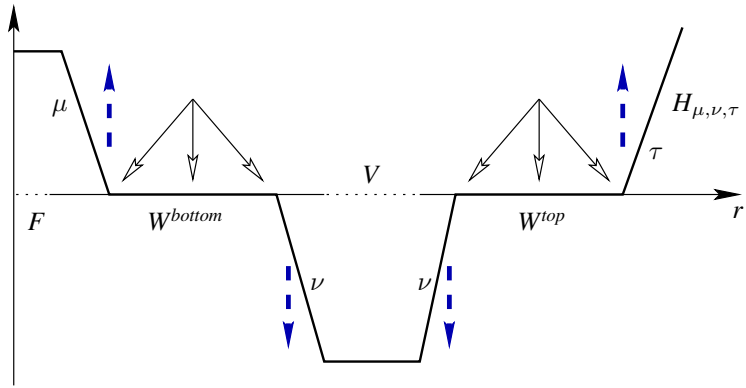}
\caption{Symplectic homology of a cobordism pair}
\label{fig:SHcobordism_pair}
\end{figure}

Important particular cases of such relative symplectic homology groups are the \emph{symplectic homology groups of a filled Liouville cobordism relative to (a part of) its boundary}. Recalling that we think of a contact type hypersurface in $W$ as a trivial collar cobordism, we obtain groups $SH_*(W,\p^\pm W)$. It turns out that these can be equivalently defined using Hamiltonians of a much simpler shape, as shown in Figure~\ref{fig:SHrelboundary} below. It is then straightforward to define also symplectic homology groups $SH_*(W,\p W)$, which play a role in the formulation of Poincar\'e duality, see~\S\ref{sec:Poincare_duality}. 
We refer to~\S\ref{sec:SHWA} for the details of the definitions. 

\begin{figure} [ht]
\centering
\input{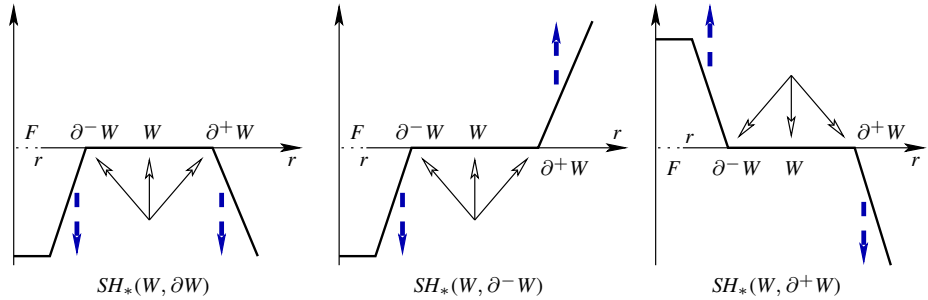}
\caption{Symplectic homology of a cobordism relative to its boundary}
\label{fig:SHrelboundary}
\end{figure}

{\color{black} 
\begin{remark}\label{rmk:multilevel} Our previous conventions for Liouville pairs do not allow to interpret $SH_*(W,\p W)$ as symplectic homology of the pair $(W,[0,1]\times\p W)$ in case $\p W$ has both negative and positive components. To remedy for this one needs to further extend the setup to \emph{pairs of multilevel Liouville cobordisms with filling}, see~\S\ref{sec:multilevel}.
\end{remark}
} 

The mnemotechnic rule for all these constructions is the following: 
\begin{center}
\emph{To compute $SH_*(W,V)$ one must use a family of Hamiltonians that vanish on $W\setminus V$, that go to $-\infty$ near $\p V$ and that go to $+\infty$ near $\p W\setminus \p V$.}
\end{center} 

Some of these shapes of Hamiltonians already appeared, if only implicitly, in Viterbo's foundational paper~\cite{Viterbo99}, as well as in~\cite{Ci02}. We make their use systematic. 

These constructions have Lagrangian analogues, which we will refer to as
the \emph{open theory}. The main notion is that of an
\emph{exact Lagrangian cobordism with filling}, meaning an exact
Lagrangian submanifold $L\subset W$ of a Liouville cobordism $W$,
which intersects $\p W$ transversally, and such that $\p^-L=L\cap
\p^-W$ is the Legendrian boundary of an exact Lagrangian submanifold
$L_F\subset F$ inside the filling $F$ of $W$. We call $L_F$ an
\emph{exact Lagrangian filling}. There is also an obvious notion of
\emph{exact Lagrangian pair with filling}. The open theory associates
to such a pair $(L,K)$ a \emph{Lagrangian symplectic homology group}
$SH_*(L,K)$, which is an invariant of the Hamiltonian isotopy class
preserving boundaries of the pair $(L,K)$ inside the Liouville pair
$(W,V)$. (In the case of a single Lagrangian $L$ with empty negative
boundary this is known under the name of \emph{wrapped Floer homology
  of $L$}.) Formally the implementation of the Lagrangian setup is the
same, using exactly the same shapes of Hamiltonians for a Lagrangian
Floer homology group. The generators of the relevant chain complexes
are then Hamiltonian chords which correspond either to Reeb chords
with endpoints on the relevant Legendrian boundaries, or to constants
in the interior of the relevant Lagrangian cobordisms. One can also
mix the closed and open theories together as in~\cite{Ekholm-Oancea},
see~\S\ref{sec:LagSH}, and there are also $S^1$-equivariant closed
theories, see~\S\ref{sec:S1equiv}. In order to streamline the
discussion, we shall restrict in this Introduction to the
non-equivariant closed theory described above.  

{\bf Remark }(grading).
For simplicity we shall restrict in this paper to
Liouville domains $W$ whose first Chern class vanishes. In this case
the filtered Floer homology groups are $\Z$-graded by the
Conley-Zehnder index, 
where the grading depends on the choice of a trivialisation of the canonical
bundle of $W$ for each free homotopy classes of loops. If $c_1(W)$ is
non-zero the groups are only graded modulo twice the minimal Chern
number. 

As announced in the title, one way to state our results is in terms of the Eilenberg-Steenrod axioms for a homology theory. We define a category which we call \emph{the Liouville category with fillings} whose objects are pairs of Liouville cobordisms with filling, and whose morphisms are exact embeddings of pairs of Liouville cobordisms with filling. Such an exact embedding of a pair $(W,V)$ with filling $F$ into a pair $(W',V')$ with filling $F'$ is an exact codimension $0$ embedding $f:W\hookrightarrow W'$, meaning that $f^*\lambda'-\lambda$ is an exact $1$-form, together with an extension $\bar f:F\circ W\hookrightarrow F'\circ W'$ which is also an exact codimension $0$ embedding, and such that $f(V)\subset V'$. A cobordism triple $(W,V,U)$ (with filling) is a topological triple such that $(W,V)$ and $(V,U)$ are cobordism pairs (with filling). 

\begin{theorem} \label{thm:ESSH}
Symplectic homology with coefficients in a field $\K$ defines a contravariant functor from the Liouville category with fillings to the category of graded $\K$-vector spaces. It associates to a pair $(W,V)$ with filling the symplectic homology groups $SH_*(W,V)$, and to an exact embedding $f:(W,V)\hookrightarrow (W',V')$ between pairs with fillings a linear map 
$$
f_!:SH_*(W',V')\to SH_*(W,V)
$$    
called \emph{Viterbo transfer map}, or \emph{shriek map}. This functor satisfies the following properties: 

(i) {\sc (homotopy)} If $f$ and $g$ are homotopic through exact embeddings, then
$$
f_!=g_!.
$$ 

(ii) {\sc (exact triangle of a pair)} Given a pair $(W,V)$ for which we denote the inclusions $V\stackrel{i}\longrightarrow W\stackrel{j}\longrightarrow (W,V)$, there is a functorial exact triangle in which the map $\p$ has degree $-1$
\begin{equation*} 
\xymatrix
@C=10pt
@R=18pt
{
SH_*(W,V) \ar[rr]^{j_!} & & 
SH_*(W) \ar[dl]^{i_!} \\ & SH_*(V) \ar[ul]^-\p_-{[-1]}  
}
\end{equation*}
Here we identify as usual a cobordism $W$ with the pair $(W,\varnothing)$.

(iii) {\sc (excision)} For any cobordism triple $(W,V,U)$, the transfer map induced by the inclusion $(W\setminus \mbox{int}(U), V\setminus \mbox{int}(U))\stackrel{i}\longrightarrow (W,V)$ is an isomorphism: 
$$
i_!:SH_*(W,V)\stackrel\simeq\longrightarrow SH_*(W\setminus \mbox{int}(U),V\setminus \mbox{int}(U)).
$$
\end{theorem} 

These are symplectic analogues of the first Eilenberg-Steenrod axioms
for a homology theory~\cite{ES52}. The one fact that may be puzzling
about our terminology is that we insist on calling this a homology
theory, though it defines a contravariant functor. Our arguments  
are the following. The first one is geometric: With
$\Z/2$-coefficients we have an isomorphism $SH_*(T^*M)\simeq H_*(\cL
M)$ between the symplectic homology of the cotangent bundle of a
closed manifold $M$ and the homology of $\cL M$, the space of free
loop on $M$. Moreover, the product structure on $SH_*(T^*M)$ is
isomorphic to the Chas-Sullivan product structure on $H_*(\cL M)$, and
the latter naturally lives on homology since it extends the
intersection product on $H_*(M)\cong H_{n+*}(T^*M,T^*M\setminus
M)$. The second one is algebraic and uses the $S^1$-equivariant
version of symplectic homology (see~\S\ref{sec:S1equiv}): We wish that
$S^1$-equivariant homology with coefficients in any ring $R$ be
naturally a $R[u]$-module, with $u$ a formal variable of degree $-2$,
and that multiplication by $u$ be nilpotent. In contrast,
$S^1$-equivariant cohomology should naturally be a $R[u]$-module, with
$u$ of degree $+2$, and
multiplication by $u$ should typically \emph{not} be nilpotent. This
is exactly the kind of structure that we have on the $S^1$-equivariant
version of our symplectic homology groups. The third one is an
algebraic argument that refers to the $0$-level part of the
$S^1$-equivariant version of a filled Liouville cobordism: Given such
a cobordism $W^{2n}$, this $0$-level part is denoted
$SH_k^{S^1,=0}(W)$ and can be expressed either as the degree $n+k$
part of $H_*(W,\p W)\otimes R[u^{-1}]$, with $R$ the ground ring and
$u$ of degree $-2$, or as the degree $n-k$ part of $H^*(W)\otimes
R[u]$. Since $H^*(W)\otimes R[u]$ is nontrivial in arbitrarily
negative degrees, it is only the first expression that allows
the interpretation of $SH_*^{S^1,=0}(W)$ as the singular (co)homology
group of a topological space via the Borel construction. This natural
emphasis on homology determines our interpretation of the induced maps
as shriek or transfer maps.  

\emph{Our bottom line is that the theory is homological in nature, but contravariant because the induced maps are shriek maps}. 

Note that in the case of a pair $(W,V)$ the simplest expression for $SH_k^{S^1,=0}(W,V)$ is obtained by identifying it with the degree $n-k$ part of the \emph{co}homology group $H^*(W,V)\otimes R[u^{-1}]$. To turn this into homology one needs to use excision followed by Poincar\'e duality, and the expression gets more cumbersome. A similar phenomenon happens for the non-equivariant version $SH_*^{=0}(W,V)$. In order to simplify the notation we always identify the $0$-level part of symplectic homology with singular cohomology throughout the paper.

{\bf Remark }(coefficients).
The symplectic homology
groups are defined with coefficients in an arbitrary ring $R$, and
statement~(i) in Theorem~\ref{thm:ESSH} is valid with arbitrary
coefficients too. Field coefficients are necessary only for
statements~(ii) and~(iii). As a general fact, the statements in this
paper which involve exact triangles are only valid with field
coefficients, and the proof of excision does require such an exact
triangle, see~\S\ref{sec:excision}. The reason is that we define our
symplectic homology groups as a first-inverse-then-direct-limit over
symplectic homology groups truncated in a finite action window. The
various exact triangles involving symplectic homology are obtained by
passing to the limit in the corresponding exact triangles for such
finite action windows, at which point arises naturally the question of
the exactness of the direct limit functor and of the inverse limit
functor. While the direct limit functor is exact, the inverse limit
functor is not. Nevertheless, the inverse limit functor is exact when
applied to inverse systems consisting of finite dimensional vector
spaces, which is the case for symplectic homology groups truncated in
a finite action window. In order to extend the exact triangle of a
pair (and also the other exact triangles which we establish in this
paper) to arbitrary coefficients one would need to modify the
definition of our groups by passing to the limit at chain level and
use a version of the Mittag-Leffler condition, a path that we shall
not pursue here. See also the discussion of factorisation homology
below, the discussion in~\S\ref{sec:cones}, and Remark~\ref{rmk:localization-chain-level}. 
More generally, one can define symplectic homology with coefficients in a local system with fibre $\K$, see~\cite{Ritter09,Abouzaid-cotangent}, and most of the results of this paper adapt in a quite straightforward way to that setup. {\color{black}One notable exception are the duality results in~\S\ref{sec:coh-and-duality}, in which the treatment of local coefficients would be more delicate.
}

{\color{black}In view of the above discussion, we henceforth adopt the following convention:}

{\color{black}{\bf Convention} (coefficients). {\em In this paper we use constant coefficients in a field $\K$.}}

Let us now discuss briefly the two other Eilenberg-Steenrod axioms, namely the direct sum axiom and the dimension axiom, and explain why they do not, and indeed cannot, have a symplectic counterpart. 
(I) The \emph{direct sum axiom} expresses the fact that the homology of an arbitrary disjoint union of topological spaces is naturally isomorphic to the direct sum of their homologies, whereas in contrast a cohomology theory would involve a direct product. The distinction between direct sums and direct products is not relevant in the setup of Liouville domains, which are by definition compact and therefore consist of at most finitely many connected components. Passing to arbitrary disjoint unions would mean to go from Liouville domains to Liouville manifolds as in~\cite{Seidel07}, and the contravariant nature of the functor would imply that it behaves as a direct product. This is one of the reasons why~\cite{Seidel07} refers to the same object as ``symplectic \emph{co}homology". However, in view of the extension of the definition to cobordisms this appears to be only a superficial distinction. The deeper fact is that, whichever way one turns it around, symplectic homology of a cobordism with nonempty negative boundary is an object of a mixed homological-cohomological nature because its definition involves both a direct limit (on $b\to\infty$) and an inverse limit (on $a\to-\infty$). We actually present in~\S\ref{sec:algebraic_duality} an example showing that  algebraic duality fails already in the case of symplectic homology of a trivial cobordism. 
(II) The \emph{dimension axiom} of Eilenberg and Steenrod expresses the fact that the value of the functor on any pair homotopy equivalent to a pair of CW-complexes is determined by its value on a point. This fact relies on the homotopy axiom and illustrates the strength of the latter: since any ball is homotopy equivalent to a point, the homotopy axiom allows one to go up in dimension for computations. As a matter of fact the dimension of a space plays no role in the definition of a homology theory in the sense of Eilenberg and Steenrod, although it is indeed visible homologically via the fact that the homology of a pair consisting of an $n$-ball and of its boundary is concentrated in degree $n$. In contrast, symplectic homology is a dimension dependent theory. Moreover, it cannot be determined by its value on a single object. \emph{No change in dimension is possible, and no dimension axiom can exist.} For example, symplectic homology vanishes on the $2n$-dimensional ball since the latter is subcritical, but the theory is nontrivial. The symplectic analogue of the class of CW-complexes is that of Weinstein manifolds, and the whole richness of symplectic homology is encoded in the way it behaves under critical handle attachments, see~\cite{Bourgeois-Ekholm-Eliashberg-1}. One could say that it is determined by its value on the elementary cobordisms corresponding to a single critical handle attachment, but that would be an essentially useless statement, since it would involve all possible contact manifolds and all their possible exact fillings. The complexity of symplectic homology reflects that of Reeb dynamics and is such that there is no analogue of the dimension axiom. 

We show in~\S\ref{sec:Poincare_duality} how to interpret Poincar\'e duality by defining an appropriate version of symplectic cohomology, and we establish in~\S\ref{sec:Mayer-Vietoris} a Mayer-Vietoris exact triangle. 

It is interesting to note at this point a formal similarity with the recent development of \emph{factorisation homology}, see the paper~\cite{Ayala-Francis} by Ayala and Francis as well as the references therein. Roughly speaking, a factorisation homology theory is a graded vector space valued \emph{monoidal} functor defined on some category of open topological manifolds of \emph{fixed dimension $n$}, with morphism spaces given by topological embeddings, and which obeys a dimension axiom involving the notion of an $E_n$-algebra. (Such a category includes in particular that of closed manifolds of dimension $n-1$, which are identified with open trivial cobordisms of one dimension higher, a procedure very much similar to our viewpoint on contact hypersurfaces as trivial cobordisms.) If one forgets the monoidal property then one essentially recovers the restriction of an Eilenberg-Steenrod homology theory to a category of manifolds of fixed dimension. Conjecturally the symplectic analogue of a factorisation homology theory should involve some differential graded algebra (DGA) enhancement of symplectic homology in the spirit of~\cite{Ekholm-Oancea}, and the axioms satisfied by factorisation homology should provide a reasonable approximation to the structural properties satisfied by such a DGA enhancement. 

One other lesson that the authors have learned from Ayala and Francis~\cite{Ayala-Francis} is that the Eilenberg-Steenrod axioms can, and probably should, be formulated at chain level. More precisely, the target of a homological functor is naturally the category of chain complexes up to homotopy rather than that of graded $R$-modules. This kind of formulation in the case of symplectic homology seems to lie at close hand using the methods of our paper, but we shall not deal with it.

A fruitful line of thought, pioneered by Viterbo in the case of Liouville domains~\cite{Viterbo99}, is to compare the symplectic homology groups of a pair $(W,V)$ with the singular cohomology groups, the philosophy being that the difference between the two measures the amount of homologically interesting dynamics on the relevant contact boundary. The singular cohomology $H^{n-*}(W,V)$ is visible through the Floer complex generated by the constant orbits in $W\setminus V$ of any of the Hamiltonians $H_{\mu,\nu,\tau}$, see Figure~\ref{fig:SHcobordism_pair}, with the degree shift being dictated by our normalisation convention for the Conley-Zehnder index, and this Floer complex coincides with the Morse complex since we work on symplectically aspherical manifolds and  the Hamiltonian is essentially constant in the relevant region~\cite{SZ92}. Note also that these constant orbits are singled out among the various types of orbits involved in the computation of $SH_*(W,V)$ by the fact that their action is close to zero, whereas all the other orbits have negative or positive action bounded away from zero. Accordingly, we denote the resulting homology group by $SH_*^{=0}(W,V)$, with the understanding that we have an isomorphism 
$$
SH_*^{=0}(W,V)\simeq H^{n-*}(W,V).
$$
In the case of a Liouville domain (Figure~\ref{fig:SHdomains}) we see that these constant orbits form a subcomplex since all the other orbits have positive action. As such, for a Liouville domain there is a natural map $H^{n-*}(W)\to SH_*(W)$. In the case of a cobordism or of a pair of cobordisms such a map does not exist anymore since the orbits on level zero do not form a subcomplex anymore. The correct way to heal this apparent ailment is to consider symplectic homology groups \emph{truncated in action with respect to the zero level}, which we denote 
$$
SH_*^{>0}(W,V),\qquad SH_*^{\ge 0}(W,V),\qquad SH_*^{\le 0}(W,V),\qquad SH_*^{<0}(W,V).
$$
Their meaning is the following. Each of them respectively takes into account, \emph{among the orbits involved in the definition of $SH_*(W,V)$}, the ones which have strictly positive action (on $\p^+V$ and $\p^+W$), non-negative action (on $\p^+V$, $\p^+W$, and $W\setminus V$), non-positive action (on $\p^-V$, $\p^-W$, and $W\setminus V$), negative action (on $\p^-V$ and $\p^-W$). We refer to~\S\ref{sec:SHWA} and~\S\ref{sec:SHpair} for the definitions. 

We have maps $SH_*^{<0}(W,V)\to SH_*^{\le 0}(W,V)\to SH_*(W,V)$
induced by inclusions of subcomplexes, and also maps $SH_*(W,V)\to
SH_*^{\ge 0}(W,V)\to SH_*^{>0}(W,V)$ induced by projections onto
quotient complexes. The group $SH_*^{=0}(W,V)$ can be thought of as a
homological cone since it completes the map $SH_*^{<0}(W,V)\to
SH_*^{\le 0}(W,V)$ to an exact triangle. The various maps which
connect these groups are conveniently depicted as forming an
octahedron as in diagram~\eqref{eq:SHoctahedron}. The continuous arrows preserve the degree, whereas the dotted arrows decrease the degree by $1$. Among the eight triangles forming the surface of the octahedron, the four triangles whose sides consist of one dotted arrow and two continuous arrows are exact triangles (see Proposition~\ref{prop:taut-triang-WV}), and the four triangles whose sides consist either of three continuous arrows or of one continuous arrow and two dotted arrows are commutative. The structure of this octahedron is exactly the same as the one involved in the octahedron axiom for a triangulated category~\cite[Chapter~1]{KS}, and for a good reason: this tautological octahedron can be deduced from the octahedron axiom of a triangulated category starting from (the chain level version of) a commuting triangle which involves $SH_*^{<0}$, $SH_*^{\le 0}$, and $SH_*$, and in which the composition of the natural maps $SH_*^{<0}\to SH_*^{\le 0}\to SH_*$ is the natural map $SH_*^{<0}\to SH_*$. Turning this around, this action-filtered octahedron can serve as an interpretation of the octahedron axiom for a triangulated category fit for readers with a preference for variational methods over homological methods.

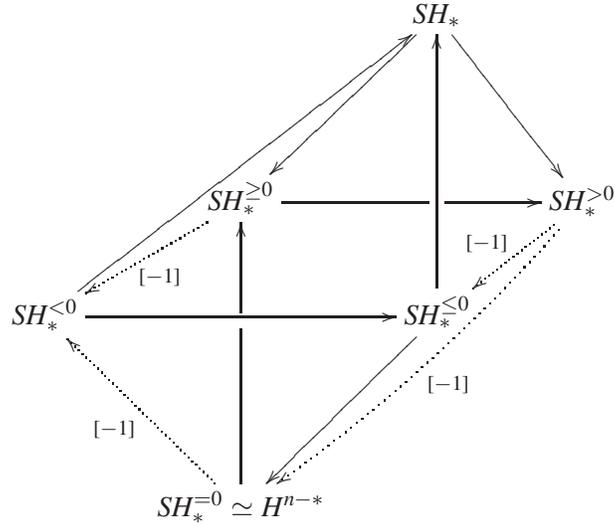
\begin{equation} \label{eq:SHoctahedron}
\xymatrix{
& & SH_* \ar[ddl] \ar[ddr] & \\
& & & \\
& SH_*^{\ge 0} \ar@{.>}[dl]^{[-1]} \ar '[r] [rr] & & SH_*^{>0} \ar@{.>}[dl]_{[-1]} \ar@{.>}@/^/[dddll]^{[-1]}  \\
SH_*^{<0} \ar[rr] \ar[uuurr] & & SH_*^{\le 0} \ar[ddl] \ar[uuu] & \\
& & & \\
& SH_*^{=0} \simeq H^{n-*} \ar@{.>}[uul]^{[-1]} \ar '[uu] [uuu] & &  
}
\end{equation}

Our uniform and emotional notation for these groups is 
$$
SH_*^\heartsuit(W,V),\qquad \heartsuit\in\{\varnothing, >0,\ge 0, =0, \le 0, <0\},
$$
with the meaning that $SH_*^\varnothing=SH_*$. 

\begin{definition}
A functor from the Liouville category with fillings to the category of graded $\K$-vector spaces satisfying the conclusions of Theorem~\ref{thm:ESSH} is called a \emph{Liouville homology theory}.
\end{definition}

\begin{theorem} For $\heartsuit\in\{\varnothing, >0,\ge 0, =0, \le 0, <0\}$ the action filtered symplectic homology group $SH_*^\heartsuit$ with coefficients in a field $\K$ defines a Liouville homology theory. 

The octahedron~\eqref{eq:SHoctahedron} defines a diagram of natural transformations which is compatible with the functorial exact sequence of a pair.
\end{theorem}

In particular, each of the symplectic homology groups $SH_*^\heartsuit$ defines a Liouville homotopy invariant of the pair $(W,V)$. Note that such an invariance statement can only hold provided we truncate the action with respect to the zero value, which is the level of constant orbits. Indeed, answering a question of Polterovich and Shelukhin, we can define symplectic homology groups $SH_*^{(a,b)}(W,V)$ truncated in an arbitrary action interval $(a,b)\subset \R$, see~\S\ref{sec:SHpair}, and the exact triangle of a pair still holds for $SH_*^{(a,b)}$. However, the homotopy axiom would generally break down and the resulting homology groups would not be Liouville homotopy invariant, except if the interval is either small and centered at $0$, or semi-infinite with the finite end close enough to zero, which are the cases that we consider. Failure of Liouville homotopy invariance for most truncations by the action can be easily detected by rescaling the symplectic form. We believe this action filtration carries interesting information for cobordisms in the form of spectral invariants, or more generally persistence modules~\cite{Polterovich-Shelukhin}.

\emph{What do we gain from this extension of the theory of symplectic homology from Liouville domains to Liouville cobordisms, and from having singled out the action filtered symplectic homology groups $SH_*^\heartsuit$~?} Firstly, a broad unifying perspective. Secondly, new computational results. We refer to~\S\ref{sec:variants}, \S\ref{sec:applications}, and~\S\ref{sec:products} for a full discussion, and give here a brief overview.

(a) Our point of view encompasses symplectic homology, wrapped Floer homology, Rabinowitz-Floer homology, $S^1$-equivariant symplectic homology, linearized contact homology, non-equivariant linearized contact homology. Indeed: 

In view of~\cite{Cieliebak-Frauenfelder-Oancea} Rabinowitz-Floer homology of a separating contact hypersurface $\Sigma$ in a Liouville domain $W$ is $SH_*(\Sigma)$, understood to be computed with respect to the natural filling $\mbox{int}(\Sigma)$. 

We show in~\S\ref{sec:S1equiv} that Viterbo's $S^1$-equivariant symplectic homology $SH_*^{S^1}$ and its flavors $SH_*^{S^1,\heartsuit}$ define Liouville homology theories, and the same is true for negative and periodic cyclic homology. The Gysin exact sequences are diagrams of natural transformations which are compatible with the exact triangles of pairs and with the octahedron~\eqref{eq:SHoctahedron}. 

In view of~\cite{BO3+4} linearized contact homology is encompassed by $SH_*^{S^1,>0}$ and non-equivariant linearized contact homology is encompassed by $SH_*^{>0}$. Moreover, our enrichment of symplectic homology to (pairs of) cobordisms indicates several natural extensions of linearized contact homology theories which blend homology with cohomology and whose definition involves the ``banana", i.e. the genus zero curve with two positive punctures, see also~\cite{BEE-product} and Remark~\ref{rmk:SFT}. Indeed, such an enrichment should exist at the level of contact homology too, i.e. non-linearized. 

(b) Most of the key exact sequences established in recent years for symplectic invariants involving pseudo-holomorphic curves appear to us as instances of the exact triangle of a pair. Examples are the critical handle attaching exact sequence~\cite{Bourgeois-Ekholm-Eliashberg-1}, the new subcritical handle attaching exact sequence of~\S\ref{sec:subcritical-handle}, see also~\cite{BvK}, the exact sequence relating Rabinowitz-Floer homology and symplectic homology~\cite{Cieliebak-Frauenfelder-Oancea}, the Legendrian duality exact sequence~\cite{Ekholm-Etnyre-Sabloff}. We discuss these in detail in~\S\ref{sec:applications}. Our point of view embeds all these isolated results into a much larger framework and establishes compatibilities between exact triangles, e.g. with Gysin exact triangles, see~\S\ref{sec:S1equiv}.

(c) Since our setup covers Rabinowitz-Floer homology, it clarifies in particular the functorial behaviour of the latter. Unlike for symplectic homology, a cobordism does not give rise to a transfer map but rather to a correspondence 
$$
SH_*(\p^- W)\longleftarrow SH_*(W)\longrightarrow SH_*(\p^+W).
$$ 
This allows us in particular to prove invariance of Rabinowitz-Floer homology under subcritical handle attachment and understand its behaviour under critical handle attachment as a formal consequence of~\cite{Bourgeois-Ekholm-Eliashberg-1}. See~\S\ref{sec:applications}.

(d) We describe in~\S\ref{sec:products} which of the symplectic
homology groups carry product structures, with respect to which
transfer maps are ring homomorphisms as in the classical case of
symplectic homology of a Liouville domain. As a consequence we
construct a degree $-n$ product on Rabinowitz-Floer homology which 
induces a degree $1-n$ coproduct on positive symplectic homology.

(e) We give a uniform treatment of vanishing and finite dimensionality results in~\S\ref{sec:vanishing}.       

(f) We establish in~\S\ref{sec:Mayer-Vietoris} Mayer-Vietoris exact triangles for all flavors $SH_*^\heartsuit$. To the best of our knowledge such exact triangles have not appeared previously in the literature.

A word about our method of proof. We already mentioned the confinement lemmas of~\S\ref{sec:confinement}. There are two other important ingredients in our construction: continuation maps and mapping cones. We now describe their roles. It turns out that the key map of the theory is the transfer map 
$$
i_!:SH_*^\heartsuit(W)\to SH_*^\heartsuit(V)
$$
induced by the inclusion $i:V\hookrightarrow W$ for a pair of Liouville cobordisms $(W,V)$ with filling, see~\S\ref{sec:transfer}. It is instrumental for our constructions to interpret this transfer map as a \emph{continuation map} determined by a suitable increasing homotopy of Hamiltonians. (Compare with the original definition~\cite{Viterbo99} of the transfer map for Liouville domains, where its continuation nature is only implicit and truncation by the action plays the main role.) The next step is to interpret the homological mapping cone of the transfer map as being isomorphic to the group $SH_*^\heartsuit(W,V)$ shifted in degree down by $1$ (Proposition~\ref{prop:SHrel-dynamical-cone}). This is achieved via a systematic use of homological algebra for mapping cones, see~\S\ref{sec:cones}, in which a higher homotopy invariance property of the Floer chain complex plays a key role (Lemma~\ref{lem:continuation_maps}). While it is possible to show directly starting from the definitions that the groups $SH_*(W,V)$, $SH_*(W)$, and $SH_*(V)$ fit into an exact triangle, we did not succeed in proving this directly for the truncated versions $SH_*^\heartsuit$. The situation was unlocked and the arguments were streamlined upon adopting the continuation map and mapping cone point of view. 

We implicitly described the structure of the paper in the body of the Introduction, so we shall not repeat it here. The titles of the sections should now be self-explanatory. We end the Introduction by mentioning two further directions that unfold naturally from the present paper. The first one is to extend symplectic homology, which is a \emph{linear theory} in the sense that its output is valued in graded $R$-modules, possibly endowed with a ring structure, to a \emph{nonlinear theory} at the level of DGAs. This is accomplished for $SH_*^{>0}$ of Liouville domains in~\cite{Ekholm-Oancea}, but the other flavors may admit similar extensions too. The second one is a further categorical extension of the theory to the level of the wrapped Fukaya category, in the spirit of~\cite{Abouzaid-Seidel} where this is again accomplished for Liouville domains. We expect in particular a meaningful theory of wrapped Fukaya categories for cobordisms, with interesting applications.

\noindent {\bf Acknowledgements}. Thanks to Gr\'egory Ginot for inspiring discussions on factorisation homology, to Mihai Damian for having kindly provided the environment that allowed us to overcome one last obstacle in the proof, to St\'ephane Guillermou and Pierre Schapira for help with correcting a sign error in~\S\ref{sec:cones},
{\color{black} and to the referee for valuable suggestions. }
The authors acknowledge the hospitality of the Institute for Advanced
Study, Princeton, NJ in 2012 and of the Simons Center for Geometry and
Physics, Stonybrook, NY in 2014, when part of this work was carried out. 
{\color{black} K.C. was supported by DFG grant CI 45/6-1.}
A.O. was supported by the European Research Council via the Starting Grant StG-259118-STEIN and by the Agence Nationale de la Recherche, France via the project ANR-15-CE40-0007-MICROLOCAL. A.O. acknowledges support from the School of Mathematics at the Institute for Advanced Study in Princeton, NJ during the Spring Semester 2017, funded by the Charles Simonyi Endowment.


\section{Symplectic (co)homology for filled Liouville cobordisms}\label{sec:SH}

Symplectic homology for Liouville domains was introduced by
Floer--Hofer~\cite{FH94,CFH95} and Viterbo~\cite{Viterbo99}. 
In this section we extend their definition to filled Liouville
cobordisms. Since symplectic homology is a well established theory, we
will omit many details of the construction and concentrate on the new
aspects. For background we refer to the excellent
account~\cite{Abouzaid-cotangent}.  

\subsection{Liouville cobordisms} \label{sec:Liouville-cobordisms}

A \emph{Liouville cobordism} $(W,\lambda)$ consists of a compact
manifold with boundary $W$ and a $1$-form $\lambda$ such that
$d\lambda$ is symplectic and $\lambda$ restricts to a contact form on
$\p W$. We refer to $\lambda$ as the \emph{Liouville form}. If the
dimension of $W$ is $2n$ the last condition means that
$\lambda\wedge(d\lambda)^{n-1}$ defines a volume form on $\p W$. We
denote by $\p^+W\subset \p W$ the union of the components for which
the orientation induced by $\lambda\wedge(d\lambda)^{n-1}$ coincides
with the boundary orientation of $W$ and call it the \emph{convex
  boundary} of $W$. We call $\p^-W=\p W\setminus \p^+W$ the
\emph{concave boundary} of $W$.  
The convex/concave boundaries of $W$ are contact manifolds $(\p^\pm
W,\alpha^\pm:=\lambda|_{\p^\pm W})$.\footnote{Unless otherwise stated
  our contact manifolds will be always cooriented and equipped with
  chosen contact forms.}  
We refer to~\cite[Chapter~11]{Cieliebak-Eliashberg-book} for an
exhaustive discussion of Liouville cobordisms and their homotopies. 
A \emph{Liouville domain} is a Liouville cobordism such that $\p W=\p^+W$. 

\begin{example} Given a Riemannian manifold $(N,g)$, its unit
    codisk bundle $D^*_rN:=\{(q,p)\in T^*N\;\bigl|\; \|p\|_g\le r\}$
    is a Liouville domain with the canonical Liouville form
    $\lambda=p\,dq$, whereas $T^*_{r,R}N:=D^*_RN\setminus
    \mbox{int}\, D^*_rN$ for $r<R$ is a Liouville cobordism with
    concave boundary given by $S^*_rN:=\p D^*_rN$.
\end{example}

Define the \emph{Liouville vector field} $Z\in\cX(W)$ by
$\iota_Zd\lambda=\lambda$ and denote by $\alpha^\pm$ the restriction
of $\lambda$ to $\p^\pm W$. It is a consequence of the definitions
that $Z$ is transverse to $\p W$ and points outwards along $\p^+W$,
and inwards along $\p^- W$. The flow $\phi_Z^t$ of the vector field
$Z$ defines Liouville trivialisations of collar neighborhoods
$\NN^\pm$ of $\p^\pm W$ 
$$
\Psi^+:\big((1-\eps,1]\times \p^+ W,r\alpha^+\big) \to (\NN^+,\lambda),
$$
$$
\Psi^-:\big([1,1+\eps)\times \p^- W,r\alpha^-\big) \to (\NN^-,\lambda),
$$
via the map 
$$
(r,x)\mapsto\varphi_Z^{\ln r}(x).
$$

Given a contact manifold $(M,\alpha)$, its
\emph{symplectization} is given by $(0,\infty)\times M$ with the
Liouville form $r\alpha$. We call $(0,1]\times M$ and
$[1,\infty)\times M$ (both equipped with the form $r\alpha$) the
  \emph{negative}, respectively \emph{positive part of the symplectization}.  

Given a Liouville cobordism $(W,\lambda)$, we define its
\emph{completion} by  
$$
\wh W = ((0,1]\times \p^- W) \sqcup _{\Psi^-} W \ _{\Psi^+}\!\!\sqcup \ [1,\infty)\times \p^+W,
$$
with the obvious Liouville form still denoted by $\lambda$.

Given a contact manifold $(M,\alpha)$ we define a \emph{
  (Liouville) filling} to be a Liouville domain $(F,\lambda)$ together
with a diffeomorphism $\varphi:\p F\to M$ such that $\varphi^*\alpha=
\lambda|_{\p F}$. 

We view a Liouville cobordism $(W,\omega,\lambda)$ as a morphism from the
concave boundary to the convex boundary, $W:(\p^- W,\alpha^-)\to (\p
^+W,\alpha^+)$.  
We view a Liouville domain $W$ as a cobordism from $\varnothing$ to its convex boundary. 
Given two Liouville cobordisms $W$ and $W'$ together with an
identification $\varphi:(\p^-W,\alpha^-)\stackrel \cong\to
(\p^+W',\alpha'^+)$, we define their \emph{composition} by 
$$
W\circ W' = W \sqcup_{\varphi:\p^-W\stackrel\cong\to \p^+W'} W'.
$$
The gluing is understood to be compatible with the trivialisations
$\Psi^-$ and $\Psi'^+$, so that the Liouville forms glue smoothly.

\subsection{Filtered Floer homology}

A contact manifold $(M,\alpha)$ carries a canonical \emph{Reeb
  vector field} $R_\alpha\in\cX(M)$ defined by the conditions
$i_{R_\alpha}d\alpha=0$ and $\alpha(R_\alpha)=1$. We refer to the
closed integral curves of $R_\alpha$ as \emph{closed Reeb orbits}, or
just {\em Reeb orbits}. We
denote by $\mbox{Spec}(M,\alpha)$ the set of periods of closed Reeb
orbits. This is the critical value set of the action functional given
by integrating the contact form on closed loops, and a version of
Sard's theorem shows that $\mbox{Spec}(M,\alpha)$ is a closed
nowhere dense subset of $[0,\infty)$. If $M$ is compact the set
$\mbox{Spec}(M,\alpha)$ is bounded away from $0$ since the Reeb
vector field is nonvanishing.  

Consider the symplectization $((0,\infty)\times M,r\alpha)$ and let
$h:(0,\infty)\times M\to \R$ be a function that depends only on the
radial coordinate, i.e. $h(r,x)=h(r)$. Its Hamiltonian vector field,
defined by $d(r\alpha)(X_h,\cdot)=-dh$, is given by
$$
X_h(r,x)=h'(r)R_\alpha(x).
$$
The $1$-periodic orbits of $X_h$ on the level $\{r\}\times M$ are
therefore in one-to-one correspondence with the closed Reeb orbits
with period $h'(r)$. Here we understand that a Reeb orbit of negative
period is parameterized by $-R_\alpha$, whereas a $0$-periodic Reeb
orbit is by convention a constant.  

Let $(W,\lambda)$ be a Liouville domain and $\wh W$ its completion. 
We define the class 
$$
\cH(\wh W)
$$ 
of \emph{admissible Hamiltonians on $\wh W$}
to consist of functions $H:S^1\times\wh W\to \R$ such that in the
complement of some compact set $K\supset W$ we have $H(r,x)=ar+c$ with
$a,c\in\R$ and $a\notin \pm\mbox{Spec}(\p W,\alpha)\cup\{0\}$. In
particular, $H$ has no $1$-periodic orbits outside the compact set
$K$.  

An almost complex structure $J$ on the symplectization
$((0,\infty)\times M,r\alpha)$ is called \emph{cylindrical } if it
preserves $\xi=\ker\alpha$, if $J|_\xi$ is independent of $r$ and compatible
with $d(r\alpha)|_\xi$, and if $J(r\p_r)=R_\alpha$. Such almost
complex structures are compatible with $d(r\alpha)$ and are
invariant with respect to dilations $(r,x)\mapsto (cr,x)$, $c>0$. In
the definition of Floer homology for admissible Hamiltonians on
$\wh W$ we shall use almost complex structures which are
cylindrical outside some compact set that contains $W$, which we
call \emph{admissible almost complex structures on $\wh W$}.

Consider an admissible Hamiltonian $H$ and an admissible almost
complex structure $J$ on the completion $\wh W$ of a Liouville domain
$W$. To define the filtered Floer homology we use the same notation
and sign conventions as in~\cite{Cieliebak-Frauenfelder-Oancea}, which
match those of \cite{Ci02,Abouzaid-Seidel,Ekholm-Oancea}:  
$$
d\lambda(\cdot,J\cdot)=g_J \qquad\text{(Riemannian metric)},
$$
$$
d\lambda(X_H,\cdot)=-dH,\qquad X_H=J\nabla H \qquad\text{(Hamiltonian vector field)},
$$
$$
\cL\wh W:=C^\infty(S^1,\wh W), \qquad S^1=\R/\Z \qquad \text{(loop space)},
$$
$$
A_H:\cL\wh W\to \R,\qquad A_H(x):=\int_{S^1}x^*\lambda -
\int_{S^1}H(t,x(t))\,dt \qquad\text{(action)},
$$
$$
\nabla A_H(x)=-J(x)(\dot x-X_H(t,x)) \qquad\text{($L^2$-gradient)},
$$
$$
u:\R\to\cL W,\qquad \p_su=\nabla A_H(u(s,\cdot)) \qquad\text{(gradient
  line)}
$$
\begin{equation}\label{eq:Floer}
\Longleftrightarrow \p_s u + J(u)(\p_t u-X_H(t,u))=0
\qquad\text{(Floer equation)},  
\end{equation}
$$
\cP(H):=\mbox{Crit}(A_H) = \{\text{$1$-periodic orbits of the Hamiltonian vector field $X_H$}\} ,
$$
$$
\hspace{.5cm} \mathcal M(x_-,x_+;H,J)=\{u:\R\times S^1\to W \mid
\p_su= \nabla A_H(u(s,\cdot)),\ u(\pm\infty,\cdot)=x_\pm\}/\R 
$$
$$
\mbox{(moduli space of Floer trajectories connecting $x_\pm\in\cP(H)$)},
$$
$$
\dim \cM(x_-,x_+;H,J)=CZ(x_+)-CZ(x_-)-1,
$$
$$
   A_H(x_+)-A_H(x_-) = \int_{\R\times S^1}|\p_su|^2ds\,dt = \int_{\R\times S^1}u^*(d\lambda-dH\wedge dt).  
$$
Here the formula expressing the dimension of the moduli space in terms
of Conley-Zehnder indices is to be understood with respect to a
symplectic trivialisation of $u^*TW$.  

Let $\K$ be a {\color{black}field} and $a<b$ with $a,b\notin\mbox{Spec}(\p W,\alpha)$. We define the filtered Floer
chain groups with coefficients in $\K$ by  
$$
FC_*^{<b}(H) = \bigoplus_{\scriptsize \begin{array}{c} x\in \cP(H)
    \\ A_H(x)<b \end{array}} \K\cdot x,\qquad
FC_*^{(a,b)}(H) = FC_*^{<b}(H)/FC_*^{<a}(H),
$$
with the differential $\p:FC_*^{(a,b)}(H)\to FC_{*-1}^{(a,b)}(H)$
given by
$$
\p x_+=\sum_{\tiny CZ(x_-)=CZ(x_+)-1} \#\mathcal M(x_-,x_+;H,J)\cdot x_-.
$$
Here $\#$ denotes the signed count of points with respect to suitable
orientations. 
We think of the cylinder $\R\times S^1$ as the twice punctured Riemann
sphere, with the positive puncture at $+\infty$ as incoming, and the
negative puncture at $-\infty$ as outgoing. This terminology makes
reference to the corresponding asymptote being an input, respectively
an output for the Floer differential. Note that the differential
decreases both the action $A_H$ and the Conley-Zenhder index. 
The filtered Floer homology is now defined as 
$$
   FH_*^{(a,b)}(H) = \ker\p/\im\p. 
$$
Note that for $a<b<c$ the short exact sequence
$$
   0 \to FC_*^{(a,b)}(H) \to FC_*^{(a,c)}(H) \to FC_*^{(b,c)}(H) \to 0
$$
induces a {\em tautological exact triangle} 
\begin{equation}\label{eq:taut1}
   FH_*^{(a,b)}(H) \to FH_*^{(a,c)}(H) \to FH_*^{(b,c)}(H)
   \to FH_*^{(a,b)}(H)[-1].
\end{equation}

{\color{black}\noindent {\bf Remark.} We will suppress the field $\K$ from the notation. As noted in the Introduction, the definition can also be given with coefficients in a commutative ring, and more generally with coefficients in a local system as
in~\cite{Ritter09,Abouzaid-cotangent}. 
}

\subsection{Restrictions on Floer trajectories} \label{sec:confinement}

We shall frequently make use of the following three lemmas to exclude
certain types of Floer trajectories.  
The first one is an immediate consequence of Lemma~7.2
in~\cite{Abouzaid-Seidel}, see also~\cite[Lemma~19.3]{Ritter}. 
{\color{black} Since our setup differs slightly from the one there,
we include the proof for completeness.}

\begin{lemma}[no escape lemma] \label{lem:no-escape}
Let $H$ be an admissible Hamiltonian on a completed Liouville domain
$(\wh W,\om,\lambda)$. Let $V\subset\wh W$ be a compact subset with
smooth boundary $\p V$ such that $\lambda|_{\p V}$ is a positive contact form, 
{\color{black} $J$ is cylindrical near $\p V$, and $H=h(r)$ in cylindrical
coordinates $(r,x)$ near $\p V=\{r=1\}$.} 
If both asymptotes of a Floer cylinder $u:\R\times
S^1\to\wh W$ are contained in $V$, then $u$ is entirely contained in $V$. 

The result continues to hold if $H_s$ depends on the coordinate
$s\in\R$ on the cylinder $\R\times S^1$ such that $\p_sH_s\leq0$ and
the action $A_{h_s}(r)=rh_s'(r)-h_s(r)$ satisfies $\p_s A_{h_s}(r)\leq
0$ for $r$ near $1$.  
\end{lemma}

{\color{black}
\begin{proof}
Assume first that $H$ is $s$-independent. 
Arguing by contradiction, suppose that $u$ leaves the set $V$. 
After replacing $V$ by the set $\{r\leq r_0\}$ for a constant $r_0>1$
close to $1$, we may assume that $u$ leaves $V$ and is transverse to $\p V$. 
In cylindrical coordinates near $\p V$ we have $X_H=h'(r)R$ and
$\lambda=r\alpha$, where $R$ is the Reeb vector field of
$\alpha=\lambda|_{\p V}$, so the functions $H=h(r)$ and
$\lambda(X_H)=rh'(r)$ are both constant along $\p V$. 
Note that their difference equals the action $A_h(r)$. 

Now $S:=u^{-1}(\wh W\setminus\mbox{Int}\,V)$ is a compact surface with
boundary. We denote by $j$ and $\beta$ the restrictions of the complex
structure and the $1$-form $dt$ from the cylinder $\R\times S^1$ to $S$,
so that on $S$ the Floer equation for $u$ can be written as
$\bigl(du-X_H(u)\otimes\beta\bigr)^{0,1}=0$. We estimate the energy of $u|_S$: 
\begin{eqnarray*}
   E(u|_S) & = & \frac12\int_S|du-X_H\otimes\beta|^2\mbox{vol}_S \\
   &=& \int_S(u^*d\lambda-u^*dH\wedge\beta) \\
   &=& \int_S d\bigl(u^*\lambda-(u^*H)\beta\bigr) + \int_S (u^*H)d\beta \\
   &=& \int_{\p S} \bigl(u^*\lambda-(u^*H)\beta\bigr) \\
   &=& \int_{\p S} \lambda\bigl(du-X_H(u)\otimes\beta\bigr) \\
   &=& \int_{\p S} \lambda\Bigl(J\circ\bigl(du-X_H(u)\otimes\beta\bigr)\circ(-j)\Bigr) \\
   &=& \int_{\p S} dr\circ du\circ(-j) \\
   & \leq & 0.
\end{eqnarray*}
Here the \textcolor{black}{equality} in the 4-th line follows from Stokes' theorem and
$d\beta\equiv 0$. The equality in the 5-th line holds because the
$r$-component of $u|_{\p S}$ equals $r_0$ and thus
$$
   \int_{\p S} u^*\bigl(\lambda(X_H)-H\bigr)\beta 
   = \int_{\p S}A_h(r_0)\beta 
   = \int_{S}A_h(r_0)d\beta = 0. 
$$
The equality in the 6-th line follows from the Floer equation,
and the equality in the 7-th line from $\lambda\circ J=dr$ and $dr(X_H)=0$
along $\p V$. The last inequality follows from the fact that for each tangent
vector $\xi$ to $\p S$ defining its boundary orientation, $j\xi$
points into $S$, so $du(j\xi)$ points out of $V$ and $dr\circ
du(j\xi)\geq 0$. Since $E(u|_S)$ is nonnegative, it follows that
$E(u|_S)=0$, and therefore $du-X_H(u)\otimes\beta\equiv 0$. So each
connected component of $u|_S$ is contained in an $X_H$-orbit, and
since $X_H$ is tangent to $\p V$, $u(S)$ is entirely contained in $\p
V$. This contradicts the hypothesis that $u$ leaves $V$ and the lemma
is proved for $s$-independent $H$.  

If $H_s$ is $s$-dependent we get an additional term
$\int_S (u^*\p_sH_s)ds\wedge dt\leq 0$ in the third line, \textcolor{black}{so the equality in the 4-th
line becomes an inequality $\leq$. The equality in the 5-th line also becomes an
inequality $\leq$} due to the nonpositive additional term in
$$
   \int_{\p S}A_{h_s}(r_0)\beta 
   = \int_{S}A_{h_s}(r_0)d\beta + \int_S\p_sA_{h_s}(r_0)ds\wedge dt \leq 0. 
$$
This proves the lemma for $s$-dependent $H_s$.
\end{proof}

{\bf Remark. }
The proof shows that Lemma~\ref{lem:no-escape} continues to hold if the cylinder
$\R\times S^1$ is replaced by a general Riemann surface $S$ with a
$1$-form $\beta$ satisfying $H\,d\beta\leq 0$ and $A_h(r)d\beta\leq 0$
for all $r$ near $1$. In this case we can allow $H$ to depend on $s$
in holomorphic coordinates $s+it$ on a region $U\subset S$ in which
$\beta=c\,dt$ for a constant $c\geq 0$, with the requirements 
$\p_sH_s\leq0$ and $\p_s A_{h_s}(r)\leq 0$ as before.  
This generalization underlies the definition of product
structures in Section~\ref{sec:products}. 
}

The second lemma summarises an argument that has appeared first
in~\cite[pages~654-655]{BOcont}. {\color{black} Since the conventions
in~\cite{BOcont} differ from ours, we include the short
proof for completeness.}

\begin{lemma}[asymptotic behaviour lemma] \label{lem:asy}
Let $(\R_+\times M,r\alpha)$ be the symplectization of a contact
manifold $(M,\alpha)$. Let $H=h(r)$ be a Hamiltonian depending only on
the radial coordinate $r\in\R_+$, and let $J$ be a cylindrical
almost complex structure. Let $u=(a,f):{\color{black}\R_\pm}\times S^1\to
\R_+\times M$ be a solution of the Floer equation~\eqref{eq:Floer} 
with $\lim_{s\to\pm\infty} u(s,\cdot)=(r_\pm,\gamma_\pm(\cdot))$ for
{\color{black} suitably parameterized} Reeb orbits $\gamma_\pm$. 

(i) Assume $h''(r_-)>0$. Then either there exists
$(s_0,t_0)\in\R\times S^1$ such that $a(s_0,t_0)>r_-$, or
$u$ is constant equal to $(r_-,\gamma_-)$. 

(ii) Assume $h''(r_+)<0$. Then either there exists
$(s_0,t_0)\in\R\times S^1$ such that $a(s_0,t_0)>r_+$, or
$u$ is constant equal to $(r_+,\gamma_+)$. 
\end{lemma}

{\color{black}\begin{proof}
In coordinates $(s,t)\in\R_\pm\times S^1$, the Floer equation for
$u=(a,f)$ with Hamiltonian $H=h(r)$ writes out as 
\begin{equation}\label{eq:Floer-symp}
   \p_sa-\alpha(\p_tf)+h'(a)=0,\quad 
   \p_ta+\alpha(\p_sf)=0, \quad
   \pi_\xi\p_sf+J(f)\pi_\xi\p_tf=0,
\end{equation}
where $\pi_\xi:TM\to\xi=\ker\alpha$ is the projection along the Reeb
vector field $R$. In case (i), suppose $h''(r_-)>0$ and $a(s,t)\leq r_-$ for
all $(s,t)\in\R_-\times S^1$. After replacing $\R_-\times S^1$ by a
smaller half-cylinder we may assume that $h''(a(s,t))\geq 0$ for all
$(s,t)\in\R_-\times S^1$. Then the average $\ol a(s):=\int_0^1a(s,t)dt$ satisfies
\begin{align*}
   \ol a'(s) &= \int_0^1\p_sa(s,t)dt \cr
   &= \int_0^1\alpha(\p_sf)(s,t)dt - \int_0^1h'\bigl(a(s,t)\bigr)dt
   \cr   
   &\geq \int_0^1f^*\alpha(s) - \int_0^1h'(r_-)dt \cr
   &\geq \int_{\gamma_-}\alpha - h'(r_-) = h'(r_-) - h'(r_-) = 0.
\end{align*}
Here the second equality follows from the first equation
in~\eqref{eq:Floer-symp}, the first inequality from $a(s,t)\leq r_-$
and $h''(a(s,t))\geq 0$, and the second inequality from Stokes' theorem
and $f^*d\alpha\geq 0$. For the third equality observe that
$x_-(t)=\bigl(r_-,\gamma_-(t)\bigr)$ is a $1$-periodic orbit of
$X_H=h'(r)R$ iff $\dot\gamma_-=h'(r_-)R$, so that $\int_{\gamma_-}\alpha=h'(r_-)$.  

Now $\ol a'(s)\geq 0$ and $\ol a(-\infty)=r_-$ imply that $\ol
a(s)\geq r_-$ for all $s$, which is compatible with $a(s,t)\leq r_-$
only if $a(s,t)=r_-$ for all $(s,t)$. Then all of the preceding
inequalities are equalities, in particular $f^*d\alpha\equiv 0$, and
therefore $u(s,t)=\bigl(r_-,\gamma_t(t)$ for all $(s,t)$. This proves
case (i). Case (ii) follows from case (i) by replacing $h$ by $-h$ and
$u(s,t)$ by $u(-s,-t)$. 
\end{proof}}

Lemma~\ref{lem:asy} can be rephrased by saying that nonconstant Floer trajectories
must rise above their output asymptote if the Hamiltonian is convex at
the asymptote, and they must rise above their input asymptote if the
Hamiltonian is concave at the asymptote. Combined with
Lemma~\ref{lem:no-escape}, it forbids Floer trajectories of the kind
shown in Figure~\ref{fig:asy}. 

\begin{figure} [ht]
\centering
\input{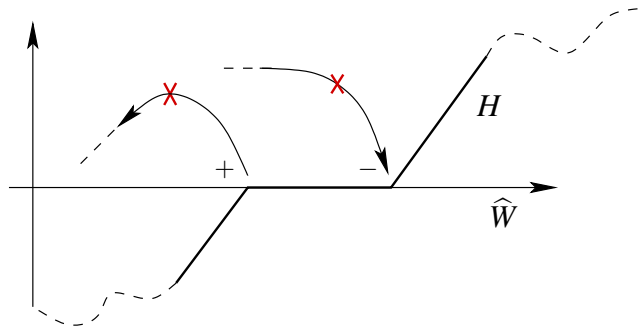}
\caption{Such Floer trajectories are forbidden by Lemma~\ref{lem:asy} in combination with Lemma~\ref{lem:no-escape}.}
\label{fig:asy}
\end{figure}

The third lemma follows from a neck stretching argument using the
compactness theorem in symplectic field theory (SFT). We refer to
Figure~\ref{fig:stretch} for a sketch of a situation in which a
certain kind of Floer trajectory is forbidden by this technique.  

\begin{lemma}[neck stretching lemma]\label{lem:neck} 
Let $H$ be an admissible Hamiltonian on a completed Liouville domain
$(\wh W,\lambda)$. Let $V\subset\wh W$ be a compact subset with
smooth boundary $\p V$ such that $H\equiv c$ near $\p V$
and $\lambda|_{\p V}$ is a positive contact form. Let $J_R$ be the
compatible almost complex structure on $\wh W$ obtained from $J$ by
inserting a cylinder of length $2R$ around $\p V$. Then for
sufficiently large $R$ there exists no $J_R$-Floer cylinder
$u:\R\times S^1\to\wh W$ with asymptotic orbits $x_\pm$ at
$\pm\infty$ such that 
\begin{enumerate}
\item $x_-\subset \mbox{int}V$ and $x_+\subset\wh W\setminus V$ with $A_H(x_+)<-c$, or
\item $x_+\subset V$ and $x_-\subset\wh W\setminus V$ with $A_H(x_-)>-c$. 
\end{enumerate}
\end{lemma}

\begin{figure} [ht]
\centering
\input{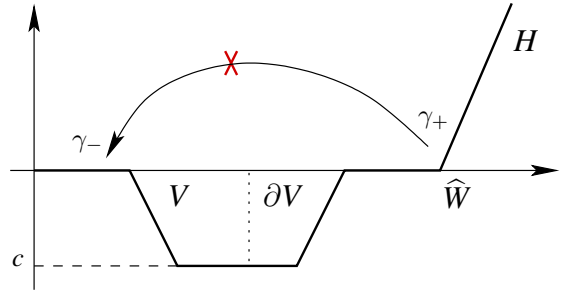}
\caption{Such Floer trajectories are forbidden if $-c>A_H(x_+)$.}
\label{fig:stretch}
\end{figure}

\begin{proof}
Let us first describe more precisely the neck stretching along $M=\p
V$. Pick a tubular neighborhood $[-\eps,\eps]\times M$ of $M$ in $\wh
W$ on which $H\equiv c$ and $\lambda=e^\rho\alpha$, where
$\alpha=\lambda|_M$ and $\rho$ denotes the coordinate on $\R$. Let $J$
be a compatible almost complex structure on $\wh W$ whose restriction
$J_0$ to $[-\eps,\eps]\times M$ is independent of $\rho$ and maps
$\xi=\ker\alpha$ to $\xi$ and $\p_\rho$ to $R_\alpha$. Let $\phi_R$ be
any diffeomorphism $[-R,R]\to[-\eps,\eps]$ with derivative $1$ near
the boundary. Then we define $J_R$ on $\wh W$ by
$(\phi_R\times\id)_*J_0$ on $[-\eps,\eps]\times M$, and by $J$ outside
$[-\eps,\eps]\times M$.  

Consider a $J_R$-Floer cylinder $u:\R\times S^1\to\wh W$ with
asymptotic orbits $x_\pm$. Its Floer energy is given by  
$$
   A_H(x_+)-A_H(x_-) = \int_{\R\times S^1}|\p_su|^2ds\,dt = \int_{\R\times S^1}u^*(d\lambda-dH\wedge dt).  
$$
Set $\Sigma=u^{-1}([-\eps,\eps]\times M)$ and write the restriction of $u$ to $\Sigma$ as 
$$
   u|_\Sigma=(\phi_R\circ a,f),\qquad (a,f):\Sigma\to[-R,R]\times M. 
$$
Let $\psi:[-R,R]\to[e^{-\eps},e^\eps]$ be any nondecreasing function which equals $e^{\phi_R}$ 
on the boundary. Using non-negativity of the integrand in the Floer energy, vanishing of $dH$ on $[-\eps,\eps]\times M$, and Stokes' theorem, we obtain
\begin{align*}
   A_H(x_+)-A_H(x_-) 
   &\geq \int_{\Sigma}u^*(d\lambda-dH\wedge dt)
   = \int_{\Sigma}u^*d\lambda\cr
   &= \int_{\Sigma}(a,f)^*d(e^{\phi_R}\alpha)
   = \int_{\Sigma}(a,f)^*d(\psi\alpha)\cr
   &= \int_{\Sigma}\Bigl(\psi'(a)da\wedge f^*\alpha+\psi(a)f^*d\alpha\Bigr).
\end{align*} 
Since $(a,f)$ is $J_0$-holomorphic, $da\wedge f^*\alpha$ and
$f^*d\alpha$ are nonnegative $2$-forms on $\Sigma$. Since
$\psi'(a)\geq 0$ and $\psi(a)\geq e^{-\eps}$,
and $\psi$ was arbitrary with the given boundary conditions, this
yields a uniform bound (independent of $R$) on the Hofer energy of
$(a,f)$ (see~\cite{BEHWZ,CM}).  

Now suppose that there exists a sequence $R_k\to\infty$ and
$J_{R_k}$-Floer cylinders $u_k:\R\times S^1\to\wh W$ with asymptotic
orbits $x_\pm$ lying on different sides of $M$. By the SFT
compactness theorem~\cite{BEHWZ,CM}, $u_k$ converges in the limit to a
broken cylinder consisting of components in the completions of $V$ and
$\wh W\setminus V$ satisfying the Floer equation and $J_0$-holomorphic
components in $\R\times M$, glued along closed Reeb orbits in
$M$. Since $x_\pm$ lie on different sides of $M$, the punctures
asymptotic to $x_\pm$ lie on different components. Hence for
large $k$ there exists a separating embedded loop
$\delta_k\subset\R\times S^1$ such that $u_k\circ\delta_k$ is
$C^1$-close to a (positively parameterized) closed Reeb orbit $\gamma$
on $M$ (which we view as a loop in $\wh W$ lying on $\p V$). Here
$\delta_k$ is parameterized as a positive boundary of the 
component of $\R\times S^1$ that is mapped to $\wh V$. Now we
distinguish two cases.  

Case (i): $x_-\subset V$ and $x_+\subset\wh W\setminus V$. Then
$\delta_k$ winds around the cylinder in the positive $S^1$-direction,
and since the Hamiltonian action increases along Floer cylinders we
conclude 
$$
   A_H(x_+)\geq A_H(\gamma)\geq A_H(x_-).  
$$ 
Since $\int_\gamma\lambda=\int_\gamma\alpha\geq 0$, we obtain 
$A_H(\gamma)=\int_\gamma\lambda-\int_0^1c\,dt\geq -c$ 
and hence $A_H(x_+)\geq -c$. 

Case (ii): $x_+\subset V$ and $x_-\subset\wh W\setminus V$. Then
$\delta_k$ winds around the cylinder in the negative $S^1$-direction,
and since the Hamiltonian action increases along Floer cylinders we
conclude 
$$
   A_H(x_+)\geq A_H(-\gamma)\geq A_H(x_-).  
$$ 
Since $\int_\gamma\lambda=\int_\gamma\alpha\geq 0$, we obtain 
$A_H(-\gamma)=-\int_\gamma\lambda-\int_0^1c\,dt\leq -c$
and hence $A_H(x_-)\leq -c$. 
\end{proof}

{\color{black}Our fourth lemma prohibits certain trajectories asymptotic to constant
Hamiltonian orbits.  
We consider the setup consisting of a completed Liouville domain $\wh W$, a cobordism $V\subset W$ such that $(W,V)$ is a Liouville pair, i.e. $W=W^{bottom}\circ V\circ W^{top}$, and a Hamiltonian $H:\wh W\to\R$ which is constant on $V$, which depends only on the radial coordinate $r$ in an open neighborhood of $\p V$, and which is either strictly convex or strictly concave as a function of $r$ outside $V$ in each component of the given neighborhood of $\p V$. Denote by $c$ the value of $H$ on $V$. 

Let $f:V\to\R$ be a Morse function which depends only on the radial coordinate $r$ in some neighborhood of $\p V$ and such that $\p^\pm V$ are regular level sets. 
We require the gradient of $f$ to point inside/outside $V$ along $\p^-V$ if $H$ is concave/convex near $\p^-V$, and 
to point inside/outside $V$ along $\p^+V$ if $H$ is concave/convex near $\p^+V$.

Given $\epsilon>0$ we denote by $V^\epsilon = ([1-\epsilon,1]\times \p^- V)\cup V \cup ([1,1+\epsilon]\times\p^+ V)$ an $\epsilon$-thickening of $V$ inside $\wh W$. For $\epsilon>0$ small enough let
$$
   H_{f,\epsilon}:S^1\times\wh W\to\R
$$ 
be a smooth Hamiltonian which is equal to $c+\epsilon^2 f$ on $V$, which is equal to $H$ outside $V^\epsilon$, and which smoothly interpolates between $H$ and $c+\epsilon^2 f$ on $[1-\epsilon,1]\times\p^- V$ and $[1,1+\epsilon]\times\p^+ V$ as a function of $r$ which is either concave or convex, according to $H$ being concave or convex on each of these regions. 

We consider admissible almost complex structures on $\wh W$ which are time-independent on $V$, cylindrical near $\p V$, and such that the gradient flow of $f$ is Morse-Smale.

\begin{lemma} \label{lem:constant} Let $f:V\to\R$ be a Morse function and $H_{f,\epsilon}$ a Hamiltonian as above. For $\epsilon>0$ small enough the following hold:

(1) If the gradient of $f$ points inside $V$ along $\p^- V$, then there is no Floer trajectory for $H_{f,\epsilon}$ which is asymptotic at the positive end to a constant orbit given by a critical point of $f$ and which is asymptotic at the negative end to an orbit in $W^{bottom}$. 

(2) If the gradient of $f$ points outside $V$ along $\p^- V$, then there is no Floer trajectory for $H_{f,\epsilon}$ which is asymptotic at the negative end to a constant orbit given by a critical point of $f$ and which is asymptotic at the positive end to an orbit in $W^{bottom}$. 
\end{lemma}

\begin{proof}

To prove (1) we argue by contradiction and assume without loss of generality that there is a sequence of positive real numbers $\epsilon_n\to 0$ and a sequence of Floer trajectories $u_n:\R\times S^1\to \wh W$ solving $\p_s u_n + J_t(u_n)(\p_t u_n -X_{H_{f,\epsilon_n}}(u_n))=0$ such that $\lim_{s\to\infty} u_n(s,t)=p_+$, $\lim_{s\to-\infty} u_n(s,t)=x_-(t)$, with $p_+$ a critical point of $f$, $x_-:S^1\to\wh W$ a $1$-periodic orbit of $H$ inside $W^{bottom}$, and $J=(J_t)$ an admissible almost complex structure which is time-independent on $V$ and such that the flow of the gradient of $f$ for the corresponding Riemannian metric is Morse-Smale.

We interpret $V$ as a Morse-Bott critical manifold with boundary for the action functional $A_H$, and we view $H_{f,\epsilon_n}$, $n\ge 1$ as determining a sequence of Morse perturbations of $A_H$ along $V$. The Morse-Bott compactness theorem proved in a more restricted Hamiltonian setting in~\cite[Proposition~4.7]{BOauto}, and in a
general SFT setting in~\cite{BEHWZ,CM}, applies to our situation. Indeed, the fact that the Morse-Bott manifold $V$ has boundary plays no role and the proof of~\cite[Proposition~4.7]{BOauto} carries over \emph{mutatis mutandis}. 

It follows that, up to extracting a subsequence, the sequence $u_n$ converges in the terminology of~\cite[Definition~4.2]{BOauto} to a 
broken Floer trajectory $\mathbf{[u]}$ with gradient fragments. The critical manifold $V$ may be disconnected, but all its components are located on the same action level $A_H=-c$. Since Floer trajectories for $H$ strictly increase the action from the asymptote at the negative puncture to the asymptote at the positive puncture, we infer that each level of the limit $\mathbf{[u]}$ contains at most one gradient trajectory of $f$. Moreover, $\mathbf{[u]}$ has a representative $\mathbf{\bar u}=(\mathbf{u}_1,\dots,\mathbf{u}_\ell)$ described as follows: there exists $1\le i\le \ell$ such that 
\begin{itemize}
\item $\mathbf{u}_1,\dots,\mathbf{u}_{i-1}$ are Floer trajectories for $H$, with $\mathbf{u}_1(-\infty)=x_-$, $\mathbf{u}_j(+\infty)=\mathbf{u}_{j+1}(-\infty)$ for $1\le j\le i-2$. 
\item $\mathbf{u}_i$ is a Floer trajectory with one gradient fragment, i.e. $\mathbf{u}_i=(u_i,\gamma_i)$ with $u_i$ a Floer trajectory for $H$ and $\gamma_i:[0,+\infty)\to V$ a negative gradient trajectory for $f$, i.e. solving $\dot\gamma_i=-\nabla f (\gamma_i)$, subject to the following conditions: $\mathbf{u}_{i-1}(+\infty)=\mathbf{u}_i(-\infty)$ if $i>1$ and $\mathbf{u}_i(-\infty)=x_-$ if $i=1$; $u_i(+\infty)=\gamma_i(0)\in V$; and $\gamma_i(+\infty)=p_+$ if $i=\ell$.
\item $\mathbf{u}_{i+1},\dots,\mathbf{u}_\ell$ are negative gradient trajectories $\mathbf{u}_j=\gamma_j:\R\to V$ for $f$, i.e. solving $\dot\gamma_j=-\nabla f(\gamma_j)$, $j=i+1,\dots,\ell$, subject to the conditions $\gamma_j(-\infty)=\gamma_{j-1}(+\infty)$ for $j=i+1,\dots,\ell$, and $\gamma_\ell(+\infty)=p_+$.
\end{itemize}

We now focus on the level $\mathbf{u}_i=(u_i,\gamma_i)$. Three situations can arise: 

Case~1: $\gamma(0)\in V\setminus \p V$.  Then the Floer trajectory $u_i$ solves the Cauchy-Riemann equation $\p_s u + J(u) \p_t u=0$ on some half-cylinder $[s_0,+\infty)\times S^1$ for $s_0\gg 0$. We identify biholomorphically $[s_0,+\infty)\times S^1$ with a punctured disc $\dot D$ and, by assumption, $u:\dot D\to V$ admits a continuous extension at the puncture. Thus $0\in D$ is a removable singularity and we can view $u_i:\R\times S^1\to \wh W$ as being defined on a Riemann sphere with a single negative puncture, on which it solves a Floer equation. The asymptote at the negative puncture is located in $W^{bottom}$ by assumption, and the image of $u_i$ intersects $\p^- V$. Then Lemma~\ref{lem:no-escape} gives a contradiction. 

Case~2: $\gamma(0)\in \p^+ V$. Pick $\delta>0$ such that $[1-\delta,1]\times\p^+V$ does not contain critical points of $f$. Since $\mathbf{[u]}$ is the limit of the sequence $u_n$, there exists $n_0\ge 1$ such that the image of $u_n$ intersects the set $(1-\delta,1]\times\p^+ V$. By assumption both asymptotes of $u_n$ are located in $W^{bottom}\cup V \setminus ([1-\delta,1]\times\p^+ V)$, and Lemma~\ref{lem:no-escape} again gives a contradiction. 

Case~3: $\gamma(0)\in \p^- V$. The map $\gamma_i:[0,\infty)\to V$ solves $\dot\gamma_i=-\nabla f(\gamma_i)$ and enters $V$ in positive time, but at the same time $-\nabla f$ points outwards along $\p V$, which is a contradiction. 

The proof of (2) is entirely analogous: cases 1 and 2 are treated exactly in the same way, while case 3 is proved similarly to (1) using that negative gradient trajectories of a Morse function on $V$ whose gradient points outwards along $\p V$ must exit $V$ in negative time. 
\end{proof}

\begin{remark} \label{rmk:constant} The conclusions of Lemma~\ref{lem:constant} most likely do not hold if one exchanges ``positive" and ``negative" in either of the statements (1) or (2). Although we do not have an explicit example involving Floer trajectories, i.e. twice punctured spheres, we can easily give an example involving pairs of pants. Consider to this effect a Liouville domain $W$ and the trivial cobordism $V=[\frac 1 2,1]\times\p W$ over the boundary. As discussed in~\S\ref{sec:products}, the symplectic homology group $SH_*^{\le 0}(V)=SH_*^{\le 0}(\p W)$ is a unital graded commutative ring, and the unit maps to $1\in H^{n-*}(\p W)$ under the 
projection $SH_*^{\le 0}(V)\to SH_*^{=0}(V)\simeq H^{n-*}(\p W)$. Assume now that the map $SH_*^{<0}(V)\to SH_*^{\le 0}(V)$ is nontrivial -- which holds for example in the case of unit cotangent bundles of closed manifolds -- and consider a class $\alpha\neq 0$ in its image. Since $1\cdot \alpha=\alpha\neq 0$ we infer the existence of at least one solution to a Floer equation defined on a pair of pants with two positive punctures and one negative puncture, asymptotic at one of the positive punctures to a constant orbit inside $V$, and asymptotic at the two other punctures to orbits located in $W^{bottom}=W\setminus V$.
\end{remark}
}

\subsection{Symplectic homology of a filled Liouville cobordism} \label{sec:SHWA}
Let $(W,\lambda)$ be a Liouville cobordism and $(F,\lambda)$ a
Liouville filling of $(\p^-W,\alpha^-=\lambda_{\p^-W})$. We compose $F$ and $W$
to the Liouville domain  
$$
   W_F := F\circ W
$$
and denote its completion by $\wh W_F$. We define the class
$$
\cH(W;F)
$$
of {\em admissible Hamiltonians on $\wh W_F$ with respect to the
  filling $F$} to consist of functions $H:S^1\times \wh W_F\to \R$ such
that $H\in\cH(\wh W_F)$ and $H= 0$ on $W$. When there is no danger of
confusion we shall use the notation  
$$
\cH(W)
$$ 
for the set $\cH(W;F)$ and refer to its elements as \emph{admissible Hamiltonians on $W$}. 

\begin{remark}
For the purposes of this section it would have been enough to
  define admissible Hamiltonians by the condition $H\le 0$ on
  $W$. This would have allowed for cofinal families consisting of
  Hamiltonians with nondegenerate $1$-periodic orbits. The definition
  that we have adopted requires to use small perturbations in order to
  define Floer homology and is slightly cumbersome in that
  respect. However, it will prove very convenient when we come to the
  definition of symplectic homology groups for pairs.
\end{remark}

Next we consider continuation maps. Let $H_-\ge H_+$ be admissible
Hamiltonians and $H_s$, $s\in\R$ be a decreasing homotopy through
admissible Hamiltonians such that $H_s=H_\pm$ near $\pm\infty$. Let
$J_s$ be a homotopy of admissible almost complex structures. Solutions
of the Floer equation $\p_s u+J_s(u)(\p_t u-X_{H_s}(u))=0$ satisfy a
maximum principle in the region where all the Hamiltonians $H_s$ are
linear and all the almost complex structures are cylindrical, and
their count defines continuation maps $FH_*(H_+)\to FH_*(H_-)$. Since
the homotopy is decreasing, the action increases along solutions of
the preceding $s$-dependent Floer equation, so it decreases under
the continuation map. We infer from this the existence of filtered
continuation maps $FH_*^{(-\infty,b)}(H_+)\to
FH_*^{(-\infty,b)}(H_-)$, $b\in\R$, and more generally the existence
of filtered continuation maps 
$$
FH_*^{(a,b)}(H_+)\to FH_*^{(a,b)}(H_-), \qquad a<b.
$$
For an admissible Hamiltonian $H$ we also have natural morphisms
determined by inclusions of and quotients by appropriate subcomplexes 
$$
FH_*^{(a,b)}(H)\to FH_*^{(a',b')}(H),\qquad a\le a',\ b\le b'.
$$
These morphisms commute with the continuation morphisms, and we obtain
more general versions of the latter  
$$
FH_*^{(a,b)}(H_+)\to FH_*^{(a',b')}(H_-), \qquad a\le a',\ b\le b'.
$$
Given real numbers $-\infty<a<b<\infty$, we define the
\emph{filtered symplectic homology groups of $W$ (with respect
  to the filling $F$)} to be  
\begin{equation}\label{eq:SH*abW}
SH_*^{(a,b)}(W)=\lim^{\longrightarrow}_{H\in\cH(W;F)} FH_*^{(a,b)}(H).
\end{equation}
The direct limit is taken here with respect to continuation maps and
with respect to the partial order $\prec$ on $\cH(W;F)$ defined as
follows: $H\prec K$ if and only if $H(t,x)\le K(t,x)$ for all $(t,x)$. Note that in a cofinal
family the Hamiltonian necessarily goes to $+\infty$ on $F\cup
([1,\infty)\times \p^+W)$. Recall also that, in order to achieve
    nondegeneracy of the $1$-periodic orbits, the Hamiltonian $H$
    needs to be perturbed on $W$ where it is constant equal to
    zero. Our convention is that we compute the direct limit using a
    cofinal family for which the size of the perturbation goes to
    zero.  

Taking the direct limit in~\eqref{eq:taut1} we obtain for $a<b<c$ the tautological
exact triangle
\begin{equation}\label{eq:taut2}
   SH_*^{(a,b)}(W) \to SH_*^{(a,c)}(W) \to SH_*^{(b,c)}(W)
   \to SH_*^{(a,b)}(W)[-1].
\end{equation}

\begin{definition} \label{defi:SH(W)}
We define six versions of \emph{symplectic homology groups of $W$
  (with respect to the filling $F$)}:
$$
SH_*(W)=\lim^{\longrightarrow}_{b\to\infty}\lim^{\longleftarrow}_{a\to -\infty} SH_*^{(a,b)}(W) \qquad \mbox{\sc (full symplectic homology)}
$$
$$
SH_*^{>0}(W)= \lim^{\longrightarrow}_{b\to\infty} \lim^{\longleftarrow}_{a\searrow 0} SH_*^{(a,b)}(W) \qquad \mbox{\sc (positive symplectic homology)}
$$
$$
SH_*^{\ge 0}(W)= \lim^{\longrightarrow}_{b\to\infty} \lim^{\longrightarrow}_{a\nearrow 0} SH_*^{(a,b)}(W) \qquad \mbox{\sc (non-negative symplectic homology)}
$$
$$
SH_*^{= 0}(W)= \lim^{\longleftarrow}_{b\searrow 0} \lim^{\longrightarrow}_{a\nearrow 0} SH_*^{(a,b)}(W) \qquad \mbox{\sc (zero-level symplectic homology)}
$$
$$
SH_*^{\le 0}(W)= \lim^{\longleftarrow}_{b\searrow 0} \lim^{\longleftarrow}_{a\to-\infty} SH_*^{(a,b)}(W) \qquad \mbox{\sc (non-positive symplectic homology)}
$$
$$
SH_*^{< 0}(W)= \lim^{\longrightarrow}_{b\nearrow 0} \lim^{\longleftarrow}_{a\to-\infty} SH_*^{(a,b)}(W) \qquad \mbox{\sc (negative symplectic homology)}
$$
\end{definition}

Since the actions of Reeb orbits are bounded away from
zero, the direct/inverse limits as $a$ (or $b$) goes to zero stabilize
for $a$ (respectively $b$) sufficiently close to zero, so they are not
actual limits. Note that the actual inverse limits as $a\to-\infty$ in these
definitions are always applied to finite dimensional vector spaces
when considering field coefficients. 
This ensures that the inverse and direct limits preserve
exactness of sequences; see~\cite{CF} for further discussion of the
order of limits, and also~\cite[Chapter~8]{ES52} for a discussion of
exactness.  

The geometric content of the definition is the following. Let $H$ be a
Hamiltonian as depicted in Figure~\ref{fig:H-heuristic}, which is
constant and very positive on $F\setminus([\delta,1]\times \p F)$ with
$0<\delta<1$, which is linear of negative slope with respect to the
$r$-coordinate on $[\delta,1]\times\p F$, which vanishes on $W$, and
which is linear of positive slope with respect to the $r$-coordinate
on $[1,\infty[\times \p^+W$. The $1$-periodic orbits of $H$ fall in
    four classes, denoted $F$ (orbits in the filling),
    ${I^-}$ (orbits that correspond to negatively parameterized
    closed Reeb orbits on $\p^-W$), ${I^0}$ (constant orbits in
    $W$), and ${I^+}$ (orbits that correspond to positively
    parameterized closed Reeb orbits on $\p^+W$). As $\delta\to 0$ and
    as the absolute values of the slopes go to $\infty$, Hamiltonians
    of this type form a cofinal family in $\mathcal{H}(W;F)$. The
    action of orbits in the class ${F}$ becomes very negative
    and falls outside any fixed and finite action window $(a,b)$, so
    that the homology groups $SH_*^{(a,b)}(W)$ take into account only
    orbits of type ${I^{-0+}}$. Each flavour of symplectic
    homology group $SH_*^{\heartsuit}(W)$,
    $\heartsuit\in\{\varnothing, >0,\ge 0, =0, \le 0, <0\}$, with
    $SH_*^\varnothing(W)$ as a notation for $SH_*(W)$, respectively
    takes into account orbits in the class ${I^{-0+}}$,
    ${I^+}$, ${I^{0+}}$, ${I^{0}}$,
    ${I^{-0}}$, ${I^-}$ for arbitrarily large values of
    the slope. As such, each of these symplectic homology groups
    corresponds to a certain count of negatively parameterized closed
    Reeb orbits on $\p^-W$, of constant orbits in $W$, and of
    positively parameterized closed Reeb orbits on $\p^+W$.  

\begin{figure} [ht]
\centering
\input{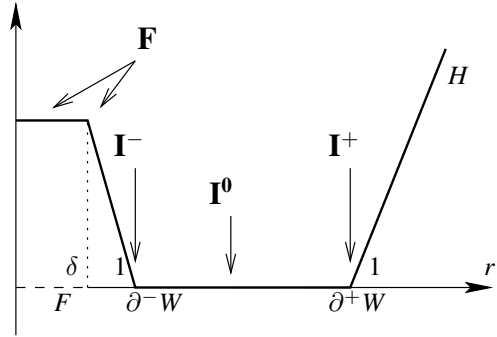}
\caption{Cofinal family of Hamiltonians for $SH_*^{\heartsuit}(W)$}
\label{fig:H-heuristic}
\end{figure}

The next proposition will be proved as
Proposition~\ref{prop:homotopy_invariance_transfer} below. 

\begin{proposition} \label{prop:invariance-SHW}
Each of the above six versions of symplectic homology is an invariant
of the Liouville homotopy type of the pair $(W;F)$.  
\end{proposition}

The following computation is fundamental in applications.

\begin{proposition} \label{prop:SH=0W}
Let $\dim\, W=2n$. Then we have a canonical isomorphism 
$$
SH_*^{=0}(W)\cong H^{n-*}(W).
$$
\end{proposition}

\begin{proof}
{\color{black} Consider a Hamiltonian $K$ of the shape as in Figure~\ref{fig:H-heuristic}. 
Since $\wh W_F$ is symplectically aspherical, it follows
from~\cite[Theorem~7.3]{SZ92} (see
also~\cite[Proposition~1.4]{Viterbo99}) that if $K$ is sufficiently
$C^2$-small on $W$, then its Floer chain complex reduces to the Morse
cochain complex for an appropriate choice of almost complex structure. 
Fix such a $K$ and denote by $c>0$ its constant value on the filling
$F$. Pick $\eps$ with $0<\eps<c$, so that the constant orbits in $F$
have action $-c<-\eps$. Since the Conley-Zehnder index of a critical
point is related to its Morse index by $\mbox{CZ}=n-\mbox{Morse}$, we
get a canonical isomorphism $FH_*^{(-\eps,\eps)}(K)\cong H^{n-*}(W)$. 

Consider any other Hamiltonian $H$ of the shape as in Figure~\ref{fig:H-heuristic}
with $K\leq H$. We choose $\eps$ smaller than the smallest action of a
closed Reeb orbit on $\p W$. Then all nonconstant orbits of $H$ have
action outside $(-\eps,\eps)$ and a monotone homotopy from $K$ to $H$
yields a continuation isomorphism $FH_*^{(-\eps,\eps)}(K)\stackrel{\cong}\to
FH_*^{(-\eps,\eps)}(H)$, which induces in the direct limit over $H$ a
canonical isomorphism $FH_*^{(-\eps,\eps)}(K)\stackrel{\cong}\to
SH_*^{(-\eps,\eps)}(W) = SH_*^{=0}(W)$. }
\end{proof}

\begin{remark} 
If $W$ is a Liouville domain we have 
$$
SH_*^{<0}(W)=0,\qquad SH_*^{\le 0}(W)=SH_*^{=0}(W),\qquad SH_*^{\ge 0}(W)=SH_*(W),
$$ 
and the group $SH_*^{>0}(W)$ coincides by definition with the group $SH_*^+(W)$ of~\cite{BOcont}.
If $W$ is a Liouville cobordism with Liouville filling $F$ we have (by
a standard continuation argument)
$$
SH_*^{>0}(W)\cong SH_*^{>0}(W_F).
$$
\end{remark}

\begin{proposition} \label{prop:taut-triang-W}
The following ``tautological" exact triangles hold for the symplectic homology groups of $W$:
\begin{equation*} 
\xymatrix
@C=10pt
@R=18pt
{
SH_*^{<0} \ar[rr] & & 
SH_* \ar[dl] \\ & SH_*^{\ge 0} \ar[ul]^{[-1]}  
}
\qquad 
\xymatrix
@C=10pt
@R=18pt
{
SH_*^{\le 0} \ar[rr] & & 
SH_* \ar[dl] \\ & SH_*^{> 0} \ar[ul]^{[-1]}  
}
\end{equation*}
\begin{equation*}
\xymatrix
@C=10pt
@R=18pt
{
SH_*^{< 0} \ar[rr] & & 
SH_*^{\le 0} \ar[dl] \\ & SH_*^{= 0} \ar[ul]^{[-1]}  
}
\qquad
\xymatrix
@C=10pt
@R=18pt
{
SH_*^{= 0} \ar[rr] & & 
SH_*^{\ge 0} \ar[dl] \\ & SH_*^{> 0} \ar[ul]^{[-1]}  
}
\end{equation*}
\end{proposition}

\begin{proof}
We prove the exactness of the triangle 
\begin{equation}\label{eq:first-ex-tr}
SH_*^{\le 0}(W)\to SH_*(W)\to SH_*^{>0}(W) \to SH_*^{\le 0}(W)[-1]\,.
\end{equation}
The proofs for the other three triangles are similar and left to the reader. 

Let $\varepsilon>0$ be smaller than the minimal period of a closed
characteristic on $\p^+W$. It follows from the definitions that  
$$
SH_*^{\le 0}(W) = \lim^{\longleftarrow}_{a\to-\infty} SH_*^{(a,\varepsilon)}(W)
$$
and 
$$
SH_*^{>0}(W) = \lim^{\longrightarrow}_{b\to\infty} SH_*^{(\varepsilon,b)}(W).
$$
For fixed $a,b\in\mathbb{R}$ such that
$-\infty<a<0<\varepsilon<b<\infty$ we have from~\eqref{eq:taut2} an exact triangle  
$$
SH_*^{(a,\varepsilon)}(W)\to SH_*^{(a,b)}(W)\to SH_*^{(\varepsilon,b)}(W)\to SH_*^{(a,\varepsilon)}(W)[-1]\,.
$$
All the terms in this exact triangle are finite dimensional vector spaces. The inverse limit functor is exact on directed systems consisting of finite dimensional vector spaces, and the direct limit functor is always exact. We then obtain~\eqref{eq:first-ex-tr} by first taking  the inverse limit on $a\to-\infty$, and then taking the direct limit on $b\to\infty$. 
\end{proof}

\noindent \emph{Symplectic homology groups relative to boundary components}. 
Let $A\subset \p W$ be a union of boundary components of $W$ and denote 
$$
A^\pm=A\cap \p^\pm W.
$$
We further assume that $A^-$ is a union of boundaries of components of $F$. We refer to such an $A$ as \emph{an admissible subset of $\p W$}.

\noindent {\bf Examples}. 
One obvious choice is $A^-=\p ^-W$, which satisfies the assumption for any $F$. If each component of $F$ has connected boundary then one can take $A^-\subset \p^-W$ arbitrary. If $F$ consists of a single connected component then the only possible choices are $A^-=\p^- W$ or $A^-=\varnothing$. Note also that, if $A$ satisfies the assumption, then $A^c:=\p W\setminus A$ also does.

Let $F_{A^-}$ denote the filling of $(A^-,\alpha^-)$ consisting of the union of the components of $F$ with boundary contained in $A^-$. Denote 
$$
(\wh W_F\setminus W)_A= \mbox{int}\,F_{A^-}\cup ((1,\infty)\times A^+), 
$$
so that 
$$
\wh W_F\setminus W = (\wh W_F\setminus W)_A \sqcup (\wh W_F\setminus W)_{A^c}.
$$
Given real numbers $-\infty<a<b<\infty$, we define the \emph{filtered symplectic homology groups of $W$ relative to $A$ (with respect to the filling $F$)} to be 
\begin{equation} \label{eq:SH*abWA}
SH_*^{(a,b)}(W,A)=
\lim^{\longrightarrow}_{\scriptsize \begin{array}{c} H\in\cH(W;F) \\ H\to\infty \mbox{ on } (\wh W_F\setminus W)_{A^c} \end{array}} 
\lim^{\longleftarrow}_{\scriptsize \begin{array}{c} H\in\cH(W;F) \\ H\to-\infty \mbox{ on } (\wh W_F\setminus W)_A \end{array}} 
FH_*^{(a,b)}(H).
\end{equation}

\begin{definition} \label{defi:SHWA}
We define six flavors of \emph{symplectic homology groups of $W$ relative to $A$, or symplectic homology groups of the pair $(W,A)$}, 
$$
SH_*^\heartsuit(W,A), \qquad \heartsuit\in \{\varnothing,>0,\ge 0, =0, \le 0, <0\},
$$
by the formulas in Definition~\ref{defi:SH(W)} with $SH_*^{(a,b)}(W)$ replaced by $SH_*^{(a,b)}(W,A)$.
The notation $SH_*^\heartsuit$ with $\heartsuit=\varnothing$ refers to $SH_*$.
\end{definition}

We refer to Figure~\ref{fig:1} for an illustration of several significant cases of Hamiltonians used in the computation of relative symplectic homology groups. The case $A=\varnothing$ corresponds to Figure~\ref{fig:H-heuristic}. In each case, in the limit the orbits that appear in the filling either fall below or fall above any fixed and finite action window, so that only orbits appearing near $W$ are taken into account. As an example, $SH_*(W,\p^-W)$ corresponds to a a certain count of positively parameterized closed Reeb orbits on $\p^-W$, of constant orbits in $W$, and of positively parameterized closed Reeb orbits on $\p^+W$. Similar interpretations hold for $SH_*(W,\p^+W)$, $SH_*(W,\p W)$, and also for all their $\heartsuit$-flavors. In Figure~\ref{fig:1} we encircled with a dashed line the region which contains the orbits that are taken into account. The mnemotechnic rule is the following: 
\begin{center}
\emph{To compute $SH_*^{\heartsuit}(W,A)$ one must use a family of Hamiltonians that go to $-\infty$ near $A$ and that go to $+\infty$ near $\p W\setminus A$.}
\end{center} 

\begin{figure} [ht]
\centering
\input{SH-rel.pstex_t}
\caption{Shape of Hamiltonians for $SH_*(W,A)$ with $A=\varnothing,\p W,\p^-W,\p^+W$}
\label{fig:1}
\end{figure}

Our notation is motivated by the following analogue of
Proposition~\ref{prop:SH=0W}, {\color{black}which is proved in the same way.}

\begin{proposition} \label{prop:SH=0WA}
Let $\dim\, W=2n$ and $A\subset \p W$ be admissible. Then we have a canonical isomorphism 
$$
SH_*^{=0}(W,A)\cong H^{n-*}(W,A).
$$
\hfill{$\square$}
\end{proposition}

The tautological exact triangles described in Proposition~\ref{prop:taut-triang-W} also exist for the relative symplectic homology groups $SH_*^{\heartsuit}(W,A)$ (same proof). Also, the relative symplectic homology groups $SH_*^\heartsuit(W,A)$ are invariants of the Liouville homotopy type of the pair $(W,F)$ (see~\S\ref{sec:ex-triangle-pair}, compare Propositions~\ref{prop:invariance-SHW} and~\ref{prop:invariance-SHWV}).

\subsection{Symplectic homology groups of a pair of filled Liouville cobordisms}\label{sec:SHpair}

A \emph{Liouville pair, or pair of Liouville cobordisms}, is a triple
$(W,V,\lambda)$ where $(W,\lambda)$ is a Liouville
cobordism and $V\subset W$ is a codimension $0$ submanifold with
boundary such that  
\renewcommand{\theenumi}{\roman{enumi}}
\begin{enumerate}
\item $(V,\lambda|_V)$ is a Liouville cobordism;
\item $\overline{W\setminus V}$ is a disjoint union of two (possibly
  empty) Liouville cobordisms $W^{bottom}$ and $W^{top}$ such that  
$$
W=W^{bottom}\circ V \circ W^{top}.
$$
\end{enumerate}
We fix a filling $F$ of $W$ and define $W_F$, $\wh W_F$ as above. 
We define the class
$$
\cH(W,V;F)
$$
of \emph{admissible Hamiltonians on $(W,V)$ with respect to the filling $F$} to consist of elements $H:S^1\times \wh W_F\to \R$ such that $H\in\cH(\wh W_F)$ and $H= 0$ on $W\setminus V$ (see Figure~\ref{fig:H-cob}).
Given real numbers $-\infty<a<b<\infty$, we define \emph{the action-filtered symplectic homology groups of $(W,V)$ (with respect to the filling $F$)} to be 
\begin{equation} \label{eq:SH*abWV}
SH_*^{(a,b)}(W,V)=
\lim^{\longrightarrow}_{\scriptsize \begin{array}{c} H\in\cH(W,V;F) \\ H\to\infty \mbox{ on } (\wh W_F\setminus W) \end{array}} 
\lim^{\longleftarrow}_{\scriptsize \begin{array}{c} H\in\cH(W,V;F) \\ H\to-\infty \mbox{ on } \mbox{int}\,V \end{array}} 
FH_*^{(a,b)}(H).
\end{equation}

\begin{definition} \label{defi:SHWV}
We define six flavors of \emph{symplectic homology groups of the
  Liouville pair $(W,V)$}, 
$$
SH_*^\heartsuit(W,V), \qquad \heartsuit\in \{\varnothing,>0,\ge 0, =0, \le 0, <0\},
$$
by the formulas in Definition~\ref{defi:SH(W)} with $SH_*^{(a,b)}(W)$ replaced by $SH_*^{(a,b)}(W,V)$.
The notation $SH_*^\heartsuit$ with $\heartsuit=\varnothing$ refers to $SH_*$.
\end{definition}

To describe the geometric content of the definition we consider a
cofinal family of Hamiltonians $H$ of the shape described in
Figure~\ref{fig:H-cob}. Heuristically, each of the groups
$SH_*^\heartsuit(W,V)$ represents a certain count of negatively
parameterized closed Reeb orbits on $\p^-W$ and $\p^-V$, of constant
orbits in $\overline{W\setminus V}$, and of positively parameterized
closed Reeb orbits on $\p^+V$ and $\p^+W$, which correspond to
generators of type ${I^{-0+}}$ and ${III^{-0+}}$ in
Figure~\ref{fig:H-cob}. However, unlike in the case of (relative)
symplectic homology groups for a single cobordism, it is not possible
to arrange the parameters of the Hamiltonians in the cofinal family so
that for a fixed and finite value of the action window $(a,b)$ the
group $FH_*^{(a,b)}(H)$ takes into account only orbits of types
${I^{-0+}}$ and ${III^{-0+}}$. Instead, we will use in
\S\ref{sec:excision} below an indirect argument relying on the 
confinement lemmas in~\S\ref{sec:confinement} and on the
properties of continuation maps in order to prove an isomorphism
between $SH_*^\heartsuit(W,V)$ and $SH_*^\heartsuit(W^{bottom},\p^-V)\oplus
SH_*^\heartsuit(W^{top},\p^+V)$ (Theorem~\ref{thm:excision}). There
we will also see (Corollary~\ref{cor:tub-nbhd}) that
Definition~\ref{defi:SHWA} is a special case of
Definition~\ref{defi:SHWV} by taking for $V$ a tubular neighbourhood
of a union of boundary components $A$. 

The following three results generalize the corresponding ones for a single cobordism. 

\begin{proposition} \label{prop:invariance-SHWV}
Each of the above six versions of symplectic homology is an invariant
of the Liouville homotopy type of the triple $(W,V,F)$.  
\end{proposition}

\begin{proof}
See Proposition~\ref{prop:invariance-pair-WV} below.
\end{proof}

\begin{proposition} \label{prop:SH=0WV}
Let $\dim\, W=2n$. Then we have a canonical isomorphism 
$$
SH_*^{=0}(W,V)\cong H^{n-*}(W,V).
$$
\end{proposition}

{\color{black}\begin{proof}
The proof of Proposition~\ref{prop:SH=0W} does not carry over to this
situation because Hamiltonians as in Figure~\ref{fig:H-cob} may have
nonconstant orbits of action zero of type $II^-$. Instead, we combine
the Excision Theorem~\ref{thm:excision} with
Proposition~\ref{prop:SH=0WA} and excision in singular cohomology to
obtain canonical isomorphisms
\begin{align*}
   SH_*^{=0}(W,V)
   &\cong SH_*^{=0}(W^{bottom},\p^-V) \oplus SH_*^{=0}(W^{top},\p^+V) \cr
   &\cong H^{n-*}(W^{bottom},\p^-V) \oplus H^{n-*}(W^{top},\p^+V) \cr
   &\cong H^{n-*}(W,V).
\end{align*}
\end{proof}}

The proof of the following proposition is verbatim the same as the one
of Proposition~\ref{prop:taut-triang-W}.  
{\color{black} Recall to this effect that we are using field coefficients, and note that $SH_*^{(a,b)}(W,V)$ is finite dimensional for any choice of parameters $-\infty<a<b<\infty$. This holds because in the nondegenerate case there are only a finite number of closed Reeb orbits on $\p(\overline{W\setminus V})$ with action smaller than $\max(|a|,|b|)$, and only these orbits contribute to the relevant Floer complex for the cofinal family of Hamiltonians described in~\S\ref{sec:excision}. } 

\begin{proposition}  \label{prop:taut-triang-WV}
The following tautological exact triangles hold for the symplectic homology groups of a pair $(W,V)$:
\begin{equation*} 
\xymatrix
@C=10pt
@R=18pt
{
SH_*^{<0} \ar[rr] & & 
SH_* \ar[dl] \\ & SH_*^{\ge 0} \ar[ul]^{[-1]}  
}
\qquad 
\xymatrix
@C=10pt
@R=18pt
{
SH_*^{\le 0} \ar[rr] & & 
SH_* \ar[dl] \\ & SH_*^{> 0} \ar[ul]^{[-1]}  
}
\end{equation*}
\begin{equation*}
\xymatrix
@C=10pt
@R=18pt
{
SH_*^{< 0} \ar[rr] & & 
SH_*^{\le 0} \ar[dl] \\ & SH_*^{= 0} \ar[ul]^{[-1]}  
}
\qquad
\xymatrix
@C=10pt
@R=18pt
{
SH_*^{= 0} \ar[rr] & & 
SH_*^{\ge 0} \ar[dl] \\ & SH_*^{> 0} \ar[ul]^{[-1]}  
}
\end{equation*}
\hfill{$\square$}
\end{proposition}

{\color{black}
\subsection{Pairs of multilevel Liouville cobordisms with filling} \label{sec:multilevel} 
As mentioned in the Introduction, according to our conventions for pairs of Liouville cobordisms the symplectic homology group $SH_*(W,\p W)$ cannot be interpreted as $SH_*(W,[0,1]\times \p W)$ in case $\p W$ has both negative and positive components. We explain in this section a further extension of the setup which removes this limitation.

Let $\ell\ge 0$ be an integer. A \emph{Liouville cobordism with $\ell$ levels} is, in case $\ell\ge 1$, a disjoint union $W=W_1\sqcup W_2\sqcup\dots\sqcup W_\ell$ of Liouville cobordisms, called \emph{levels}, and is the empty set if $\ell=0$. We think of $W_1$ as being the ``bottom-most" level, and of $W_\ell$ as being the ``top-most" level. Each $W_i$ may itself be disconnected. Our previous definition of Liouville cobordisms corresponds to the case $\ell=1$. We also refer to such a $W$ as being a \emph{multilevel Liouville cobordism}. 

Let $V$ and $W$ be two Liouville cobordisms with $\ell$ levels. We say that $V$ and $W$ \emph{can be interweaved} if $\p^+V_i=\p^-W_i$ for $i=1,\dots,\ell$ and $\p^+W_i=\p^-V_{i+1}$ for $i=1,\dots,\ell-1$. The \emph{interweaving of $V$ and $W$}, denoted $V\diamond W$, is the Liouville cobordism with one level $V_1\circ W_1\circ \dots\circ V_\ell\circ W_\ell$. We allow in the definition the bottom-most or the top-most level of $V$ or $W$ to be empty, and in that case the condition for interweaving $V$ and $W$ which involves that level has to be understood as being void. In the case of cobordisms with one level, interweaving specialises to composition. See Figure~\ref{fig:interweave}.

\begin{figure}
         \begin{center}
\input{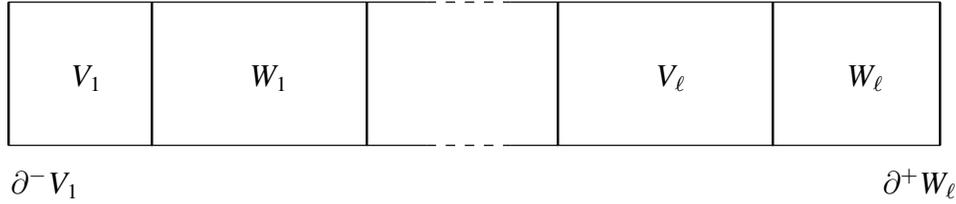}
         \end{center}
\caption{Interweaving of two multilevel cobordisms  \label{fig:interweave}}
\end{figure}

Given a Liouville cobordism $W$ with $\ell\ge 1$ levels, a \emph{Liouville filling} for $W$ is a Liouville cobordism with $\ell$ levels $F=F_1\sqcup\dots\sqcup F_\ell$ such that $F_1$ is a nonempty Liouville domain and $F$ and $W$ can be interweaved. In the case $\ell=1$, this notion specialises to our previous notion of a Liouville filling. 

Given a Liouville cobordism $W$ with one level, a \emph{Liouville sub-cobordism $V\subset W$} is a codimension $0$ submanifold such that with respect to the induced Liouville form $V$ and $V^c=\overline{W\setminus V}$ are multilevel Liouville cobordisms that can be interweaved. If $V$ has only one level then $(W,V)$ is a Liouville pair in the sense of~\S\ref{sec:SHpair}. 

Given a multilevel Liouville cobordism $W$, a \emph{Liouville sub-cobordism $V\subset W$} consists of a collection of (possibly empty) multilevel Liouville sub-cobordisms, one for each of the levels of $W$. We speak in such a situation of a \emph{pair of multilevel Liouville cobordisms}. In case $W$ has a filling, we speak of a \emph{pair of multilevel Liouville cobordisms with filling}. 

Let $(W,V)$ be a pair of multilevel Liouville cobordisms with filling $F$. Denote $W_F=F\diamond W$ and consider the symplectization $\wh W_F$. We define the class
$$
\cH(W,V;F)
$$
of \emph{admissible Hamiltonians on $(W,V)$ with respect to the filling $F$} to consist of elements $H:S^1\times \wh W_F\to \R$ such that $H\in\cH(\wh W_F)$ and $H= 0$ on $W\setminus V$ (see Figure~\ref{fig:H-cob-multi}).
Given real numbers $-\infty<a<b<\infty$, we define \emph{the action-filtered symplectic homology groups of $(W,V)$ (with respect to the filling $F$)} to be 
\begin{equation} \label{eq:SH*abWV-multilevel}
SH_*^{(a,b)}(W,V)=
\lim^{\longrightarrow}_{\scriptsize \begin{array}{c} H\in\cH(W,V;F) \\ H\to\infty \mbox{ on } (\wh W_F\setminus W) \end{array}} 
\lim^{\longleftarrow}_{\scriptsize \begin{array}{c} H\in\cH(W,V;F) \\ H\to-\infty \mbox{ on } \mbox{int}\,V \end{array}} 
FH_*^{(a,b)}(H).
\end{equation}

\begin{definition} \label{defi:SHWV-multilevel}
We define six flavors of \emph{symplectic homology groups of the
  multilevel Liouville pair $(W,V)$}, 
$$
SH_*^\heartsuit(W,V), \qquad \heartsuit\in \{\varnothing,>0,\ge 0, =0, \le 0, <0\},
$$
by the formulas in Definition~\ref{defi:SH(W)} with $SH_*^{(a,b)}(W)$ replaced by $SH_*^{(a,b)}(W,V)$.
The notation $SH_*^\heartsuit$ with $\heartsuit=\varnothing$ refers to $SH_*$.
\end{definition}

The above definition obviously specialises to Definition~\ref{defi:SHWV} in case $W$ is a filled Liouville cobordism with one level.

\begin{figure}
         \begin{center}
\input{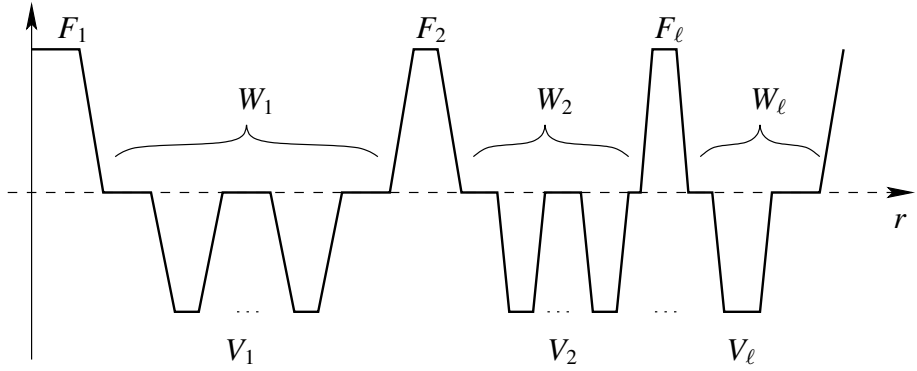}
         \end{center}
\caption{Hamiltonian in $\mathcal{H}(W,V;F)$ for a multilevel cobordism  \label{fig:H-cob-multi}}
\end{figure}

Within the paper we state and prove all the results for pairs of one level Liouville cobordisms with filling. However, all these results hold more generally for pairs $(W,V)$ of multilevel Liouville cobordisms with filling. The formulation of these more general statements is verbatim the same. The proofs are only superficially more involved: a repeated application of the Excision Theorem~\ref{thm:excision} allows one to restrict to the case where $W$ is a one level cobordism with filling, and the case of a multilevel sub-cobordism $V$ is treated in exactly the same way as that of a one level sub-cobordism. For these reasons, we will not give in the sequel any more details regarding multilevel Liouville cobordisms and will restrict to one level pairs. 

}

\section{Cohomology and duality} \label{sec:coh-and-duality}

\subsection{Symplectic cohomology for a pair of filled Liouville cobordisms} 

We continue with the notation of the previous section. Our definition of symplectic cohomology for a pair of filled Liouville cobordisms 
extends the one for Liouville domains
used in~\cite[\S2.5]{Cieliebak-Frauenfelder-Oancea}.  

The starting point of the definition is the dualization of the Floer
chain complex with coefficient {\color{black}field $\K$}. We denote 
$$
FC^*_{>a}(H) =\prod_{\scriptsize \begin{array}{c} x\in\cP(H)\\ 
    A_H(x)>a\end{array}} \K\cdot x.
$$
The grading is given by the Conley-Zehnder index, and the differential 
$\delta : FC^k_{>a}(H)\to FC^{k+1}_{>a}(H)$ is defined by  
$$
\delta x_- = \sum_{CZ(x_+)=CZ(x_-)+1} \#\cM(x_-,x_+;H,J)\cdot x_+.
$$
The differential increases the action, so that $FC^*_{>b}(H)\subset FC^*_{>a}(H)$ is a subcomplex if $a<b$. We define \emph{filtered Floer cochain groups}
$$
FC^*_{(a,b)}(H)=FC^*_{>a}(H)/FC^*_{>b}(H).
$$
We have a natural identification
$$
FC^*_{(a,b)}(H)\cong FC_*^{(a,b)}(H)^\vee,\qquad \delta = \p^\vee,
$$
where $FC_*^{(a,b)}(H)^\vee = \Hom_R(FC_*^{(a,b)}(H),R)$. 

We have natural morphisms at filtered cochain level defined by shifting the action window
$$
FC^*_{(a',b')}(H)\to FC^*_{(a,b)}(H),\qquad a\le a',\ b\le b'. 
$$
These morphisms are dual to the ones defined on Floer chain groups. Also, given admissible Hamiltonians $H_-\ge H_+$ and a decreasing homotopy from $H_-$ to $H_+$, we have filtered continuation maps which commute with the differentials 
$$
FC^*_{(a,b)}(H_-)\to FC^*_{(a,b)}(H_+).
$$
These continuation maps are dual to the ones defined on Floer chain groups, and commute with the morphisms defined by shifting the action window. The homotopy type of the continuation maps does not depend on the choice of decreasing homotopy with fixed endpoints. 

Let $W$ be a Liouville cobordism with filling $F$, and let $A\subset
\p W$ be an admissible union of boundary components as
in~\S\ref{sec:SHWA}. Recall also the notation $A^c=\p W\setminus A$ 
and $(\wh W_F\setminus W)_A=\mbox{int}\,F_{A^-}\cup ((1,\infty)\times A^+)$, and
  recall also the class $\cH(W;F)$ of admissible Hamiltonians
  from~\S\ref{sec:SHWA}. Let $-\infty<a<b<\infty$ be real numbers. We
  define the \emph{filtered symplectic cohomology groups of $W$
    relative to $A$ (with respect to the filling $F$)} to be 
\begin{equation} \label{eq:cohSH*abWA}
SH^*_{(a,b)}(W,A)= \lim^{\longrightarrow}_{\scriptsize \begin{array}{c} H\in\cH(W;F) \\ H\to-\infty \mbox{ on } (\wh W_F\setminus W)_{A} \end{array}} 
\lim^{\longleftarrow}_{\scriptsize \begin{array}{c} H\in\cH(W;F) \\ H\to\infty \mbox{ on } (\wh W_F\setminus W)_{A^c} \end{array}} 
FH^*_{(a,b)}(H).
\end{equation}
The mnemotechnic rule is the same as in the case of symplectic homology:
\begin{center}
\emph{To compute $SH^*_{(a,b)}(W,A)$ one must use a family of Hamiltonians that go to $-\infty$ near $A$ and that go to $+\infty$ near $\p W\setminus A$.}
\end{center} 

\begin{definition} \label{defi:cohWA}
We define six flavors of \emph{symplectic cohomology groups of $W$ relative to $A$}, or \emph{symplectic cohomology groups of the pair $(W,A)$}, 
$$
SH^*_\heartsuit(W,A),\qquad \heartsuit\in\{\varnothing,\ge 0, >0,=0,\le 0,<0\},
$$
by the following formulas (the notation $SH^*_\varnothing$ refers to $SH^*)$:
$$
SH^*(W,A)=\lim^{\longrightarrow}_{a\to-\infty}\lim^{\longleftarrow}_{b\to \infty} SH^*_{(a,b)}(W,A) \qquad \mbox{\sc (full symplectic cohomology)}
$$
$$
SH^*_{<0}(W,A)= \lim^{\longrightarrow}_{a\to-\infty} \lim^{\longleftarrow}_{b\nearrow 0} SH^*_{(a,b)}(W,A) \qquad \mbox{\sc (negative symplectic cohomology)}
$$
$$
SH^*_{\le 0}(W,A)= \lim^{\longrightarrow}_{a\to-\infty} \lim^{\longrightarrow}_{b\searrow 0} SH^*_{(a,b)}(W,A) \quad \! \mbox{\sc (non-positive symplectic cohomology)}
$$
$$
SH^*_{= 0}(W,A)= \lim^{\longleftarrow}_{a\nearrow 0} \lim^{\longrightarrow}_{b\searrow 0} SH^*_{(a,b)}(W,A) \qquad \mbox{\sc (zero-level symplectic cohomology)}
$$
$$
SH^*_{\ge 0}(W,A)= \lim^{\longleftarrow}_{a\nearrow 0} \lim^{\longleftarrow}_{b\to\infty} SH^*_{(a,b)}(W,A) \quad \mbox{\sc (non-negative symplectic cohomology)}
$$
$$
SH^*_{> 0}(W,A)= \lim^{\longrightarrow}_{a\searrow 0} \lim^{\longleftarrow}_{b\to\infty} SH^*_{(a,b)}(W,A) \qquad \mbox{\sc (positive symplectic cohomology)}
$$
\end{definition}

Let now $(W,V)$ be a pair of Liouville cobordisms with filling $F$ as in~\S\ref{sec:SHpair}, and recall the class $\cH(W,V;F)$ of admissible Hamiltonians for the pair $(W,V)$ with respect to the filling $F$. Let $-\infty<a<b<\infty$ be real numbers. We define the \emph{filtered symplectic cohomology groups of $(W,V)$ (with respect to the filling $F$)} to be
\begin{equation} \label{eq:cohSH*abWV}
SH^*_{(a,b)}(W,V)=
\lim^{\longrightarrow}_{\scriptsize \begin{array}{c} H\in\cH(W,V;F) \\ H\to-\infty \mbox{ on } V \end{array}} 
\lim^{\longleftarrow}_{\scriptsize \begin{array}{c} H\in\cH(W,V;F) \\ H\to\infty \mbox{ on } (\wh W_F\setminus W) \end{array}} 
FH^*_{(a,b)}(H).
\end{equation}

\begin{definition} \label{defi:cohWV}
We define six flavors of \emph{symplectic cohomology groups of the
  Liouville pair $(W,V)$}, 
$$
SH^*_\heartsuit(W,V), \qquad \heartsuit\in \{\varnothing,>0,\ge 0, =0, \le 0, <0\},
$$
by the formulas in Definition~\ref{defi:cohWA} with $SH^*_{(a,b)}(W,A)$ replaced by $SH^*_{(a,b)}(W,V)$. The notation $SH^*_\varnothing$ refers to $SH^*$.
\end{definition}

The discussion from~\S\ref{sec:SHpair} regarding the geometric content of the definition holds for cohomology as well. The following proposition is proved similarly to Proposition~\ref{prop:SH=0WV}. 

\begin{proposition} \label{prop:cohSH=0WA}
Let $(W,V)$ be a pair of Liouville cobordisms with filling of
dimension $2n$. Then we have a canonical isomorphism 
$$
SH^*_{=0}(W,V)\cong H_{n-*}(W,V).
$$
\hfill{$\square$}
\end{proposition} 

\subsection{Poincar\'e duality}\label{sec:Poincare_duality} 

The differences and the similarities between symplectic homology and symplectic cohomology are
mainly dictated by the order in which we consider direct and inverse
limits. We illustrate this by the following theorem, which was one of
our guidelines for the definitions.

\begin{theorem}[Poincar\'e duality]\label{thm:Poincare} 
Let $W$ be a filled Liouville cobordism and $A\subset \p W$ be an
admissible union of connected components. Then we have a canonical isomorphism
$$
SH_*^\heartsuit(W,A)\cong SH^{-*}_{-\heartsuit}(W,A^c).
$$
Here the symbol $\heartsuit$ takes the values $\varnothing,>0,\ge 0, =0, \le 0, <0$, and $-\heartsuit$ is by convention equal to $\varnothing,< 0,\le 0,=0,\ge 0,> 0$, respectively.
\end{theorem}

\begin{proof} Given a time-dependent $1$-periodic Hamiltonian $H:S^1\times\wh W\to\R$ we denote $\bar H:S^1\times\wh W\to \R$, $\bar H(t,x)=-H(-t,x)$. Given a time-dependent $1$-periodic family of almost complex structures $J=(J_t)_{t\in S^1}$ on $\wh W$, we denote $\bar J=(\bar J_t)$, $t\in S^1$ with $\bar J_t=J_{-t}$. Given a loop $x:S^1\to \wh W$, we denote $\bar x:S^1\to \wh W$, $\bar x(t)=x(-t)$. Given a cylinder $u:\R\times S^1\to \wh W$, we denote $\bar u:\R\times S^1\to \wh W$, $\bar u(s,t)=u(-s,-t)$. 

The key to the proof of Poincar\'e duality for symplectic homology is the canonical isomorphism, which will be also referred to as Poincar\'e duality, 
\begin{equation} \label{eq:PDchainlevel}
FC_*^{(a,b)}(H,J)\cong FC^{-*}_{(-b,-a)}(\bar H,\bar J),
\end{equation} 
obtained by mapping each $1$-periodic orbit
$x$ of $H$ to the $1$-periodic orbit of $\bar H$ given by the oppositely parameterized loop $\bar x$, and
each Floer cylinder $u$ for $(H,J)$ to the cylinder $\bar u$, which is a Floer cylinder for $(\bar H,\bar J)$. Note that the positive and negative punctures
get interchanged when passing from $u$ to $\bar u$, so that a chain complex is transformed into a cochain complex. It is straightforward that $A_{\bar H}(\bar x)=-A_H(x)$. It is less straightforward, but true, that $CZ(\bar x)=-CZ(x)$. The proof follows from~\cite[Lemma~2.3]{Cieliebak-Frauenfelder-Oancea}, taking into account that the flows of $\bar H$ and $H$ satisfy the relation $\varphi^t_{\bar H}=\varphi^{-t}_H$. We refer to~\cite[Proposition~2.2]{Cieliebak-Frauenfelder-Oancea} for a discussion of this Poincar\'e duality isomorphism in the context of autonomous Hamiltonians, and for a precise statement of its compatibility with continuation maps. 

The isomorphism~\eqref{eq:PDchainlevel} directly implies a canonical isomorphism 
\begin{equation} \label{eq:PDSHab}
SH_*^{(a,b)}(W,A)\cong SH^{-*}_{(-b,-a)}(W,A^c).
\end{equation}
To see this, note that the class of admissible Hamiltonians $\cH(W;F)$
is stable under the involution $H\mapsto \bar H$. It follows that we can present $SH^{-*}_{(-b,-a)}(W,A^c)$ as a first-inverse-then-direct limit on $FH^{-*}_{(-b,-a)}(\bar H)$ for $H\in \cH(W;F)$, whereas $SH_*^{(a,b)}(W,A)$ is presented as a first-inverse-then-direct limit on $FH_*^{(a,b)}(H)$. In view of~\eqref{eq:PDchainlevel} it is enough to see that the inverse and direct limits in the definitions  are taken over the same sets. Indeed, for $SH_*^{(a,b)}(W,A)$ the inverse limit is taken over Hamiltonians $H$ that go to $-\infty$ on $(\wh W_F\setminus W)_A$, which is equivalent to $\bar H$ going to $\infty$ on $(\wh W_F\setminus W)_A$, and this is precisely the directed set for the inverse limit in the definition of $SH^{-*}_{(-b,-a)}(W,A^c)$. A similar discussion holds for the direct limit. 

The isomorphisms $SH_*^\heartsuit(W,A)\cong SH^{-*}_{-\heartsuit}(W,A^c)$ follow from~\eqref{eq:PDSHab} and from the definitions. We analyse the case $\heartsuit=``>0"$ and leave the other cases to the reader. In the definition of $SH_*^{>0}(W,A)$ the inverse limit is taken over $a\searrow 0$ and the direct limit is taken over $b\to\infty$, which is equivalent to $-a\nearrow 0$ and $-b\to-\infty$. After relabelling $(-b,-a)=(a',b')$, this is the same as $b'\nearrow 0$ and $a'\to-\infty$, which corresponds to the definition of $SH^{-*}_{<0}(W,A^c)$. 
\end{proof}

\subsection{Algebraic duality and universal coefficients} \label{sec:algebraic_duality}

We discuss in this section the algebraic duality between homology and
cohomology in the symplectic setting that we consider. {\color{black}Recall that we use field coefficients.} 

The starting observation is that, given a degree $k$, real numbers
$a<b$, admissible Hamiltonians $H\le H'$, an admissible decreasing
homotopy $(H_s)$, $s\in \R$ connecting $H'$ to $H$, and a regular
homotopy of almost complex structures $(J_s)$, $s\in \R$ connecting an
almost complex structure $J'$ which is regular for $H'$ to an almost
complex structure $J$ which is regular for $H$, there are canonical
identifications  
$$
FC^k_{(a,b)}(H,J)\cong FC_k^{(a,b)}(H,J)^\vee,\qquad \sigma^k\cong (\sigma_k)^\vee,
$$
where $\sigma_k:FC_k^{(a,b)}(H,J)\to FC_k^{(a,b)}(H',J')$,
$\sigma^k:FC^k_{(a,b)}(H',J')\to FC^k_{(a,b)}(H,J)$ are the
continuation maps induced by the homotopy $(H_s,J_s)$. These
identifications follow from the definitions and hold with arbitrary
coefficients.  

We now turn to the relationship between $SH_*^{(a,b)}(W,V)$ and
$SH^*_{(a,b)}(W,V)$. Since we work in a finite action window $(a,b)$,
both the direct and the inverse limits in the definition of
$SH_*^{(a,b)}(W,V)$ and $SH^*_{(a,b)}(W,V)$ eventually stabilize, so
that we can compute these groups using only one suitable
Hamiltonian. The universal coefficient theorem then implies with
coefficients in a field $\K$ the existence of a canonical
isomorphism (see for example~\cite[\S V.7]{Br}) 
\begin{equation} \label{eq:dualitykab}
SH^k_{(a,b)}(W,V;\K)\cong  SH_k^{(a,b)}(W,V;\K)^\vee.
\end{equation}

The issue of comparing $SH^k_\heartsuit(W,V)$ and $SH_k^\heartsuit(W,V)$ becomes therefore a purely algebraic one, as it amounts to comparing via duality the various double limits involved in Definitions~\ref{defi:SH(W)} and~\ref{defi:cohWA} (see also Definitions~\ref{defi:SHWV} and~\ref{defi:cohWV}). The key property is the following: given a direct system of modules $M_\alpha$ and a module $N$ over some ground ring $R$, the natural map 
\begin{equation}\label{eq:Homlim}
\Hom_R(\lim_{\longrightarrow} M_\alpha,N)\stackrel \simeq \longrightarrow \lim_{\longleftarrow} \Hom_R(M_\alpha,N)
\end{equation}
is an isomorphism. However, it is generally \emph{not true} that, given an inverse system $M_\alpha$, the natural map 
$$
\Hom_R(\lim_{\longleftarrow} M_\alpha,N)\longleftarrow \lim_{\longrightarrow} \Hom_R(M_\alpha,N)
$$
is an isomorphism (the two sets actually have different cardinalities
in general). In our situation, {\color{black}$N=R$ is the coefficient field
$\K$.}

We omit in the sequel the field $\K$ from the notation. 

\begin{proposition}\label{prop:dualityk}
Let $(W,V)$ be a pair of Liouville cobordisms with filling. Using
field coefficients we have canonical isomorphisms  
$$
SH^k_\heartsuit(W,V)\cong SH_k^\heartsuit(W,V)^\vee,\qquad \heartsuit\in \{>0,\ge 0,=0\}
$$
and
$$
SH_k^\heartsuit(W,V)\cong SH^k_\heartsuit(W,V)^\vee,\qquad
\heartsuit\in \{<0,\le 0,=0\}. 
$$ 
\end{proposition}

\begin{proof}
Assume first $\heartsuit\in \{>0,\ge 0,=0\}$. In all three cases, the
limit over $a\to 0$ in the definition of $SH_*^\heartsuit(W,V)$ and
$SH^*_\heartsuit(W,V)$ stabilizes, and the result follows
from~\eqref{eq:dualitykab} and~\eqref{eq:Homlim} applied to the limit $b\to\infty$.
 
Assume now $\heartsuit\in \{<0,\le 0,=0\}$. In all three cases the
limit over $b\to 0$ in the definition of $SH_*^\heartsuit(W,V)$ and
$SH^*_\heartsuit(W,V)$ stabilizes, and the result follows again
from~\eqref{eq:Homlim} applied to the limit $a\to-\infty$, by
rewriting~\eqref{eq:dualitykab} as  
$$
SH_k^{(a,b)}(W,V)\cong SH^k_{(a,b)}(W,V)^\vee.
$$
This holds because the vector spaces which are involved are finite dimensional. 
\end{proof}

\begin{corollary}\label{cor:dualityk}
(a) Let $(W,V)$ be a pair of filled Liouville cobordisms with vanishing first Chern class. Suppose
that $\p V$ and $\p W$ carry only finitely many closed Reeb orbits of
any given degree. Then with field coefficients we have for all flavors
$\heartsuit$ canonical isomorphisms  
$$
   SH^k_\heartsuit(W,V)\cong SH_k^\heartsuit(W,V)^\vee\qquad\text{and}\qquad
   SH_k^\heartsuit(W,V)\cong SH^k_\heartsuit(W,V)^\vee. 
$$ 
(b) Let $W$ be a Liouville domain. Then with field coefficients we have
canonical isomorphisms  
$$
   SH^k(W)\cong SH_k(W)^\vee.
$$
\end{corollary}

\begin{proof}
Part (a) follows from the proof as Proposition~\ref{prop:dualityk},
using that all inverse limits remain finite dimensional. Part (b)
holds because for a Liouville domain we have $SH_k(W)=SH_k^{\ge 0}(W)$. 
\end{proof}

\begin{remark}
Proposition~\ref{prop:dualityk} illustrates the fact
that the full symplectic homology and cohomology groups of a cobordism
or of a pair of cobordisms are of a mixed homological-cohomological
nature. This is due to the presence of both a direct and of an
inverse limit in the definitions. As such, the full version
$SH_*(W,V)$ does not satisfy in general any form of algebraic
duality. In fact, in Example~\ref{ex:Peter} below we construct a
Liouville cobordism $W$ for which in some degree $k$ (and with
$\Z_2$-coefficients) neither $SH^k(W)\cong SH_k(W)^\vee$ nor
$SH_k(W)\cong SH^k(W)^\vee$ holds. 
\end{remark}

\section{Homological algebra and mapping cones} \label{sec:cones}

\subsection{Cones and distinguished triangles}\label{sec:triangles} 

Let $R$ be a ring. Let {\sf Ch} denote the category of chain complexes of $R$-modules. The objects of this category are chain complexes of $R$-modules, and the morphisms are chain maps of degree $0$. Let {\sf Kom} denote the category of chain complexes of $R$-modules up to homotopy. The objects are the same as the ones of {\sf Ch}, and the morphisms are equivalence classes of degree $0$ chain maps with respect to the equivalence relation given by homotopy equivalence. We use homological $\mathbb Z$-grading, and we use the following notational conventions : 
\renewcommand{\theenumi}{\roman{enumi}}
\begin{enumerate}
\item given a morphism $X\longrightarrow Y$ in {\sf Kom}, we use the notation $X\stackrel f\longrightarrow Y$ for a specific representative $f$ of this morphism. Thus $f$ is a morphism in {\sf Ch}. 
\item all diagrams are understood to be commutative in {\sf Kom}. If we specify representatives in {\sf Ch} for the morphisms, we say that a diagram is \emph{strictly commutative} if it commutes in {\sf Ch}. 
\item we use the notation 
$$
\xymatrix
@C=30pt
@R=30pt
{\ar @{} [dr] |s X \ar[r]^f \ar[d]_\varphi & Y \ar[d]^\psi \\ 
X' \ar[r]_g & Y' 
} 
$$
for a diagram in {\sf Ch} which is commutative modulo a \emph{specified} homotopy $s$, i.e. such that $\psi f - g\varphi = s\p_X + \p_{Y'}s$. In particular, the diagram 
$$
\xymatrix
@C=30pt
@R=30pt
{X \ar[r]^f \ar[d]_\varphi & Y \ar[d]^\psi \\ 
X' \ar[r]_g & Y' 
} 
$$
is commutative in {\sf Kom}. 
\item given a chain complex $X=\{(X_n),\p_X\}$ and $k\in\mathbb Z$, we define the shifted complex $X[k]$ by 
$$
X[k]_n=X_{n+k},\quad n\in\mathbb Z,\qquad \p_{X[k]}=(-1)^k\p_X.
$$ 
Given a morphism $f:X\to Y$, we define $f[k]:X[k]\to Y[k]$ as $f[k]=f$. 
\end{enumerate}

Our conventions for cones and distinguished triangles follow the ones of Kashiwara and Schapira~\cite[Chapter~1]{KS}, except that we use dual homological grading. Given a chain map $f:X \to Y$, we define its \emph{cone} to be the chain complex 
$$
C(f)=Y\oplus X[-1],\qquad \p_{C(f)}=\left(\begin{array}{cc} \p_{Y} & f \\ 0 & \p_{X[-1]}\end{array}\right) = \left(\begin{array}{cc} \p_Y & f \\ 0 & -\p_X\end{array}\right)
$$
We have in particular a short exact sequence of chain complexes 
\begin{equation} \label{les:cone}
\xymatrix
{
0\ar[r] &  Y \ar[r]^-{\alpha(f)} & C(f) \ar[r]^-{\beta(f)} & X[-1] \ar[r] &  0
} 
\end{equation}
where $\alpha(f)=\left(\begin{array}{c} \mbox{Id}_{Y} \\ 0 \end{array}\right)$ is the canonical inclusion, and $\beta(f)=\left(\begin{array}{cc} 0 & \mbox{Id}_{X[-1]} \end{array}\right)$ is the canonical projection. For simplicity we abbreviate in the sequel the identity maps by $1$, e.g. we write $\alpha(f)=\left(\begin{array}{c} 1 \\ 0 \end{array}\right)$ and 
$\beta(f)=\left(\begin{array}{cc} 0 & 1 \end{array}\right)$. 

One of the key features of the cone construction is that the connecting homomorphism in the homology long exact sequence associated to the short exact sequence~\eqref{les:cone} is equal to $f_*$, the morphism induced by $f$. 

By definition, a \emph{triangle} in {\sf Kom} is a sequence of morphisms
\begin{equation} \label{les:triangle}
\xymatrix
@C=50pt
{
X \ar[r]^-f & Y \ar[r]^-g & Z \ar[r]^-h & X[-1] 
}
\end{equation}
A \emph{distinguished triangle} is a triangle which is isomorphic in {\sf Kom} to a triangle of the form 
\begin{equation} \label{les:model}
\xymatrix
@C=50pt
{
X \ar[r]^f &  Y \ar[r]^-{\alpha(f)} & C(f) \ar[r]^-{\beta(f)} & X[-1]
}
\end{equation}
We call~\eqref{les:model} a \emph{model distinguished triangle}. 

It follows from the definition that a distinguished triangle~\eqref{les:triangle} induces a long exact sequence in homology 
\begin{equation}\label{eq:hlescone}
\xymatrix
{
\cdots H_*(X) \ar[r]^-{f_*} & H_*(Y) \ar[r]^-{g_*} & H_*(Z) \ar[r]^-{h_*} & H_{*-1}(X) \ar[r]^-{f_*} & \cdots
} 
\end{equation}
We shall often represent such a long exact sequence as 
$$
\xymatrix
{H(X) \ar[rr]^{f_*} && H(Y) \ar[dl]^{g_*} \\
& H(Z) \ar[ul]^-{h_*}_-{[-1]} &
}
$$
We call such a diagram an \emph{exact triangle}. 

The above definition of the class of distinguished triangles makes {\sf Kom} into a \emph{triangulated category} in the sense of Verdier. This means that the class of distinguished triangles satisfies Verdier's axioms (TR0)--(TR5) (see for example~\cite[\S\S1.4-1.5]{KS}). One of the essential axioms is (TR3): a triangle~\eqref{les:triangle} is distinguished if and only if the triangle 
$$
\xymatrix
@C=50pt
{
Y \ar[r]^-g & Z \ar[r]^-h & X[-1] \ar[r]^-{-f[-1]} & Y[-1]
}
$$
is distinguished. This follows from Lemma~\ref{lem:rotate}(i) below, see also~\cite[Lemma~1.4.2]{KS}. 

\begin{lemma} \label{lem:rotate}

Let $f:X\to Y$ be a morphism in {\sf Ch}. 

(i) \cite[Lemma~1.4.2]{KS} There exists a morphism in {\sf Ch} 
$$
\Phi : X[-1]\to C(\alpha(f))
$$
which is an isomorphism in {\sf Kom}, with an explicit homotopy inverse in {\sf Ch} denoted
$$
\Psi:C(\alpha(f))\to X[-1],
$$ 
and such that the diagram below commutes in {\sf Kom}:
$$
\xymatrix
@C=50pt
@R=30pt 
{
Y \ar[r]^-{\alpha(f)} \ar@{=}[d] & C(f) \ar[r]^-{\beta(f)} \ar@{=}[d] & X[-1] \ar[r]^-{-f[-1]} \ar@<.5ex>[d]^{\Phi} &  
Y[-1] \ar@{=}[d] \\
Y \ar[r]_-{\alpha(f)} & C(f) \ar[r]_-{\alpha(\alpha(f))} & C(\alpha(f)) \ar[r]_-{\beta(\alpha(f))} \ar@<.5ex>[u]^{\Psi} & Y[-1]
}
$$

(ii) There exists a morphism in {\sf Ch} 
$$
\tau:Y[-1]\to C(\beta(f))
$$
which is an isomorphism in {\sf Kom}, with an explicit homotopy inverse in {\sf Ch} denoted 
$$
\sigma:C(\beta(f))\to Y[-1],
$$
and such that the diagram below commutes in {\sf Kom} 
$$
\xymatrix
@C=50pt
@R=30pt 
{
C(f) \ar[r]^-{\beta(f)} \ar@{=}[d] & X[-1] \ar[r]^-{-f[-1]} \ar@{=}[d] & Y[-1] \ar@<.5ex>[d]^\tau \ar[r]^-{-\alpha(f)[-1]} & C(f)[-1] \ar@{=}[d] \\
C(f) \ar[r]_-{\beta(f)} & X[-1] \ar[r]_-{\alpha(\beta(f))[-1]} & C(\beta(f)) \ar@<.5ex>[u]^{\sigma} \ar[r]_-{\beta(\beta(f))}  & C(f)[-1]
}
$$
\end{lemma}

\begin{proof}
(i) (following \cite{KS}) Taking into account that $C(\alpha(f))=Y\oplus X[-1]\oplus Y[-1]$, we define in matrix form 
$$
\Phi=\left(\begin{array}{c} 0 \\ 1 \\ -f \end{array}\right),\qquad \Psi=\left(\begin{array}{ccc} 0 & 1 & 0 \end{array}\right).
$$
(Here $1$ stands for $\mbox{Id}_{X[-1]}$ according to our convention.) A direct verification shows that these are chain maps, and also that the third square in the diagram commutes in {\sf Ch}, i.e. $\beta(\alpha(f))\Phi=-f[-1]$. Such verifications formally amount to elementary multiplications of matrices. For example: 
$$
\p_{C(\alpha(f))}\Phi =
{ \left(\begin{array}{ccc} \p_Y & f & 1 \\ 0 & \p_{X[-1]} & 0 \\ 0 & 0 & \p_{Y[-1]} \end{array}\right)}
{\left(\begin{array}{c} 0 \\ 1 \\ -f \end{array}\right)} 
= 
{\left(\begin{array}{c} 0 \\ \p_{X[-1]} \\ -\p_{Y[-1]}f \end{array}\right)} 
$$
and
$$
\beta(\alpha(f))\Phi = \left(\begin{array}{ccc} 0 & 0 & 1 \end{array}\right) \left(\begin{array}{c} 0 \\ 1 \\ -f \end{array}\right) = -f.
$$

The second square in the diagram is commutative in {\sf Kom}. Indeed, direct verification shows that $\Psi\alpha(\alpha(f))=\beta(f)$. On the other hand, the maps $\Phi$ and $\Psi$ are homotopy inverses to each other. Indeed, direct verification shows that $\Psi\Phi=\mbox{Id}_{X[-1]}$ and 
$$
\mbox{Id}_{C(\alpha(f))}-\Phi\Psi = \left(\begin{array}{ccc} 1 & 0 & 0 \\ 0 & 0 & 0 \\ 0 & f & 1 \end{array}\right) = \p_{C(\alpha(f))}K + K\p_{C(\alpha(f))},
$$
where $K:C(\alpha(f))\to C(\alpha(f))[1]$ is a homotopy given in matrix form by 
$$
K=\left(\begin{array}{ccc} 0 & 0 & 0 \\ 0 & 0 & 0 \\ 1 & 0 & 0 \end{array}\right).
$$

(ii) Taking into account that $C(\beta(f))=X[-1]\oplus Y[-1]\oplus X[-2]$ we define in matrix form 
$$
\tau=\left(\begin{array}{c} 0 \\ -1 \\ 0\end{array}\right),\qquad \sigma=\left(\begin{array}{ccc} -f & -1 & 0 \end{array}\right).
$$
Here $1$ stands for $\mbox{Id}_{Y[-1]}$. Direct verification shows that these are chain maps, that $\beta(\beta(f))\tau=-\alpha(f)[-1]$ so that the third square is commutative in {\sf Ch}, and that $\sigma\alpha(\beta(f))=-f[-1]$. 

Commutativity in {\sf Kom } of the second square follows again from the fact that $\sigma$ and $\tau$ are homotopy inverses to each other. Indeed, we have $\sigma\tau=\mbox{Id}_{Y[-1]}$, whereas 
$$
\mbox{Id}_{C(\beta(f))}-\tau\sigma = \left(\begin{array}{ccc} 1 & 0 & 0 \\ -f & 0 & 0 \\ 0 & 0 & 1 \end{array}\right) = \p_{C(\beta(f))} L + L \p_{C(\beta(f))},
$$
where $L:C(\beta(f))\to C(\beta(f))[1]$ is a homotopy defined in matrix form by 
$$
L = \left(\begin{array}{ccc} 0 & 0 & 0 \\ 0 & 0 & 0 \\ 1 & 0 & 0 \end{array}\right).
$$
\end{proof}

\begin{remark}
One consequence of Lemma~\ref{lem:rotate} (i.e. axiom (TR3)) is that a triangle 
$$
\xymatrix
@C=50pt
{
X \ar[r]^-f & Y \ar[r]^-g & Z \ar[r]^-h & X[-1] 
}
$$ 
is distinguished if and only if the triangle
$$
\xymatrix
@C=50pt
{
X[-1] \ar[r]^-{-f[-1]} & Y[-1] \ar[r]^-{-g[-1]} & Z[-1] \ar[r]^-{-h[-1]} & X[-2] 
}
$$
is distinguished. The triangle 
$$
\xymatrix
@C=50pt
{
X[-1] \ar[r]^-{f[-1]} & Y[-1] \ar[r]^-{g[-1]} & Z[-1] \ar[r]^-{h[-1]} & X[-2] 
}
$$
is in general \emph{not} distinguished, but rather \emph{anti-}distinguished in the sense of~\cite[Definition~1.5.9]{KS}. The class of distinguished triangles is distinct from that of anti-distinguished triangles, as explained to us by S.~Guillermou. 
\end{remark}

We use Lemma~\ref{lem:rotate} in order to replace by cones in {\sf Kom} the kernels and cokernels of certain maps in {\sf Ch}.

\begin{lemma} \label{lem:ip} 
Let 
\begin{equation} \label{les:1}
0\longrightarrow A \stackrel i \longrightarrow B \stackrel p \longrightarrow C \longrightarrow 0
\end{equation}
be a short exact sequence in {\sf Ch} which is split as a short exact sequence of $R$-modules. 

(i) Given a splitting $s:C\to B$, i.e. a {\color{black} degree $0$} map such that $ps=\mbox{Id}_C$, there is a canonical chain map $f:C[1]\to A$ and there are canonical identifications in {\sf Ch}
$$
B= C(f),\qquad i= \alpha(f),\qquad p= \beta(f).
$$

(ii) 
The maps
$$
\Phi:C\stackrel \simeq \longrightarrow C(i), \qquad \tau:A[-1]\stackrel \simeq \longrightarrow C(p)
$$
defined in (i) and (ii) of Lemma~\ref{lem:rotate} are isomorphisms in {\sf Kom} and they determine isomorphisms of distinguished triangles
$$
\xymatrix
@C=50pt
@R=30pt
{
A \ar[r]^-i \ar@{=}[d] & B \ar[r]^-p \ar@{=}[d] & C \ar[r]^-{-f[-1]} \ar[d]^{\Phi}_\simeq & A[-1] \ar@{=}[d] \\
A \ar[r]_-i & B \ar[r]_-{\alpha(i)} & C(i) \ar[r]_-{\beta(i)} & A[-1]
}
$$
and
$$
\xymatrix
@C=50pt
@R=30pt
{
B \ar[r]^-p \ar@{=}[d] & C \ar[r]^-{-f[-1]} \ar@{=}[d] & A[-1] \ar[r]^-{-i[-1]} \ar[d]^\tau_\simeq & B[-1] \ar@{=}[d] \\
B \ar[r]_{p} & C \ar[r]_-{\alpha(p)} & C(p) \ar[r]_-{\beta(p)} & B[-1]
}
$$
In particular, the homology long exact sequences determined by the top and bottom line in each of the above diagrams are isomorphic. 

(iii) Assume that the splitting $s:C\to B$ is a chain map. We then have an isomorphism in {\sf Kom} 
$$
A\stackrel \simeq \longrightarrow C(s).
$$ 
(The same holds if we assume that the splitting $s$ is homotopic to a chain map.) 
\end{lemma}

\begin{proof}
{\color{black}For item (i) let $(i\ s):C(f)=A\oplus C\stackrel{\cong}\to B$ be the
isomorphism of $R$-modules induced by $s$. Since $p(\p_Bs-s\p_C)=0$
and $i$ is injective, we can define $f:C[1]\to A$ uniquely by
$if=\p_Bs-s\p_C$ and one checks that this map has the desired properties. 
Item (ii) is simply a rephrasal of Lemma~\ref{lem:rotate}. }

Item (iii) is a consequence of (ii) as follows. Let us write $s=\left(\begin{array}{c} \varphi \\ 1 \end{array}\right)$ with $\varphi:C\to A$. Viewing $B$ as the cone of $f$ as in (i), the condition that $s$ is a chain map translates into $\varphi\p_C=\p_A\varphi +f$. (This in turn can be reinterpreted as saying that $-\varphi$ is a chain homotopy between $f$ and $0$.) 

We consider the map $\pi:B=A\oplus C\to A$ given by $\pi=\left(\begin{array}{cc} 1 & -\varphi \end{array}\right)$. Then $\pi$ is a chain map and $\ker\pi = \im s$, so that we have a split short exact sequence 
$$
0\longrightarrow C \stackrel s \longrightarrow B \stackrel \pi \longrightarrow A\longrightarrow 0
$$
and we conclude using the first assertion in (ii). 

The class of chain maps is closed under homotopies: if $s$ is homotopic to a chain map, then it is an actual chain map. 
\end{proof}

\noindent {\bf Remark.} It is \emph{not} true that a short exact sequence of complexes $0\to A\stackrel{i}\longrightarrow B\stackrel p \longrightarrow C\to 0$ can always be completed to a distinguished triangle $A\stackrel i \longrightarrow B \stackrel p \longrightarrow C \longrightarrow A[-1]$. Thus the splitting assumption in Lemma~\ref{lem:ip} is necessary. Indeed, consider the example of the short exact sequence of $\Z$-modules
$$
\xymatrix{0 \ar[r] & \Z \ar[r]^i_{\times 2} & \Z \ar[r]^p & \Z/2\ar[r] & 0}
$$ 
where $p$ is the canonical projection and $i$ is multiplication by $2$, thought of as an exact sequence of chain complexes supported in degree $0$. The cone of $i$ is equal to $\Z$ in degrees $0$ and $1$, with differential $\left(\begin{array}{cc} 0 & \times 2\\ 0 & 0 \end{array}\right)$. The map $\left(\begin{array}{cc} p & 0 \end{array}\right):C(i)\to \Z/2$ is a quasi-isomorphism, yet $\Z/2$ is not homotopy equivalent to $C(i)$ since the only morphism $\Z/2\to C(i)$ is the zero map. This shows that the above short exact sequence cannot be completed to a distinguished triangle.

\begin{proposition} \label{prop:diag}
Let 
$$
\xymatrix
@C=30pt
@R=30pt
{X \ar[r]^f \ar[d]_\varphi & Y \ar[d]^\psi \\ 
X' \ar[r]_g & Y' 
} 
$$
be a commutative diagram in {\sf Kom}.
This can be completed to a diagram whose rows and columns are distinguished triangles in
{\sf Kom} and in which all squares are commutative (in {\sf Kom}), except the bottom right square which is anti-commutative
$$
\xymatrix
@C=30pt
@R=30pt
{X \ar[r]^f \ar[d]_\varphi & Y \ar[r] \ar[d]_\psi & Z \ar[r] \ar[d]_\chi & X[-1] \ar[d] \\
X' \ar[r]^g \ar[d] & Y' \ar[r] \ar[d] & Z' \ar[r] \ar[d] & X'[-1] \ar[d] \\
X'' \ar[r]^h \ar[d] & Y'' \ar[r] \ar[d] & Z'' \ar @{} [dr] |{-} \ar[r] \ar[d] & X''[-1] \ar[d] \\
X[-1] \ar[r] & Y[-1] \ar[r] & Z[-1] \ar[r] & X[-2] 
}
$$
\end{proposition}

\begin{remark}
This statement, attributed to Verdier, is proved in Beilinson, Bernstein, Deligne~\cite[Proposition~1.1.11]{BBD} by a repeated use of the octahedron axiom (TR5). This is also proved in~\cite[Lemma~2.6]{May_additivity} under the name ``$3\times 3$ Lemma", where it is shown that it is actually equivalent to the octahedron axiom. The same statement appears as Exercise~10.2.6 in~\cite{Weibel}. Our proof is more explicit and produces a diagram in which all the squares except the initial one and the bottom right one are commutative in {\sf Ch}, and in which the bottom right square is anti-commutative in {\sf Ch}. This result encompasses~\cite[Lemma~2.18]{BO3+4} and~\cite[Lemma~5.7]{BOGysin}. {\color{black}For completeness, we will reprove~\cite[Lemma~5.7]{BOGysin} as Lemma~\ref{lem:diag} below as a consequence of Proposition~\ref{prop:diag} (under an additional splitting assumption).}
\end{remark}

\begin{proof}[Proof of Proposition~\ref{prop:diag}] We start with the square 
$$
\xymatrix
@C=30pt
@R=30pt
{\ar @{} [dr] |s X \ar[r]^f \ar[d]_\varphi & Y \ar[d]^\psi \\ 
X' \ar[r]_g & Y' 
} 
$$
which is commutative modulo the homotopy $s$, meaning in our notation that 
\begin{equation} \label{eq:homotopy-s}
\psi f - g\varphi = s\p_X + \p_{Y'}s.
\end{equation}
We construct the grid diagram in the statement by a repeated use of the cone construction. 
The first two lines and the first two columns are constructed as model distinguished triangles. More precisely, we define 
$$
Z=C(f)=Y\oplus X[-1],\qquad Z'=C(g)=Y'\oplus X'[-1], \qquad \chi=\left(\begin{array}{cc} \psi & s \\ 0 & \varphi \end{array}\right).
$$
The condition that $\chi$ is a chain map is equivalent to equation~\eqref{eq:homotopy-s}, and the second and third square formed by the first two lines are then commutative in {\sf Ch}: 
$$
\xymatrix
@C=30pt
@R=30pt
{X \ar[r]^f \ar[d]_\varphi & Y \ar[r]^-{\alpha(f)} \ar[d]_\psi & Z \ar[r]^-{\beta(f)} \ar[d]_\chi & X[-1] \ar[d]_{\varphi[-1]} \\
X' \ar[r]^g & Y' \ar[r]^-{\alpha(g)} & Z' \ar[r]^-{\beta(g)} & X'[-1] 
}
$$
Similarly, we define 
$$
X''=C(\varphi)=X'\oplus X[-1],\qquad Y''=C(\psi)=Y'\oplus Y[-1],\qquad h= \left(\begin{array}{cc} g & -s \\ 0 & f \end{array}\right).
$$
Again, the condition that $h$ is a chain map is equivalent to equation~\eqref{eq:homotopy-s} and the first two columns determine a diagram in which the second and third square are commutative in {\sf Ch}:
$$
\xymatrix
@C=30pt
@R=30pt
{X \ar[r]^f \ar[d]_\varphi & Y \ar[d]_\psi \\
X' \ar[r]^g \ar[d]_-{\alpha(\varphi)} & Y' \ar[d]_-{\alpha(\psi)} \\
X'' \ar[r]^h \ar[d]_-{\beta(\varphi)} & Y'' \ar[d]_-{\beta(\psi)} \\
X[-1] \ar[r]^{f[-1]} & Y[-1]  
}
$$

We define 
$$
Z''=C(\chi).
$$
We construct the third and fourth columns of the grid diagram as model distinguished triangles, and we are left to specify the morphisms $A,B,C,D$ below:
$$
\xymatrix
@C=30pt
@R=30pt
{X \ar[r]^f \ar[d]_\varphi & Y \ar[r]^-{\alpha(f)} \ar[d]_\psi & C(f) \ar[r]^-{\beta(f)} \ar[d]_\chi & X[-1] \ar[d]^-{\varphi[-1]} \\
X' \ar[r]^g \ar[d]_-{\alpha(\varphi)} & Y' \ar[r]^-{\alpha(g)} \ar[d]_-{\alpha(\psi)} & C(g) \ar[r]^-{\beta(g)} \ar[d]_-{\alpha(\chi)} & X'[-1] \ar[d]^-{\alpha(\varphi[-1])} \\
C(\varphi) \ar[r]^h \ar[d]_-{\beta(\varphi)} & C(\psi) \ar@{.>}[r]^{A} \ar[d]_-{\beta(\psi)} & C(\chi) \ar@{.>}[r]^{B} \ar[d]_-{\beta(\chi)} & C(\varphi[-1]) \ar[d]^-{\beta(\varphi[-1])} \\
X[-1] \ar[r]^-{f[-1]} & Y[-1] \ar@{.>}[r]^{C} & C(f)[-1] \ar@{.>}[r]^{D} & X[-2] 
}
$$

The key point is that we have isomorphisms of chain complexes 
$$
\xymatrix
@C=30pt
@R=8pt
{
I:C(\chi) \ar[r]^-{\simeq} \ar@{=}[d] & C(h), \ar@{=}[d] \\
Y'\oplus X'[-1]\oplus Y[-1]\oplus X[-2] & Y'\oplus Y[-1] \oplus X'[-1] \oplus X[-2]
}
$$
$$
I:=\left(\begin{array}{cccc} 1 & 0 & 0 & 0 \\ 0 & 0 & 1 & 0 \\ 0 & 1 & 0 & 0 \\ 0 & 0 & 0 & -1
\end{array}\right)
$$
and 
\begin{equation} \label{eq:Jf}
\xymatrix
@C=30pt
@R=8pt
{
J(f):C(f)[-1] \ar[r]^-{\simeq} \ar@{=}[d] & C(f[-1]), \ar@{=}[d] \\
Y[-1]\oplus X[-2] & Y[-1]\oplus X[-2]
}
\end{equation}
$$
J(f):=\left(\begin{array}{cc} 1 & 0 \\ 0 & -1 \end{array}\right).
$$
One checks directly that the maps $I$ and $J(f)$ commute with the differentials. 

The third line in our diagram, involving the maps $A$ and $B$, is defined using the isomorphisms $I$ and $J(\varphi)$ from the model distinguished triangle associated to $h$, i.e. $A=I^{-1}\alpha(h)$, $B=J(\varphi)\beta(h)I$:
$$
\xymatrix
@C=30pt
@R=30pt
{C(\varphi) \ar[r]^h \ar@{=}[d] & C(\psi) \ar@{.>}[r]^-{A} \ar@{=}[d] & C(\chi) \ar@{.>}[r]^-{B} \ar[d]_-{I}^-{\simeq} & C(\varphi[-1]) \\
C(\varphi) \ar[r]^h  & C(\psi) \ar[r]^-{\alpha(h)} & C(h) \ar[r]^-{\beta(h)} & C(\varphi)[-1] \ar[u]^-{J(\varphi)}_-{\simeq} 
}
$$
In matrix form we have 
$$
A=\left(\begin{array}{cc} 1 & 0 \\ 0 & 0 \\ 0 & 1 \\ 0 & 0 \end{array}\right),\qquad 
B=\left(\begin{array}{cccc} 0 & 1 & 0 & 0 \\ 0 & 0 & 0 & 1 \end{array}\right).
$$

The fourth line in our diagram, involving the maps $C$ and $D$, is defined using the isomorphism $J(f)$ from the model distinguished triangle associated to $f[-1]$, i.e. $C=J(f)^{-1}\alpha(f[-1])$, $D=\beta(f[-1])J(f)$:
$$
\xymatrix
@C=30pt
@R=30pt
{X[-1] \ar[r]^-{f[-1]} \ar@{=}[d] & Y[-1] \ar@{.>}[r]^-{C} \ar@{=}[d] & C(f)[-1] \ar@{.>}[r]^-{D} \ar[d]_-{J(f)}^-{\simeq} & X[-2] \ar@{=}[d] \\
X[-1] \ar[r]^-{f[-1]} & Y[-1] \ar[r]^-{\alpha(f[-1])} & C(f[-1]) \ar[r]^-{\beta(f[-1])} & X[-2] 
}
$$
In matrix form we have 
$$
C=\left(\begin{array}{c} 1 \\ 0 \end{array}\right),\qquad D=\left(\begin{array}{cc} 0 & -1 \end{array}\right).
$$

A direct check shows that 
$$
A\alpha(\psi)=\alpha(\chi)\alpha(g),\qquad B\alpha(\chi)=\alpha(\varphi[-1])\beta(g),\qquad 
C\beta(\psi)=\beta(\chi)A,
$$
and
$$
D\beta(\chi)=-\beta(\varphi[-1])B.
$$
\end{proof}

{\color{black} 
For later use, we recall Lemma~5.7 from~\cite{BOGysin} and show how it follows from Proposition~\ref{prop:diag} under an additional assumption. 

\begin{lemma}[{\cite[Lemma~5.7]{BOGysin}}] \label{lem:diag} 
Let  
\begin{equation} \label{eq:fgh} 
\xymatrix{ 
0\ar[r] & A \ar[r]^i \ar[d]^f & B \ar[r]^p \ar[d]^g & C \ar[r] 
\ar[d]^h & 0 \\ 
0\ar[r] & A' \ar[r]^{i'}  & B' \ar[r]^{p'}  & C' \ar[r] 
 & 0 
} 
\end{equation}  
be a morphism of short exact sequences of complexes. 
We then have a diagram whose rows and columns are exact and in which all squares are commutative, except the bottom right one which is anti-commutative. 
$$
\xymatrix
@C=30pt
@R=30pt
{H_*(A) \ar[r]^{i_*} \ar[d]^{f_*} & H_*(B) \ar[r]^{p_*} \ar[d]^{g_*} & H_*(C) \ar[r] \ar[d]^{h_*} & H_{*-1}(A) \ar[d]^{f_*} \\
H_*(A') \ar[r]^{i'_*} \ar[d]^{\alpha(f)_*} & H_*(B') \ar[r]^{p'_*} \ar[d]^{\alpha(g)_*} & H_*(C') \ar[r] \ar[d]^{\alpha(h)_*} & H_{*-1}(A') \ar[d]^{\alpha(f)_*} \\
H_*(C(f)) \ar[r] \ar[d]^{\beta(f)_*} & H_*(C(g)) \ar[r] \ar[d]^{\beta(g)_*} & H_*(C(h)) \ar @{} [dr] |{-} \ar[r] \ar[d]^{\beta(h)_*} & H_{*-1}(C(f)) \ar[d]^{\beta(f)_*} \\
H_{*-1}(A) \ar[r]^{i_*} & H_{*-1}(B) \ar[r]^{p_*} & H_{*-1}(C) \ar[r] & H_{*-2}(A) 
}
$$
\end{lemma}

\begin{proof}
Up to changes in notation, this is exactly Lemma~5.7 in~\cite{BOGysin}. 

To wrap up the story, we show here how this result follows from Proposition~\ref{prop:diag} under the additional assumption that the short exact sequences are split as sequences of $R$-modules (this is always the case if $R$ is field or, more generally, if we work with chain complexes of free $R$-modules). 

Choose splittings $s:C\to B$ and $s':C'\to B'$. By Lemma~\ref{lem:ip}, these determine canonical chain maps $\varphi:C[1]\to A$ and $\varphi':C'[1]\to A'$, together with canonical identifications $B=C(\varphi)$, $i=\alpha(\varphi)$, $p=\beta(\varphi)$, $B'=C(\varphi')$, $i'=\alpha(\varphi')$, $p'=\beta(\varphi')$. 

The map $g:B\to B'$ can then be identified with a map $C(\varphi)\to C(\varphi')$ written in matrix form as 
$$
g=\left(\begin{array}{cc} f & t \\ 0 & h \end{array}\right) : A\oplus C \to A'\oplus C'. 
$$

The condition that $g$ is a chain map is then equivalent to the three relations 
$$
f\p_A = \p_{A'} f,\qquad h\p_C = \p_{C'}h,\qquad f\varphi - \varphi'h=\p_{A'} t - t\p_C.
$$
We interpret the last relation as $f\varphi-\varphi'h[1]=\p_{A'}t+t\p_{C[1]}$, which means that the square 
\begin{equation}\label{eq:square-t}
\xymatrix
@C=30pt
@R=30pt
{\ar @{} [dr] |t C[1] \ar[r]^\varphi \ar[d]_{h[1]} & A \ar[d]^f \\ 
C'[1] \ar[r]_{\varphi'} & A' 
} 
\end{equation}
is commutative up to a homotopy given by $t:C\to A'$. The initial diagram~\eqref{eq:fgh} appears then as the horizontal extension of this commutative square in {\sf Kom} to a map of distinguished triangles. 

We now apply Proposition~\ref{prop:diag} to the square~\eqref{eq:square-t} in order to obtain the grid diagram
$$
\xymatrix
@C=30pt
@R=30pt
{\ar @{} [dr] |t C[1] \ar[r]^\varphi \ar[d]_{h[1]} & A \ar[r]^i \ar[d]_f & B \ar[r]^p \ar[d]_g & C \ar[d]_h \\
C'[1] \ar[r]^{\varphi'} \ar[d] & A' \ar[r]^{i'} \ar[d] & B' \ar[r]^{p'} \ar[d] & C' \ar[d] \\
C(h[1]) \ar[r] \ar[d] & C(f) \ar[r] \ar[d] & C(g) \ar @{} [dr] |{-} \ar[r] \ar[d] & C(h) \ar[d] \\
C \ar[r] & A[-1] \ar[r] & B[-1] \ar[r] & C[-1] 
}
$$
The anti-commutativity of the bottom right corner can be traded for anti-commu\-tativity of the bottom left corner by changing the sign of the two bottom middle vertical arrows. The grid diagram in the statement of the lemma is then obtained by passing to homology.  
\end{proof}
}

\subsection{Uniqueness of the cone}\label{sec:cone-uniqueness}

We now spell out what is the additional piece of structure that is needed in order for the cone of a map to be uniquely and canonically defined up to homotopy. 

(i) \noindent \emph{$\operatorname{Hom}$ complexes}. Let $X,Y$ be chain complexes of $R$-modules and denote 
$$
\operatorname{Hom}_d(X,Y),\qquad d\in \Z
$$
the $R$-module of $R$-linear maps of degree $d$. This is a chain complex with differential 
$$
\p:\operatorname{Hom}_d(X,Y)\to \operatorname{Hom}_{d-1}(X,Y),
$$
$$
\p\Phi=\p_Y\Phi - (-1)^{|\Phi|}\Phi\p_X,
$$
where $|\Phi|=d$ denotes the degree of a map $\Phi\in\operatorname{Hom}_d(X,Y)$. The space of degree $d$ cycles 
$$
Z_d(X,Y)=\ker(\p:\operatorname{Hom}_d(X,Y)\to \operatorname{Hom}_{d-1}(X,Y))
$$ 
is the space of degree $d$ chain maps $X\to Y$. Two degree $d$ chain maps are homologous, i.e. they differ by an element of 
$$
B_d(X,Y):=\mbox{Im}(\p:\operatorname{Hom}_{d+1}(X,Y)\to\operatorname{Hom}_d(X,Y)),
$$ 
if and only if they are chain homotopic.

\noindent {\bf Remark/Notation.} We denote a degree $d$ map $f$ from $X$ to $Y$ by 
$$
f:X\stackrel d\longrightarrow Y.
$$ 
We \emph{do not} use the notation $f:X\to Y[d]$, which we reserve for chain maps. This distinction is relevant in practice when using cones because the differential of the complex $Y[d]$ is not $\p_Y$, but $(-1)^d\p_Y$.

(ii) \noindent \emph{Chain maps between cones}. Let
$$
\xymatrix
@C=30pt
@R=30pt
{\ar @{} [dr] |s X \ar[r]^f \ar[d]_\varphi & Y \ar[d]^\psi \\ 
X' \ar[r]_g & Y' 
} 
$$
be a diagram of degree $0$ chain maps which is commutative modulo a prescribed degree $1$ homotopy $s\in\operatorname{Hom}_1(X,Y')$, meaning that $\psi f - g\varphi = \p(s)$. 
We have an induced chain map 
$$
\xymatrix
@C=5pt
@R=8pt
{
\chi_s={\left(\begin{array}{cc} \psi & s \\ 0 & \varphi[-1] \end{array}\right)} \qquad : & C(f)\ar[rrr]\ar@{=}[d] & & & C(g). \ar@{=}[d] \\
& Y\oplus X[-1] & & & Y'\oplus X'[-1]
}
$$
The homotopy class of the map $\chi_s$ depends only on the equivalence class of the homotopy $s$ modulo $B_1(X,Y')$. Indeed, if $t\in\operatorname{Hom}_1(X,Y')$ is another map such that $\psi f - g\varphi = \p(t)$ then $s-t\in Z_1(X,Y')$. If 
$s-t\in B_1(X,Y')$, meaning that 
$$
s-t=\p(b)
$$ 
with $b\in\operatorname{Hom}_2(X,Y')$, then 
$$
\chi_s-\chi_t=\p\left(\begin{array}{cc} 0 & b \\ 0 & 0 \end{array}\right)\in B_0(C(f),C(g)),
$$
meaning that $\chi_s$ and $\chi_t$ are chain homotopic.

(iii) \emph{Lifts of $B_0$ modulo $B_1$}. Denote $B_1=B_1(X,Y)$, $Z_1=Z_1(X,Y)$, $\operatorname{Hom}_1=\operatorname{Hom}_1(X,Y)$. Let 
$$
V_1\subset \operatorname{Hom}_1
$$ 
be a subspace such that $V_1\cap Z_1=B_1$ and $V_1+Z_1=\operatorname{Hom}_1$. Equivalently, $B_1\subset V_1$ is a subspace and $\p$ induces an isomorphism $V_1/B_1\stackrel\simeq\rightarrow B_0$. We call $V_1$ a \emph{linear lift of $B_0$ modulo $B_1$}.

Let such a linear lift $V_1\subset \operatorname{Hom}_1(X,Y)$ be given. Given two homotopic maps $f,g\in \operatorname{Hom}_0(X,Y)$, i.e. $f-g=\p(s)$, we can assume without loss of generality that $s\in V_1$. The map $s$ is uniquely defined modulo $B_1$, which implies that the homotopy class of the map $\chi_s:C(f)\to C(g)$ is well-defined. 

Thus, given a lift $V_1\subset \operatorname{Hom}_1(X,Y)$, the cone of any map $X\to Y$ is uniquely defined in {\sf Kom}.

\subsection{Directed, bi-directed, and doubly directed systems}\label{sec:directed-systems}

We now explain a setup in which one can speak of limits of ordered systems of mapping cones. {\color{black}The motivation for the definitions to follow lies in the definition of symplectic homology as a direct/inverse limit over directed systems in which the morphisms are Floer continuation maps in Floer homology. To this effect, the reader my find it useful to refer to \S\ref{sec:Floer-continuation} and~\S\ref{sec:transfer}.}
We begin with a few definitions. 

A \emph{directed set} is a partially ordered set $(I,\prec)$ such that for any $i,j$ there exists $k$ with $i,j\prec k$. An \emph{inversely directed set} is a partially ordered set $(I,\prec)$ such that for any $i,j$ there exists $\ell$ with $\ell\prec i,j$. Equivalently, we require that $I$ with the opposite order be a directed set. A \emph{bi-directed set} is a partially ordered set $(I,\prec)$ which is both directed and inversely directed. Our typical example is $I=\R$. 

A \emph{system in {\sf Kom} indexed by $I$ } is a collection of chain complexes $X(i)$, $i\in I$ together with chain maps $\varphi_i^j:X(i)\to X(j)$, $i\prec j$ such that
$\varphi_j^k \varphi_i^j=\varphi_i^k$ for $i\prec j\prec k$ and $\varphi_i^i=\mbox{Id}_{X(i)}$ in {\sf Kom}. More precisely, there exist maps $x_{ijk}\in \mbox{Hom}_1(X(i),X(k))$, $i\prec j \prec k$ and $x_i\in \mbox{Hom}_1(X(i),X(i))$ such that 
$$
\varphi_i^k-\varphi_j^k \varphi_i^j=\p(x_{ijk}),\qquad \mbox{Id}_{X(i)}-\varphi_i^i=\p(x_i).
$$
We speak of a \emph{directed system}, of an \emph{inversely directed system}, and of a \emph{bi-directed system} if $(I,\prec)$ is a directed set, an inversely directed set, respectively a bi-directed set. We call the maps $\varphi_i^j$ \emph{structure maps}. 

More generally, let $(I^+,\prec)$ be a directed set and $(I^-,\prec)$
be an inversely directed set. A \emph{doubly directed set modelled on
  $I^\pm$} is a subset $I\subset I^-\times I^+$ with the following two
properties: 
\begin{itemize}
\item if $(i,j)\in I$ then $(i',j)\in I$ for all $i'\prec i$ and
$(i,j')\in I$ for all $j'\prec j$; 
\item for every $j\in I^+$ there exists $i\in I^-$ such that $(i,j)\in
I$. 
\end{itemize}
Our typical example is $I^\pm=\R_\pm^*$ and $I=\{(a,b)\in
\R_-^*\times\R_+^*\, : \, a\le f(b)\}$, where $f:\R_+^*\to \R_-^*$ is
a decreasing function such that $f(b)\to-\infty$ as $b\to\infty$.  

 A \emph{doubly directed system in {\sf Kom} indexed by the doubly directed set $I$} is a collection of chain complexes $X(i,j)$, $(i,j)\in I$ together with chain maps $\varphi_{i'j}^{ij}:X(i',j)\to X(i,j)$ for $i'\prec i$ and $\varphi_{ij}^{ij'}:X(i,j)\to X(i,j')$ for $j\prec j'$ with respect to which every $X(i,\cdot)$ is a directed system and every $X(\cdot,j)$ is an inversely directed system, and such that all diagrams 
\begin{equation} \label{eq:diag-doubly-directed}
 \xymatrix{
 X(i',j) \ar[r] \ar[d] & X(i,j) \ar[d] \\
 X(i',j') \ar[r] & X(i,j') 
 }
 \end{equation} 
are commutative in {\sf Kom}, for any choice of indices such that $i'\prec i$, $j\prec j'$ and $(i,j),(i',j),(i,j'),(i',j')\in I$. We call the maps $\varphi_{i'j}^{ij}$ and $\varphi_{ij}^{ij'}$ \emph{structure maps}.

Given a map of bi-directed systems or a map of doubly directed systems, which means a collection of chain maps indexed by the relevant indexing set which commute in {\sf Kom} with the chain maps defining each of the systems, we are interested in understanding conditions under which the cone of that map is itself a bi-directed, respectively a doubly directed system. The two situations are similar, except for more cumbersome notation in the case of doubly directed systems since we need to work with two indexing variables $(i,j)$ rather than with just one index variable $i$. For this reason we shall focus in the sequel on bi-directed systems and indicate how the discussion adapts to doubly directed systems.

Let $\{X(i),\varphi_i^j\}$, $\{Y(i),\psi_i^j\}$ be two bi-directed systems in {\sf Kom} with the same index set $I$. A map of bi-directed systems in {\sf Kom} is a collection of chain maps $f_i:X(i)\to Y(i)$, $i\in I$ such that $\psi_i^jf_i$ and $f_j\varphi_i^j$ are homotopic for all $i\prec j$. Given $s_i^j\in\operatorname{Hom}_1(X(i),Y(j))$, $i\prec j$ such that $\psi_i^jf_i-f_j\varphi_i^j=\p(s_i^j)$, denote $\chi_i^j=\chi_{s_i^j}$. We then have a commutative diagram
$$
\xymatrix
@C=30pt
@R=30pt
{\ar @{} [dr] |{s_i^j} X(i) \ar[r]^{f_i} \ar[d]_{\varphi_i^j} & Y(i) \ar[d]^{\psi_i^j} \ar[r] & C(f_i) \ar[d]^{\chi_i^j} \ar[r] & X(i)[-1] \ar[d] \\ 
X(j) \ar[r]_{f_j} & Y(j) \ar[r] & C(f_j) \ar[r] & X(j)[-1] 
} 
$$ 
We are interested in finding conditions under which $\{C(f_i),\chi_i^j\}$ is a bi-directed system in {\sf Kom}. 

Let us consider the following condition: 

(B) \emph{There exists a collection $\{b_{ijk}\}$, $i\prec j \prec k$ with $b_{ijk}\in\mbox{Hom}_1(X(i),Y(k))$ such that}
$$
s_i^k-\psi_j^ks_i^j-s_j^k\varphi_i^j+f_k x_{ijk} - y_{ijk}f_i=\p(b_{ijk}),\qquad i,j,k.
$$
Here it is understood that $\{x_{ijk}\}$, $\{y_{ijk}\}$ and $\{s_i^j\}$ are given as above.
A direct computation then shows that  
$$
\chi_i^k-\chi_j^k\chi_i^j=\p\left(\begin{array}{cc} y_{ijk} & b_{ijk} \\ 0 & -x_{ijk} \end{array}\right), \qquad i,j,k.
$$ 
Indeed, the off-diagonal term on the left hand side is $s_i^k-\psi_j^ks_i^j-s_j^k\varphi_i^j$, while the off-diagonal term on the right hand side is $\p(b_{ijk})-f_k x_{ijk} + y_{ijk}f_i$. 

\noindent {\bf Remark.} Condition (B) is motivated both by the outcome of preliminary computations for bi-directed systems in {\sf Ch} and by the example of Floer continuation maps discussed below.

Condition (B) is clearly independent of the choice of $\{s_i^j\}$, $\{x_{ijk}\}$, and $\{y_{ijk}\}$ up to homotopy. This motivates the stronger condition (C) below, of a more intrinsic nature. \textcolor{black}{For the statement, recall the notion of a lift of $B_0$ mod $B_1$ from~\S\ref{sec:cone-uniqueness}.(iii).}

(C) \emph{We are given the data of collections of lifts of $B_0$ mod $B_1$:}
$$
\{X_i^j\subset \operatorname{Hom}_1(X(i),X(j))\},\qquad i\prec j,
$$
$$
\{Y_i^j\subset \operatorname{Hom}_1(Y(i),Y(j))\},\qquad i\prec j,
$$
$$
\{V_i^j\subset \operatorname{Hom}_1(X(i),Y(j))\},\qquad i\prec j
$$
\emph{such that $(\psi_{j}^k)_*V_i^j\subset V_i^k$, $(\varphi_i^{j})^*V_j^k\subset V_i^k$, $(f_k)_*X_i^k\subset V_i^k$, and $(f_i)^*Y_i^k\subset V_i^k$}.

We claim that
$$
(C) \Longrightarrow (B).
$$
For the proof we start by choosing $s_i^j\in V_i^j$, $x_{ijk}\in X_i^k$, $y_{ijk}\in Y_i^k$. We then remark that $-y_{ijk}f_i+s_i^k+f_kx_{ijk}$ and $\psi_j^ks_i^j+s_j^k\varphi_i^j$ are both contracting homotopies for $\psi_j^k\psi_i^jf_i-f_k\varphi_j^k\varphi_i^j$, so that their difference is a cycle. Now condition (C) implies that both these homotopies lie in $V_i^k$, which implies that their difference is a boundary $\p(b_{ijk})$. 

Condition (B) implies that $\{C(f_i),\chi_i^j\}$ is a bi-directed system in {\sf Kom}. The same holds in particular under condition (C). 

We now indicate how the discussion adapts to the case of a map $\{f_{ij}:X(i,j)\to Y(i,j)\}$ between doubly directed systems indexed by the same doubly directed set $I$. Denote $\varphi_{i'j}^{ij},\varphi_{ij}^{ij'}$ the structure maps for $\{X(i,j)\}$, and denote $\psi_{i'j}^{ij},\psi_{ij}^{ij'}$ the structure maps for $\{Y(i,j)\}$. Denote $\sigma_{i'j}^{ij'}$, $\tau_{i'j}^{ij'}$ the homotopies that express the commutativity in {\sf Kom} of the diagrams~\eqref{eq:diag-doubly-directed}: 
$$
\varphi_{ij}^{ij'}\varphi_{i'j}^{ij} - \varphi_{i'j'}^{ij'}\varphi_{i'j}^{i'j'}=\p(\sigma_{i'j}^{ij'}),\qquad 
\psi_{ij}^{ij'}\psi_{i'j}^{ij} - \psi_{i'j'}^{ij'}\psi_{i'j}^{i'j'}=\p(\tau_{i'j}^{ij'}).
$$
Denote $s_{i'j}^{ij}$ and $s_{ij}^{ij'}$ the homotopies that express the fact that $f_{\cdot j}$ and $f_{i\cdot}$ are maps of directed systems. 

The analogue of condition~(B) for doubly-directed systems is the following:

($\tilde{\mbox{B}}$) \emph{We require condition $\mbox{(B)}$ to hold for each of the maps of directed systems 
$f_{i\cdot}$ and $f_{\cdot j}$, and in addition we require that there exists a collection $\{B_{i'j}^{ij'}\}$ with $B_{i'j}^{ij'}\in \mbox{Hom}_1(X(i',j),Y(i,j'))$ such that}
$$
\psi_{ij}^{ij'}s_{i'j}^{ij} + s_{ij}^{ij'}\varphi_{i'j}^{ij} - \psi_{i'j'}^{ij'}s_{i'j}^{i'j'}-s_{i'j'}^{ij'}\varphi_{i'j}^{i'j'} + f_{ij'}\sigma_{i'j}^{ij'}-\tau_{i'j}^{ij'}f_{i'j}=\p(B_{i'j}^{ij'}).
$$
Similarly to the case of bi-directed systems, a direct computation shows that 
$$
\chi_{ij}^{ij'}\chi_{i'j}^{ij}-\chi_{i'j'}^{ij'}\chi_{i'j}^{i'j'}=\p\left(\begin{array}{cc} \tau_{i'j}^{ij'} & B_{i'j}^{ij'} \\ 0 & -\sigma_{i'j}^{ij'}\end{array}\right),
$$
where $\chi_{ab}^{cd}:C(f_{ab})\to C(f_{cd})$ are the maps induced between cones, as before. It is important to note that condition~($\tilde{\mbox{B}}$) is of the same nature as condition (B), and the only difference between the two is that condition ($\tilde{\mbox{B}}$)  takes into account the additional conditions of commutativity up to homotopy which are involved in the definition of a doubly directed system.

One can also phrase for doubly directed systems an analogue ($\tilde{\mbox{C}}$) of condition (C) for bi-directed systems, but we shall not need it and therefore we do not make it explicit.

\noindent \emph{Limiting objects.} Let now the coefficient ring be a field $\K$, and recall~\cite{ES52} that the inverse limit functor is exact on inversely directed systems consisting of finite dimensional vector spaces. Let $\{f_{ij}:X(i,j)\to Y(i,j)\}$ be a map of doubly directed systems, and assume that each $X(i,j)$ and $Y(i,j)$ has finite dimensional homology in each degree. Under condition ($\tilde{\mbox{B}}$) we obtain in the first-inverse-then-direct-limit a homology exact triangle
$$
\xymatrix
{\lim\limits^{\longrightarrow}_{j}\lim\limits^{\longleftarrow}_{i} H(X(i,j)) \ar[rr]^{\lim\limits^{\longrightarrow}_{j}\lim\limits^{\longleftarrow}_{i} (f_{ij})_*} && \lim\limits^{\longrightarrow}_{j}\lim\limits^{\longleftarrow}_{i} H(Y(i,j)) \ar[dl] \\
& \lim\limits^{\longrightarrow}_{j}\lim\limits^{\longleftarrow}_{i} H(C(f_{ij})) \ar[ul]_-{[-1]} &
}
$$

\noindent {\bf Remark.} The following question is relevant. When is  
$$
\lim\limits^{\longrightarrow}_{j}\lim\limits^{\longleftarrow}_{i}X(i,j) \longrightarrow
\lim\limits^{\longrightarrow}_{j}\lim\limits^{\longleftarrow}_{i}Y(i,j) \longrightarrow 
\lim\limits^{\longrightarrow}_{j}\lim\limits^{\longleftarrow}_{i} C(f_{ij}) \longrightarrow
\lim\limits^{\longrightarrow}_{j}\lim\limits^{\longleftarrow}_{i} X(i,j)[-1] 
$$
a (model) distinguished triangle? This is related to exactness criteria for the inverse limit functor and to the so-called Mittag-Leffler condition, see for example~\cite{CF} and the references therein. 

\subsection{Floer continuation maps}\label{sec:Floer-continuation}

We now show how condition ($\tilde{\mbox{B}}$) above is satisfied in the case of Floer continuation maps for a doubly directed system of Hamiltonians. In order to streamline the discussion we shall actually treat the case of a directed system of Hamiltonians, the case of doubly directed systems being conceptually equivalent, except for the more complicated notation.

\noindent \emph{Higher continuation maps.} Let $K\le L$ be two Hamiltonians and let $(FC(K),\p_K)$, $(FC(L),\p_L)$ be the Floer complexes for some choice of regular almost complex structures $J_K$ and $J_L$. An $s$-dependent Hamiltonian $H=H_s$, $s\in \R$ such that $H_s=L$ for $s\ll 0$, $H_s=K$ for $s\gg 0$, and $\p_sH\le 0$, together with an $s$-dependent almost complex structure interpolating between $J_L$ and $J_K$, determines a degree $0$ chain map  
$$
\varphi_H: FC(K)\to FC(L).
$$
We refer to $H$ as a \emph{decreasing Hamiltonian homotopy (from $L$ to $K$)}, and to $\varphi_H$ as the associated \emph{continuation map}. 

Given two decreasing Hamiltonian homotopies $H^0$ and $H^1$ from $L$ to $K$, the choice of a homotopy $\{H^\lambda\}$, $\lambda\in [0,1]$ between the two, together with the choice of a homotopy of almost complex structures which we ignore from the notation, determines a degree $1$ map 
$$
\varphi_{\{H^\lambda\}}:FC(K)\stackrel {+1}\longrightarrow  FC(L).
$$
We refer to $\{H^\lambda\}$ as a \emph{homotopy of homotopies, or $1$-homotopy}, and to $\varphi_{\{H^\lambda\}}$ as the associated \emph{degree $1$ continuation map}. This is in general not a chain map. However, it is a chain homotopy between $\varphi_{H^0}$ and $\varphi_{H^1}$: 
$$
\varphi_{H^1}-\varphi_{H^0} =\p(\varphi_{\{H^\lambda\}}) = \p_K\varphi_{\{H^\lambda\}} + \varphi_{\{H^\lambda\}}\p_H.
$$

We now go one step further. Given two $1$-homotopies $\{H_\mu^0\}$ and $\{H_\mu^1\}$, $\mu\in [0,1]$ the choice of a homotopy $\{H_\mu^\lambda\}$, $\lambda\in[0,1]$ connecting them, together with the choice of a homotopy of homotopies of almost complex structures which we ignore from the notation, determines a degree $2$ map
$$
\varphi_{\{H_\mu^\lambda\}}:FC(K)\stackrel {+2}\longrightarrow  FC(L). 
$$
We refer to $\{H_\mu^\lambda\}$ as a \emph{$2$-homotopy}, and to $\varphi_{\{H_\mu^\lambda\}}$ as the associated \emph{degree $2$ continuation map}. This is in general not a chain map. However, if $\{H^0_\mu\}$ and $\{H^1_\mu\}$ coincide at $\mu=0$ and at $\mu=1$, and if $\{H_\mu^\lambda\}$ is constant at $\mu=0$ and at $\mu=1$, the map $\varphi_{\{H_\mu^\lambda\}}$ is a contracting chain homotopy for $\varphi_{\{H_\mu^1\}}-\varphi_{\{H_\mu^0\}}$:
$$
\varphi_{\{H_\mu^1\}}-\varphi_{\{H_\mu^0\}} = \p(\varphi_{\{H_\mu^\lambda\}}).
$$

More generally, denote $I=[0,1]$ and, for $d\ge 0$, consider the $d$-dimensional cube $I^d$. (If $d=0$ then $I^d$ consists of a single point.) A generic pair $\{H_{s,z},J_{s,z}\}$, $z\in I^d$, $s\in\R$ consisting of an $I^d$-family of decreasing Hamiltonian homotopies from $L$ to $K$ and of an $I^d$-family of $s$-dependent almost complex structures which all coincide with $J_L$ for $s\ll 0$ and with $J_K$ for $s\gg 0$, 
determines a map 
$$
\varphi_{\{H_{s,z},J_{s,z}\}}\in \operatorname{Hom}_d(FC(K),FC(L)). 
$$
This map is defined on a generator $x\in FC(K)$ by 
$$
x\mapsto \sum_{|x|-|y|=-d}\# \mathcal{M}(y,x;\{H_{s,z},J_{s,z}\})y
$$
and then extended by linearity. Here $\mathcal{M}(y,x;\{H_{s,z},J_{s,z}\})$ denotes the moduli space of solutions to the Floer equation in the chosen $I^d$-family, asymptotic to $y$ at $-\infty$ and asymptotic to $x$ at $+\infty$. 
In other words, the map $\varphi_{\{H_{s,z},J_{s,z}\}}$ counts index $-d$ solutions of the Floer equation within the $d$-dimensional family parameterized by $I^d$. 
We refer to $\{H_{s,z},J_{s,z}\}$ as a \emph{$d$-homotopy}, and to $\varphi_{\{H_{s,z},J_{s,z}\}}$ as the associated \emph{degree $d$ continuation map}. 

Let $\{H^0,J^0\}$ and $\{H^1,J^1\}$ be two $d$-homotopies \emph{which are equal on $\p I^d$}. For any choice of a $(d+1)$-homotopy $\{H^\lambda,J^\lambda\}$, $\lambda\in [0,1]$ which interpolates between the two, and which is \emph{constant on $(\p I^d)\times I\subset I^d\times I=I^{d+1}$}, the associated degree $d+1$ continuation map $\varphi_{\{H^\lambda,J^\lambda\}}$ is a contracting chain homotopy for $\varphi_{\{H^1,J^1\}}-\varphi_{\{H^0,J^0\}}$: 
$$
\varphi_{\{H^1,J^1\}}-\varphi_{\{H^0,J^0\}} = \p(\varphi_{\{H^\lambda,J^\lambda\}}).
$$

We have thus proved the following
\begin{lemma} \label{lem:continuation_maps}
The difference between any two degree $d$ continuation maps determined by $d$-homotopies which coincide on $\p I^d$ is homotopic to zero. A contracting homotopy is provided by any degree $d+1$ continuation map determined by an interpolating $(d+1)$-homotopy which is constant on $(\p I^d)\times I\subset I^d\times I=I^{d+1}$.

\hfill{$\square$}  
\end{lemma}

This statement generalizes to higher homotopies the well-known fact that any two continuation maps in Floer theory are homotopic, so that the morphism that they induce in homology is independent of all choices. This last property is sometimes referred to as Floer homology being a connected simple system in the sense of Conley.

\bigskip

\noindent \emph{Directed systems of continuation maps.} 

Let $\{K_i\}$, $\{L_i\}$ be two directed systems of Hamiltonians, meaning that $K_i\le K_j$ and $L_i\le L_j$ for $i\prec j$. Let $\{K_i^j\}$, $\{L_i^j\}$, $i\prec j$ be decreasing homotopies from $K_j$ to $K_i$, respectively from $L_j$ to $L_i$, yielding continuation maps $\varphi_i^j:FC(K_i)\to FC(K_j)$, $\psi_i^j:FC(L_i)\to FC(L_j)$. Then 
$$
\{FC(K_i),\varphi_i^j\},\qquad \{FC(L_i),\psi_i^j\}
$$  
are bi-directed systems in {\sf Kom}. 

Assume further that $K_i\le L_i$ for all $i$. Let $H_i$ be a decreasing homotopy from $L_i$ to $K_i$, yielding continuation maps $f_i:FC(K_i)\to FC(L_i)$. The collection $\{f_i\}$ is then a map of bi-directed systems in {\sf Kom}. 

Indeed, the maps $\psi_i^j f_i$ and $f_j\varphi_i^j$ are homotopic via a degree $1$ continuation map
$$
s_i^j:FC(K_i)\stackrel{+1}\longrightarrow FC(L_j)
$$
that is associated to a $1$-homotopy $\cH_i^j$ connecting $L_i^j\# H_i$ and $H_j\# K_i^j$. Here $\#$ denotes the gluing of Hamiltonians for a large enough value of the gluing parameter. 

Similarly, the maps $\varphi_i^k$ and $\varphi_j^k\varphi_i^j$, respectively $\psi_i^k$ and $\psi_j^k\psi_i^j$, are homotopic via degree $1$ maps
$$
x_{ijk}:FC(K_i)\stackrel{+1}\longrightarrow FC(K_k),\qquad y_{ijk}:FC(L_i)\stackrel{+1}\longrightarrow FC(L_k),
$$
that are associated to $1$-homotopies $K_{ijk}$ connecting $K_i^k$ and $K_j^k\# K_i^j$, respectively $L_{ijk}$ connecting $L_i^k$ and $L_j^k\# L_i^j$. 

We claim that condition (B) is satisfied in this setup. In view of Lemma~\ref{lem:continuation_maps} it is enough to show that both $\psi_j^ks_i^j + s_j^k\varphi_i^j$ and $f_kx_{ijk}+s_i^k-y_{ijk}f_i$ are degree $1$ Floer continuation maps induced by $1$-homotopies parameterized by $\lambda\in [0,1]$ with the same endpoints $L_j^k\#L_i^j\#H_i$ at $\lambda=0$ and $H_k\# K_j^k\# K_i^j$ at $\lambda=1$. Consider the following diagram, where in each entry we have indicated a composition of Floer continuation maps and the $0$-homotopy which induces it, and where on each arrow we have indicated a homotopy between the target and source maps, together with the $1$-homotopy which induces it. The main point is that a concatenation of $1$-homotopies induces the sum of the corresponding degree $1$ maps, and the reversal of the direction of a $1$-homotopy induces minus the corresponding degree $1$ map. The composition of the bottom horizontal arrows is thus a degree $1$-continuation map which equals $\psi_j^ks_i^j+s_j^k\varphi_i^j$, while the composition of the other three arrows is a degree $1$ continuation map which equals $f_kx_{ijk}+s_i^k-y_{ijk}f_i$. The corresponding $1$-homotopies do have the same endpoints at $\lambda=0$ and $\lambda=1$, as expected. 
$$
\xymatrix
@C=50pt
@R=5pt
{
\psi_i^kf_i \ar[rr]^{s_i^k}_{\cH_i^k}& & f_k\varphi_i^k \\
L_i^k\# H_i \ar[dddddddd]_{y_{ijk}f_i}^{L_{ijk}\# H_i} & & {\color{black}H_k}\# K_i^k \ar[dddddddd]^{f_kx_{ijk}}_{H_k\#K_{ijk}} \\
& & \\
& & \\
& & \\
& & \\
& & \\
& & \\
& & \\
\psi_j^k\psi_i^jf_i \ar[r]^{\psi_j^ks_i^j}_{L_j^k\#\cH_i^j} & \psi_j^kf_j\varphi_i^j \ar[r]^{s_j^k\varphi_i^j}_{\cH_j^k\# K_i^j} & f_k\varphi_j^k\varphi_i^j \\
L_j^k\# L_i^j\# H_i & L_j^k\# H_j \# K_i^j & H_k\# K_j^k\# K_i^j
}
$$
{\color{black} It follows from the results in Section~\ref{sec:directed-systems} that
the system 
$$
\{C(f_i),\chi_i^j\}
$$
of cones $C(f_i)$ and induced maps $\chi_i^j:C(f_i)\to C(f_i)$} is a directed system in {\sf Kom}. In particular the homotopy type of the maps $\chi_i^j$ does not depend on the choice of $1$-homotopies. 

Similarly, for a doubly directed system of Hamiltonians we obtain a doubly directed system 
$$
\{C(f_{ij}),\chi_{ab}^{cd}\}
$$
in {\sf Kom}, together with the fact that the homotopy type of the maps $\chi_{ab}^{cd}$ does not depend on the choice of $1$-homotopies.

\section{The transfer map and homotopy invariance}

Given a Liouville cobordism pair $(W,V)$ we construct in this section a transfer map 
$$
f_!^\heartsuit:SH_*^\heartsuit(W)\to SH_*^\heartsuit(V)
$$
for $\heartsuit\in\{\varnothing,\ge 0, >0,=0,\le 0,<0\}$ that is
invariant under homotopy of Liouville structures. This generalizes to cobordisms the transfer map defined
for Liouville domains by Viterbo in~\cite{Viterbo99}. The whole
structure that we exhibit on symplectic homology is actually governed
by the underlying chain level map. Indeed, we prove
in~\S\ref{sec:exact-triangle-pair} that the shifted symplectic
homology groups of the pair $SH_*^\heartsuit(W,V)[-1]$ are isomorphic
to the homology of the cone of the chain level transfer map. 

We recall that we use coefficients in a field $\K$.

\subsection{The transfer map} \label{sec:transfer}

{\color{black}Let $(W,V)$ be a Liouville cobordism pair with filling $F$. Recall from~\S\ref{sec:SHWA} the definition of the symplectic homology groups
$$
SH_*^\heartsuit(W)=\lim_b \lim_a \lim^{\longrightarrow}_{H\in\cH(W;F)} FH_*^{(a,b)}(H),
$$
where $\cH(W;F)$ is the class of Hamiltonians $H:S^1\times \wh W_F\to \R$ which are zero on $W$ and are linear of non-critical slope in the complement of $W_F$, and the meaning of the limits involving $a$ and $b$ is determined by the value of $\heartsuit$. In the previous formula the first direct limit is considered with respect to continuation maps $FH_*^{(a,b)}(H_+)\to FH_*^{(a,b)}(H_-)$ for $H_+\le H_-$ induced by non-increasing homotopies $H_s$, $s\in \R$ which are equal to $H_\pm$ for $s$ near $\pm\infty$.

The transfer map will be defined as a limit of a directed system of continuation maps. For that purpose the definition of $SH_*^\heartsuit(V)$, which involves Hamiltonians defined on $\wh V_{F\circ W^{bottom}}=F\circ W^{bottom}\circ V\circ [1,\infty)\times\p^+V$, needs to be recast in terms of Hamiltonians defined on $\wh W_F=F\circ W\circ [1,\infty)\times\p^+W$. The manifold $\wh W_F$ is the domain of the Hamiltonians involved in the definition of $SH_*^\heartsuit(W)$. 

Denote by $\cH^W(V;F)$ the space of Hamiltonians $H:S^1\times \wh W_F\to\R$ such that $H\in\cH(\wh W_F)$ and $H=0$ on $V$. 

\begin{lemma} \label{lem:SHV-alternate-direct-limit}
For any two real numbers $-\infty < a < b <\infty$ we have 
$$
SH_*^{(a,b)}(V)= \lim^{\longrightarrow}_{H\in\cH^W(V;F)} FH_*^{(a,b)}(H). 
$$ 
\end{lemma}

\begin{proof}
By definition we have 
$$
SH_*^{(a,b)}(V)= \lim^{\longrightarrow}_{H\in\cH(V;F)} FH_*^{(a,b)}(H), 
$$
and we claim that the two limits are equal. Recall that the space $\cH(V;F)$ consists of Hamiltonians $H:\wh V_{F\circ W^{bottom}}\to\R$ which are linear outside a compact set and such that $H=0$ on $V$. The claim is a consequence of the existence of a special cofinal family in $\cH^W(V;F)$ constructed as follows. See Figure~\ref{fig:KH-new}.
Consider a sequence $(\nu_k)$, $k\in\Z_-$ of positive real numbers such that $\nu_k\notin\mbox{Spec}(\p^+V)$ and $\nu_k\to\infty$ as $k\to\infty$, and let $H^V_k:\wh V_{F\circ W^{bottom}}\to\R$ be a cofinal family in $\cH(V;F)$ such that $H^V_k(r,x)=\nu_k(r-1)$ on $[1,\infty)\times\p^+V$. Consider further sequences 
$$
(\eta_k),\quad (R_k),\quad (\tau_k),\qquad k\in\Z_+ 
$$
such that 
\begin{itemize}
\item $\eta_k>0$ is smaller than the distance from $\nu_k$ to $\mbox{Spec}(\p^+V)$, and $\eta_k\to 0$ as $k\to\infty$;
\item $R_k>\max(1,(\nu_k - a)/\eta_k)$;
\item $\nu_k/4 <\tau_k<\nu_k/2$ and $\tau_k\notin\mbox{Spec}(\p^+W)$.  
\end{itemize}
Let $H_k:\wh W_F\to\R$ be a Hamiltonian which is equal to $H^V_k$ on $F\circ W^{bottom}\circ V \circ [1,R_k]\times\p^+V$, which is constant equal to $\nu_k(R_k-1)$ on $R_kW^{top}$, and which is equal to $\nu_k(R_k-1)+\tau_k(r-R_k)$ on $[R_k,\infty)\times\p^+W$. Here $R_kW^{top}$ stands for the image of $W^{top}$ by the flow of the Liouville vector field at time $\ln R_k$.

The Hamiltonian $H_k$ has three more groups of $1$-periodic orbits in addition to those of the Hamiltonian $H^V_k$: 
\begin{itemize}
\item[($III^-$)]orbits corresponding to positively parameterized closed Reeb orbits on $\p^+V=\p^-W^{top}$ and located near $R_k\p^+V$. 
\item [($III^0$)]constants in $R_kW^{top}$.
\item [($III^+$)]orbits corresponding to positively parameterized closed Reeb orbits on $\p^+W=\p^+W^{top}$ and located near $R_k\p^+W^{top}$. 
\end{itemize}
The orbits in group $III^0$ have action $-\nu_k(R_k-1)$, the maximal
action of an orbit in group $III^-$ is smaller than
$-\nu_k(R_k-1)+R_k(\nu_k-\eta_k)=\nu_k-R_k\eta_k$, and the maximal
action of an orbit in group $III^+$ is smaller than
$-\nu_k(R_k-1)+R_k\nu_k/2 = -\nu_k(R_k/2-1)$. The largest of these
actions is the one in group $III^-$, which however falls below the
action window $(a,b)$ due to the condition $R_k>\max(1,(\nu_k -
a)/\eta_k)$, so that the orbits contributing to the Floer complex in
the action window $(a,b)$ are the same for $H^V_k$ and for $H_k$. 
Lemma~\ref{lem:no-escape} for $s$-dependent Hamiltonians (decreasing
in $s$ outside $V_{F\circ W^{bottom}}$)
shows that the continuation Floer trajectories for the family $H^V_k$ and for the family $H_k$ stay within a neighborhood of $V_{F\circ W^{bottom}}$, where the two Hamiltonians coincide. These continuation Floer trajectories are therefore the same, and they define the same continuation maps in the two directed systems at hand. We obtain 
$$
SH_*^{(a,b)}(V) = \lim^{\longrightarrow}_{k\to\infty}FH_*^{(a,b)}(H^V_k) = \lim^{\longrightarrow}_{k\to\infty}FH_*^{(a,b)}(H_k). 
$$
Since $H_k$, $k\in\Z_-$ is a cofinal family in $\cH^W(V;F)$, the conclusion of the Lemma follows. 
\end{proof}

We obviously have $\cH(W;F)\subset \cH^W(V;F)$, and for each Hamiltonian $K$ in $\cH(W;F)$ there exists a Hamiltonian $H$ in $\cH^W(V;F)$ such that $K\le H$ (while the converse is \emph{not} true). For any two such Hamiltonians we have continuation maps 
$$
f_{HK}^{(a,b)}:FC_*^{(a,b)}(K)\to FC_*^{(a,b)}(H)
$$ 
induced by non-increasing homotopies which are linear at infinity, and these continuation maps define a morphism between the directed systems determined by $\cH(W;F)$ and $\cH^W(V;F)$. 

\begin{definition}
The \emph{Viterbo transfer map in the action window $(a,b)$} is the limit continuation map 
$$
f_!^{(a,b)} : SH_*^{(a,b)}(W)\to SH_*^{(a,b)}(V), \qquad f_!^{(a,b)}
:= \lim^{\longrightarrow}_{\stackrel{K\le H}{K\in\cH(W;F),\, H\in\cH^W(V;F)}}f_{HK}^{(a,b)}. 
$$
\end{definition}

By general properties of the continuation maps the Viterbo transfer maps $f_!^{(a,b)}$ fit into a doubly-directed system, inverse on $a$ and direct on $b$.

\begin{definition} \label{defi:transfer-map}
For $\heartsuit\in\{\varnothing,\ge 0, >0,=0,\le 0,<0\}$ the \emph{Viterbo transfer map} 
$$
f_!^\heartsuit:SH_*^\heartsuit(W)\to SH_*^\heartsuit(V)
$$
is defined as
$$
f_!^\heartsuit = \lim_b \lim_a f_!^{(a,b)},
$$ 
where the limits are inverse or direct according to the value of $\heartsuit$, as in Definition~\ref{defi:SH(W)}. 
\end{definition}

}

\begin{figure}
         \begin{center}
\input{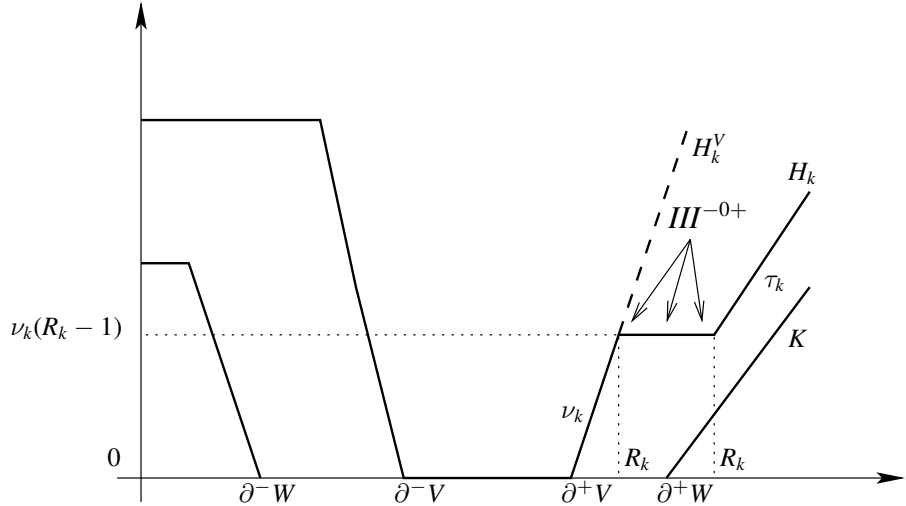}
         \end{center}
\caption{Hamiltonians for the definition of the transfer map  \label{fig:KH-new}}
\end{figure}

\begin{proposition}[Functoriality of the transfer map] \label{prop:func-transfer}
Let $U\subset V\subset W$ be a triple of Liouville cobordisms with
filling. Let 
$f_{VW}^\heartsuit$, $f_{UW}^\heartsuit$, $f_{UV}^\heartsuit$ be the transfer maps for the pairs $(W,V)$, $(W,U)$, and $(V,U)$ respectively, for $\heartsuit\in\{\varnothing,\ge 0, >0,=0,\le 0,<0\}$. Then 
$$
f_{UW}^\heartsuit=f_{UV}^\heartsuit\circ f_{VW}^\heartsuit.
$$
\end{proposition}

\begin{proof} 
{\color{black}This is a direct consequence of the definition of the transfer map as a limit continuation map, together with functoriality of continuation maps. To see this, we recall the notation $W=W^{bottom}\circ V\circ W^{top}$ and $V=V^{bottom}\circ U \circ V^{top}$, and consider on $W$ the following three types of Hamiltonians, see Figures~\ref{fig:KH-new} and~\ref{fig:KHG-new}: 
\begin{itemize}
\item Hamiltonians $K$ which are admissible for $W$, and thus vanish on $W$ and are linear increasing towards $\p^+ W$. 
\item \emph{one step} Hamiltonians $H$ which vanish on $V$, take a positive constant value on $W^{top}$, and are linear increasing towards $\p^+ V$ and $\p^+ W$.
\item \emph{two step} Hamiltonians $G$ which vanish on $U$, take a constant value on $V^{top}$, take a constant value on $W^{top}$, and are linear increasing towards $\p^+ U$, $\p^+ V$, and $\p^+ W$.
\end{itemize}
The transfer maps $f^\heartsuit_{VW}$ are defined above as limit
continuation maps induced by monotone homotopies from $K$ (at
$+\infty$) to $H$ (at $-\infty$). Similarly, the transfer maps
$f^\heartsuit_{UW}$ can be obtained as limit continuation maps induced
by monotone homotopies from $K$ (at $+\infty$) to $G$ (at $-\infty$),
and the transfer maps $f^\heartsuit_{UV}$ can be obtained as limit
continuation maps induced by monotone homotopies from $H$ (at
$+\infty$) to $G$ (at $-\infty$). We can choose the homotopies from
$K$ to $G$ to factor through $H$, so that they can be expressed as
concatenation of homotopies from $K$ to $H$, and from $H$ to $G$. The
composition of the continuation maps induced by each of these last two
homotopies is equal to the continuation map induced by the
concatenation of the two homotopies -- this is what we call
\emph{functoriality of continuation maps} -- and the same property
holds in the limit. This proves $f_{UW}^\heartsuit=f_{UV}^\heartsuit\circ f_{VW}^\heartsuit$. 
}
\end{proof}

In the sequel we shall often drop the symbol $\heartsuit$ from the notation for the transfer map, and simply write $f_!$ instead of $f^{\heartsuit}_!$.

\begin{figure}
         \begin{center}
\input{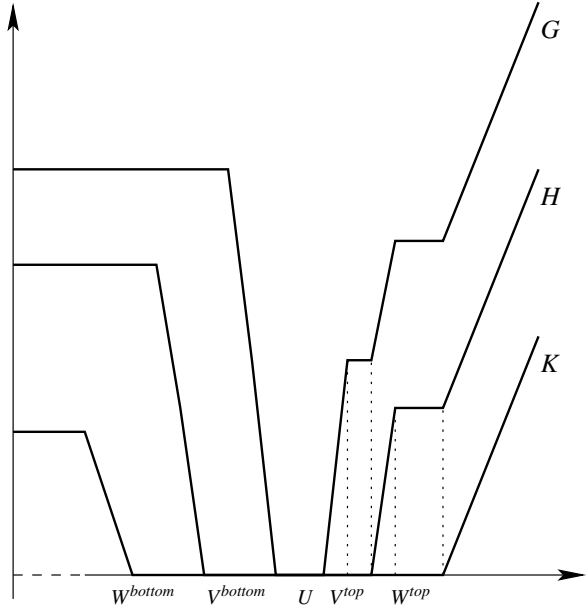}
         \end{center}
\caption{Hamiltonians for the proof of functoriality of the transfer map  \label{fig:KHG-new}}
\end{figure}

\subsection{Homotopy invariance of the transfer map} \label{sec:homotopy_invariance_transfer} 
Given a pair of Liouville cobordisms $(W,V)$ with filling, we denote the transfer map for a given Liouville structure $\lambda$ by 
$$
\xymatrix{
SH_*^\heartsuit(W;\lambda)\ar[r]^{f_{!,\lambda}} & SH_*^\heartsuit(V;\lambda).
}
$$

\begin{proposition}[homotopy invariance of the transfer map]\label{prop:homotopy_invariance_transfer}
Let $(W,V)$ be a pair of Liouville cobordisms with filling. Given a
homotopy of Liouville structures $\lambda_t$ on $W$, $t\in [0,1]$, there are induced isomorphisms $h_W:SH_*^\heartsuit(W;\lambda_0)\to SH_*^\heartsuit(W;\lambda_1)$, $h_V:SH_*^\heartsuit(V;\lambda_0)\to SH_*^\heartsuit(V;\lambda_1)$, and a commutative diagram 
$$
\xymatrix
@C=25pt
{
SH_*^\heartsuit(W;\lambda_0) \ar[r]^{f_{!,\lambda_0}} \ar[d]_\cong^{h_W} & SH_*^\heartsuit(V;\lambda_0) \ar[d]_\cong^{h_V}  \\
SH_*^\heartsuit(W;\lambda_1) \ar[r]_{f_{!,\lambda_1}} & SH_*^\heartsuit(V;\lambda_1) 
}
$$
The isomorphisms $h_W$ and $h_V$ do not depend on the choice of homotopy $\lambda_t$ with fixed endpoints.
\end{proposition}

\begin{proof}[Proof] 
The homotopy invariance of the transfer
  map under deformations of the Liouville structure which are constant
  along the boundaries of $W$ and $V$ is a consequence of its
  definition as a limit continuation map. 
In particular, given a Liouville cobordism $W$ with two Liouville structures $\lambda$ and $\lambda'$ which coincide along $\p W$, the transfer map 
$$
SH_*^\heartsuit(W;\lambda)\to SH_*^\heartsuit(W;\lambda')
$$
is an isomorphism. 

The homotopy invariance in the general case is obtained using the functoriality of the transfer map, by a classical geometric construction which consists in attaching to $\p W$ topologically trivial cobordisms with Liouville structures that interpolate between any two given Liouville structures on the boundary of $W$, see~\cite[Lemma~3.7]{Ci02}. A detailed argument is given in~\cite{Gutt15} in an $S^1$-equivariant setting.

That the isomorphisms $h_W$ and $h_V$ do not depend on the choice of homotopy $(\lambda_t)$, $t\in [0,1]$ is a consequence of the fact that any two such homotopies with the same endpoints are homotopic, together with the usual ``homotopy of homotopies" argument in Floer theory (see also the discussion of Floer continuation maps at the end of~\S\ref{sec:cones}). 
\end{proof}

\section{Excision} \label{sec:excision}

Let $(W,V)$ be a pair of Liouville cobordisms and $F$ a filling of $W$, and define $W_F$, $\wh W_F$ as in~\S\ref{sec:SHWA}. Recall the class $\cH(W,V;F)$ of admissible Hamiltonians defined in~\S\ref{sec:SHpair}. 
For $0<r_1<r_2$ and a subset $A\subset\wh W_F$, we denote by $[r_1,r_2]\times A=\phi_{[\log r_1,\log r_2]}(A)$ the image of $A$ under the Liouville flow $\phi_t$ on the time interval $[\log r_1,\log r_2]$. For parameters
$$
   \mu,\nu,\tau>0,\qquad 0<\delta,\eps<1 
$$
(that will be specified later), let $H\in\cH(W,V;F)$ be a ``staircase Hamiltonian" on $\wh W_F$, defined up to smooth approximation as follows (see Figure~\ref{fig:H-cob}):
\begin{itemize}
\item $H\equiv(1-\delta)\mu$ on $F\setminus(\delta,1]\times\p^-W$,
\item $H$ is linear of slope $-\mu$ on $[\delta,1]\times\p^-W$,
\item $H\equiv 0$ on $W^{bottom}$,
\item $H$ is linear of slope $-\nu$ on $[1,1+\eps]\times\p^-V$,
\item $H\equiv -\eps\nu$ on $V\setminus\bigl([1,1+\eps]\times\p^-V\cup [1-\eps,1]\times\p^+V\bigr)$,
\item $H$ is linear of slope $\nu$ on $[1-\eps,1]\times\p^+V$,
\item $H\equiv 0$ on $W^{top}$,
\item $H$ is linear of slope $\tau$ on $[1,\infty)\times\p^+W$. 
\end{itemize} 
{\color{black} A smooth approximation of $H$ will thus be of the form
$H(r,y)=h(r)$ on $[0,\infty)\times\p^+W$ (and similarly near the other
boundary components of $W$ and $V$). Hence $1$-periodic orbits of $X_H$ 
on $\{r\}\times\p^+W$ correspond to Reeb orbits on $\p^+W$ of period
  $h'(r)$, and their Hamiltonian action equals
$$
   rh'(r)-h(r).
$$}
We assume that $\mu,\nu,\nu,\tau$ do not lie in the action spectrum of
$\p^-W,\p^-V,\p^+V,\p^+W$, respectively. We denote by $\eta_\nu>0$
{\color{black}a positive real number smaller than} the distance from
$\nu$ to the union of the action spectra of $\p^-V$ and $\p^+V$, and we define similarly $\eta_\mu,\eta_\tau>0$. The $1$-periodic orbits of $H$ fall into $11$ classes:
\begin{itemize}
\item[($F^0$)] constants in $F\setminus ([\delta,1]\times \p F)$,
\item[($F^+$)] orbits corresponding to negatively parameterized
  closed Reeb orbits on\break $\p F=\p^-W$ and located near $\delta\times\p^-W$,
\item[($I^-$)] orbits corresponding to negatively parameterized closed Reeb orbits on\break $\p^-W^{bottom}=\p^-W$ and located near $\p^-W$,
\item[($I^0$)] constants in $W^{bottom}$,
\item[($I^+$)] orbits corresponding to negatively parameterized closed Reeb orbits on\break $\p^+W^{bottom}=\p^-V$ and located near $\p^-V$,
\item[($II^-$)] orbits corresponding to negatively parameterized closed Reeb orbits on $\p^-V$ and located near $(1+\eps)\times\p^-V$,
\item[($II^0$)] constants in $V\setminus\bigl([1,1+\eps]\times\p^-V\cup [1-\eps,1]\times\p^+V\bigr)$,
\item[($II^+$)] orbits corresponding to positively parameterized closed Reeb orbits on $\p^+V$ and located near $(1-\eps)\times\p^+V$,
\item[($III^-$)] orbits corresponding to positively parameterized closed Reeb orbits on\break $\p^-W^{top}=\p^+V$ and located near $\p^+V$,
\item[($III^0$)] constants in $W^{top}$,
\item[($III^+$)] orbits corresponding to positively parameterized closed Reeb orbits on $\p^+W$ and located near $\p^+W^{top}=\p^+W$. 
\end{itemize}

\begin{figure}
         \begin{center}
\input{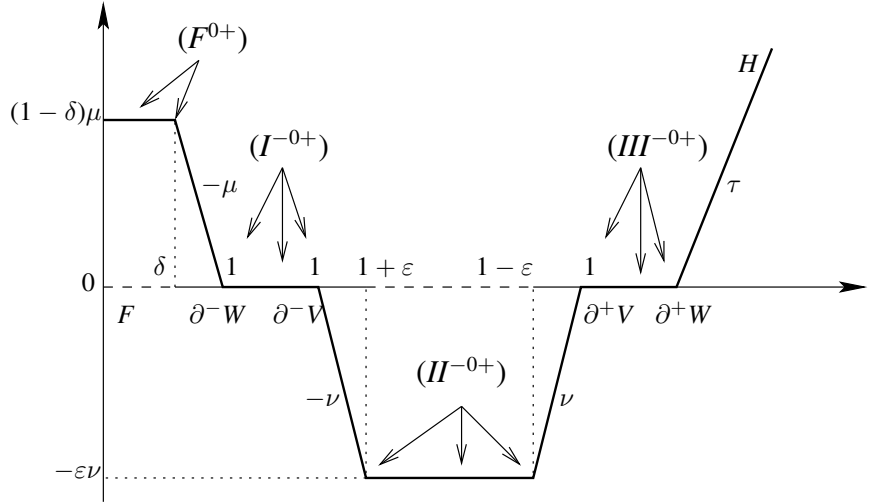}
         \end{center}
\caption{Hamiltonian in $\mathcal{H}(W,V;F)$ \label{fig:H-cob}}
\end{figure}

{\bf Notational convention.} For two classes of orbits $A,B$ we write $A\prec B$ if the homological Floer boundary operator maps no orbit from $A$ to an orbit from $B$. A priori, this relation is not transitive. However, when we write $A\prec B\prec C$ we also mean that $A\prec C$. We write $A<B$ if all orbits in $A$ have smaller action than all orbits in $B$. Note that $A<B$ implies $A\prec B$, and $A<B<C$ implies $A\prec B\prec C$.    

\begin{lemma}\label{lem:order}
Fix $a<b$. If the parameters $\mu,\nu,\tau,\delta,\eps$ above satisfy
\begin{equation}\label{eq:par}
   (1-\delta)\mu>\min\{-a,\nu-\eta_\nu\} \quad\text{and}\quad \eps\nu > \min\{b,\tau-\eta_\tau\},
\end{equation}
and if we use an almost complex structure that is cylindrical and has a long enough neck 
near $(1-2\varepsilon)\times \p^+V$, 
then the four groups of orbits in the action interval $[a,b]$ satisfy
\begin{equation}\label{eq:order}
   F \prec I \prec III\prec II\quad\text{and}\quad III\prec I.
\end{equation}
Moreover, within each group of orbits we have the relations
\begin{equation}\label{eq:order-within-groups}
\begin{gathered}
   F^+\prec F^0, \qquad
   I^+\prec I^-\prec I^0, \cr 
   II^-\prec II^0\prec II^+, \qquad 
   III^0\prec III^-\prec III^+. 
\end{gathered}
\end{equation}
\end{lemma}

\begin{proof}
The combination of Lemmas~\ref{lem:no-escape} and~\ref{lem:asy} yields the relations
\begin{gather*}
   F\prec I^-,\qquad F,I\prec II^{-+},\qquad F,I,II,III^{-0}\prec III^+, \cr
   I^+\prec F,I^{-0},\qquad III^-\prec F,I,II.
\end{gather*}
For any choice of parameters, the actions satisfy
$$
   F^+<F^0,\qquad F,I^{-+} < I^0=III^0 < II^{0+},III^{-+},\qquad II^-<II^0<II^+.
$$
We see that $F\prec I^{-0},II,III$. The remaining relation $F\prec I^+$ follows if the actions satisfy $F^0<I^+$, i.e., $-(1-\delta)\mu < \max\{a,-(\nu-\eta_\nu)\}$, which is the first condition in~\eqref{eq:par}. Next we see that $I\prec II,III$ and $III^-\prec I,II$. For the remaining relation $III^{0+}\prec I,II$ we arrange the actions to satisfy $III^+<II^0$, i.e., $\min\{b,\tau-\eta_\tau\}<\eps\nu$, which is the second condition in~\eqref{eq:par}. Then we have $III^0<III^+<II^0<II^+$. The relations $I^0\prec III^0$ and $III^0\prec I^0$ follow from monotonicity: there is an \emph{a priori} strictly positive lower bound on the energy of trajectories traversing $V$, and this rules out trajectories running between $III^0$ and $I^0$ which after small Morse perturbation of $H$ have arbitrarily small energy. The remaining relation $III^{0+}\prec I,II^-$ now follows from Lemma~\ref{lem:neck}, stretching the neck at the hypersurface $(1-2\eps)\times\p^+V$ where $H\equiv-\eps\nu$, and $\eps\nu$ is bigger than all actions in the groups $III^0$ and $III^+$. This proves~\eqref{eq:order}. 
The relations in~\eqref{eq:order-within-groups} also follow from the preceding discussion.
\end{proof}

\begin{remark}
Under the conditions of Lemma~\ref{lem:order}, the Floer boundary operator has upper triangular form if the periodic orbits are ordered by increasing action within each class and the classes are ordered (for example) as 
$$
   F^+\prec F^0\prec I^+\prec I^-\prec I^0 \prec
   III^0\prec III^-\prec III^+\prec
   II^-\prec II^0\prec II^+.
$$
\end{remark}

Let us fix $a<0<b$ and $0<\delta,\eps<1$ and consider $\mu,\nu,\tau>0$ subject to the conditions
\begin{equation}\label{eq:par2}
   \mu>-a/(1-\delta),\qquad \tau>b,\qquad \nu>\max\{-a,b/\eps\}. 
\end{equation}
Note that these conditions allow us to make $\mu,\nu,\tau$ arbitrarily large, independently of each other. They ensure condition~\eqref{eq:par} in Lemma~\ref{lem:order}. Moreover, the actions of all orbits in the classes $F,II^0,II^+$ lie outside the interval $[a,b]$. So the Floer chain complex can be written as
{\color{black}
$$
   FC^{(a,b)} = FC^{(a,b)}_{III} \oplus FC^{(a,b)}_{I} \oplus FC^{(a,b)}_{II^-}  
$$
and with respect to this decomposition the Floer boundary operator has the form
\begin{equation}\label{eq:upper-triang}
        \begin{pmatrix}* & 0 & *\\ 0 & * & *\\ 0 & 0 & *
        \end{pmatrix}\;.
\end{equation}
Let us fix $\mu,\tau$ and consider $\nu<\nu'$ both satisfying~\eqref{eq:par2}. We denote the corresponding Hamiltonians by $H_{\nu'}\leq H_\nu$ and consider the continuation maps 
$$
\phi_{\nu\nu'}:FC^{(a,b)}(H_{\nu'})\to FC^{(a,b)}(H_\nu)
$$
induced by convex interpolation between $H_\nu$ and $H_{\nu'}$. These
continuation maps 
{\color{black}may not have the upper triangular form~\eqref{eq:upper-triang} 
since the combination of Lemmas~\ref{lem:no-escape} and~\ref{lem:asy}
does not apply to the current homotopy situation. Therefore, we
decompose the above chain complex instead as 
$$
   FC^{(a,b)} = FC^{(a,b)}_{III} \oplus FC^{(a,b)}_{I,II^-},   
$$
with differential written in upper triangular form as
$\left(\begin{array}{cc} * & * \\ 0 &
  * \end{array}\right)$. The continuation maps $\phi_{\nu\nu'}$ have
upper triangular form with respect to this decomposition and we obtain
the commuting diagram with exact rows 
}
\begin{equation}\label{eq:filt}
\xymatrix
{
0\ar[r] & FC^{(a,b)}_{III}(H_{\nu'}) \ar[r] \ar[d] & FC^{(a,b)}(H_{\nu'}) \ar[r] \ar[d] & FC^{(a,b)}_{I,II^-}(H_{\nu'}) \ar[r] \ar[d] & 0 \\
0\ar[r] & FC^{(a,b)}_{III}(H_{\nu}) \ar[r] & FC^{(a,b)}(H_{\nu}) \ar[r] & FC^{(a,b)}_{I,II^-}(H_{\nu}) \ar[r] & 0\;, \\
}
\end{equation}
where $FC^{(a,b)}_{I,II^-}$ denotes the quotient complex $FC^{(a,b)} / FC^{(a,b)}_{III}$.

\begin{lemma}\label{lem:Hnu-III}
$$
   \lim^{\longleftarrow}_{\scriptsize \nu\to\infty}FH^{(a,b)}_{III}(H_\nu) \cong SH^{(a,b)}(W^{top},\p^+V).
$$ 
\end{lemma}

\begin{proof}
We consider a homotopy of Hamiltonians which on $V\cup W^{top}\cup [1,\infty)\times\p^+W$ is constant and which on $F\cup W^{bottom}$ is a convex interpolation between the Hamiltonian $H_\nu$ and the Hamiltonian $\ol H_\nu$ that is constant equal to $-\varepsilon\nu$. Since the homotopy is constant on the cobordism $V$, Lemma~\ref{lem:neck} applies and shows that there is no interaction between the orbits in $III$ and the orbits appearing in $F\cup W^{bottom}$. The usual continuation argument then shows that the homology $FH^{(a,b)}_{III}$ is invariant during this homotopy. Since $\lim\limits^{\longleftarrow}_{\scriptsize \nu\to\infty} FH^{(a,b)}_{III}(\ol H_\nu)=SH^{(a,b)}(W^{top},\p^+V)$ by definition, we obtain the desired isomorphism.
\end{proof}

\begin{lemma}\label{lem:Hnu-I,II-}
$$
   \lim^{\longleftarrow}_{\scriptsize \nu\to\infty}FH^{(a,b)}_{I,II^-}(H_\nu) \cong SH^{(a,b)}(W^{bottom},\p^-V).
$$ 
\end{lemma}

\begin{proof}
We consider a homotopy of Hamiltonians which on $F\cup W^{bottom}\cup V$ is constant and which on $W^{top}\cup [1,\infty)\times\p^+W$ is a convex interpolation between the Hamiltonian $H_\nu$ and the Hamiltonian $K_\nu$ that is constant equal to $-\varepsilon \nu$ on $V\cup W^{top}$ and is linear of slope $\tau$ (the same as the slope of $H_\nu$) on $[1,\infty)\times\p^+W$. 
See Figure~\ref{fig:H-cob-K}. 

We have $FH^{(a,b)}(K_\nu) = FH^{(a,b)}_{I,II^-}(K_\nu)$ and so we have a well-defined continuation map $\phi^{HK}_\nu:FH^{(a,b)}_{I,II^-}(K_\nu)\to FH^{(a,b)}_{I,II^-}(H_\nu)$ obtained by composing the continuation map $FH^{(a,b)}(K_\nu)\to FH^{(a,b)}(H_\nu)$ with the map induced by projection $FH^{(a,b)}(H_\nu)\to FH^{(a,b)}_{I,II^-}(H_\nu)$. Since the homotopy is constant in the region $F\cup W^{bottom}\cup V$, which contains the orbits of type $I,II^-$, it follows that this continuation map is an isomorphism. Indeed, the generators of the two chain complexes are canonically identified and upon arranging them in increasing order by the action the continuation map at chain level has upper triangular form with $+1$ on the diagonal. (Note that we do not use at this point Lemma~\ref{lem:neck}.)
 
For $\nu\leq\nu'$ we get commutative diagrams in which all maps are continuation morphisms
$$
\xymatrix{
FH^{(a,b)}_{I,II^-}(H_\nu) & \ar[l]^\cong_{\phi^{HK}_\nu} FH^{(a,b)}_{I,II^-}(K_\nu) \\
FH^{(a,b)}_{I,II^-}(H_{\nu'}) \ar[u]^{\phi_{\nu\nu'}} & FH^{(a,b)}_{I,II^-}(K_{\nu'}) \ar[l]^\cong_{\phi^{HK}_{\nu'}} \ar[u]_{\psi_{\nu\nu'}}\,.
}
$$
Here $\psi_{\nu\nu'}:FH^{(a,b)}_{I,II^-}(K_{\nu'})\to FH^{(a,b)}_{I,II^-}(K_\nu)$ is the continuation map induced by a convex interpolation between $K_\nu$ and $K_{\nu'}$. As a consequence we have a canonical isomorphism 
\begin{equation} \label{eq:inverse-limit-H-K}
\xymatrix{
\lim\limits^{\longleftarrow}_{\scriptsize \nu\to\infty}FH^{(a,b)}_{I,II^-}(H_\nu) & & \ar[ll]^\cong_{\lim\limits^{\longleftarrow}_{\scriptsize \nu\to\infty}\phi^{HK}_\nu} \lim\limits^{\longleftarrow}_{\scriptsize \nu\to\infty}FH^{(a,b)}_{I,II^-}(K_\nu).
} 
\end{equation}

The complex $FC^{(a,b)}_{I,II^-}(K_\nu)$ can be decomposed as 
\begin{equation} \label{eq:decomposition-FCKnu}
FC^{(a,b)}_{I,II^-}(K_\nu)=FC^{(a,b)}_I(K_\nu)\oplus FC^{(a,b)}_{II^-}(K_\nu),
\end{equation}
with differential of upper triangular form $\left(\begin{array}{cc} * & * \\ 0 & *\end{array}\right)$. Lemma~\ref{lem:no-escape-homotopies} below shows that this decomposition is preserved by the continuation maps $\psi_{\nu\nu'}$, which also have upper triangular form. (That this precise property could \emph{a priori} fail for the Hamiltonians $H_\nu$ was the reason to deform them to the Hamiltonians $K_\nu$.) In particular, there is a well-defined inverse system of quotient homologies $FH^{(a,b)}_{II^-}(K_\nu)$, $\nu\to\infty$. Lemma~\ref{lem:vanishing} below shows that the inverse limit of this system vanishes, and we thus obtain a canonical isomorphism 
\begin{equation} \label{eq:inverse-limit-K-I}
\xymatrix{
\lim\limits^{\longleftarrow}_{\scriptsize \nu\to\infty}FH^{(a,b)}_{I}(K_\nu) \ar[rr]^\cong & &  \lim\limits^{\longleftarrow}_{\scriptsize \nu\to\infty}FH^{(a,b)}_{I,II^-}(K_\nu),
}
\end{equation} 
the map being induced in the limit by the inclusions $FC^{(a,b)}_I(K_\nu)\hookrightarrow FC^{(a,b)}_{I,II^-}(K_\nu)$. 

We now prove the isomorphism 
\begin{equation} \label{eq:inverse-limit-K-I-SH}
\xymatrix{
\lim\limits^{\longleftarrow}_{\scriptsize \nu\to\infty}FH^{(a,b)}_{I}(K_\nu) \cong SH^{(a,b)}(W^{bottom},\p^- V).
}
\end{equation} 
The Floer trajectories which are involved in the definition of the
Floer differential for $FC^{(a,b)}_I(K_\nu)$ are contained in a
neighborhood of $F\cup W^{bottom}$ by Lemma~\ref{lem:no-escape}. The
key point is that the Floer trajectories involved in the definition of
the continuation maps $FC^{(a,b)}_I(K_{\nu'})\to FC^{(a,b)}_I(K_\nu)$
are also contained in a neighborhood of $F\cup W^{bottom}$. For this
purpose we choose the Hamiltonians $K_\nu$ such that for $\nu'\ge \nu$
the Hamiltonian $K_{\nu'}$ coincides with $K_\nu$ on a neighborhood of
$F\cup W^{bottom}$ where the orbits in group $I$ for $K_\nu$ are
located. This ensures that the assumptions in the last paragraph of
Lemma~\ref{lem:no-escape} are satisfied for the homotopy obtained by
convex interpolation between $K_\nu$ and $K_{\nu'}$. Denote $\ol
K_\nu$ the Hamiltonian defined on $F\cup W^{bottom}\cup
[1,\infty)\times\p^-V$ which is equal to $K_\nu$ on $F\cup W^{bottom}$
and linear of slope $-\nu$ (the same as the slope of $K_\nu$) on
$[1,\infty)\times\p^- V$. The previous argument then shows the equality  
$$
\xymatrix{
\lim\limits^{\longleftarrow}_{\scriptsize \nu\to\infty}FH^{(a,b)}_{I}(K_\nu)  = \lim\limits^{\longleftarrow}_{\scriptsize \nu\to\infty}FH^{(a,b)}_{I}(\ol K_\nu),
}
$$
and the right hand side is $SH^{(a,b)}(W^{bottom},\p^- V)$ by definition. 

The conclusion of Lemma~\ref{lem:Hnu-I,II-} now follows by combining the isomorphisms~\eqref{eq:inverse-limit-H-K},~\eqref{eq:inverse-limit-K-I}, and~\eqref{eq:inverse-limit-K-I-SH}.
\end{proof}

\begin{figure}
         \begin{center}
\input{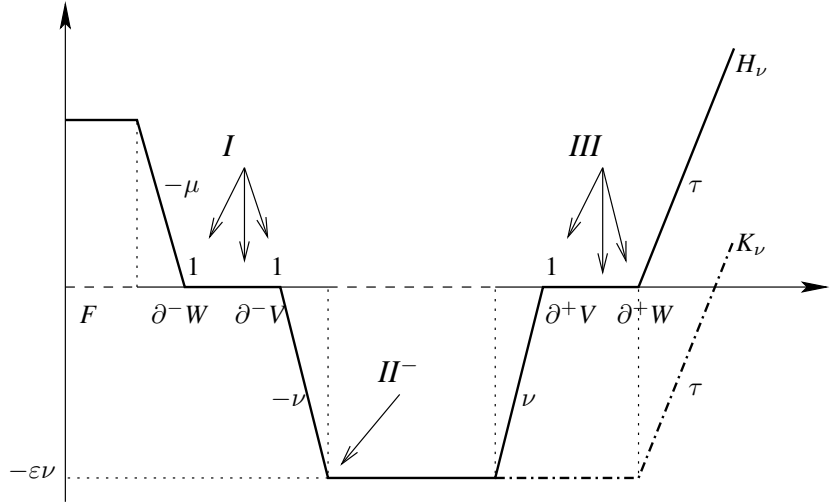}
         \end{center}
\caption{The Hamiltonians $H_\nu$ and $K_\nu$  \label{fig:H-cob-K}}
\end{figure}

The next lemma was used in the previous proof. We recall that $K_\nu$ denotes a Hamiltonian which coincides with $H_\nu$ on $F\cup W^{bottom}\cup V$, is constant equal to $-\varepsilon \nu$ on $V\cup W^{top}$, and is linear of slope $\tau$ (the same as the slope of $H_\nu$) on $[1,\infty)\times\p^+ W$. We choose the smoothings of the Hamiltonians $K_{\nu'}$ and $K_{\nu}$ to coincide up to a translation by $\epsilon(\nu'-\nu)$ in the region $II^-$ \textcolor{black}{but only for} slopes in the interval $(-\nu+\eta_\nu, 0)$. We recall the decomposition~\eqref{eq:decomposition-FCKnu} of $FC^{(a,b)}_{I,II^-}(K_\nu)$, with respect to which the differential has upper triangular form. 

\begin{lemma} \label{lem:no-escape-homotopies} 
The Floer continuation map $\psi_{\nu\nu'}:FC^{(a,b)}_{I,II^-}(K_{\nu'})\to FC^{(a,b)}_{I,II^-}(K_\nu)$ 
induced by a non-increasing $s$-dependent convex interpolation from $K_\nu$ at $-\infty$ to $K_{\nu'}$ at $+\infty$ has upper-triangular form with respect to the decompositions $FC^{(a,b)}_{I,II^-}=FC^{(a,b)}_I\oplus FC^{(a,b)}_{II^-}$ for $K_\nu$ and $K_{\nu'}$. 
\end{lemma}

\begin{proof} The only problematic relation is $I_{K_{\nu'}}\prec
II^-_{K_\nu}$. To prove it we use the fact that in the region $II^-$
the two Hamiltonians coincide up to a translation, so in this region
the homotopy is simply given by adding to the Hamiltonian $K_\nu$ some
function $\R\to [-\epsilon(\nu'-\nu),0]$ of $s$ with compactly supported
derivative. As such, the constant trajectories at the orbits in
$II^-_{K_\nu}$ solve the $s$-dependent continuation Floer equation.  

Assume there exists a continuation Floer trajectory $u:\R\times S^1\to
\wh W_F$ from some orbit $x_+=\lim_{s\to+\infty} u(s,\cdot)$ in
$I_{K_{\nu'}}$ to some orbit $x_-=\lim_{s\to-\infty} u(s,\cdot)$ in
$II^-_{K_\nu}$. By Lemma~\ref{lem:asy}, either $u$ is constant equal
to $x_-$ for very negative values of the parameter $s$, or there
exists $(s,t)\in\R\times S^1$ with $s$ very negative such that 
{\color{black}$r(u(s,t))>r_-=r(x_-(t))$.} 
In the first situation the Floer trajectory would need to be constant
equal to $x_-$ for all values of $s$ because of unique continuation
and the fact that the constant trajectory at $x_-$ solves the same
equation. This is a contradiction since $x_+\neq x_-$. In the second
situation we reach a contradiction using Lemma~\ref{lem:no-escape}, 
{\color{black}which we can apply in the $s$-independent case because the
homotopy is just given by a shift by a function of $s$ on $V\cup
W^{top}\cup[1,\infty)\times\p^+W$.}    
\end{proof}

The next lemma was used in the proof of Lemma~\ref{lem:Hnu-I,II-} as well. By Lemma~\ref{lem:no-escape-homotopies} we have a well-defined inverse system $FH^{(a,b)}_{II^-}(K_\nu)$, $\nu\to\infty$. 

\begin{lemma}\label{lem:vanishing}
$$
   \lim^{\longleftarrow}_{\scriptsize \nu\to\infty}FH^{(a,b)}_{II^-}(K_\nu) = 0.
$$ 
\end{lemma}

\begin{proof}
For $\nu'>\nu$, generators of $FC^{(a,b)}_{II^-}(K_{\nu'})$ correspond to closed Reeb orbits $\gamma$ on $\p^-V$ with Hamiltonian action satisfying
$$
   A_{K_{\nu'}}(\gamma) = -(1+\eps)\Bigl(\int_\gamma\lambda\Bigr) + \eps\nu' \in (a,b). 
$$
Since this condition is equivalent to
$$
   A_{K_{\nu}}(\gamma) = -(1+\eps)\Bigl(\int_\gamma\lambda\Bigr) + \eps\nu \in (a+\eps(\nu-\nu'),b+\eps(\nu-\nu')),
$$
we see that the same Reeb orbits also correspond to generators of the Floer chain group $FC^{(a+\eps(\nu-\nu'),b+\eps(\nu-\nu'))}_{II^-}(K_{\nu})$. Varying the slope continuously from $\nu'$ to $\nu$, we obtain a continuation isomorphism between these two groups fitting into the commuting diagram
\begin{equation*}
\xymatrix{
   FH^{(a+\eps(\nu-\nu'),b+\eps(\nu-\nu'))}_{II^-}(K_{\nu}) \ar[r]^-{\cong} \ar[rd]^-{\pi} & FH^{(a,b)}_{II^-}(K_{\nu'}) \ar[d]^-{\psi_{\nu\nu'}} \\
   & FH^{(a,b)}_{II^-}(K_{\nu}).
}
\end{equation*}
That the horizontal map is an isomorphism follows from the fact that
the Hamiltonian is deformed outside a compact set only by a global shift by a constant, and from the
fact that there are no orbits that cross the boundary of the moving
action window during the homotopy. The horizontal map can be expressed
as a composition of small-time continuation maps induced by homotopies
for fixed action windows, which are isomorphisms since each of these
homotopies can be followed backwards, and of tautological isomorphisms
given by shifting the action window by some small amount in the
complement of the action spectrum.  

Now if $b+\eps(\nu-\nu')<a$, then the intervals $[a+\eps(\nu-\nu'),b+\eps(\nu-\nu')]$ and $[a,b]$ do not overlap and thus the projection $\pi$ vanishes in homology. Hence the Floer chain map $\psi_{\nu\nu'}$ vanishes whenever $\nu'-\nu>(b-a)/\eps$, from which the lemma follows. 
\end{proof}

\begin{proposition}[excision for filtered symplectic homology]\label{prop:excision}
Let $(W,V)$ be a pair of Liouville cobordisms with filling and consider parameters $-\infty<a<b<\infty$. There is a short exact sequence 
$$
0\to SH^{(a,b)}_*(W^{top},\p^+V)\to SH^{(a,b)}_*(W,V)\to SH^{(a,b)}_*(W^{bottom},\p^-V)\to 0.
$$
Moreover, this short exact sequence splits canonically, so that we have a canonical isomorphism
$$
SH^{(a,b)}_*(W,V)\cong SH^{(a,b)}_*(W^{top},\p^+V)\oplus SH^{(a,b)}_*(W^{bottom},\p^-V). 
$$
\end{proposition}

\begin{proof} We fix the parameters $0<\delta,\eps<1$ and $\mu,\tau>0$ such that the first
two conditions in~\eqref{eq:par2} hold, and we work with the family of Hamiltonians $H_\nu=H_{\mu,\nu,\tau}$, $\nu\to\infty$ discussed above. Then 
$$
\lim^{\longleftarrow}_{\scriptsize \nu\to\infty}FH^{(a,b)}_*(H_\nu) \cong SH^{(a,b)}_*(W, V)
$$
by definition. The short exact sequence of inverse
systems~\eqref{eq:filt} determines an inverse system of homology exact
triangles in which each term is a finite dimensional vector space. In
this case the inverse limit preserves exactness and we obtain using
Lemmas~\ref{lem:Hnu-III} and~\ref{lem:Hnu-I,II-} an exact triangle  
$$
\xymatrix
@C=20pt
{
SH^{(a,b)}_*(W^{top},\p^+V) \ar[rr] & & 
SH^{(a,b)}_*(W,V)  \ar[dl] \\ & SH^{(a,b)}_*(W^{bottom},\p^-V) \ar[ul]^{[-1]}\;.  
}
$$

The proof of Lemma~\ref{lem:Hnu-I,II-} shows that each class in
$SH^{(a,b)}_*(W^{bottom},\p^-V)$ is represented by a sequence 
{\color{black}(indexed by $\nu$ and representing an element of the inverse limit)} 
of classes in $FH^{(a,b)}_{I,II^-}(H_\nu)$ which are each
represented by a cycle that is a linear combination of orbits in
$I_{H_\nu}$. Indeed, the proof provides such a representative by a
cycle in $FC^{(a,b)}_I(K_\nu)$, and we have
$FC^{(a,b)}_I(K_\nu)=FC^{(a,b)}_I(H_\nu)$; on the other hand, since
$I_{H_\nu}\prec II^-_{H_\nu}$ as already seen
in~\eqref{eq:upper-triang}, this continues to be a cycle in
$FC^{(a,b)}_{I,II^-}(H_\nu)$.  

To prove the existence of the short exact sequence in the statement we use that the degree $-1$ connecting map $FH^{(a,b)}_{I,II^-}(H_\nu) \to FH^{(a,b)}_{III}(H_\nu)$ vanishes on elements of $I_{H_\nu}$ by~\eqref{eq:upper-triang}. Thus the connecting map in the above exact triangle vanishes, and the latter becomes the short exact sequence 
$$
0\to SH^{(a,b)}_*(W^{top},\p^+V)\to SH^{(a,b)}_*(W,V)\to SH^{(a,b)}_*(W^{bottom},\p^-V)\to 0.
$$

To prove the existence of a canonical splitting for this exact
sequence we use again that $I\prec III$ for $H_\nu$. Thus a cycle in $FC^{(a,b)}_{I,II^-}(H_\nu)$ which is a linear combination of orbits in $I_{H_\nu}$ is canonically also a cycle in $FC^{(a,b)}(H_\nu)$. The splitting $SH^{(a,b)}_*(W^{bottom},\p^-V)\to SH^{(a,b)}_*(W,V)$ associates to each class, represented by a sequence of classes of cycles in $FC^{(a,b)}_{I,II^-}(H_\nu)$ which are linear combinations of orbits in $I_{H_\nu}$, the sequence of classes represented by the same cycles viewed in $FC^{(a,b)}(H_\nu)$. The latter represents indeed an element in the inverse limit of $FH^{(a,b)}(H_\nu)$, $\nu\to\infty$ because the continuation maps $\phi_{\nu\nu'}:FC^{(a,b)}(H_{\nu'})\to FC^{(a,b)}(H_\nu)$ preserve the relation $I\prec III$. 
\end{proof}

Taking limits over $a$ and $b$, Proposition~\ref{prop:excision} implies

\begin{theorem}[excision]\label{thm:excision}
Let $(W,V)$ be a pair of Liouville cobordisms with filling.   
Then for each flavour $\heartsuit$ we have canonical isomorphisms 
$$
   SH_*^\heartsuit(W,V) \cong SH_*^\heartsuit(W^{bottom},\p^-V) \oplus SH_*^\heartsuit(W^{top},\p^+V). 
$$
\hfill{$\square$}
\end{theorem}

}

In Proposition~\ref{prop:excision} and Theorem~\ref{thm:excision} we
allow $W^{bottom}$ or $W^{top}$ to be empty, in which case the 
corresponding term is not present in the diagram. In particular,
taking $V$ to be a collar neighbourhood of some boundary components we
obtain  

\begin{corollary}\label{cor:tub-nbhd}
Given a Liouville cobordism $W$ and an admissible union of connected
components $A\subset \p W$, we have  
$$
SH_*^\heartsuit(W,A)\cong SH_*^\heartsuit(W,I\times A),
$$
where $I\times A$ is a collar neighborhood of $A$ in $W$ which we view
as a trivial cobordism, so that $(W,I\times A)$ is a Liouville
pair. 
\hfill$\square$
\end{corollary}

This is the precise sense in which Definitions~\ref{defi:SHWA}
and~\ref{defi:SHWV} are compatible. 

In order to make the excision theorem resemble the one in algebraic topology,
we introduce the following notion. 

\begin{definition}
A 
{\color{black}\emph{Liouville cobordism triple}} 
$(W,V,U)$ consists of three Liouville cobordisms $U\subset V\subset W$
such that $(W,V)$ and $(V,U)$ are Liouville cobordism pairs.   
{\color{black}A {\em filling} of a Liouville cobordism triple is a filling of $W$,
which induces fillings of $V$ and $U$ in the obvious way.}
\end{definition}

Then we have

\begin{theorem}[excision for triples] \label{thm:excision-triple} 
Let $(W,V,U)$ be a 
{\color{black}filled Liouville cobordism triple.} 
Then for each flavour $\heartsuit$ we have canonical isomorphisms  
$$
SH_*^\heartsuit(W,V)\cong SH_*^\heartsuit(\overline{W\setminus U},\overline{V\setminus U})\,.
$$
\end{theorem}

Here if some boundary component $A$ of $V$ and $U$ coincides, then the
homology on the right hand side has to be understood relative to $A$. 
{\color{black} (Alternatively, one can use Proposition~\ref{prop:triv-cob} below 
to move $U$ into the interior of $V$ and avoid this situation.)}
Also, if $\overline{W\setminus U}$ contains both a bottom and an upper
part then the right hand side has to be understood  
{\color{black}according to Section~\ref{sec:multilevel}} 
as a direct sum, as in the statement of Theorem~\ref{thm:excision}. 

\begin{proof}
Let us write 
$$
   \ol{W\setminus V} = W^{bottom}\amalg W^{top},\qquad 
   \ol{V\setminus U} = V^{bottom}\amalg V^{top}. 
$$
Then 
$$
   \ol{W\setminus U} = (W^{bottom}\cup V^{bottom})\amalg (W^{top}\cup V^{top}) 
$$
and we find
\begin{align*}
   SH_*^\heartsuit(\ol{W\setminus U},\ol{V\setminus U}) 
   &= SH_*^\heartsuit(W^{bottom}\cup V^{bottom},V^{bottom}) \oplus
   SH_*^\heartsuit(W^{top}\cup V^{top},V^{top}) \cr
   &\cong SH_*^\heartsuit(W^{bottom},\p^-V) \oplus
   SH_*^\heartsuit(W^{top},\p^+V) \cr
   &\cong SH_*^\heartsuit(W,V), 
\end{align*}
{\color{black} where the first equality is the definition and the other two
isomorphisms follow from Theorem~\ref{thm:excision}.} 
\end{proof}

\pagebreak

\section{The exact triangle of a pair of filled Liouville cobordisms} \label{sec:exact-triangle-pair}

The main result of this section is

\begin{theorem}[exact triangle of a pair]\label{thm:exact-triangle-pair}
For each filled Liouville cobordism pair $(W,V)$ and $\heartsuit\in\{\varnothing,\ge 0, >0,=0,\le 0,<0\}$ there exist exact triangles 
\begin{equation*} 
\xymatrix
@C=10pt
@R=18pt
{
SH_*^\heartsuit(W,V) \ar[rr] & & 
SH_*^\heartsuit(W) \ar[dl] \\ & SH_*^\heartsuit(V) \ar[ul]^{[-1]}  
}
\end{equation*}
and
\begin{equation*}
\xymatrix
@C=10pt
@R=18pt
{
SH^*_\heartsuit(W,V) \ar[dr]_{[+1]} & & 
SH^*_\heartsuit(W) \ar[ll] \\ & SH^*_\heartsuit(V) \ar[ur]   
}
\end{equation*}
These triangles are functorial with respect to inclusions of Liouville pairs. 
\end{theorem}

This theorem will be proved in Section~\ref{sec:ex-triangle-pair} below.

\subsection{Cofinal families of Hamiltonians} \label{sec:adapted} 
{\color{black}
As a preparation, we now recast the definition of the symplectic homology groups $SH_*^\heartsuit(W)$, $SH_*^\heartsuit(V)$ and of the transfer map $f_!^\heartsuit:SH_*^\heartsuit(W)\to SH_*^\heartsuit(V)$ at chain level in terms of some carefully chosen cofinal families of Hamiltonians. This will allow us to further express the relative symplectic homology groups $SH_*^\heartsuit(W,V)$ in terms of the cone construction. 

Let $(W,V)$ be a Liouville pair with filling $F$.

{\bf Notational convention.} Let $P$, $Q$ denote sets of $1$-periodic
orbits of a given Hamiltonian $H$. Recall that we write $Q<P$
if all the orbits in group $Q$ have strictly smaller action than all
the orbits in group $P$, and we write $Q\prec P$ 
if there is no Floer trajectory for $H$ asymptotic at the positive puncture to an orbit in $Q$ and asymptotic at the negative puncture to an orbit in $P$. This implies that the expression of the Floer boundary operator on any orbit in $Q$ does not contain any element in $P$. It is understood that the Floer equation involves some almost complex structure which is not specified in the notation.  

Similarly, given two Hamiltonians $H_\pm$ and a homotopy $H_s$, $s\in \R$ equal to $H_\pm$ near $\pm\infty$, and given sets of $1$-periodic orbits $P_{H_\pm}$ for $H_\pm$, we write 
$$
P_{H_+}\prec P_{H_-}
$$
if there is no Floer continuation trajectory for the homotopy $H_s$ asymptotic at the positive puncture to an orbit in $P_{H_+}$ and asymptotic at the negative puncture to an orbit in $P_{H_-}$. This implies that the expression of the Floer continuation map on any orbit in $P_{H_+}$ does not contain any element in $P_{H_-}$. Here it is again understood that the Floer continuation equation involves some almost complex structure which is not specified in the notation. 
In the previous context, we write 
$$
P_{H_+}<P_{H_-}
$$
if the $H_+$-action of any orbit in $P_{H_+}$ is smaller than the $H_-$-action of any orbit in $P_{H_-}$. This implies $P_{H_+}\prec P_{H_-}$ if $H_+\le H_-$ and the homotopy $H_s$ is non-increasing with respect to the $s$-variable. 

Given $c\in \R$, we write 
$$
P_{H_+}< P_{H_-} - c
$$
if the difference between the $H_+$-action of any orbit in $P_{H_+}$
and the $H_-$-action of any orbit in $P_{H_-}$ is smaller than $-c$. 

{\color{black}
\begin{lemma}\label{lem:gap}
Consider Hamiltonians $H_+\ge H_-$ and a homotopy $H_s$ which is a
convex interpolation between $H_+$ and $H_-$ given by a non-decreasing
$s$-dependent function, i.e., $H_s= H_-+f(s)(H_+-H_-)$ with $f:\R\to[0,1]$, $f'\ge
0$, $f=0$ near $-\infty$, $f=1$ near $+\infty$. Then $P_{H_+} <
P_{H_-} -\|H_+-H_-\|_\infty$ implies $P_{H_+}\prec P_{H_-}$. 
\end{lemma}
}

\begin{proof}
If there is a continuation Floer trajectory $u:\R\times S^1\to \wh W_F$ 
solving $\p_s u + J_{s,t}(u)(\p_t u - X_{H_s}(u))=0$ 
with $\lim_{s\to\pm\infty}u(s,\cdot)=x_\pm(\cdot)$,
where $x_\pm$ are $1$-periodic orbits of $H_\pm$, 
then we have
\begin{eqnarray*}
A_{H_+}(x_+)-A_{H_-}(x_-) & = & \int_{-\infty}^\infty \frac d{ds}A_{H_s}(u(s,\cdot))\, ds \\
& = & \int_{-\infty}^\infty dA_{H_s}(u(s,\cdot))\cdot \p_s u \,ds - \int_{-\infty}^\infty \int_0^1 \p_s H_s(t,u(s,t))\, dt\, ds \\
& = & \int_{-\infty}^\infty\int_0^1 \|\p_s u(s,t)\|^2\, dt\, ds \\
& & -\int_{-\infty}^\infty\int_0^1 f'(s) \Bigl(H_+(t,u(s,t))-H_-(t,u(s,t))\Bigr)\, dt \, ds  \\
& \ge & -\int_{-\infty}^\infty\int_0^1 f'(s) \sup_{t,x} \Bigl(H_+(t,x)-H_-(t,x)\Bigr) \, dt \, ds \\
& = & -\|H_+-H_-\|_\infty. 
\end{eqnarray*}
\end{proof}

Since the domain of definition of the Hamiltonians that we use in this paper is a noncompact manifold, it is appropriate to discuss the degree of applicability of the previous principle: it holds for compactly supported homotopies, so that $\|H_+-H_-\|_\infty$ is finite (and can be explicitly computed), but it also holds for non-compactly supported homotopies if one knows \emph{a priori} that the continuation Floer trajectories are contained in a compact set, in which case it is enough to estimate $\|H_+-H_-\|_\infty$ on that compact set.

\subsubsection{Hamiltonians for $SH_*^\heartsuit(W)$.}\label{sec:Ham-W} 
Let 
$$
\mu,\tau>0
$$ 
be such that $\mu\notin \mbox{Spec}(\p^-W)$ and $\tau\notin \mbox{Spec}(\p^+W)$. Denote by
$\eta_\mu>0$ 
the distance from $\mu$ to $\mbox{Spec}(\p^- W)$ and let $\delta>0$ 
be such that 
\begin{equation} \label{eq:deltamuetamu}
\delta\mu<\eta_\mu. 
\end{equation}
We denote by
$$
K_{\mu,\tau}=K_{\mu,\tau,\delta}:\wh W_F\to\R
$$
the Hamiltonian which is defined up to smooth approximation as follows: it is constant equal to $\mu(1-\delta)$ on $F\setminus [\delta,1]\times\p F$, it is linear equal to $\mu(1-r)$ on $[\delta,1]\times\p F$, it is constant equal to $0$ on $W$, and it is linear equal to $\tau(r-1)$ on $[1,\infty)\times\p^+ W$. See Figure~\ref{fig:Kmutau}. 

\begin{figure}
         \begin{center}
\input{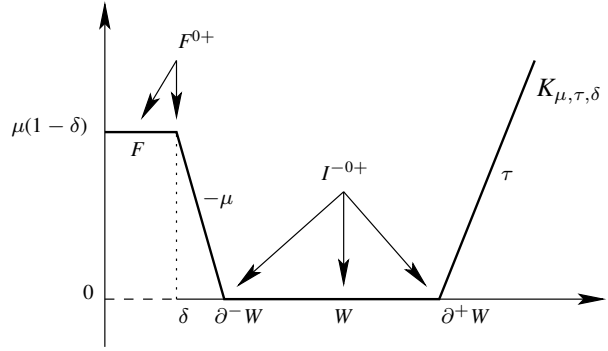}
         \end{center}
\caption{Hamiltonians $K_{\mu,\tau,\delta}$ for the definition $SH_*^\heartsuit(W)$  \label{fig:Kmutau}}
\end{figure}

A smooth approximation of $K_{\mu,\tau}$ will thus be of the form $K_{\mu,\tau}(r,y)=k(r)$ on $[1,\infty)\times\p^+W$ (and similarly near the negative boundary $\p^-W$). The $1$-periodic orbits of $X_{K_{\mu,\tau}}$ on $\{r\}\times \p^+W$ correspond to Reeb orbits on $\p^+W$ of period $k'(r)$, and their Hamiltonian action equals 
$$
rk'(r)-k(r). 
$$
Since we assumed that $\mu$ and $\tau$ are not equal to the period of a closed Reeb orbit on the respective boundaries of $W$, it follows that $K_{\mu,\tau}$ has no $1$-periodic orbits in the regions where it is linear. 

The $1$-periodic orbits of the Hamiltonian $K_{\mu,\tau}$ naturally fall into 5 classes as follows:
\begin{itemize}
\item[($F^0$)] constants in $F\setminus [\delta,1]\times\p F$. 
\item [($F^+$)] orbits corresponding to negatively parameterized closed Reeb orbits on $\p F=\p^-W$ and located near $\{\delta\}\times\p^-W$. 
\item[($I^-$)] orbits corresponding to negatively parameterized closed Reeb orbits on $\p^-W$ and located near $\p^-W$. 
\item [($I^0$)] constants in $W$.
\item [($I^+$)] orbits corresponding to positively parameterized closed Reeb orbits on $\p^+W$ and located near $\p^+W$. 
\end{itemize}
We denote by $F$ the group of orbits $F^{0+}$, and by $I$ the group of orbits $I^{-0+}$. The maximal action of an orbit in group $F$ is $-\mu(1-\delta)=-\mu+\delta\mu$, while the minimal action of an orbit in group $I$ is $-\mu+\eta_\mu$. Condition~\eqref{eq:deltamuetamu} implies  
$F<I$, and in particular 
$$
F\prec I.
$$
This last relation actually holds regardless of the choice of parameters by combining Lemmas~\ref{lem:no-escape} and~\ref{lem:asy} in order to prohibit trajectories from $F$ to $I^-$ with the relation $F<I^{0+}$ which prohibits trajectories from $F$ to $I^{0+}$. Alternatively, the relation $F\prec I^0$ is also a consequence of Lemma~\ref{lem:constant}, while $F\prec I^+$ is also a consequence of Lemmas~\ref{lem:no-escape} and~\ref{lem:asy}. 

Let now $(\mu_i)$, $i\in\Z_-$ and $(\tau_j)$, $j\in\Z_+$ be two
sequences which do not contain elements in
${\mbox{Spec}}(\p^-W)\cup{\mbox{Spec}}(\p^+W)$ and such that
$\mu_{i'}>\mu_i$ for $i'<i$ and $\tau_j<\tau_{j'}$ for $j<j'$. We
moreover require that $\mu_i\to\infty$ as $i\to-\infty$ and
$\tau_j\to\infty$ as $j\to\infty$. Choose a sequence $(\delta_i)$,
$i\in\Z_-$ of positive numbers such that $\delta_{i'}<\delta_i$ for
$i'<i$, such that $\delta_i\to 0$ as $i\to-\infty$, and such that
condition~\eqref{eq:deltamuetamu} is satisfied: 
$$
\delta_i\mu_i<\eta_{\mu_i} \qquad \mbox{ for all } i\in\Z_-. 
$$
We denote 
$$
K_{i,j}:=K_{\mu_i,\tau_j,\delta_i},\qquad i\in \Z_-,\quad j\in\Z,
$$
so that $K_{i',j}\ge K_{i,j}$ for $i'\le i$, and $K_{i,j}\le K_{i,j'}$ for $j\le j'$. We consider $FC_*(K_{i,j})$ as a doubly-directed system in {\sf Kom}, inverse on $i\to-\infty$ and direct on $j\to\infty$, with maps
$$
FC_*(K_{i',j})\to FC_*(K_{i,j}),\qquad i'\le i
$$
induced by non-decreasing homotopies, and maps 
$$
FC_*(K_{i,j})\to FC_*(K_{i,j'}),\qquad j\le j'
$$
induced by non-increasing homotopies. 
Denote $FC_F(K_{i,j})$ the Floer subcomplex of $FC_*(K_{i,j})$ generated by orbits in the group $F$, and denote $FC_I(K_{i,j})$ the Floer quotient complex generated by orbits in the group $I$. The groups of orbits $I^-$, $I^0$, $I^+$ are ordered by action as $I^-<I^0<I^+$ within the group of orbits $I$, so that we have corresponding sub- and quotient complexes
$FC_{I^\heartsuit}(K_{i,j})$
for $\heartsuit\in\{\varnothing,\ge 0, >0,=0,\le 0,<0\}$, where $I^\heartsuit$ has the following meaning: 
$$
I^\varnothing=I, \quad  I^{\le 0}=I^{-0}, \quad  I^{>0}=I^+, \quad  I^{<0}=I^-, \quad  I^{=0}=I^0, \quad  I^{\ge 0}=I^{0+}.
$$ 

\begin{lemma} The homotopies that define the doubly-directed system $FC_*(K_{i,j})$, $i\in\Z_-$, $j\in\Z_+$ induce doubly-directed systems  
$$
FC_{I^\heartsuit}(K_{i,j}),\qquad i\in\Z_-,\quad j\in \Z_+,\quad \heartsuit\in\{\varnothing,\ge 0, >0,=0,\le 0,<0\}.
$$
\end{lemma}

\begin{proof}
Our choice of parameters ensures that 
\begin{equation}\label{eq:FprecI}
F_{K_{i',j}}\prec I_{K_{i,j}},\qquad F_{K_{i,j}}\prec I_{K_{i,j'}}
\end{equation}
for $i'\le i$ and $j\le j'$. To prove these relations denote
$\mu'=\mu_{i'}$, $\tau'=\tau_{j'}$, $\delta'=\delta_{i'}$, and
similarly $\mu,\tau,\delta$ for the corresponding numbers not
decorated with primes. The first relation follows from
Lemma~\ref{lem:gap} and the relation $F_{K_{i',j}}<I_{K_{i,j}} -
\|K_{i',j}-K_{i,j}\|_\infty$: the maximal action of an orbit in
$F_{K_{i',j}}$ is $-\mu'(1-\delta')$, the minimal action of an orbit
in $I_{K_{i,j}}$ is $-\mu+\eta_\mu$, and the maximal oscillation of
the homotopy is $\|K_{i',j}-K_{i,j}\|_\infty =
\mu'(1-\delta')-\mu(1-\delta)$; the desired relation then follows
from~\eqref{eq:deltamuetamu}. The second relation in~\eqref{eq:FprecI}
follows from $F_{K_{i,j}}<I_{K_{i,j'}}$ because in this case the
homotopy is non-increasing. Now we have already seen that
$F_{K_{i,j}}<I_{K_{i,j}}$, while the action of the orbits in
$I_{K_{i,j'}}$ is never smaller than the action of the orbits in
$I_{K_{i,j}}$. This proves the relations~\eqref{eq:FprecI}. They imply
that we have a doubly-directed subsystem $FC_F(K_{i,j})$   
and a doubly-directed quotient system $FC_I(K_{i,j})$, $i\in\Z_-$, $j\in\Z_+$.

To prove that we have doubly-directed systems $FC_{I^\heartsuit}(K_{i,j})$, $i\in\Z_-$, $j\in\Z_+$ for all values of $\heartsuit$ we need to show the relations 
$$
I^-_{K_{i',j}}\prec I^{0+}_{K_{i,j}} \quad \mbox{and }\quad I^{-0}_{K_{i',j}}\prec I^+_{K_{i,j}} \quad \mbox{ for }\quad i'\le i,
$$
$$
I^-_{K_{i,j}}\prec I^{0+}_{K_{i,j'}} \quad \mbox{and }\quad I^{-0}_{K_{i,j}}\prec I^+_{K_{i,j'}} \quad \mbox{ for }\quad j\le j'. 
$$
The last two relations follow from the fact that the non-increasing homotopies which induce maps $FC_*(K_{i,j})\to FC_*(K_{i,j'})$ for $j\le j'$ preserve the filtration by the action. In contrast, this argument cannot be used to prove the first two relations since non-decreasing homotopies typically do not preserve the action filtration. Instead we argue using the confinement lemmas in~\S\ref{sec:confinement}: the first relation follows from Lemma~\ref{lem:constant}, and the second relation follows from Lemmas~\ref{lem:no-escape} and~\ref{lem:asy}.
\end{proof}

\begin{lemma}\label{lem:SHWgeometric-W} We have isomorphisms 
$$
SH_*^\heartsuit(W)\cong \lim^{\longrightarrow}_j \lim^{\longleftarrow}_i FH_{I^\heartsuit}(K_{i,j})
$$
for $\heartsuit\in\{\varnothing,\ge 0, >0,=0,\le 0,<0\}$.
\end{lemma}

\begin{proof}
Recall that the slopes of $K_{i,j}$ are  $-\mu_i$ and $\tau_j$, with
$-\mu_i<0<\tau_j$. We claim that
\begin{equation}\label{eq:tradeactionforI}
SH_*^{(-\mu_i,\tau_j)}(W)\cong FH_{I}(K_{i,j}). 
\end{equation}
To prove~\eqref{eq:tradeactionforI} recall that $SH_*^{(a,b)}(W)={\displaystyle\lim^{\longrightarrow}_K} \, FH_*^{(a,b)}(K)$, where $K$ ranges over the space $\cH(W;F)$ of admissible Hamiltonians on $\wh W_F$ with respect to the filling $F$ and the direct limit is considered with respect to non-increasing homotopies, see~\S\ref{sec:SHWA}. Consider a decreasing sequence $i_k\to -\infty$ and an increasing sequence $j_k\to\infty$ as $k\to\infty$. The sequence of Hamiltonians $K_{i_k,j_k}$, $k\in\Z_+$ is then cofinal in $\cH(W;F)$ and we have $SH_*^{(a,b)}(W)={\displaystyle \lim^{\longrightarrow}_{k\to\infty}} \, FH_*^{(a,b)}(K_{i_k,j_k})$, where the direct limit is considered with respect to continuation maps $FH_*^{(a,b)}(K_{i_k,j_k})\to FH_*^{(a,b)}(K_{i_{k'},j_{k'}})$ induced by non-increasing homotopies. We can assume without loss of generality that $-\mu_{i_k}\le a$ and $\tau_{j_k}\ge b$. The smoothings of any such two Hamiltonians $K_{i_k,j_k}$ and $K_{i_{k'},j_{k'}}$, $k\le k'$ can be constructed so that they coincide in the neighborhood of $W$ where the periodic orbits in group $I$ for $K_{i_k,j_k}$ appear. As such, the continuation map $FC_*^{(a,b)}(K_{i_k,j_k})\to FC_*^{(a,b)}(K_{i_{k'},j_{k'}})$, which is upper triangular if we arrange the generators in increasing order of the action, has diagonal entries equal to $+1$ and is therefore an isomorphism. This proves that the canonical map $FH_*^{(a,b)}(K_{i_k,j_k})\to SH_*^{(a,b)}(W)$ is an isomorphism for all $k$ (such that $-\mu_{i_k}\le a$ and $\tau_{j_k}\ge b$). 

The isomorphism~\eqref{eq:tradeactionforI} is proved by considering the following three isomorphisms: we have $FH_I(K_{i,j})=FH_*^{(-\mu_i+\eta,\tau_j)}(K_{i,j})$ for any $\eta>0$ such that $\delta_i\mu_i<\eta<\eta_{\mu_i}$; we have $FH_*^{(-\mu_i+\eta,\tau_j)}(K_{i,j})\cong SH_*^{(-\mu_i+\eta,\tau_j)}(W)$ by the above; and we have 
$SH_*^{(-\mu_i+\eta,\tau_j)}(W)\cong SH_*^{(-\mu_i,\tau_j)}(W)$ since there is no periodic Reeb orbit on $\p^-W$ with period in the interval $(\mu_i-\eta,\mu_i)$. 

A variant of the same argument shows that, under the
isomorphism~\eqref{eq:tradeactionforI}, the continuation maps
$FH_I(K_{i',j})\to FH_I(K_{i,j})$, $i'\leq i$ and $FH_I(K_{i,j})\to
FH_I(K_{i,j'})$, $j\leq j'$ induced by a non-decreasing homotopy,
respectively by a non-increasing homotopy, coincide with the canonical
maps $SH_*^{(-\mu_{i'},\tau_j)}(W)\to SH_*^{(-\mu_i,\tau_j)}(W)$ and
$SH_*^{(-\mu_i,\tau_j)}(W)\to SH_*^{(\mu_i,\tau_{j'})}(W)$,
respectively. From this the conclusion of the lemma follows in the case $\heartsuit=\varnothing$. 

The proof in the case $\heartsuit\neq\varnothing$ is similar in view of the isomorphisms 
\begin{gather*}
SH_*^{(0^+,\tau_j)}(W)\cong FH_{I^{>0}}(K_{i,j}),\qquad
SH_*^{(0^-,\tau_j)}(W)\cong FH_{I^{\ge 0}}(K_{i,j}),\cr
SH_*^{(0^-,0^+)}(W)\cong FH_{I^{= 0}}(K_{i,j}),\qquad
SH_*^{(-\mu_i,0^+)}(W)\cong FH_{I^{\le 0}}(K_{i,j}),\cr
SH_*^{(-\mu_i,0^-)}(W)\cong FH_{I^{< 0}}(K_{i,j}).
\end{gather*}
Here $0^-$ and $0^+$ denote a negative, respectively a positive real number which is close enough to zero (with absolute value smaller than the minimal period of a closed Reeb orbit on $\p^-W$, respectively $\p^+W$). 
\end{proof}

{\color{black}
\subsubsection{Hamiltonians for $SH_*^\heartsuit(W,\p^\pm W)$.} 
We shall need in the sequel (Lemma~\ref{lem:SHWVgeometric-WV-bottom-top}) alternative descriptions of the homology groups $SH_*^\heartsuit(W,\p^\pm W)$ in the spirit of the previous section, which we now explain. We refer freely to the notation of~\S\ref{sec:Ham-W}. 

Given $\mu,\tau>0$ such that $\mu\notin\mbox{Spec}(\p^-W)$ and $\tau\notin\mbox{Spec}(\p^+W)$, and given $\delta\in (0,1)$, we consider Hamiltonians $K^\pm=K^\pm_{\mu,\tau,\delta}:\wh W_F\to\R$ defined as follows: 
\begin{itemize}
\item the Hamiltonian $K^-_{\mu,\tau,\delta}$ coincides with the Hamiltonian $K_{\mu,\tau,\delta}$ of~\S\ref{sec:Ham-W} on $W\cup[1,\infty)\times\p^+W$ and is equal to $-K_{\mu,\tau,\delta}$ on $F$. See Figure~\ref{fig:Kplusminus}.
\item the Hamiltonian $K^+_{\mu,\tau,\delta}$ coincides with the Hamiltonian $K_{\mu,\tau,\delta}$ on $F\cup W$ and is equal to $-K_{\mu,\tau,\delta}$ on $[1,\infty)\times\p^+W$. See Figure~\ref{fig:Kplusminus}.
\end{itemize}

\begin{figure}
         \begin{center}
\input{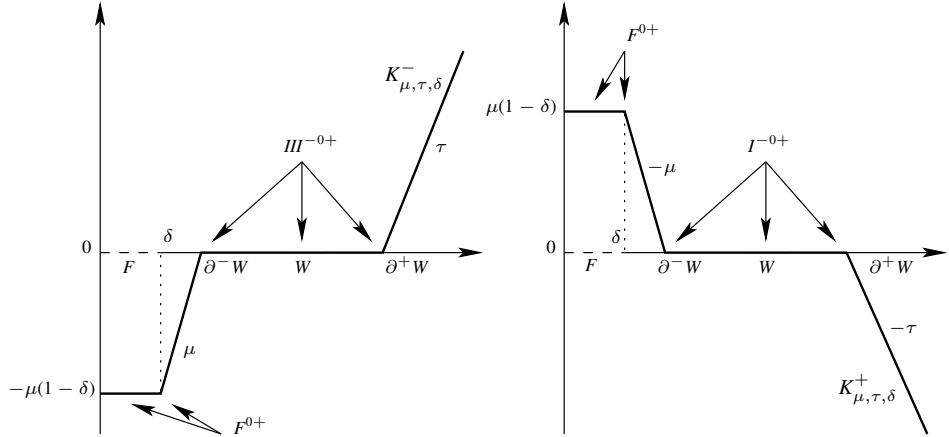}
         \end{center}
\caption{Hamiltonians $K^\pm$ for the definition of $SH_*^\heartsuit(W,\p^\pm W)$ \label{fig:Kplusminus}}
\end{figure} 

The $1$-periodic orbits of each of these Hamiltonians naturally fall into 5 groups, which we denote by $F^{0+},III^{-0+}$ for $K^-$, and by $F^{0+},I^{-0+}$ for $K^+$. We denote as usual by $\eta_\mu,\eta_\tau>0$ positive numbers smaller than the distance from $\mu$ to $\mbox{Spec}(\p^- W)$, respectively smaller than the distance from $\tau$ to $\mbox{Spec}(\p^+W)$. If the parameters are chosen such that 
\begin{equation} \label{eq:another-equation}
\delta\mu<\eta_\mu \qquad \mbox{and}\qquad \mu-\eta_\mu>\tau-\eta_\tau
\end{equation}
then we have $F< I$ for $K^+$, respectively $III<F$ for $K^-$. We denote $III^{=0}=III^0$, $III^{>0}=III^{-+}$, and also $I^{=0}=I^0$, $I^{<0}=I^{-+}$. 

This construction is well-behaved in families, just like the construction in the previous section. 
Consider first an indexing parameter $j\in\Z_+$. We choose sequences $\mu_j\to\infty$, $\tau_j\to\infty$, $\delta_j\to 0$ as $j\to\infty$, such that $\mu_j\notin\mbox{Spec}(\p^-W)$, $\tau_j\notin\mbox{Spec}(\p^+W)$, such that $(\mu_j)$ and $(\tau_j)$ are increasing and $(\delta_j)$ is decreasing, and such that~\eqref{eq:another-equation} is satisfied for each $j$. We define $K^-_j=K^-_{\mu_j,\tau_j,\delta_j}$. Given $j\le j'$ we consider the interpolating homotopy from $K^-_j$ at $+\infty$ to $K^-_{j'}$ at $-\infty$ 
which is the concatenation of the following two monotone homotopies: first keep $K^-_j$ fixed on $W\cup[1,\infty)\times\p^+W$ and interpolate between $K^-_j$ and $K^-_{\mu_{j'},\tau_j,\delta_{j'}}$ on $F$, then keep the Hamiltonian fixed on $F\cup W$ and interpolate between $K^-_{\mu_{j'},\tau_j,\delta_{j'}}$ and $K^-_{j'}$ on $[1,\infty)\times \p^+W$. We claim that for such a homotopy we have
$$
III_{K^-_j}\prec F_{K^-_{j'}}, \qquad III^{=0}_{K^-_j}\prec III^{>0}_{K^-_{j'}}.
$$
The proof of the first relation uses Lemma~\ref{lem:gap}. Since the
homotopy from $K^-_{j'}$ to $K^-_{j'}$ is non-increasing on
$[1,\infty)\times\p^+W$, the continuation Floer trajectories are
contained in $F\cup W$, where the gap between the Hamiltonians is 
$$
gap = \|(K^-_j - K^-_{j'}) \big|_{F\cup W}\|_\infty = \mu_{j'}(1-\delta_{j'})-\mu_j(1-\delta_j). 
$$
In view of Lemma~\ref{lem:gap} it is enough to show that the maximal action of an orbit in $III_{K^-_j}$ is smaller than the minimal action of an orbit in $F_{K^-_{j'}}$ minus the $gap$. This is equivalent to the inequality $\mu_j-\eta_{\mu_j}<\mu_{j'}(1-\delta_{j'})-\big(\mu_{j'}(1-\delta_{j'})-\mu_j(1-\delta_j)\big)$, which is in turn equivalent to $\delta_j\mu_j<\eta_{\mu_j}$. 

To prove the second relation we observe that the map induced by the homotopy is the composition of the maps induced by each of the monotone homotopies which constitute it. For the first homotopy, supported in $F$, there are no trajectories from $III^0_{K^-_j}$ to $III^-_{K^-_{\mu_{j'},\tau_j,\delta_{j'}}}$ by Lemma~\ref{lem:constant}, and there are no trajectories from $III^0_{K^-_j}$ to $III^+_{K^-_{\mu_{j'},\tau_j,\delta_{j'}}}$ by Lemmas~\ref{lem:no-escape} and~\ref{lem:asy}. For the second homotopy, there are no trajectories from $III^0_{K^-_{\mu_{j'},\tau_j,\delta_{j'}}}$ to $III^{>0}_{K^-_{j'}}$ because the homotopy is non-increasing and $III^0_{K^-_{\mu_{j'},\tau_j,\delta_{j'}}}<III^{>0}_{K^-_{j'}}$. This proves the second relation. (Note that one could not argue here using the $gap$.)

As a consequence, we obtain well-defined directed systems in {\sf Kom}
$$
FC_{III^\heartsuit}(K^-_j),\qquad j\to\infty,\qquad \heartsuit\in\{\varnothing,=0,>0\}.
$$

Consider now an indexing parameter $i\in\Z_-$. Given sequences $\mu_i\to\infty$, $\tau_i\to\infty$, $\delta_i\to 0$ as $i\to-\infty$, such that $\mu_i\notin\mbox{Spec}(\p^- W)$, $\tau_i\notin\mbox{Spec}(\p^+W)$, such that $(\mu_i)$ and $(\tau_i)$ are increasing with $|i|$ and $(\delta_i)$ is decreasing with $|i|$, and such that~\eqref{eq:another-equation} is satisfied for each $i$, we define $K^+_i=K^+_{\mu_i,\tau_i,\delta_i}$. Given $i'\le i$ the homotopy from $K^+_{i'}$ at $+\infty$ to $K^+_{i}$ at $-\infty$ defined as the concatenation of the two monotone homotopies from $K^+_{i'}$ to $K^+_{\mu_{i'},\tau_i,\delta_{i'}}$ and from $K^+_{\mu_{i'},\tau_i,\delta_{i'}}$ to $K^+_{i}$ is such that 
$$
F_{K^+_{i'}}\prec I_{K^+_i},\qquad I^{<0}_{K^+_{i'}}\prec I^{=0}_{K^+_i}.
$$
The proof involves arguments entirely similar to the previous ones for the Hamiltonians $K^-$, hence we omit the details. We obtain well-defined inverse systems in {\sf Kom} 
$$
FC_{I^\heartsuit}(K^+_i),\qquad i\to-\infty,\qquad \heartsuit\in\{\varnothing,<0,=0\}. 
$$

\begin{lemma}\label{lem:SHWgeometric-W-rel}
(a) For $\heartsuit\in\{\varnothing,\ge 0, >0,=0,\le 0,<0\}$ we have isomorphisms
$$
SH_*^\heartsuit(W,\p^-W) \cong \lim^{\longrightarrow}_j FH_{III^\heartsuit}(K^-_{j}).
$$

(b) For $\heartsuit\in\{\varnothing,\ge 0, >0,=0,\le 0,<0\}$ we have isomorphisms
$$
SH_*^\heartsuit(W,\p^+W) \cong \lim^{\longleftarrow}_i FH_{I^\heartsuit}(K^+_{i}).
$$
\end{lemma}

\begin{proof}
The proof is similar to the one of Lemma~\ref{lem:SHWgeometric-W}. For
part (a) observe first that the right hand side does not depend on the
choice of the family $K_j^-$ subject to conditions~\eqref{eq:another-equation}.
We pick $\mu_j=\tau_j$ outside the action spectra of $\p^-W$ and
$\p^+W$ such that $\eta_{\mu_j}<\eta_{\tau_j}$, and then $\delta_j$
sufficiently small so that~\eqref{eq:another-equation} holds for all $j$.
Then a similar proof to that of equation~\eqref{eq:tradeactionforI} yields
\begin{equation*}
   SH_*^{(-\infty,\tau_j)}(W,\p^-W)\cong FH_*^{(-\infty,\tau_j)}(K_j^-)\cong FH_{III}(K_j^-). 
\end{equation*}
In the direct limit over $j$ we obtain part (a) for $\heartsuit=\varnothing$.
The cases $\heartsuit=''>0''$ and $\heartsuit=''=0''$ are proved
similarly, and the remaining cases are a formal consequence of these three. 
The proof of part (b) is analogous, where now it suffices to treat the
cases $\heartsuit\in\{\varnothing,=0,<0\}$.
\end{proof}
}

\subsubsection{Hamiltonians for $SH_*^\heartsuit(V)$ inside $\wh W_F$.}\label{sec:Ham-V} 
{\color{black}Heuristically, the construction presented in this section can be viewed as the ``gluing" of the three constructions presented in the two previous sections.}

 We consider a Liouville cobordism pair $(W,V)$ with filling $F$ and write $W=W^{bottom}\circ V\circ W^{top}$. Let 
$$
\mu, \quad \nu_\pm,\quad \tau>0
$$
be such that $\mu\notin\mbox{Spec}(\p^-W)$,
$\nu_\pm\notin\mbox{Spec}(\p^\pm V)$,
$\tau\notin\mbox{Spec}(\p^+W)$. Let $\eta_\mu$, $\eta_{\nu_\pm}$,
$\eta_\tau>0$ be positive real numbers smaller than $1/2$ and smaller
than the distances from $\mu$, $\nu_\pm$, $\tau$ to the corresponding
action spectra. Let 
$$
\delta,\epsilon\in (0,1),\quad R\in(1,\infty)
$$
be such that 
\begin{equation}\label{eq:Hmunutau-conditions-1}
\delta\mu<\eta_\mu,\quad \epsilon\nu_-<\eta_{\nu_-},\quad \nu_+<R\,\eta_{\nu_+},
\end{equation}
and
\begin{equation}\label{eq:Hmunutau-conditions-2}
R(\tau-\eta_\tau)< R(\nu_+-\eta_{\nu_+}) < \nu_+(R-1)< \nu_--\eta_{\nu_-}<\mu-\eta_\mu.
\end{equation}
Note that the second inequality in~\eqref{eq:Hmunutau-conditions-2} is automatic in view of~\eqref{eq:Hmunutau-conditions-1}.
\textcolor{black}{Also note that the inequalities in~\eqref{eq:Hmunutau-conditions-2} impose relations between $\mu$, $\nu_+$, $\nu_-$ and $\tau$. Typically, an ordering of the kind 
$$
\tau\le \nu_+,\qquad \nu_+ R\le \nu_-,\qquad \nu_-\le \mu
$$
is enough to ensure condition~\eqref{eq:Hmunutau-conditions-2} if $\eta_\tau > \eta_{\nu_+}$, $\eta_{\nu_-} > \eta_\mu$ and $\nu_+>1$. These last three conditions are not in the least restrictive, since the parameters $\eta_\tau,\eta_{\nu_\pm},\eta_\mu$ are to be thought of as arbitrarily small, and the slope $\nu_+$ is to be thought of as large. However, the previous three conditions on $\tau,\nu_\pm,\mu$ are restrictive, and among these three the most restrictive one is $\nu_+ R\le \nu_-$: it forces $\nu_-$ to be larger than $\nu_+$, and indeed much larger, in an uncontrolled way. This has implications on the kind of doubly directed systems that we will construct, namely systems for which we can consider first an inverse limit as the negative slopes go to $-\infty$, then a direct limit as the positive slopes go to $+\infty$, but not the other way around.}

We denote by 
$$
H_{\mu,\nu_\pm,\tau}=H_{\mu,\nu_\pm,\tau,\delta,\epsilon,R}:\wh W_F\to \R
$$
the Hamiltonian which is defined up to smooth approximation as
follows: it is constant equal to $\epsilon\mu(1-\delta)+
\nu_-(1-\epsilon)$ on $F\setminus [\delta\epsilon,1]\times\p F$, it is
linear equal to $\mu(\epsilon-\delta\epsilon)+ \nu_-(1-\epsilon) +
\mu(\delta\epsilon-r)$ on $[\delta\epsilon,\epsilon]\times\p F$, it is
constant equal to $\nu_-(1-\epsilon)$ on $\epsilon W^{bottom}$, it is
linear equal to $\nu_-(1-\epsilon) + \nu_-(\epsilon-r)$ on
$[\epsilon,1]\times\p^-V$, it is constant equal to $0$ on $V$, it is
linear equal to $\nu_+(r-1)$ on $[1,R]\times \p^+V$, it is constant
equal to $\nu_+(R-1)$ on $R W^{top}$, and it is linear equal to
$\nu_+(R-1)+\tau(r-R)$ on $[R,\infty)\times\p^+ W$. See
  Figure~\ref{fig:H-one-step}.  

\begin{figure}
         \begin{center}
\input{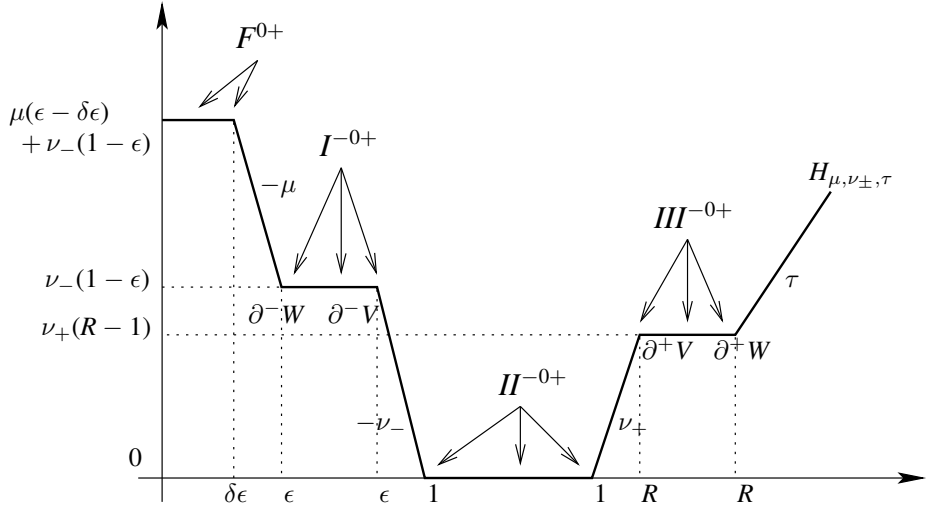}
         \end{center}
\caption{Hamiltonian adapted to the construction of the transfer map $SH_*^\heartsuit(W)\to SH_*^\heartsuit(V)$  \label{fig:H-one-step}}
\end{figure}

The $1$-periodic orbits of the Hamiltonian $H_{\mu,\nu_\pm,\tau}$ fall into 11 classes as follows: 
\begin{itemize}
\item[($F^0$)] constants in $F\setminus ([\delta\epsilon,1]\times \p F)$,
\item[($F^+$)] orbits corresponding to negatively parameterized
  closed Reeb orbits on\break $\p F=\p^-W$ and located near $\delta\epsilon\p^-W$,
\item[($I^-$)] orbits corresponding to negatively parameterized closed Reeb orbits on\break $\p^-W^{bottom}=\p^-W$ and located near $\epsilon\p^-W$,
\item[($I^0$)] constants in $\epsilon W^{bottom}$,
\item[($I^+$)] orbits corresponding to negatively parameterized closed Reeb orbits on\break $\p^+W^{bottom}=\p^-V$ and located near $\epsilon \p^-V$,
\item[($II^-$)] orbits corresponding to negatively parameterized closed Reeb orbits on $\p^-V$ and located near $\p^-V$,
\item[($II^0$)] constants in $V$,
\item[($II^+$)] orbits corresponding to positively parameterized closed Reeb orbits on $\p^+V$ and located near $\p^+V$,
\item[($III^-$)] orbits corresponding to positively parameterized closed Reeb orbits on\break $\p^-W^{top}=\p^+V$ and located near $R\p^+V$,
\item[($III^0$)] constants in $RW^{top}$,
\item[($III^+$)] orbits corresponding to positively parameterized closed Reeb orbits on $\p^+W$ and located near $R\p^+W^{top}=R\p^+W$. 
\end{itemize}

We denote by $F$ the group of orbits $F^{0+}$, and by $J$ the group of orbits $J^{-0+}$ for $J=I,II,III$. 

\begin{lemma}\label{lem:ordering} 
For the previous choices of parameters the above groups of orbits for $H_{\mu,\nu_\pm,\tau}$ are ordered as
$$
   F\prec I\prec III\prec II \quad\text{and}\quad III\prec I,
$$
provided the almost complex structure is cylindrical and stretched enough on a collar neighborhood of $\p^+V$ in $V$.
\end{lemma}

\begin{proof}
The relation $F\prec I$ holds because $F<I$. Indeed, the maximal action of an orbit in $F$ equals $-\epsilon\mu(1-\delta)-\nu_-(1-\epsilon)$ (and is attained on $F^0$). The minimal action of an orbit in $I$ is larger than $-\nu_-(1-\epsilon) + \min(-\epsilon(\mu-\eta_\mu),-\epsilon(\nu_--\eta_{\nu_-}))$. The conclusion follows in view of $\delta\mu<\eta_\mu$ and $\mu(1-\delta)>\mu-\eta_\mu>\nu_--\eta_{\nu_-}$.

{\color{black} The relation $I\prec III$ holds because $I<III$. Indeed, the maximal action of an orbit in $I$ equals $-\nu_-(1-\epsilon)$ (and is attained on $I^0$). The minimal action of an orbit in $III$ is equal to $-\nu_+(R-1)$ (and is attained on $III^0$). The conclusion follows in view of $\nu_+(R-1)<\nu_--\eta_{\nu_-}<\nu_-(1-\epsilon)$. 

The relation $F\prec III$ holds because $F<I<III$ by the above. 

The relation $I\prec II$ holds because $I<II$. Indeed, the maximal action of an orbit in $I$ equals $-\nu_-(1-\epsilon)$. The minimal action of an orbit in $II$ is larger than $-\nu_-+\eta_{\nu_-}$. The conclusion follows in view of $\epsilon\nu_-<\eta_{\nu_-}$. 

The relation $F\prec II$ holds because $F<I<II$ by the above.}

The relation $III\prec II$ is seen as follows. On the one hand we have $III<II^{0+}$. Indeed, the maximal action of an orbit in $III$ is smaller than $-\nu_+(R-1)+\max(R(\nu_+-\eta_{\nu_+}),R(\tau-\eta_\tau))$. The minimal action of an orbit in $II^{0+}$ equals $0$, and the conclusion follows in view of $R(\tau-\eta_\tau)<R(\nu_+-\eta_{\nu_+})<\nu_+(R-1)$. On the other hand we have $III\prec II^-$ by Lemma~\ref{lem:neck} for an almost complex structure which is cylindrical and stretched enough within a collar neighborhood of $\p^+V$ in $V$. 

The relation $III\prec I$ (and actually also $III\prec F$) follows also from Lemma~\ref{lem:neck}.
\end{proof}

{\color{black}
{\bf Remark. }Lemma~\ref{lem:ordering} should be compared to Lemma~\ref{lem:order}
which asserts the same ordering of groups of orbits. The latter concerns 
the simpler Hamiltonians in Figure~\ref{fig:H-cob} and its proof
crucially uses Lemmas~\ref{lem:no-escape} and \ref{lem:asy}. 
The former concerns the more complicated Hamiltonians in
Figure~\ref{fig:H-one-step} (with two additional parameters $\epsilon,R$) and its
proof uses only action estimates and Lemma~\ref{lem:neck}. 
This has the advantage that the ordering in Lemma~\ref{lem:ordering} 
is preserved by continuation maps (see the proof of
Lemma~\ref{lem:ordering-homotopy} below), whereas the one in
Lemma~\ref{lem:order} is not. 
}

We now define a special cofinal family of Hamiltonians in $\cH^W(V;F)$ of the form above. Besides conditions~\eqref{eq:Hmunutau-conditions-1} and~\eqref{eq:Hmunutau-conditions-2}, we will also need a finer relation, stated as~\eqref{eq:special-equation} below, which will be used in order to show that the continuation maps preserve the decomposition into groups of orbits given by Lemma~\ref{lem:ordering}. We will first choose the parameters $\nu_+$, $R$, $\tau$ in the region with positive slopes, and then choose the parameters $\nu_-$, $\epsilon$, $\mu$, $\delta$ in the region with negative slopes. 

\emph{(a) Choice of the parameters in the region with positive slopes.} We start with a sequence $(\nu_{+,j})$, $j\in\Z_+$ consisting of real numbers $\nu_{+,j}\ge 1$, which does not contain elements in $\mbox{Spec}(\p^+V)$, such that $\nu_{+,j}<\nu_{+,j'}$ for $j<j'$, and such that $\nu_{+,j}\to\infty$ as $j\to\infty$. 

We further consider a sequence $(\tau_j)$, $j\in\Z_+$ consisting of
positive real numbers such that $\tau_j\in(\nu_{+,j}/4,\nu_{+,j}/2)$, 
which does not contain elements in $\mbox{Spec}(\p^+W)$, and such that $\tau_j<\tau_{j'}$ for $j<j'$.  

We choose the parameters $\eta_{\nu_{+,j}},\eta_{\tau_j}\in(0,1/2)$ such that they form monotone sequences which converge to $0$.

We then choose a sequence $(R_j)$, $j\in\Z_+$ consisting of numbers $R_j\ge 1$, such that $R_j<R_{j'}$ for $j<j'$ and $R_j\to\infty$, $j\to\infty$, and such that the last condition in~\eqref{eq:Hmunutau-conditions-1} is satisfied under the stronger form: 
\begin{equation} \label{eq:Hmunutau-conditions-1-stronger}
R_j\eta_{\nu_{+,j}}>2\nu_{+,j} \qquad \mbox{ for all } j\in\Z_+.
\end{equation}
(This stronger form of~\eqref{eq:Hmunutau-conditions-1} will be used in Lemma~\ref{lem:SHVgeometric-V}.)
The first two inequalities in~\eqref{eq:Hmunutau-conditions-2} are then satisfied. 

\emph{(b) Choice of the parameters in the region with negative slopes.} We start with a sequence $(\nu_{-,i})$, $i\in\Z_-$ consisting of real numbers $\nu_{-,i}\ge 1$, which does not contain elements in $\mbox{Spec}(\p^-V)$, such that 
{\color{black}
\begin{equation}\label{eq:nu-stronger}
   \nu_{-,i-1}\geq \nu_{-,i}+2\qquad\text{ for all }i\in\Z_-.
\end{equation}
This implies $\nu_{-,i'}\geq \nu_{-,i}+2$ for $i'<i$ and $\nu_{-,i}\to\infty$ as $i\to-\infty$.}
We choose the parameters $\eta_{\nu_{-,i}}\in(0,1/2)$ and such that they form a monotone sequence which converges to $0$. We require that the third inequality in~\eqref{eq:Hmunutau-conditions-2} is satisfied: 
$$
\nu_{+,j}(R_j-1)<\nu_{-,i}-\eta_{\nu_{-,i}} \qquad \mbox{ for all } i\le -j. 
$$
This last condition is implied by $\nu_{-,i}>\nu_{+,-i}(R_{-i}-1)+1/2$, $i\in\Z_-$,  which provides an explicit recipe for the construction. 

We choose a sequence $(\epsilon_i)$, $i\in\Z_-$ of numbers 
{\color{black} $\epsilon_i\in(0,1/2)$}
such that $\epsilon_{i'}<\epsilon_i$ for $i'<i$, such that $\epsilon_i\to 0$, $i\to-\infty$, and such that the second condition in~\eqref{eq:Hmunutau-conditions-1} is satisfied:
$$
\epsilon_i\nu_{-,i}<\eta_{\nu_{-,i}} \qquad \mbox{for all } i\in\Z_-.
$$
We also require that the sequence $1/\epsilon_i$ does not contain any element in $\mbox{Spec}(\p^-W)$, which is a generic property.

{\color{black}We then consider two sequences $(\mu_i)$, $(\delta_i)$, $i\in\Z_-$ such that  
\begin{equation} \label{eq:special-equation}
   \epsilon_{i}\mu_{i}(1-\delta_{i}) = 1\qquad \mbox{for all }i,\in\Z_-
\end{equation}
}
and which moreover satisfy the following conditions: the sequence $(\mu_i)$ consists of positive numbers and does not contain elements of $\mbox{Spec}(\p^-W)$, we have $\mu_{i'}>\mu_i$ for $i'<i$ and $\mu_i\to\infty$, $i\to-\infty$; the sequence $(\delta_i)$ is such that $\delta_i\in(0,1)$ for all $i\in\Z_-$, we have $\delta_{i'}<\delta_i$ for $i'<i$ and $\delta_i\to 0$, $i\to-\infty$; the first condition in~\eqref{eq:Hmunutau-conditions-1} is satisfied: 
$$
\delta_i\mu_i<\eta_{\mu_i} \qquad \mbox{ for all } i\in\Z_-.
$$
Such sequences are easily constructed by choosing $\mu_i$ slightly larger than $1/\epsilon_i$ for all $i\in\Z_-$. 

These conditions imply $\mu_i>1/\epsilon_i>\nu_{-,i}/\eta_{\nu_{-,i}}\ge 2\nu_{-,i}$ for all $i\in\Z_-$, so that the last inequality in~\eqref{eq:Hmunutau-conditions-2} is also satisfied since $\nu_{-,i}\ge 1$.

Let now
$$
H_{i,j}:=H_{\mu_i,\nu_{-,i},\nu_{+,j},\tau_j,\delta_i,\epsilon_i,R_j},\qquad i\in \Z_-,\quad j\in\Z, \quad i\le -j.
$$
{\color{black} 
Then we have $H_{i',j}\ge H_{i,j}$ for $i'\leq i$ and $H_{i,j}\le H_{i,j'}$ for $j\le j'$.
Indeed, the first inequality follows from conditions~\eqref{eq:nu-stronger} and~\eqref{eq:special-equation}, which imply that for $i'<i$ the value of $H_{i',j}$ on $\epsilon_i\p_-V$ satisfies $\nu_{-,i'}(1-\epsilon_i)\geq (\nu_{-,i}+2)(1-\epsilon_i)\geq \nu_{-,i}(1-\epsilon_i)+1 = \nu_{-,i}(1-\epsilon_i)+\epsilon_i\mu_i(1-\delta_i)=\max_FH_{i,j}$.
The second inequality follows from the conditions $\nu_{+,j'}\geq\nu_{+,j}\geq\tau_j$ and $R_{j'}\geq R_j\geq 1$, which imply $(\nu_{+,j'}-\tau_j)(R_{j'}-1)\geq (\nu_{+,j}-\tau_j)(R_{j}-1)$, or equivalently $\nu_{+,j'}(R_{j'}-1)\geq \nu_{+,j}(R_{j}-1)+\tau_j(R_{j'}-R_j)$, so $H_{i,j'}\geq H_{i,j}$ on $R_{j'}\p^+W$ and therefore everywhere.}

We consider $FC_*(H_{i,j})$ as a doubly-directed system in {\sf Kom}, inverse on $i\to-\infty$ and direct on $j\to\infty$, with maps
$$
FC_*(H_{i',j})\to FC_*(H_{i,j}),\qquad i'\le i\le -j
$$
induced by non-decreasing homotopies, and maps 
$$
FC_*(H_{i,j})\to FC_*(H_{i,j'}),\qquad j\le j',\quad i\le -j'
$$
induced by non-increasing homotopies. 
{\color{black}(The non-decreasing homotopies will actually be chosen more specifically, as a composition of ``small distance'' homotopies, see the proof of Lemma~\ref{lem:ordering-homotopy} below.) 
}
{\color{black}The choice of parameters ensures that for each $H_{i,j}$ the groups of
orbits are ordered as in Lemma~\ref{lem:ordering}.}
Denote $FC_F(H_{i,j})$ the Floer subcomplex of $FC_*(H_{i,j})$ generated by orbits in the group $F$, denote $FC_{I,II,III}(H_{i,j})$ the Floer quotient complex generated by orbits in the groups $I,II,III$, and consider similarly $FC_{I,III}(H_{i,j})$ and $FC_{II}(H_{i,j})$. The groups of orbits $II^-$, $II^0$, $II^+$ are ordered by the action as $II^-<II^0<II^+$ within the group of orbits $II$, so that we have corresponding sub- and quotient complexes
$FC_{II^\heartsuit}(H_{i,j})$
for $\heartsuit\in\{\varnothing,\ge 0, >0,=0,\le 0,<0\}$, where $II^\heartsuit$ has the following meaning: 
$$
II^\varnothing=II,  \,   II^{\le 0}=II^{-0},  \,   II^{>0}=II^+,  \,   II^{<0}=II^-,  \,   II^{=0}=II^0,  \,   II^{\ge 0}=II^{0+}.
$$ 
Similarly, we have orderings by the action $I^{-+}<I^0$ within the group $I$, and $III^0<III^{-+}$ within the group $III$, as well as orderings $I\prec III$ and $III\prec I$ from Lemma~\ref{lem:ordering}. We thus define $FC_{(I,III)^\heartsuit}(H_{i,j})$ for $\heartsuit\in\{\varnothing,\ge 0, >0,=0,\le 0,<0\}$ via 
$$
(I,III)^\varnothing=(I,III),\quad (I,III)^{\le 0}=(I,III^0),\quad (I,III)^{>0}=III^{-+},
$$
$$
(I,III)^{<0}=I^{-+},\quad (I,III)^{=0}=(I^0,III^0),\quad (I,III)^{\ge 0}=(I^0,III).
$$

\begin{lemma}\label{lem:ordering-homotopy} 
The homotopies that define the doubly-directed system $FC_*(H_{i,j})$ 
{\color{black} can be chosen so that they}
induce doubly-directed systems  
$$
FC_{II^\heartsuit}(H_{i,j}), \qquad FC_{I^\heartsuit}(H_{i,j}), \qquad FC_{III^\heartsuit}(H_{i,j}) 
\qquad \mbox{and }\qquad FC_{(I,III)^\heartsuit}(H_{i,j}) 
$$
for $i\in\Z_-$, $j\in \Z_+$, $i\le -j$ and $\heartsuit\in\{\varnothing,\ge 0, >0,=0,\le 0,<0\}$.
\end{lemma}

\begin{proof} 
(1) We consider first the continuation maps 
$$
FC_*(H_{i',j})\to FC_*(H_{i,j}),\qquad {\color{black}i'\le i \le -j}
$$ 
induced by non-decreasing homotopies equal to $H_{i',j}$ near $+\infty$ and equal to $H_{i,j}$ near $-\infty$. The positive slopes $\nu_{+,j}$, $\tau_j$ are fixed, as well as the parameter $R_j$, and the homotopy is constant outside $F\circ W^{bottom}$. 

Denote for simplicity $H=H_{i,j}$, $H'=H_{i',j}$, and $\nu_-=\nu_{-,i}$, $\nu_-'=\nu_{-,i'}$, $\epsilon=\epsilon_i$, $\epsilon'=\epsilon_{i'}$, $\mu=\mu_i$, $\mu'=\mu_{i'}$. The \emph{gap} $\|H-H'\|_\infty$ between the two Hamiltonians is equal to the biggest value among $(1-\epsilon')\nu'_--(1-\epsilon)\nu_-$ (the difference of values in the region $I^0$) and $(1-\epsilon')\nu'_-+\epsilon'\mu'(1-\delta') - (1-\epsilon)\nu_--\epsilon\mu(1-\delta)$ (the difference of values in the region $F^0$). Condition~\eqref{eq:special-equation} ensures that these two values are equal, hence 
$$
gap:=\|H-H'\|_\infty = (1-\epsilon')\nu'_--(1-\epsilon)\nu_-\,. 
$$
{\color{black}In the sequel we will repeatedly apply Lemma~\ref{lem:gap} (without
further mentioning it), which asserts that for two groups of orbits
$P_{H_+} < P_{H_-} - gap$ implies $P_{H_+}\prec P_{H_-}$.}

We first prove that 
$$
{\color{black}F_{H'},I_{H'}\prec II_H},
$$ 
so that we have induced maps $FC_{II}(H')\to FC_{II}(H)$. We have 
$F^0_{H'}+gap < I^0_{H'}+gap <II^-_H$: the first inequality is obvious, and the second inequality is equivalent to $-(1-\epsilon)\nu_-<-\nu_-+\eta_{\nu_-}$, which is implied by $\epsilon\nu_-<\eta_{\nu_-}$. This ensures $F_{H'}\prec II_H$ and $I_{H'}\prec II_H$. 

{\color{black} We now prove 
$$
III_{H'}\prec (F,I,II)_H.
$$
Note that $H$ and $H'$ coincide in the regions $II^{0+}$ and $III$, and from the proof of Lemma~\ref{lem:ordering} we know that $III_H<II^{0+}_H$. 
{\color{black}The conditions $III_{H'}\prec (F,I,II^-)_H$ follow from Lemma~\ref{lem:neck}.} 
To prove the condition $III_{H'}\prec II^{0+}_H$, we cannot argue directly by action considerations as in the proof of $III_H\prec II^{0+}_H$ since the gap between $H$ and $H'$ could be arbitrarily large. Instead, we use again $III_H<II^{0+}_H$, so we can find some $\epsilon>0$ such that $III_{H}<II^{0+}_H -\epsilon$. \emph{We specialize now to non-decreasing homotopies from $H$ to $H'$ which are compositions of ``small distance" homotopies} with gap smaller than $\epsilon$. (This can alway be achieved by cutting and reparametrizing a given homotopy.) Note that all the homotopies are fixed on $II^{0+}$ and $III$. For each of these small distance homotopies, say running from $H_-$ at $-\infty$ to $H_+$ at $+\infty$, we then have $III_{H_+}\prec II^{0+}_{H_-}$ by Lemma~\ref{lem:gap}, and we also have $III_{H_+}\prec (F,I,II^-)_{H_-}$ by Lemma~\ref{lem:neck}. In other words $III_{H_+}\prec (F,I,II)_{H_-}$ and the image through the continuation map of a generator in $III_{H_+}$ lies in $III_{H_-}$. As a result, the image of a generator in $III_{H'}$ through a composition of such ``small distance" homotopies lies in $III_{H}$ and we have $III_{H'}\prec (F,I,II)_H$. (This reproves in particular $III_{H'}\prec (F,I,II^-)_H$). 
}

We now prove that 
$$
F_{H'}\prec I_{H},III_H,
$$
wherefrom induced maps $FC_{I,II,III}(H')\to FC_{I,II,III}(H)$ and $FC_{I,III}(H')\to FC_{I,III}(H)$. The relation $F_{H'}\prec I_H$ follows from $F^0_{H'}+gap <\min(I^-_H,I^+_H)$, which is $-\epsilon'(1-\delta')\mu'-(1-\epsilon)\nu_-<-(1-\epsilon)\nu_- + \min(-\epsilon(\mu-\eta_\mu),-\epsilon(\nu_--\eta_{\nu_-}) = -(1-\epsilon)\nu_- - \epsilon(\mu-\eta_\mu)$. This is equivalent to $-(1-\delta)\mu<-(\mu-\eta_\mu)$ in view of~\eqref{eq:special-equation}, and holds in view of $\delta\mu<\eta_\mu$. The relation $F_{H'}\prec III_H$ follows from the previous one: indeed $I_H<III_H$, hence $F^0_{H'}+gap<III_H$. 

We also have 
$$
{\color{black}I_{H'}\prec III_H}.
$$
This is a consequence of $I^0_{H'}+gap < III^0_H$, which is $-(1-\epsilon')\nu'_- + (1-\epsilon')\nu'_--(1-\epsilon)\nu_- < -\nu_+(R-1)$, which is equivalent to $\nu_+(R-1)<(1-\epsilon)\nu_-$ and is implied by~\eqref{eq:Hmunutau-conditions-1} and~\eqref{eq:Hmunutau-conditions-2}. {\color{black}Since we already proved $III_{H'}\prec I_H$, we infer that} the continuation maps therefore preserve the decomposition $FC_{I,III}(H)=FC_I(H)\oplus FC_{III}(H)$. 

We now prove that 
$$
II^-_{H'}\prec II^{0+}_H \qquad \mbox{ and } \qquad II^{-0}_{H'}\prec II^+_H,
$$
so that we have induced maps $FC_{II^\heartsuit}(H')\to
FC_{II^\heartsuit}(H)$ for all values of $\heartsuit$. The first
relation follows from Lemmas~\ref{lem:no-escape},~\ref{lem:asy},
and~\ref{lem:constant}, while the last relation follows from
Lemmas~\ref{lem:no-escape} and~\ref{lem:asy}
{\color{black}(using $H'=H$ outside $F\circ W^{bottom}$).} 
Note that in this
situation we cannot argue using the action because the homotopy only
preserves the action filtration up to an error given by the $gap$, and
the latter can be arbitrarily large.  

We now prove that 
$$
I^{-+}_{H'}\prec (I^0_H,III_H) \qquad \mbox{ and } \qquad (I_{H'},III^0_{H'})\prec III^{-+}_H, 
$$
which implies that we have induced maps $FC_{(I,III)^\heartsuit}(H')\to FC_{(I,III)^\heartsuit}(H)$ for all values of $\heartsuit$. 

In view of $I_{H'}\prec III_H$, the first relation is a consequence of
$I^{-+}_{H'}\prec I^0_H$, which is in turn implied by $I^{-+}_{H'}+gap
< I^0_H$. The latter is seen to hold as follows. Denote by
$T_{\p^-V}$, $T_{\p^-W}$ the minimal period of a closed Reeb orbit on
$\p^-V$, respectively on ${\p^-W}$, and set $T_-:=\min(T_{\p^-V},
T_{\p^-W})>0$. The desired inequality is implied by
$-(1-\epsilon')\nu'_--\epsilon'T_- + (1-\epsilon')\nu'_- -
(1-\epsilon)\nu_- < -(1-\epsilon)\nu_-$, which holds because
$-\epsilon'T_-<0$.  

In view of $I_{H'}\prec III_H$, the second relation is a consequence
of $III^0_{H'}\prec III^{-+}_H$. The relation $III^0_{H'}\prec
III^+_H$ is a consequence of Lemmas~\ref{lem:no-escape}
and~\ref{lem:asy} in view of the fact that the homotopy is constant
{\color{black}outside $F\circ W^{bottom}$.} 
The relation $III^0_{H'}\prec III^-_H$ is a consequence of Lemma~\ref{lem:constant}. Note that in both situations we cannot argue using the action because the homotopy only preserves the action filtration up to an error given by the $gap$, and the latter can be arbitrarily large. 

(2) We now consider the continuation maps 
$$
FC_*(H_{i,j})\to FC_*(H_{i,j'}),\qquad {\color{black}j\le j'\le -i} 
$$ 
induced by non-increasing homotopies equal to $H_{i,j}$ near $+\infty$
and equal to $H_{i,j'}$ near $-\infty$. The negative slopes
$\nu_{-,i}$, $\mu_i$ are fixed, as well as the parameters
$\epsilon_i,\delta_i$, and the homotopy is constant on $F\circ
W^{bottom}\circ V$.  
{\color{black}This situation is easier than the one in (1) because 
here the continuation maps preserve the action filtration.}

Denote again for simplicity $H=H_{i,j}$, $H'=H_{i,j'}$, and $\nu_+=\nu_{+,j}$, $\nu_+'=\nu_{+,j'}$, $R=R_j$, $R'=R_j'$, $\tau=\tau_j$, $\tau'=\tau_{j'}$. 

The relations 
$$
F_H\prec I_{H'},II_{H'},III_{H'} \qquad \mbox{ and } \qquad I_H\prec II_{H'}
$$
follow as in Lemma~\ref{lem:ordering}. On the one hand we have
$I_{H'}=I_H$ and $II_{H'}^{-0}=II_H^{-0}$, so that $F_H\prec
I_{H'},II^{-0}_{H'}$ and $I_H\prec II^{-0}_{H'}$. On the other hand we
have $F^0_H<II^0_H=II^0_{H'}<II^+_{H'}$ and
$F^0_H=F^0_{H'}<III^0_{H'}<III^{-+}_{H'}$ for $i\le -j'$ which implies
$F_H\prec II^+_{H'},III_{H'}$. Finally, we also have 
{\color{black}$I_H=I{_H'}<II^0_{H'}<II^+_{H'}$,} 
which implies $I_H\prec II_{H'}$. 

The relation 
$$
III_H\prec II_{H'}
$$
is proved as follows. We have $III_H\prec II^-_{H'}$ as in
Lemma~\ref{lem:ordering}, using Lemma~\ref{lem:neck}. We have
$III^{0+}_H<II^{0+}_{H'}$ by~\eqref{eq:Hmunutau-conditions-2}, namely
$R(\tau-\eta_\tau)<\nu_+(R-1)$. Finally we have 
{\color{black}$III^-_H<II^{0+}_{H'}$}
by~\eqref{eq:Hmunutau-conditions-1}, namely $R\eta_{\nu_+}>\nu_+$. 

The relation 
$$
III_H\prec I_{H'}
$$
is proved as in Lemma~\ref{lem:ordering}, using Lemma~\ref{lem:neck}.

The continuation map 
$$
FC_{II}(H)\to FC_{II}(H') 
$$
is induced by a non-increasing homotopy hence preserves the filtration by the action. As a consequence we obtain well-defined continuation maps 
$$
FC_{II^\heartsuit}(H)\to FC_{II^\heartsuit}(H') 
$$
for all values of $\heartsuit$. 

Let us now prove that the continuation map 
$$
FC_{I,III}(H)\to FC_{I,III}(H')
$$
induces maps 
$$
FC_{(I,III)^\heartsuit}(H)\to FC_{(I,III)^\heartsuit}(H')
$$
for all values of $\heartsuit$. We need to show the relations
$I^{-+}_H\prec I^0_{H'},III_{H'}$ and $I_H,III^0_H\prec
III^{-+}_{H'}$. The first relation follows from
$I^{-+}_H<I^0_H=I^0_{H'}<III^0_{H'}<III^{-+}_{H'}$, where the middle
inequality is ensured 
{\color{black}by~\eqref{eq:Hmunutau-conditions-1} and~\eqref{eq:Hmunutau-conditions-2}, namely 
$\nu_+(R-1)<\nu_{-}-\eta_{\nu_{-}}<\nu_{-}(1-\epsilon_i)$.} 
The second relation follows from $I^0_H<III^0_{H'}<III^{-+}_{H'}$.  

The above shows that we actually have non-interacting doubly-directed systems
$$
FC_{I^\heartsuit}(H_{i,j})\qquad \mbox{and }\qquad FC_{III^\heartsuit}(H_{i,j}) 
$$
for all values of $\heartsuit$ and Lemma~\ref{lem:ordering-homotopy}
is proved.
\end{proof}

\begin{lemma}\label{lem:SHVgeometric-V} We have isomorphisms 
$$
SH_*^\heartsuit(V)\cong \lim^{\longrightarrow}_j \lim^{\longleftarrow}_i FH_{II^\heartsuit}(H_{i,j})
$$
for $\heartsuit\in\{\varnothing,\ge 0, >0,=0,\le 0,<0\}$.
\end{lemma}

\begin{proof} The proof is very much similar to that of Lemma~\ref{lem:SHWgeometric-W}.  
Recalling that the slopes near $\p^\pm V$ for $H_{i,j}$ are $-\nu_{-,i}$ and $\nu_{+,j}$, the key identity is 
\begin{equation} \label{eq:tradeactionforII}
SH_*^{(-\nu_{-,i},\nu_{+,j})}(V)\cong FH_{II}(H_{i,j}).
\end{equation}
To prove~\eqref{eq:tradeactionforII} recall from
Lemma~\ref{lem:SHV-alternate-direct-limit} that $SH_*^{(a,b)}(V)$ can
be expressed as a direct limit over Hamiltonians in $H^W(V;F)$ of
Floer homology groups truncated in the action window $(a,b)$. In
particular, considering a decreasing sequence $i_k\to-\infty$ and an
increasing sequence $j_k\to\infty$ as $k\to\infty$
with $i_k\leq -j_k$, we have $SH_*^{(a,b)}(V)={\displaystyle
  \lim^{\longrightarrow}_{k\to\infty}} \,
FH_*^{(a,b)}(H_{i_k,j_k})$. Here the direct limit is understood with
respect to continuation maps $FH_*^{(a,b)}(H_{i_k,j_k})\to
FH_*^{(a,b)}(H_{i_{k'},j_{k'}})$ induced by non-increasing homotopies. 

We claim that for $k$ large enough such that $\nu_{+,j_k}\geq-a$ 
we have $FC_*^{(a,b)}(H_{i_k,j_k})=FC_{II}^{(a,b)}(H_{i_k,j_k})$. The
proof is similar to the proof of Lemma~\ref{lem:SHV-alternate-direct-limit}:
We need to show that the actions of orbits in groups $F$, $I$ and
$III$ are below $a$. For the groups $F$ and $I$ this is obvious. The 
actions within group $III$ are ordered as $III^0<III^{-+}$. The maximal action of 
the orbits in group $III^{-}$ is bounded above by 
$-\nu_+(R-1)+R(\nu_+-\eta_{\nu_+}) = \nu_+-R\eta_{\nu_+} < -\nu_+\leq a$,
where we have dropped the index $j_k$ and the first inequality follows
from condition~\eqref{eq:Hmunutau-conditions-1-stronger}. Similarly, 
the maximal action of the orbits in group $III^+$ is bounded above by 
$-\nu_+(R-1)+R(\tau-\eta_{\tau}) < -\nu_+(R-1)+R(\nu_+-\eta_{\nu_+})<a$, 
where the first inequality follows from~\eqref{eq:Hmunutau-conditions-2} 
and the second one from the one for group $III^-$.   
Combining this with the previous paragraph we obtain
$$
SH_*^{(a,b)}(V)={\displaystyle \lim^{\longrightarrow}_{k\to\infty}}
\, FH_{II}^{(a,b)}(H_{i_k,j_k}).  
$$
Assume now without loss of generality that $-\nu_{-,i_k}\le a$ and
$\nu_{+,j_k}\ge b$. The smoothings of any such two Hamiltonians
$H_{i_k,j_k}$ and $H_{i_{k'},j_{k'}}$, $k\le k'$ can be constructed so
that they coincide in the neighborhood of $V$ where the periodic
orbits in group $II$ for $H_{i_k,j_k}$ appear. As such, the
continuation map $FC_{II}^{(a,b)}(H_{i_k,j_k})\to
FC_{II}^{(a,b)}(H_{i_{k'},j_{k'}})$, which is upper triangular if we
arrange the generators in increasing order of the action, has diagonal
entries equal to $+1$ and is therefore an isomorphism. This proves
that we have a canonical isomorphism
$FH_{II}^{(a,b)}(H_{i_k,j_k})\cong SH_*^{(a,b)}(V)$ for all $k$ (such
that $-\nu_{-,i_k}\le a$ and $\nu_{+,j_k}\ge b$). 
This implies~\eqref{eq:tradeactionforII} by choosing $a=-\nu_{-,i}$ and $b=\nu_{+,j}$. 

A variant of this argument shows that, under the isomorphism~\eqref{eq:tradeactionforII}, the continuation maps $FH_{II}(H_{i',j})\to FH_{II}(H_{i,j})$, $i'\leq i$ and $FH_{II}(H_{i,j})\to FH_{II}(H_{i,j'})$, $j\leq j'$ induced by a non-decreasing homotopy, respectively by a non-increasing homotopy, coincide with the canonical maps $SH_*^{(-\nu_{-,i'},\nu_{+,j})}(V)\to SH_*^{(-\nu_{-,i},\nu_{+,j})}(V)$ and $SH_*^{(-\nu_{-,i},\nu_{+,j})}(V)\to SH_*^{(-\nu_{-,i},\nu_{+,j'})}(V)$, respectively. The conclusion of the Lemma follows in the case $\heartsuit=\varnothing$.

The proof in the case $\heartsuit\neq\varnothing$ is similar, as in Lemma~\ref{lem:SHWgeometric-W}.
\end{proof}

\begin{lemma}\label{lem:SHWVgeometric-WV-bottom-top} We have isomorphisms 
$$
SH_*^\heartsuit(W^{bottom},\p^+W^{bottom})\cong \lim^{\longrightarrow}_j \lim^{\longleftarrow}_i FH_{I^\heartsuit}(H_{i,j})
$$
and
$$
SH_*^\heartsuit(W^{top},\p^-W^{top})\cong \lim^{\longrightarrow}_j \lim^{\longleftarrow}_i FH_{III^\heartsuit}(H_{i,j})
$$
for $\heartsuit\in\{\varnothing,\ge 0, >0,=0,\le 0,<0\}$.
\end{lemma}

\begin{proof}
(1) We prove the first isomorphism. Since the group of orbits $I$ is
  located in the region where the Hamiltonians $H_{i,j}$ have negative
  slope the direct limit over $j$ plays no role and we can assume
  without loss of generality that $j=j_0$ is constant. The Floer
  trajectories involved in the differential for $FC_I(H_{i,j})$ and
  also the relevant continuation Floer trajectories are confined to a
  neighborhood of $F\circ W^{bottom}$ by Lemma~\ref{lem:no-escape}. We
  can thus replace the Hamiltonians $H_i=H_{i,j_0}$ by Hamiltonians
  $\wt H_i$ which coincide with $H_i$ in $F\circ W^{bottom}\circ V$
  and are constant equal to $0$ on $V\circ
  W^{top}\circ[1,\infty)\times\p^+W$. We can further shift these
    Hamiltonians to $\ol H_i=\wt H_i-\nu_{-,i}(1-\epsilon_i)$ so that
    the orbits in group $I$ lie on level $0$, and further replace $\ol
    H_i$ by $\cH_i=\epsilon_i\ol H_i\circ \varphi_Z^{\ln
      1/\epsilon_i}$, so that the orbits in group $I$ for $\cH_i$ lie
    in a neighborhood of $W^{bottom}$, and the slopes of $\cH_i$ in
    the linear regions are the same as the slopes of $\ol
    H_i$. Finally, we can further replace the Hamiltonians $\cH_i$ by
    $\wt \cH_i$ defined on $\wh W^{bottom}_F$ which coincide with
    $\cH_i$ on $F\circ W^{bottom}$ and continue on
    $[1,\infty)\times\p^+W^{bottom}$ linearly with the same slope
      $-\nu_{-,i}$. 
The resulting inverse system is cofinal and, by Lemma~\ref{lem:SHWgeometric-W-rel}(b), it computes $SH_*^\heartsuit(W^{bottom},\p^+W^{bottom})$.

(2) We prove the second isomorphism. Since the group of orbits $III$
is located in the region where the Hamiltonians $H_{i,j}$ have
positive slope, the inverse limit over $i$ plays no role. {\color{black}
Consider the Hamiltonian $\wt H_j$ which
coincides with $H_{i,j}$ on $V\circ W^{top}\circ
[1,\infty)\times\p^+W$, and is constant equal to $0$ on $F\circ
  W^{bottom}\circ V$. The complex $FC_{III}(\wt H_j)$ is well-defined by the same action considerations which show that $III_{H_{i,j}}<II^{0+}_{H_{i,j}}$. Consider a non-increasing homotopy from $H_{i,j}$ at $-\infty$ to $\wt H_j$ at $+\infty$, and also the reverse non-decreasing homotopy from $\wt H_j$ at $-\infty$ to $H_{i,j}$ at $+\infty$. We claim that these homotopies induce chain maps between $FC_{III}(H_{i,j})$ and $FC_{III}(\wt H_j)$ which are homotopy inverses to each other. We first prove that $III_{\wt H_j}\prec (F,I,II)_{H_{i,j}}$ and $III_{H_{i,j}}\prec (F,II)_{\wt H_j}$, where in the latter case $F$ stands for critical points in $F\circ W^{bottom}$ and $II=II^{0+}$. The first relation follows from Lemma~\ref{lem:neck} for $(F,I,II^-)_{H_{i,j}}$ and from action considerations for $II^{0+}_{H_{i,j}}$ since the homotopy is non-increasing. The second relation follows from Lemmas~\ref{lem:no-escape} and~\ref{lem:asy} for $III^-_{H_{i,j}}$, from Lemma~\ref{lem:constant} for $III^0_{H_{i,j}}$, and it also follows for $III^+_{H_{i,j}}$ by specializing to homotopies which are compositions of ``small distance" homotopies as in the proof of Lemma~\ref{lem:ordering-homotopy}. As a result, the induced chain maps between $FC(H_{i,j})$ and $FC(\wt H_j)$ preserve the subcomplexes generated by $III_{H_{i,j}}$ and $III_{\wt H_j}$. These chain maps are homotopy inverses of each other, and a similar argument shows that the corresponding chain homotopies also preserve the subcomplexes generated by $III_{H_{i,j}}$ and $III_{\wt H_j}$. This proves the claim. 
}

We can now further shift these Hamiltonians $\wt H_j$ to $\ol
  H_j=\wt H_j-\nu_{+,j}(R_j-1)$ so that the orbits in group $III$ lie
  on level $0$, and further replace $\ol H_j$ by $\cH_j=R_j\ol
  H_j\circ \varphi_Z^{\ln 1/R_j}$, so that the orbits in group $III$
  for $\cH_j$ lie in a neighborhood of $W^{top}$. 
The resulting direct system is cofinal and, by Lemma~\ref{lem:SHWgeometric-W-rel}(a), it computes $SH_*^\heartsuit(W^{top},\p^- W^{top})$. 
\end{proof}

{\color{black}
Lemmas~\ref{lem:SHVgeometric-V} and~\ref{lem:SHWVgeometric-WV-bottom-top} imply
that for all flavors $\heartsuit$ we have isomorphisms 
$$
SH_*^\heartsuit(W^{top},\p^-W^{top})\oplus SH_*^\heartsuit(W^{bottom},\p^+W^{bottom})\cong \lim^{\longrightarrow}_j \lim^{\longleftarrow}_i FH_{(I,III)^\heartsuit}(H_{i,j}).
$$
On the other hand, by the Excision Theorem~\ref{thm:excision} we have isomorphisms
$$
SH_*^\heartsuit(W,V) \cong SH_*^\heartsuit(W^{bottom},\p^-V) \oplus SH_*^\heartsuit(W^{top},\p^+V). 
$$
Combining these isomorphisms we obtain

\begin{corollary} \label{cor:SHWVgeometric-WV} 
We have isomorphisms 
$$
SH_*^\heartsuit(W,V) \cong \lim^{\longrightarrow}_j \lim^{\longleftarrow}_i FH_{(I,III)^\heartsuit}(H_{i,j})
$$
for $\heartsuit\in\{\varnothing,\ge 0, >0,=0,\le 0,<0\}$.
\hfill{$\square$}
\end{corollary}
}
}

{\color{black}
\subsubsection{The transfer map revisited} \label{sec:transfer-map-revisited}
Consider again a Hamiltonian $H=H_{\mu,\nu_\pm,\tau}$ as in Figure~\ref{fig:H-one-step} above. 
We associate to it a new Hamiltonian $L\leq H$ defined as follows: it
is constant equal to $\mu(\eps-\delta\eps) + \nu_-(1-\eps)$ on
$F\setminus[\delta\eps,1]\times\p F$, it is linear of slope $-\mu$ on
$[\delta\eps,\xi]\times\p F$, it is constant equal to $0$ on
$[\xi,1]\times\p F\cup W\cup[1,R]\times\p^+W$, and it is linear of
slope $\tau$ on $[R,\infty)\times\p^+W$. See Figure~\ref{fig:L}. 

\begin{figure}
         \begin{center}
\input{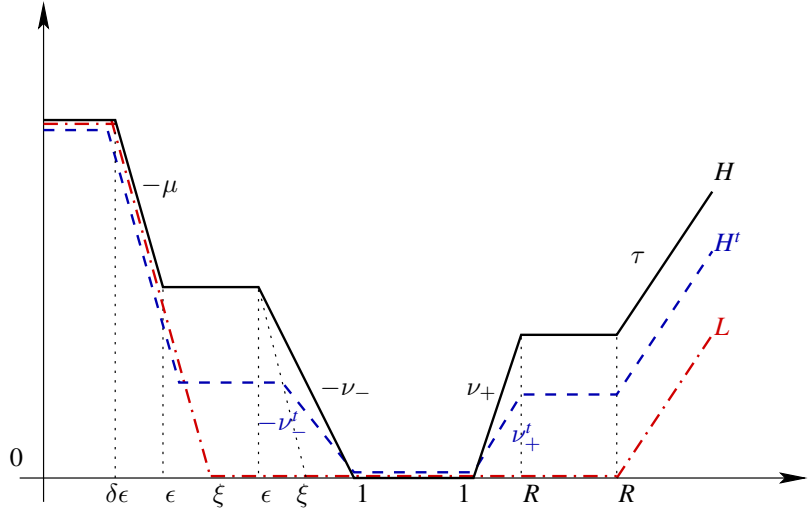}
         \end{center}
\caption{Hamiltonian $L$ for the construction of the transfer map \label{fig:L}}
\end{figure}

Here the constant $\xi$ is determined by the construction and given by
$$
   \xi = \frac{\nu_-}{\mu}(1-\eps)+\eps \in (\eps,1). 
$$
The orbits of the Hamiltonian $L$ fall as usual into 5 groups
$F^{0+},I^{-0+}$ and we have $F<I^-<I^0<I^+$. Indeed, the smallest
action of an orbit in group $I^-$ is $-\xi(\mu-\eta_\mu)$,
whereas the largest action of an orbit in group $F$ is
$-\mu(\xi-\delta\epsilon)$, and we have
$-\mu(\xi-\delta\epsilon)<-\xi(\mu-\eta_\mu)$, which is equivalent to
$\mu\delta\eps<\xi\eta_\mu$, in view of $\mu\delta<\eta_\mu$ and
$\epsilon<\xi$. {\color{black}Arguing differently, for the Hamiltonian $L$ we have $F\prec I$ regardless of the choice of parameters using Lemmas~\ref{lem:no-escape},~\ref{lem:asy} and~\ref{lem:constant}, and the orbits within each of the groups $F$ and $I$ are naturally ordered by the action as $F^+<F^0$ and $I^-<I^0<I^+$.}

Consider now a Hamiltonian $K:=K_{\mu,\tau,\delta'}$ as in
Figure~\ref{fig:Kmutau}, with $\delta'\in(0,1)$ such that
$\mu\delta'<\eta_\mu$ and $\mu(1-\delta')>\mu(\xi-\delta\eps)$,
i.e. the maximal level of $K$ is larger than the maximal level of $L$.  
We then have $L\leq K$.

{\color{black}\begin{lemma}\label{lem:L-K}
The homotopy from $K$ to $L$ given by slow convex interpolation induces 
for all flavors $\heartsuit$ homotopy equivalences
$$
   FC_{I^\heartsuit}(L)\stackrel{\sim}\longrightarrow FC_{I^\heartsuit}(K).
$$
\end{lemma}

\begin{proof}
Although the homotopy is decreasing in the convex end, the Floer equation remains unchanged in the region $\{r\ge R\}$ where the Hamiltonians $L$ and $K$ have the same slope. So the maximum principle applies and the continuation map $FC(L)\to FC(K)$ is well-defined. It is a homotopy equivalence with homotopy inverse given by the continuation map induced by the reverse homotopy from $L$ to $K$. 

We assume without loss of generality that $L$ has no critical points in $[\xi,1]\times \p F \cup [1,R]\times \p^+W$ and that it coincides with $K$ on $W$. 

It is useful to define the following Hamiltonians: $LK$ is equal to $L$ on $F$ and is equal to $K$ on $W \cup [1,\infty)\times \p_+W$, and $KL$ is equal to $K$ on $F\circ W$ and is equal to $L$ on $[1,\infty)\times \p^+W$. We accordingly have chain homotopy equivalences $FC(L)\to FC(KL)\to FC(K)$ and also $FC(L)\to FC(LK)\to FC(K)$ induced respectively by homotopies supported in the positive/negative end. We will show that we have corresponding chain homotopy equivalences $FC_{I^\heartsuit}(L)\to FC_{I^\heartsuit}(KL)\to FC_{I^\heartsuit}(K)$ for all flavors $\heartsuit$. The same statement holds if we replace $KL$ with $LK$, but we will not use it. 

We first consider the homotopies connecting $L$ and $KL$, supported in the negative end, and show that they induce chain maps $FC_{I^\heartsuit}(L)\to FC_{I^\heartsuit}(KL)$ and $FC_{I^\heartsuit}(KL)\to FC_{I^\heartsuit}(L)$ which are homotopy inverses of each other for all flavors $\heartsuit$. We first consider the non-decreasing homotopy from $L$ to $KL$, constant on $W\cup [1,\infty)\times\p^+W$. Each element in the homotopy is of the following form: outside $F$ it coincides with $L$, and inside $F$ it is  linear of slope $-\mu$ in some region $[a,b]\times \p F$ with $0<a<b\le 1$ depending continuously on the Hamiltonian; it is constant equal to $0$ on $\{b\le r\le 1\}$ and it is constant equal to $\mu(b-a)$ on $\{r\le a\}$. Also, each element in the homotopy satisfies $F<I^-<I^0<I^+$. We can decompose the homotopy into ``small distance" homotopies of gap $e>0$ small enough so that, at the endpoints $L_\pm$ of each such homotopy, we have $F_{L_+}<I_{L_-}-e$, $I^-_{L_+}<I^{0+}_{L_-}-e$, $I^0_{L_+}<I^+_{L_-}-e$. This ensures that we have induced chain maps $FC_{I^\heartsuit}(L_+)\to FC_{I^\heartsuit}(L_-)$ for all flavors $\heartsuit$, and the result of the composition is a continuation chain map $FC_{I^\heartsuit}(KL)\to FC_{I^\heartsuit}(L)$. By considering the reverse homotopy, the same argument produces a chain map $FC_{I^\heartsuit}(L)\to FC_{I^\heartsuit}(KL)$. The same argument applied in $1$-parametric families shows that each of the small distance chain maps $FC_{I^\heartsuit}(L_+)\to FC_{I^\heartsuit}(L_-)$ is a chain homotopy equivalence, and so is their composition. 

The same arguments show that we have chain homotopy equivalences $FC_{I^\heartsuit}(KL)\to FC_{I^\heartsuit}(K)$ for all flavors $\heartsuit$. By composition we obtain chain homotopy equivalences $FC_{I^\heartsuit}(L)\to FC_{I^\heartsuit}(K)$ for all flavors $\heartsuit$. 
\end{proof}

\noindent {\bf Remark.} We have used an argument based on ``small distance" isomorphisms also in the proof of Lemma~\ref{lem:ordering-homotopy}. It is likely that it can be used in order to simplify further the proof of Lemma~\ref{lem:ordering-homotopy}. 

}

Consider now a doubly-directed system $H_{i,j}$ as in Section~\ref{sec:Ham-V}.
Let $L_{i,j}$ and $K_{i,j}$ be the Hamiltonians associated
to $H_{i,j}$ as in the previous paragraph. We turn $L_{i,j}$ into a doubly
directed system in {\sf Kom} by composing the continuation maps $FC(K_{i',j})\to
FC(K_{i,j})$ and $FC(K_{i,j})\to FC(K_{i,j'})$ with the canonical maps in
Lemma~\ref{lem:L-K} and their inverses. (Note that in general we do
not have $L_{i',j}\geq L_{i,j}$ for $i'\leq i\leq -j$.) Then all the
results for the system $K_{i,j}$ in \S\ref{sec:Ham-W} carry over
to the system $L_{i,j}$. 

Recall that $L_{i,j}\leq H_{i,j}$ and the orbits in group $F$ for
$L_{i,j}$ and $H_{i,j}$ coincide. Therefore, by Lemma~\ref{lem:ordering} 
the actions of the orbit groups satisfy $F_{L_{i,j}}<
(I,II,III)_{H_{i,j}}$. We thus obtain induced chain maps  
$$
   f_{i,j}:FC_{I}(L_{i,j})\to FC_{I,II,III}(H_{i,j}) \to FC_{II}(H_{i,j})
$$
which define a morphism of doubly-directed systems in {\sf Kom}. Here
the first map is the continuation map and the second one the
projection onto the quotient complex in view of Lemma~\ref{lem:ordering}. 
Since these maps preserve the filtration by action, we also have
induced chain maps
$$
   f_{i,j}^\heartsuit:FC_{I^\heartsuit}(L_{i,j})\to FC_{II^\heartsuit}(H_{i,j})
$$
for $\heartsuit\in\{\varnothing,\ge 0, >0,=0,\le 0,<0\}$, which define
morphisms of doubly-directed systems in {\sf Kom}.  
We denote $(f_{i,j}^\heartsuit)_*$ the maps induced in homology. 

\begin{lemma}
Under the isomorphisms of Lemmas~\ref{lem:SHWgeometric-W},
\ref{lem:SHVgeometric-V} and~\ref{lem:L-K} we have 
$$
f_!^\heartsuit = \lim^{\longrightarrow}_j \lim^{\longleftarrow}_i \, (f_{i,j}^\heartsuit)_*,
$$
where $f_!^\heartsuit: SH_*^\heartsuit(W)\to SH_*^\heartsuit(V)$ is
the transfer map from Definition~\ref{defi:transfer-map}. 
\end{lemma}

\begin{proof} 
Recall from~\eqref{eq:tradeactionforI} and Lemma~\ref{lem:L-K} the isomorphisms 
$$
   SH_*^{(-\mu_i,\tau_j)}(W)\cong FH_{I}(K_{i,j})\cong FH_{I}(L_{i,j}).
$$
Recall also from~\eqref{eq:tradeactionforII} the isomorphism
$$
SH_*^{(-\nu_{-,i},\nu_{+,j})}(V)\simeq FH_{II}(H_{i,j}).
$$
Recall that $\mu_i\ge\nu_{-,i}$ and $\tau_j\le \nu_{+,j}$.
It follows from the proofs of Lemmas~\ref{lem:SHWgeometric-W}
and~\ref{lem:SHVgeometric-V} that the continuation map
$(f_{i,j})_*:FH_{I}(L_{i,j})\to FH_{II}(H_{i,j})$ coincides via the
above isomorphisms with the composition of the transfer map
$f_!^{(-\mu_i,\tau_j)}:SH_*^{(-\mu_i,\tau_j)}(W)\to
SH_*^{(-\mu_i,\tau_j)}(V)$ with the canonical map given by enlarging/restric\-ting
the action window $SH_*^{(-\mu_i,\tau_j)}(V)\to
SH_*^{(-\nu_{-,i},\nu_{+,j})}(V)$, i.e. 
$$
\xymatrix
{
SH_*^{(-\mu_i,\tau_j)}(W) \ar[rr]^{(f_{i,j})_*} \ar[dr]_{f_!^{(-\mu_i,\tau_j)}} & & SH_*^{(-\nu_{-,i},\nu_{+,j})}(V) \\
& SH_*^{(-\mu_i,\tau_j)}(V) \ar[ur] &
}. 
$$
Since $-\nu_{-,i}\to -\infty$ as $i\to-\infty$ and $\tau_j\to+\infty$ as $j\to+\infty$, and since the continuation maps in the doubly-directed systems for $L_{i,j}$ and $H_{i,j}$ correspond under the previous isomorphisms to enlarging/restricting the action windows (Lemmas~\ref{lem:SHWgeometric-W} and~\ref{lem:SHVgeometric-V}), we obtain 
$$
f_! = \lim^{\longrightarrow}_j \lim^{\longleftarrow}_i \, (f_{i,j})_*.
$$
This proves the lemma for $\heartsuit=\varnothing$. The proof for the
other values of $\heartsuit$ is entirely analogous. 
\end{proof}
}

\subsection{Symplectic homology of a pair as a homological mapping cone}

Let $f_{i,j}^\heartsuit$ be the chain maps
constructed in~\S\ref{sec:transfer-map-revisited}. The
discussion in~\S\ref{sec:cones} shows that the cones
$C(f_{i,j}^\heartsuit)$ form a doubly-directed system, and we define
(compare with Corollary~\ref{cor:SHWVgeometric-WV})
$$
   SH^{\heartsuit,cone}_*(W,V):=\lim^{\longrightarrow}_j \lim^{\longleftarrow}_i \, H_*(C(f_{i,j}^\heartsuit)).
$$

The goal of this section is to prove the following proposition.

\begin{proposition} \label{prop:SHrel-dynamical-cone}
Let $(W,V)$ be a cobordism pair. Then we have an isomorphism 
$$
   SH_*^{\heartsuit,cone}(W,V)\cong SH_*^\heartsuit(W,V)[-1]
$$
for $\heartsuit\in\{\varnothing,\ge 0, >0,=0,\le 0,<0\}$. 
\end{proposition}

\begin{proof}
In view of Corollary~\ref{cor:SHWVgeometric-WV} it will be enough to prove 
\begin{equation}\label{eq:cone-lim}
   \lim^{\longrightarrow}_j \lim^{\longleftarrow}_i \,H_*(C(f_{i,j}^\heartsuit)) 
   = \lim^{\longrightarrow}_j\lim^{\longleftarrow}_i \, FH_{(I,III)^\heartsuit}(H_{i,j})[-1]  
\end{equation}
for all values of $\heartsuit$. 

{\color{black}
We recall the notation $W=W^{bottom}\circ V\circ W^{top}$. Recall the
families of Hamiltonians $H_{i,j}$ and $L_{i,j}$ from~\S\ref{sec:transfer-map-revisited}.
For a fixed value of the double index $(i,j)$ we denote for
readability $H=H_{i,j}$ and $L=L_{i,j}$.

Let $\heartsuit=\varnothing$. We claim that any monotone homotopy from
$L$ to $H$ induces a homotopy equivalence 
$$
   FC_{I}(L)\stackrel\sim \longrightarrow FC_{I,II,III}(H).
$$ 
To see this, consider for $t\in[0,1]$ the non-increasing homotopy of
Hamiltonians $H^t$ as in Figure~\ref{fig:L} from $H^0=H$ to $H^1=L$. 
Each $H^t$ has the shape considered in Section~\ref{sec:Ham-V} with parameters
$$
   \mu^t=\mu,\ \nu^t_-\in[0,\nu_-],\ \nu^t_+\in[0,\nu_+],\ \tau^t=\tau,\ \delta^t>0,\ \eps^t\in[\eps,\xi],\ R^t=R
$$
satisfying
$$
   \delta^t\eps^t=\delta\eps,\qquad \nu^t_-(1-\eps^t)=\mu(\xi-\eps^t).
$$
Thus $\eps^t$ increases with $t$, while $\delta^t$ and $\nu^t_-$ decrease with $t$. 
The actions of orbits in the regions $I$, $II$ and $III$ are
bounded below by $-\mu(\xi-\eps^t)-\eps^t(\mu-\eta_\mu)=-\mu\xi+\eps^t\eta_\mu$,
$-\wt\nu_-^t$ and $-\nu^t_+(R-1)$, respectively, all of which
increase with $t$. Here 
$\wt\nu_-^t$ denotes $\nu_--\eta_{\nu_-}$ for $\nu_-^t\geq\nu_--\eta_{\nu_-}$ and $\nu_-^t$ otherwise.
Since the action of orbits in region $F$ is independent of $t$
and the actions satisfy $F<I,II,III$ for $t=0$, it follows that
$F<I,II,III$ holds for all $t\in[0,1]$. 
Considering a moving action window separating the orbit group $F$ from
the groups $I,II,III$,
we see that the continuation map $FH_{I}(L)\to FH_{I,II,III}(H)$ is a
composition of small distance {\color{black}homotopy equivalences} and thus an isomorphism. 
This proves the claim.
}

Let us consider the commutative diagram 
$$
\xymatrix{
FC_I(L)\ar[rr]^f \ar[dr]^\sim_{h.e.}& & FC_{II}(H) \\
& FC_{I,II,III}(H)\ar[ur]_p & 
}
$$
in which $p$ is the projection induced by the ordering $I,III\prec II$. By Lemma~\ref{lem:ip}(ii) we have an isomorphism in {\sf Kom} 
$$
C(f)\cong C(p)\cong FC_{I,III}(H)[-1]. 
$$
This isomorphism is compatible with continuation maps, and hence with the structure of a doubly-directed system. 
In the first-inverse-then-direct limit this yields~\eqref{eq:cone-lim} for $\heartsuit=\varnothing$.

Let $\heartsuit=``=0"$. The orbits of $L$ in the group $I^0$ are constants, and we separate them as $I^0=I^{0bottom}\sqcup I^{0V}\sqcup I^{0top}$, according to whether they lie in $W^{bottom}$, $V$, respectively $W^{top}$, with the orbits lying in $W^{bottom}\cup W^{top}$ forming a subcomplex, and the orbits lying in $V$ forming a quotient complex (this is achieved by perturbing $L$ along $W$ by a Morse function whose restriction to $V$ is smaller than its restriction to $W^{bottom}\cup W^{top}$). The Floer complex reduces to the Morse complex by symplectic asphericity~\cite{SZ92}, and we therefore have canonical identifications $FC_{I^{0bottom}}(L)\equiv FC_{I^0}(H)$, $FC_{I^{0V}}(L)\equiv FC_{II^0}(H)$, and $FC_{I^{0top}}(L)\equiv FC_{III^0}(H)$. 

The continuation map $f^{=0}:FC_{I^0}(L)\to FC_{II^0}(H)$ is identified with the projection $FC_{I^0}(L)\to FC_{I^{0V}}(L)$, and by Lemma~\ref{lem:ip}(ii) we have an isomorphism in {\sf Kom}
$$
C(f^{=0})\cong FC_{I^{0bottom,0top}}(L)[-1] \equiv FC_{I^0,III^0}(H)[-1].
$$
This identification is compatible with continuation maps, and hence with the structure of a doubly-directed system. 
In the first-inverse-then-direct limit this yields~\eqref{eq:cone-lim} for $\heartsuit=``=0"$.

Let $\heartsuit=``<0"$. 
{\color{black} 
We denote $FC_{I^{0bottom}}(L)$ the complex generated by the critical
points of $L$ inside $W^{bottom}$, and we recall the canonical
identification $FC_{I^{0bottom}}(L)\simeq FC_{I^0}(H)$ which we
already discussed in the case $\heartsuit=``=0"$ above. 
We claim that any monotone homotopy from $L$ to $H$ induces a homotopy
equivalence 
$$
   FC_{I^{-,0bottom}}(L)\stackrel\sim\longrightarrow FC_{I,II^-}(H). 
$$
To see this, consider the composition
$$
   g:FC_{I^{-,0bottom}}(L)\longrightarrow FC_{I,II^-,III}(H)\longrightarrow FC_{I,II^-}(H),
$$
where the first map is the continuation map and the second one is the
quotient projection according to Lemma~\ref{lem:ordering}. Note that
the subcomplexes $FC_{I^{-,0bottom}}(L)$ and $FC_{I,II^-,III}(H)$
correspond to the negative action parts if we choose the perturbing
Morse functions to be positive on $W^{bottom}$ and negative on $V\cup W^{top}$.
Since the homotopy is constant on $V$, Lemma~\ref{lem:no-escape} shows
that the Floer cylinders counted by the map $g$ lie entirely in $F\cup W^{bottom}$.
Therefore, the map $g$ agrees with the continuation map $FC^{<0}(\wt
L)\to FC^{<0}(\wt H)$, where $\wt L$, $\wt H$ are the Hamiltonians
that agree with $L$, $H$ on $F\cup W^{bottom}$ and are equal to zero on
$V\cup W^{top}$. The argument in the case $\heartsuit=\varnothing$, 
setting the Hamiltonians $H^t$ also to zero on $V\cup W^{top}$, shows that
this map is a homotopy equivalence. This proves the claim.
}

Consider now the commutative diagram  
$$
\xymatrix{
FC_{I^{-,0bottom}}(L)\ar[rr]^\varphi \ar[dr]^\sim_{h.e.}& & FC_{II^-}(H) \\
& FC_{I,II^-}(H) \ar[ur]_p & 
}
$$
in which $p$ is the projection determined by the ordering $I\prec II^-$. It follows from Lemma~\ref{lem:ip}(ii) that we have an isomorphism in {\sf Kom} 
$$
C(\varphi)\cong C(p)\cong FC_{I}(H)[-1].
$$
We then consider the diagram of short exact sequences of complexes
$$
\xymatrix
@C=40pt
{
FC_{I^{-}}(L)\ar[r] \ar[d]^{f^{<0}} & FC_{I^{-,0bottom}}(L) \ar[r] \ar[d]^\varphi & FC_{I^{0bottom}}(L) \ar[d]\\
FC_{II^-}(H) \ar@{=}[r] \ar[d] & FC_{II^{-}}(H) \ar[r] \ar[d] & 0 \ar[d] \\
C(f^{<0}) & C(\varphi)\cong FC_{I}(H)[-1]  \ar[r]^{\cong \, proj[-1]} & C(0)\cong FC_{I^0}(H)[-1]
}
$$
The top right square commutes up to homotopy by Proposition~\ref{prop:diag} because the cone of the identity map on the second line is homotopic to zero. 
The cone of $\varphi$ has been identified above, and the bottom right map induced between the cones is homotopic to the projection $FC_I(H)[-1]\stackrel{proj[-1]}\longrightarrow FC_{I^0}(H)[-1]$. 
It then follows from Proposition~\ref{prop:diag} and Lemma~\ref{lem:ip}(ii) that we have isomorphisms in {\sf Kom}
$$
C(f^{<0})\cong C(proj[-1])[1]\cong C(proj) \cong FC_{I^{-+}}(H)[-1].
$$
For the middle isomorphism see~\eqref{eq:Jf}. The identification
$C(f^{<0})\cong FC_{I^{-+}}(H)[-1]$ is compatible with continuation
maps, and hence with the structure of a doubly-directed system. 
In the first-inverse-then-direct limit this yields~\eqref{eq:cone-lim} for $\heartsuit=``<0"$.

Let $\heartsuit=``\ge 0"$. This is a consequence of the cases
$\heartsuit=\emptyset$ and $\heartsuit=``<0"$. 
For this, we consider the diagram
$$
\xymatrix
@C=40pt
{
FC_{I^-}(L)\ar[r] \ar[d]^{f^{<0}} & FC_{I}(L)\ar[r] \ar[d]^{f} & FC_{I^{0+}}(L) \ar[d]^{f^{\ge 0}} \\
FC_{II^-}(H)\ar[r] \ar[d] & FC_{II}(H) \ar[r] \ar[d] & FC_{II^{0+}}(H) \\
C(f^{<0})\cong FC_{I^{-+}}(H)[-1] \ar[r]^{\cong\, incl[-1]} & C(f) \cong FC_{I,III}(H)[-1] &
}
$$
The cones of $f^{<0}$ and $f$ have been identified above, and the map induced between the cones is homotopic to the inclusion $FC_{I^{-+}}(H)[-1] \stackrel{incl[-1]}\longrightarrow FC_{I,III}(H)[-1]$. It then follows from Proposition~\ref{prop:diag} and Lemma~\ref{lem:ip}(ii) that we have isomorphisms in {\sf Kom} 
$$
C(f^{\ge 0})\cong C(incl[-1])\cong C(incl)[-1]\cong FC_{I^0,III}(H)[-1].
$$
For the middle isomorphism see~\eqref{eq:Jf}. The
identification $C(f^{\ge 0})\cong FC_{I^0,III}(H)[-1]$ is compatible
with continuation maps, and hence with the structure of a
doubly-directed system. 
In the first-inverse-then-direct limit this yields~\eqref{eq:cone-lim} for $\heartsuit=``\ge 0"$.

Let $\heartsuit=``>0"$. This is a consequence of the cases $\heartsuit=``=0"$ and $\heartsuit=``\ge 0"$. The proof is similar to that of the case $\heartsuit=``\ge 0"$.  

Let $\heartsuit=``\le 0"$. This is a consequence of the cases $\heartsuit=``>0"$ and $\heartsuit=\varnothing$. The proof is similar to that of the case $\heartsuit=``\ge 0"$.   
\end{proof}

\emph{Remarks on the proof of Proposition~\ref{prop:SHrel-dynamical-cone}.} It is worth noting that we really needed to consider only three cases: $\heartsuit=\varnothing$, $\heartsuit=``=0"$, and $\heartsuit=``<0"$, the other three cases being in a sense formal consequences. As a matter of fact, given $\heartsuit=\varnothing$ and $\heartsuit=``=0"$, any of the four remaining cases suffices in order to deal with the other remaining three. A strategy that would have worked is to have considered the case $\heartsuit=``>0"$, i.e. work our way from the convex end throughout the cobordism (instead of starting from the concave end as in the proof). Should one wish to start with one of the cases $\heartsuit=``\le 0"$ or $\heartsuit=``\ge 0"$, an additional argument would be needed, related to excision, that would allow to decouple $I$ from $III^0$, respectively $I^0$ from $III$. 

We can see \emph{a posteriori} that the proof consists in a suitable iterative application of  the following two elementary steps. (i) Identify suitable complexes for $L$ and $H$ which are homotopy equivalent via the continuation map. (ii) Embed the maps $f^\heartsuit$ whose cone we wish to compute inside grid diagrams of the type considered in Proposition~\ref{prop:diag}, in which the other maps are either some of the homotopy equivalences of Step (i), or maps $f^\heartsuit$ whose cones have been already computed, or natural projections/inclusions for which the cones are known via Lemma~\ref{lem:ip}.

\subsection{The exact triangle of a pair} \label{sec:ex-triangle-pair}

The homotopy invariance of the transfer map, together with the identification between the dynamical definition of the relative symplectic homology groups and the definition using cones given by Proposition~\ref{prop:SHrel-dynamical-cone} implies that for any exact inclusion of pairs $(W,V)\stackrel{f}\longrightarrow (W',V')$ and for any $\heartsuit\in\{\varnothing,\ge 0, >0,=0,\le 0,<0\}$ we have an induced transfer map 
$$
SH_*^\heartsuit(W',V')\stackrel{f_!}\longrightarrow SH_*^\heartsuit(W,V).
$$
The following proposition establishes Theorem~\ref{thm:exact-triangle-pair} (the case of symplectic cohomology is completely analogous to that of symplectic homology). 

\begin{proposition} \label{prop:invariance-pair-WV}
Let $(W,V)$ be a cobordism pair for which we denote the inclusions $V\stackrel{i}\longrightarrow W\stackrel{j}\longrightarrow (W,V)$. Given $\heartsuit\in\{\varnothing,\ge 0, >0,=0,\le 0,<0\}$ the following hold. 

(i) For any Liouville structure $\lambda$ there exists an exact triangle  
\begin{equation*} 
\xymatrix
@C=10pt
@R=18pt
{
SH_*^\heartsuit(W,V;\lambda) \ar[rr]^-{j_!} & & 
SH_*^\heartsuit(W;\lambda) \ar[dl]^-{i_!} \\ & SH_*^\heartsuit(V;\lambda) \ar[ul]^-\p_-{[-1]}  
}
\end{equation*}
where the various symplectic homology groups are understood to be computed with respect to the Liouville structure $\lambda$.

(ii) Given a homotopy of Liouville structures $\lambda_t$, $t\in
[0,1]$, there are induced isomorphisms
$h_W:SH_*^\heartsuit(W;\lambda_0)\to
SH_*^\heartsuit(W;\lambda_1)$,
$h_V:SH_*^\heartsuit(V;\lambda_0)\to
SH_*^\heartsuit(V;\lambda_1)$, and
$h_{W,V}:SH_*^\heartsuit(W,V;\lambda_0)\to
SH_*^\heartsuit(W,V;\lambda_1)$ which define a morphism between
the exact triangles in~(i) corresponding to $\lambda_0$ and $\lambda_1$. 

(iii) Given an exact inclusion of pairs $(W,V)\stackrel{f}\longrightarrow (W',V')$, the transfer maps $f_!:SH_*^\heartsuit(V')\to SH_*^\heartsuit(V)$, $f_!:SH_*^\heartsuit(W')\to SH_*^\heartsuit(W)$, and $f_!:SH_*^\heartsuit(W',V')\to SH_*^\heartsuit(W,V)$ define a morphism between the exact triangles of the pairs $(W',V')$ and $(W,V)$.
\end{proposition}

\begin{proof}
The existence of the exact triangle in (i) is a consequence of the tautological homology exact triangle of a cone~\eqref{eq:hlescone} and of the identification between $SH_*^\heartsuit(W,V)[-1]$ and $SH_*^{\heartsuit,cone}(W,V)$ proved in Proposition~\ref{prop:SHrel-dynamical-cone}.  

Part (ii) follows from the naturality of the homology exact triangle of a cone with respect to chain maps, and from the naturality of the absolute transfer map $SH_*^\heartsuit(W;\lambda)\to SH_*^\heartsuit(V;\lambda)$ with respect to homotopies of Liouville structures. 

Part (iii) follows from the naturality of the homology exact triangle of a cone and from the functoriality of transfer maps (Proposition~\ref{prop:func-transfer}). 
\end{proof}

The Excision Theorem~\ref{thm:excision-triple} can also be reinterpreted using transfer maps. The proof uses the same kind of arguments as above and we shall omit it. 

\begin{proposition}
Given a Liouville cobordism triple $(W,V,U)$, denote the inclusion 
$$
(\overline{W\setminus U},\overline{V\setminus U})\stackrel{i}\longrightarrow (W,V).
$$ 
The excision isomorphism in Theorem~\ref{thm:excision-triple} is induced by the transfer map $i_!$.

\hfill{$\square$}
\end{proposition}

\subsection{Exact triangle of a triple and Mayer-Vietoris exact triangle}\label{sec:Mayer-Vietoris}

\begin{proposition}[Exact triangle of a triple]\label{prop:triple}
Let $U\subset V\subset W$ be a triple of Liouville cobordisms with filling, meaning that $(V,U)$ and $(W,V)$ are pairs of Liouville cobordisms with filling, and denote the inclusions by $(V,U)\stackrel{i}\longrightarrow (W,U)\stackrel{j}{\longrightarrow} (W,V)$. For $\heartsuit\in\{\varnothing,\ge 0, >0,=0,\le 0,<0\}$ there exists an exact triangle
\begin{equation*} 
\xymatrix
@C=10pt
@R=18pt
{
SH_*^{\heartsuit}(W,V) \ar[rr]^-{j_!} & & 
SH_*^\heartsuit(W,U) \ar[dl]^-{i_!} \\ & SH_*^\heartsuit(V,U) \ar[ul]^-\p_-{[-1]}  
}
\end{equation*}
which is functorial with respect to inclusions of triples, and which is invariant under homotopies of the Liouville structure that preserve the triple. 
\end{proposition}

\begin{proof}
The proof is a formal consequence of the functorial properties of the long exact sequence of a pair. The proof of Theorem~I.10.2 in~\cite{ES52} holds verbatim. 
\end{proof}

\begin{theorem}[Mayer-Vietoris exact triangle] \label{thm:MV}
Let $U,V\subset W$ be Liouville cobordisms such that $W=U\cup V$ and $Z:=U\cap V$ is a Liouville cobordism such that 
$$
U=U^{bottom}\circ Z,\qquad V=Z\circ V^{top},\qquad W=U^{bottom}\circ Z\circ V^{top},
$$
with $U^{bottom}=\overline{U\setminus Z}$, $V^{top}=\overline{V\setminus Z}$. We denote the inclusion maps by 
$$
\xymatrix
@C=50pt
@R=10pt
{& U \ar[dr]^-{j_U} & \\
Z \ar[ur]^-{i_U} \ar[dr]_-{i_V} & & W\\
& V \ar[ur]_-{j_V} & 
}
$$
There is a functorial \emph{Mayer-Vietoris exact triangle} 
\begin{equation*} 
\xymatrix
@C=20pt
@R=18pt
{
SH_*^\heartsuit(W) \ar[rr]^-{(j_{U!},j_{V!})} & & 
SH_*^\heartsuit(U)\oplus SH_*^\heartsuit(V) \ar[dl]^-{i_{U!}-i_{V!}} \\ & SH_*^\heartsuit(Z) \ar[ul]_-{[-1]}^\delta  
}
\end{equation*}
For $SH^{=0}$ this exact triangle is isomorphic to the Mayer-Vietoris exact triangle in singular cohomology. 
\end{theorem} 

\begin{figure}
         \begin{center}
\input{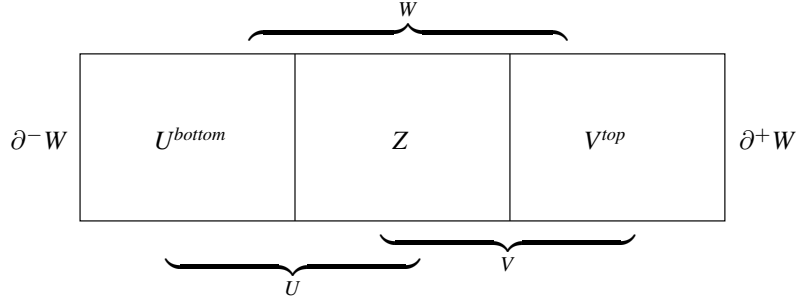}
         \end{center}
\caption{Cobordisms for the Mayer-Vietoris theorem \label{fig:MV}}
\end{figure}

\begin{proof}
The Mayer-Vietoris exact triangle follows by a purely algebraic argument from the exact triangle of a pair and its naturality, and from the Excision Theorem~\ref{thm:excision-triple}. The idea is to consider the following commutative diagram.
$$
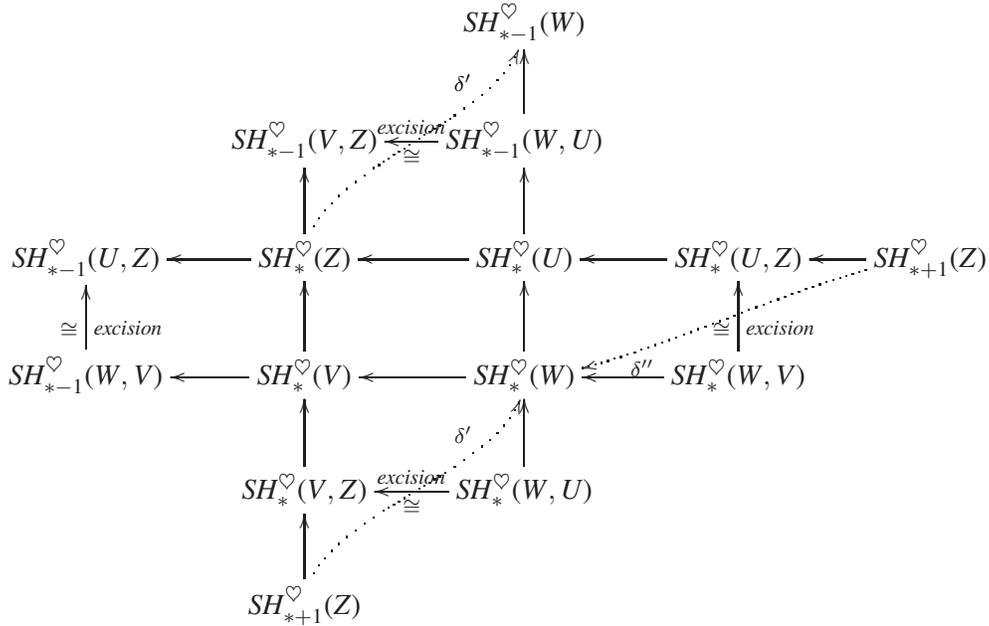
{
\xymatrix
@C=20pt
{
& & SH^\heartsuit_{*-1}(W) 
\ar@(d,u) @{<..} [ldd]_(.3){\color{black}\delta'}
& &  \\
& SH^\heartsuit_{*-1}(V,Z) & SH^\heartsuit_{*-1}(W,U) \ar[u] \ar[l]_-{excision}^-\cong & & \\
SH^\heartsuit_{*-1}(U,Z) & SH^\heartsuit_*(Z) \ar[l] \ar[u] & SH^\heartsuit_*(U) \ar[l] \ar[u] & SH^\heartsuit_*(U,Z) \ar[l] & 
SH^\heartsuit_{*+1}(Z) \ar[l] \\
SH^\heartsuit_{*-1}(W,V) \ar[u]_(.4){excision}^(.4)\cong & SH^\heartsuit_*(V) \ar[l] \ar[u] & SH^\heartsuit_*(W)  \ar[l] \ar[u] 
 \ar@(r,l) @{<..} [urr]_(.3){\delta''}
 \ar@(d,u) @{<..} [ldd]_(.3){\delta'}
 & SH^\heartsuit_*(W,V) \ar[l] \ar[u]^(.4)\cong_(.4){excision} & \\
& SH^\heartsuit_*(V,Z) \ar[u] & SH^\heartsuit_*(W,U) \ar[l]_-{excision}^-\cong \ar[u] & & \\
& SH^\heartsuit_{*+1}(Z) \ar[u] & 
}
}
$$
The isomorphism $SH^\heartsuit_*(W,V)\stackrel\sim \longrightarrow SH^\heartsuit_*(U,Z)$ follows from the Excision Theorem~\ref{thm:excision-triple} for the exact triple $(W,V,V^{top})$. Similarly, we have an isomorphism
$SH^\heartsuit_*(W,U)\stackrel \sim\longrightarrow SH^\heartsuit_*(V,Z)$. The maps $\delta'$ and $\delta''$ are obtained by inverting the corresponding excision isomorphisms, and we actually have $\delta''=-\delta'$ by the ``hexagonal lemma" of Eilenberg and Steenrod~\cite[Lemma~I.15.1]{ES52} which we recall below. We define the map $\delta$ in the statement of Theorem~\ref{thm:MV} to be equal to $\delta''$, and a direct check by diagram chasing shows that the Mayer-Vietoris triangle is exact, see~\cite[\S~I.15]{ES52} for details. 
\end{proof}

\begin{lemma} \cite[Hexagonal Lemma~I.15.1]{ES52} Consider the following diagram of groups and homomorphisms
$$
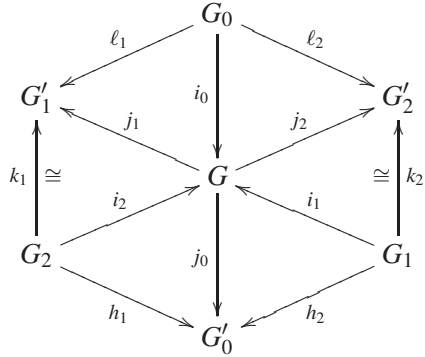
\xymatrix
@C=50pt
@R=15pt
{& G_0 \ar[dl]_{\ell_1} \ar[dr]^{\ell_2} \ar[dd]_{i_0} & \\
G'_1 & & G'_2 \\
& G \ar[ul]_{j_1} \ar[ur]^{j_2} \ar[dd]_{j_0} & \\
G_2 \ar[uu]_\cong^{k_1} \ar[ur]^{i_2} \ar[dr]_{h_1} & & G_1 \ar[uu]^\cong_{k_2} \ar[ul]_{i_1} \ar[dl]^{h_2} \\
& G'_0 &
}
$$
Assume that each triangle is commutative, that $k_1$ and $k_2$ are isomorphisms, that the two diagonal sequences are exact at $G$, and that $j_0i_0=0$. Then the two homomorphisms from $G_0$ to $G'_0$ obtained by skirting the sides of the hexagon differ in sign only. 
\hfill{$\square$}
\end{lemma}

The hexagonal lemma of Eilenberg and Steenrod is applied in the proof of Theorem~\ref{thm:MV} to the following configuration. 
$$
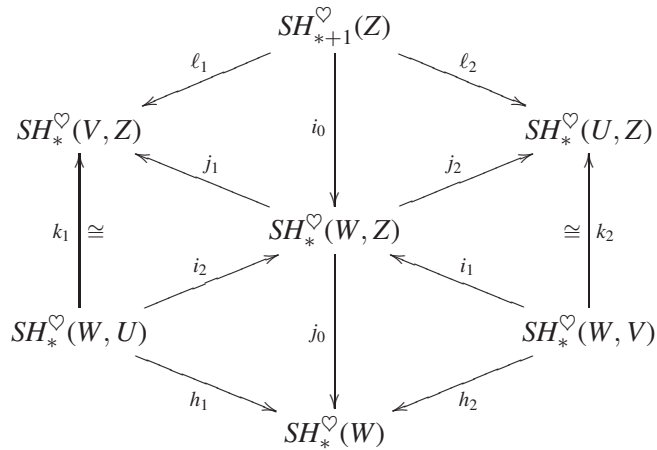
\xymatrix
@C=40pt
@R=20pt
{& SH_{*+1}^\heartsuit(Z) \ar[dl]_{\ell_1} \ar[dr]^{\ell_2} \ar[dd]_{i_0} & \\
SH_*^\heartsuit(V,Z) & & SH_*^\heartsuit(U,Z) \\
& SH_*^\heartsuit(W,Z) \ar[ul]_{j_1} \ar[ur]^{j_2} \ar[dd]_{j_0} & \\
SH_*^\heartsuit(W,U) \ar[uu]_\cong^{k_1} \ar[ur]^{i_2} \ar[dr]_{h_1} & & SH_*^\heartsuit(W,V) \ar[uu]^\cong_{k_2} \ar[ul]_{i_1} \ar[dl]^{h_2} \\
& SH_*^\heartsuit(W) &
}
$$
The vertical isomorphisms $k_1$ and $k_2$ are the excision isomorphisms. The connecting homomorphism $\delta$ in the Mayer-Vietoris exact sequence, or the homomorphism $\delta''$ in the notation of the proof of Theorem~\ref{thm:MV}, is defined to be $h_2k_2^{-1}\ell_2$.

\subsection{Compatibility between exact triangles}
In this section we use the notation $(\heartsuit,\heartsuit',\heartsuit'/\heartsuit)$ for any one of the triples $(<0,\varnothing,\ge 0)$, $(\le 0,\varnothing,>0)$, $(<0,\le 0,=0)$, or $(=0,\ge 0,>0)$. To any such triple there corresponds a tautological exact triangle (see Propositions~\ref{prop:taut-triang-W} and~\ref{prop:taut-triang-WV})
$$
\xymatrix
@C=30pt
@R=30pt
{
SH_*^{\heartsuit} \ar[rr] & & 
SH_*^{\heartsuit'} \ar[dl] \\ & SH_*^{\heartsuit'/\heartsuit} \ar[ul]^{[-1]}  
}
$$

\begin{proposition}\label{prop:compatibility-exact-triangles} 
Let $(W,V)$ be a Liouville pair of cobordisms with filling. Let $(\heartsuit,\heartsuit',\heartsuit'/\heartsuit)$ be a triple as above. 

(i) The transfer maps $f^\heartsuit_{WV}$, $f^{\heartsuit'}_{WV}$, and $f^{\heartsuit'/\heartsuit}_{WV}$ induce a morphism between the tautological exact triangles corresponding to $(\heartsuit,\heartsuit',\heartsuit'/\heartsuit)$ for $W$ and $V$. 

(ii) The exact triangles of the pair $(W,V)$ for $\heartsuit,\heartsuit',\heartsuit'/\heartsuit$ determine ``triangles of triangles" together with the corresponding tautological exact triangles. More precisely, upon expanding the exact triangles of a pair and the tautological ones into long exact sequences, we obtain the following diagram in which all squares are commutative, except the bottom right one which is anti-commutative
$$
\xymatrix
@C=30pt
@R=30pt
{
SH_*^\heartsuit(W,V) \ar[r] \ar[d] & SH_*^\heartsuit(W) \ar[r]^{f^\heartsuit_!} \ar[d] & SH_*^\heartsuit(V) \ar[r] \ar[d] & SH_{*-1}^\heartsuit(W,V) \ar[d] \\
SH_*^{\heartsuit'}(W,V) \ar[r] \ar[d] & SH_*^{\heartsuit'}(W) \ar[r]^{f^{\heartsuit'}_!} \ar[d] & SH_*^{\heartsuit'}(V) \ar[r] \ar[d] & SH_{*-1}^{\heartsuit'}(W,V) \ar[d] \\
SH_*^{\heartsuit'/\heartsuit}(W,V) \ar[r] \ar[d] & SH_*^{\heartsuit'/\heartsuit}(W) \ar[r]^{f^{\heartsuit'/\heartsuit}_!} \ar[d] & SH_*^{\heartsuit'/\heartsuit}(V) \ar @{} [dr] |{-} \ar[r] \ar[d] & SH_{*-1}^{\heartsuit'/\heartsuit}(W,V) \ar[d] \\
SH_{*-1}^\heartsuit(W,V) \ar[r] & SH_{*-1}^\heartsuit(W) \ar[r]^{f^\heartsuit_!} & SH_{*-1}^\heartsuit(V) \ar[r] & SH_{*-2}^\heartsuit(W,V) 
}
$$

(iii) The exact triangle of a pair $(W,V)$ for $SH_*^{=0}$ is isomorphic to the exact triangle of the pair $(W,V)$ in singular cohomology $H^{n-*}$. 
\end{proposition}

\begin{proof}
Assertion (i) follows from the fact that continuation maps induced by increasing homotopies respect the action filtration.

Assertion (ii) follows from Lemma~\ref{lem:diag}, and from our identification of the relative symplectic homology groups with limit homology groups of mapping cones corresponding to chain level continuation maps (Proposition~\ref{prop:SHrel-dynamical-cone}). 

Lemma~\ref{lem:diag} is applied to the following morphism between action filtration short exact sequences given by the chain level continuation maps:
$$
\xymatrix
{
0 \ar[r] & FC_{I^\heartsuit}(K_{i,j}) \ar[r] \ar[d]^{f^\heartsuit_{i,j}} & FC_{I^{\heartsuit'}}(K_{i,j}) \ar[r] \ar[d]^{f^{\heartsuit'}_{i,j}} & FC_{I^{\heartsuit'/\heartsuit}}(K_{i,j}) \ar[r] \ar[d]^{f^{\heartsuit'/\heartsuit}_{i,j}} & 0 \\
0 \ar[r] & FC_{II^\heartsuit}(K_{i,j}) \ar[r] & FC_{II^{\heartsuit'}}(H_{i,j}) \ar[r] & FC_{II^{\heartsuit'/\heartsuit}}(H_{i,j}) \ar[r] & 0 
}
$$

Assertion (iii) is proved \emph{mutatis mutandis} like~\cite[Proposition~1.4]{Cieliebak-Frauenfelder-Oancea}. We omit the details. 
\end{proof}

Finally, we prove the following compatibility between the tautological
exact triangles. 

\begin{proposition}\label{prop:taut-triangles}
For every filled Liouville pair $(W,V)$ the four tautological
exact triangles fit into the commuting diagram
$$
\xymatrix
@C=30pt
@R=30pt
{
   & SH_{*+1}^{>0}(W,V) \ar@{=}[r] \ar[d] & SH_{*+1}^{>0}(W,V) \ar[d] & \\
   SH_*^{<0}(W,V) \ar[r] \ar@{=}[d] & SH_*^{\leq 0}(W,V) \ar[r] \ar[d] &
   SH_*^{=0}(W,V) \ar[r] \ar[d] & SH_{*-1}^{<0}(W,V) \ar@{=}[d] \\ 
   SH_*^{<0}(W,V) \ar[r] & SH_*(W,V) \ar[r] \ar[d] & SH_*^{\geq
    0}(W,V) \ar[r] \ar[d] & SH_{*-1}^{<0}(W,V)\\ 
   & SH_*^{>0}(W,V) \ar@{=}[r] & SH_*^{>0}(W,V) & 
}
$$
\end{proposition}

\begin{proof}
Fix $\epsilon>0$ small enough. For any choice of real numbers $a,b$ such that $-\infty<a<-\epsilon<\epsilon<b<\infty$, and for any choice of admissible Hamiltonian and almost complex structure, we have a commutative diagram of short exact sequences
$$
\xymatrix{
0\ar[r]& FC_*^{(a,-\epsilon)}\ar[r] \ar@{=}[d] & FC_*^{(a,\epsilon)}\ar[r]\ar[d] & FC_*^{(-\epsilon,\epsilon)}\ar[r]\ar[d] & 0 \\
0\ar[r]& FC_*^{(a,-\epsilon)}\ar[r] & FC_*^{(a,b)}\ar[r]& FC_*^{(-\epsilon,b)}\ar[r]& 0 
}
$$
in which the various maps are inclusions or projections. This induces a commutative diagram between the corresponding long exact sequences in homology, and by passing to the limit on the Hamiltonian and then on $a\to-\infty$, $b\to\infty$ as in Section~\ref{sec:SHpair} we obtain the commutativity of the diagram formed by the two horizontal lines in the statement. 

The commutativity of the diagram formed by the two vertical lines in the statement is proved analogously.
\end{proof}

{\color{black}
We conclude this subsection with a compatibility result between the
exact triangle of a triple and Poincar\'e duality. 

\begin{proposition}[Poincar\'e duality and long exact sequence of a triple]\label{prop:PD-triple}
For every triple $(W,V,U)$ of filled Liouville cobordisms and $\heartsuit\in\{\varnothing, >0,\ge
0, =0, \le 0, <0\}$ there exists a commuting diagram
\begin{equation*}
{\small
\xymatrix
@C=10pt
@R=20pt
{
     SH_*^\heartsuit(W,V) \ar[r] \ar[d]^\cong_{exc} & SH_*^\heartsuit(W,U)
   \ar[r] \ar[d]^\cong_{exc} & SH_*^\heartsuit(V,U) \ar[r] \ar[d]^\cong_{exc} &
   SH_{*-1}^\heartsuit(W,V)   \ar[d]^\cong_{exc} \\
     SH_*^\heartsuit(W\setminus V,\p V) \ar[r] \ar[d]^\cong_{PD} &
   SH_*^\heartsuit(W\setminus U,\p U)
   \ar[r] \ar[d]^\cong_{PD} & SH_*^\heartsuit(V\setminus U,\p U) \ar[r] \ar[d]^\cong_{PD} &
   SH_{*-1}^\heartsuit(W\setminus V,\p V)   \ar[d]^\cong_{PD} \\
     SH^{-*}_\heartsuit(W\setminus V,\p W) \ar[r] \ar@{=}[d] &
   SH^{-*}_\heartsuit(W\setminus U,\p W)
   \ar[r] \ar@{=}[d] & SH^{-*}_\heartsuit(V\setminus U,\p V) \ar[r] \ar[d]^\cong_{exc} &
   SH^{1-*}_\heartsuit(W\setminus V,\p W)   \ar@{=}[d] \\
     SH^{-*}_{-\heartsuit}(W\setminus V,\p W) \ar[r] &
   SH^{-*}_{-\heartsuit}(W\setminus U,\p W) \ar[r] &
   SH^{1-*}_{-\heartsuit}(W\setminus U,W\setminus V) \ar[r] &
   SH^{1-*}_{-\heartsuit}(W\setminus V,\p W)   \\
}
}
\end{equation*}
where the first and last row are the long exact sequences of the
triples $(W,V,U)$ and $(W\setminus U,W\setminus V,\p W)$,
respectively, and the vertical arrows are the Poincar\'e duality
and excision isomorphisms from Theorem~\ref{thm:Poincare} and
Theorem~\ref{thm:excision}. 
(The remaining horizontal maps are defined by this diagram.)  
\end{proposition}

{\color{black}
\begin{proof} 
The conclusion follows directly 
from the definition of the Poincar\'e
duality isomorphism in Theorem~\ref{thm:Poincare} and the observation
that for a Hamiltonian $G$ as in Figure~\ref{fig:KHG-new} adapted to the
triple $(W,V,U)$, the Hamiltonian $-G$ is adapted to the triple
$(W\setminus U,W\setminus V,\p W)$. 

Alternatively, one can reduce the general case by a purely algebraic argument to the case $U=\varnothing$, as in the proof of Proposition~\ref{prop:triple}. The case $U=\varnothing$ is in turn treated by noting that for a Hamiltonian $H$ as in Figures~\ref{fig:KH-new} or~\ref{fig:H-one-step} adapted to the
pair $(W,V)$, the Hamiltonian $-H$ is adapted to the triple
$(W,W\setminus V,\p W)$.
\end{proof}
}
}

\subsection{The exact triangle of a pair of Liouville domains revisited}

The exact triangle 
\begin{equation*} 
\xymatrix
@C=10pt
@R=18pt
{
SH_*^\heartsuit(W,V) \ar[rr] & & 
SH_*^\heartsuit(W) \ar[dl] \\ & SH_*^\heartsuit(V) \ar[ul]^-\p_-{[-1]}  
}
\end{equation*}
can be established in a more direct way for a pair $(W,V)$ of Liouville domains since there is no need to first identify the symplectic homology of the pair with a homological mapping cone. Instead, one can argue directly on the chain complexes using truncation by the action. We find it instructive to spell out the argument. This proof is only apparently simpler: since the transfer maps induced by the inclusions $V\hookrightarrow W$ and $W\hookrightarrow (W,V)$ are only implicitly constructed, this proof would require additional arguments in order to incorporate it into the larger framework that we discuss in this paper, and these additional arguments would essentially amount to reinterpret this diagram in terms of transfer maps.

For a pair of Liouville domains we only need to consider three flavors $\heartsuit\in\{\varnothing,=0,>0\}$. We prove below the compatibility of the exact triangle of the pair with the tautological exact triangle relating these three flavors.

Let $V\subseteq W$ be an inclusion of Liouville domains and denote by
$\wh W$ the symplectic completion of $W$. Let $H=H_{\nu,\tau}$,
$\nu>0$, $\tau>0$ be a one step Hamiltonian on $\wh W$, defined up to
smooth approximation as follows (Figure~\ref{fig:SHpairdomains}):  
\begin{itemize}
\item $H=0$ on $W\setminus V$,
\item $H$ is linear of slope $\tau$ on $\wh W\setminus W$,
\item $H$ is linear of slope $\nu$ on a collar $]\delta,1]\times \p V\subseteq V$ for some $0<\delta<1$,
\item $H$ is constant equal to $-\nu(1-\delta)$ on the complement of this collar in $V$.
\end{itemize} 

\begin{figure}
         \begin{center}
\input{SHpairdomains.pstex_t}
         \end{center}
\caption{Hamiltonian for a pair of Liouville domains \label{fig:SHpairdomains}}
\end{figure}

For $\nu$ and $\tau$ not lying in the action spectrum of $\p V$, respectively $\p W$, the $1$-periodic orbits of $H$ fall into five classes:
\begin{itemize}
\item[($II^0$)] constants in the complement of the collar in $V$,
\item[($II^+$)] orbits corresponding to characteristics on $\p V$ and located in the region $\{\delta\}\times\p V$,
\item[($III^-$)] orbits corresponding to characteristics on $\p V$ and located in the region $\p V$,
\item[($III^0$)] constants in $W\setminus V$,
\item[($III^+$)] orbits corresponding to characteristics on $\p W$ and located in the region $\p W$.
\end{itemize}
Suitable choices of the parameters $\tau$ and $\delta$ as a function of $\nu$ ensure that the various classes of orbits are ordered according to the action as follows: 
$$
III^0 < III^-,III^+ < II^0 < II^+.
$$
As $\nu\to \infty$ we can allow $\tau\to\infty$. In general we need to let $\delta\to 0$ if we wish to acquire $III^-<II^0$. However, by Lemmas~\ref{lem:asy} and~\ref{lem:no-escape} we have 
$$
III^-\prec II^0,II^+
$$
for any fixed choice of $\delta>0$, independently of the choice of $\nu$.   
Also, by Lemma~\ref{lem:asy} we have 
$$
III^-\prec III^+,\qquad II^0,II^+\prec III^+.
$$
The outcome is that for suitable choices of the parameters we have 
$$
III^0< III^-\prec II^0 < II^+ \prec III^+
$$
and 
$$
III^0<III^-\prec III^+ <II^0<II^+.
$$

Let $FC_{tot}$ be the total Floer complex for the Hamiltonian $H$. For a subset $\mathcal I\subset \{II^0,II^+,III^-,III^0,III^+\}$ denote by $FC_{\mathcal I}$ the complex generated by the orbits in the classes belonging to $\mathcal I$. For example, $FC_{III^-,III^0,III^+}$ stands for the subcomplex generated by the orbits in the classes $III^-,III^0,III^+$, and $FC_{III^-,III^+}$ stands for its quotient complex modulo $FC_{III^0}$ etc. We will also abbreviate $FC_{II}=FC_{II^0,II^+}$ and $FC_{III}=FC_{III^-,III^0,III^+}$.

Let us consider the following diagram whose first two rows and first two columns are exact
$$
\xymatrix
{
& 0 \ar[d] & 0 \ar[d] & & \\
0\ar[r] & FC_{III^0} \ar[r] \ar[d] & FC_{II,III^-,III^0} \ar[r] \ar[d] & FC_{II,III^-} \ar[r] \ar[d]^p & 0 \\
0\ar[r] & FC_{III} \ar[r] \ar[d] & FC_{tot} \ar[r] \ar[d] & FC_{II} \ar[r] \ar@{.>}[d]^-f & 0 \\
& FC_{III^-,III^+} \ar[r]_q \ar[d] & FC_{III^+} \ar@{.>}[r]_-g \ar[d] & FC_{III^-}[-1] & \\
& 0 & 0 & & 
}
$$

Here the chain maps $f:FC_{II}\to FC_{III^-}[-1]$ and $g:FC_{III^+}\to FC_{III^-}[-1]$ are uniquely determined so that we have natural identifications 
$$
FC_{II,III^-}=C(f)[1], \quad p=\beta(f),\qquad \qquad FC_{III^-,III^+}=C(g)[1],\quad q=\beta(g).
$$

Proposition~\ref{prop:diag} and its proof ensure that the bottom right square in the above diagram is commutative in {\sf Kom}, and moreover the diagram can be completed to a diagram in {\sf Kom}  whose lines and columns are distinguished triangles, and all of whose squares are commutative except the bottom-right one which is anti-commutative:  
\begin{equation} \label{eq:grid-chain}
\xymatrix
{
FC_{III^0} \ar[r] \ar[d] & FC_{II,III^-,III^0} \ar[r] \ar[d] & FC_{II,III^-} \ar[r] \ar[d]^p & FC_{III^0}[-1] \ar[d] \\
FC_{III} \ar[r] \ar[d] & FC_{tot} \ar[r] \ar[d] & FC_{II} \ar[r] \ar@{.>}[d]^-f & FC_{III}[-1] \ar[d] \\
FC_{III^-,III^+} \ar[r]_q \ar[d] & FC_{III^+} \ar@{.>}[r]_-g \ar[d] & FC_{III^-}[-1] \ar@{} [dr] |{-} \ar[r] \ar[d] &      FC_{III^-,III^+}[-1] \ar[d] \\
FC_{III^0}[-1] \ar[r] & FC_{II,III^-,III^0}[-1] \ar[r] & FC_{II,III^-}[-1] \ar[r] & FC_{III^0}[-2] 
}
\end{equation}

Indeed, the term $FC_{III^-}[-1]$ is isomorphic in {\sf Kom} to $C(p)[-1]$ on the one hand, and to $C(-q)[-1]$ on the other hand, and these two complexes are isomorphic as seen in the proof of Proposition~\ref{prop:diag}. 

We now remark that we have a homotopy equivalence that is well-defined up to homotopy 
$$
FC_{III^-}[-1]\cong FC_{II^+}.
$$
This follows again from Proposition~\ref{prop:diag}. For the proof we
consider a homotopy from a Hamiltonian $K=K_\tau$ which is zero on $W$
and coincides with $H_{\nu,\tau}$ outside $W$ to the Hamiltonian
$H$. We denote $FC_V(K)$ the subcomplex of $FC(K)$ generated by critical
points inside the domain $V$, so that the continuation map induces a
homotopy equivalence $FC_V(K)\simeq FC_{II,III^-}$. On the other hand we
have a canonical identification $FC_V(K)\equiv FC_{II^0}$, and a commutative
diagram up to homotopy  
$$
\xymatrix{
FC_V(K) \ar@{=}[r] \ar[d]_\simeq^{h.e.} & FC_{II^0} \ar[d]^{incl} \\
FC_{II,III^-} \ar[r]_{proj} & FC_{II}\;.
}
$$
Then Proposition~\ref{prop:diag} yields the desired homotopy equivalence $FC_{III^-}[-1]\cong FC_{II^+}$.

\begin{remark}
This chain homotopy equivalence provides one point of view on the vanishing of $SH_*(I\times \p V,\p^-(I\times \p V))$ proved in Proposition~\ref{prop:triv-cob}. 
\end{remark}

Diagram~\eqref{eq:grid-chain} can now be used as a building block to prove the existence of a diagram with exact lines and columns and in which all squares are commutative except the one marked ``$-$'', which is anti-commutative. 
\begin{equation} \label{eq:grid-SH}
\xymatrix
{
H^{n-*}(W,V) \ar[r] \ar[d] & H^{n-*}(W) \ar[r] \ar[d] & H^{n-*}(V) \ar[r] \ar[d] & H^{n-*+1}(W,V) \ar[d] \\
SH_*(W,V) \ar[r] \ar[d] & SH_*(W) \ar[r] \ar[d] & SH_*(V) \ar[r] \ar[d] & SH_{*-1}(W,V) \ar[d] \\
SH_*^{>0}(W,V) \ar[r] \ar[d] & SH_*^{>0}(W) \ar[r] \ar[d] & SH_*^{>0}(V) \ar@{} [dr] |{-} \ar[r] \ar[d] & SH_{*-1}^{>0}(W,V) \ar[d] \\
H^{n-*+1}(W,V) \ar[r] & H^{n-*+1}(W) \ar[r] & H^{n-*+1}(V) \ar[r] & H^{n-*+2}(W,V) 
}
\end{equation}

This grid diagram expresses the compatibility between the exact triangle of a pair of Liouville domains $(W,V)$ and the tautological exact triangle involving singular cohomology, symplectic homology, and positive symplectic homology. One relevant ingredient here is the chain homotopy equivalence $C^{III}[-1]\cong C^{II}$. The other ingredient is that all the above homological constructions are compatible with continuation maps and with direct limits.

\section{Variants of symplectic homology groups} \label{sec:variants}

\subsection{Rabinowitz-Floer homology}
Given a pair of Liouville domains $(W,V)$, Rabinowitz-Floer
homology $RFH_*(\p V,W)$ was defined in~\cite{Cieliebak-Frauenfelder}
as a Floer-type theory associated to the Rabinowitz action functional  
$$
\tilde A_H:\cL\wh W\times \R\to \R,\qquad \tilde A_H(\gamma,\eta)=A_{\eta H}(\gamma),
$$
where $H:\wh W\to \R$ is a Hamiltonian such that $\p V=H^{-1}(0)$ is a regular level, $H|_V\le 0$, and $H|_{\wh W\setminus V}\ge 0$. The dynamical significance of Rabinowitz-Floer homology is that it counts leafwise intersection points of $\p V$ under Hamiltonian motions~\cite{AF10}, and one of its most useful properties is that Hamiltonian displaceability of $\p V$ (and hence of $V$) implies vanishing. 

It was proved in~\cite{Cieliebak-Frauenfelder-Oancea} that $RFH_*(\p
V,W)$ does not depend on $W$, so we will denote it by $RFH(\p V)$ (it
does however depend on the filling $V$ of $\p V$). The main result
of~\cite{Cieliebak-Frauenfelder-Oancea} is that, with our current notation, we have an isomorphism  
\begin{equation}\label{eq:RFH}
   RFH_*(\p V)\cong SH_*(\p V),
\end{equation}
i.e. \emph{Rabinowitz-Floer homology is the symplectic homology of the
  trivial cobordism over $\p V$}. As such, Rabinowitz-Floer homology is naturally incorporated within the setup that we develop in this paper.  

\subsection{$S^1$-equivariant symplectic homologies}\label{sec:S1equiv} 
The circle $S^1=\R/\Z$ acts on the free loop space by shifting the
parametrisation. As such, one can define $S^1$-equivariant flavors of
symplectic homology groups. In the case of Liouville domains relevant
instances have been defined
in~\cite{Viterbo99,Seidel07,BO3+4,Zhao14,ACF14}. Following
Seidel~\cite{Seidel07} and~\cite{BO3+4,Zhao14}, the relevant structure
is that of an \emph{$S^1$-complex}, meaning a $\Z$-graded chain
complex $(C_*,\p)$ together with a sequence of maps $\p_i:C_*\to
C_{*+2i-1}$, $i\ge 0$ such that $\p_0=\p$ and   
\begin{equation} \label{eq:S1complex}
\sum_{i+j=k} \p_i\p_j=0
\end{equation}
for all $k\ge 0$. An $S^1$-complex for which $\p_i=0$ for $i\ge 2$ is
called a \emph{mixed complex} in the literature on cyclic
homology. One should view $S^1$-complexes as being
$\infty$-mixed complexes, or mixed complexes up to
homotopy, see~\cite{BO3+4} and the references therein. Given a
Hamiltonian $H$ one can endow $FC_*^{(a,b)}(H)$ with the structure of an $S^1$-complex
that is canonical up to homotopy equivalence. Moreover, a
homotopy of Hamiltonians induces a morphism between the
$S^1$-complexes defined on the Floer chain groups at the endpoints. 

{\color{black}Recall that we work with coefficients in a field $\K$.} Denote by $u$ a
formal variable of degree $-2$. 
Given an $S^1$-complex $\cC=(C_*,\{\p_i\}_{i\ge 0})$ we define
following Jones~\cite{Jones} and Zhao~\cite{Zhao14} 
the \emph{periodic cyclic chain complex}
$$
   C_*[u,u^{-1}],\qquad \p_u = \sum_{i\ge 0}u^{i}\p_i,\qquad |u|=-2.
$$
Here elements in $C_*[u,u^{-1}]$ of degree $k$ are by definition
Laurent polynomials $\sum_{j=-N}^N x_ju^j$ with $x_j\in
C_{k+2j}$. Then $\p_u^2=0$ as a consequence of~\eqref{eq:S1complex}
and the map $\p_u$ is $\K[u]$-linear. 
We consider 
the sub/quotient complexes 
$$
   C_*[u],\qquad 
   C_*[u^{-1}] = C_*[u,u^{-1}]/uC_*[u]
$$
with differential induced by $\p_u$ and the induced 
$\K[u]$-module structure. The homologies 
\begin{gather*}
   HC_*^{[u]}(\cC):=H_*(C_*[u]),\qquad 
   HC_*^{[u,u^{-1}]}(\cC) := H_*(C_*[u,u^{-1}]),\cr 
   HC_*(\cC) := HC_*^{[u^{-1}]}(\cC) := H_*(C_*[u^{-1}])
\end{gather*}
correspond to certain versions of the \emph{negative cyclic homology},  
\emph{periodic (or Tate) cyclic homology}, 
respectively \emph{cyclic homology} 
of the $S^1$-complex $\cC$ in the literature. We will not use these
names but rather indicate in the notation which version of
(Laurent) polynomials we are using.    
Due to the short exact sequence of complexes of $\K[u]$-modules  
$$
   0\to C_*[u]\to C_*[u,u^{-1}] \to C_*[u,u^{-1}]/C_*[u]\cong C_*[u^{-1}][-2]\to 0, 
$$
these homology groups fit into the {\em fundamental exact triangle} 
$$
\xymatrix
@C=20pt
{
HC_*^{[u]}(\cC) \ar[rr] & & 
HC_*^{[u,u^{-1}]}(\cC)  \ar[dl]^{[-2]} \\ & HC_{*}(\cC) \ar[ul]^{[+1]}\;.  
}
$$

\begin{example} 
Given an $S^1$-space $X$, its singular chain complex with arbitrary
coefficients $C_*=(C_*(X),\p)$ carries the structure of a mixed complex
$\cC=(C_*,\p,\p_1)$ such that~\cite{Hingston-Delta,Hess12} 
$$
   HC_*(\cC)\cong H_*^{S^1}(X).
$$
Here $H_*^{S^1}(X)=H_*(X\times_{S^1}ES^1)$ is the
usual $S^1$-equivariant homology group of $X$ defined by the Borel construction. 
The map $\p_1:C_*\to C_{*+1}$ is defined by inserting a suitable
representative of the fundamental class of the oriented circle $S^1$ into the first
argument of the composite map $C_*(S^1)\otimes
C_*(X)\stackrel{EZ}\longrightarrow C_*(S^1\times
X)\stackrel{\mu_*}\longrightarrow C_*(X)$, where $\mu:S^1\times X\to
X$ is the $S^1$-action and $EZ$ is the Eilenberg-Zilber equivalence,
explicitly described by the Eilenberg-McLane shuffle
map~\cite[p.64]{EMacL53b}. Define the homology groups
$$
   H_*^{[u,u^{-1}]}(X) = HC_*^{[u,u^{-1}]}(\cC),\qquad 
   H_*^{[u]}(X) = HC_*^{[u]}(\cC).
$$ 
While these groups cannot be described as homology groups of a
topological space in the manner of $H_*^{S^1}(X)$ -- they typically
have infinite support in the negative range -- they are nevertheless
unavoidable should one wish to formulate duality. More precisely, let
us assume that $X$ is an oriented manifold of dimension $n$ with
boundary preserved by the $S^1$-action. Denoting by
$H^*_{S^1}(X)=H^*(X\times_{S^1} ES^1)$ the usual $S^1$-equivariant
cohomology groups, Poincar\'e duality in the
$S^1$-equivariant setting takes the form  
$$
   H^i_{S^1}(X)\cong H_{n-i}^{[u]}(X,\p X).
$$
More generally, dualizing the mixed complex structure on $C_*(X)$ and
changing the degree of $u$ to $+2$, one can define two other versions
$H^*_{[u,u^{-1}]}(X)$ and $H^*_{[u^{-1}]}(X)$ of \emph{$S^1$-equivariant
cohomology}, with Poincar\'e duality isomorphisms  
$$
   H^i_{[u,u^{-1}]}(X)\cong H_{n-i}^{[u,u^{-1}]}(X,\p X),\qquad 
   H^i_{[u^{-1}]}(X)\cong H_{n-i}^{[u^{-1}]}(X,\p X) = H_{n-i}^{S^1}(X).
$$
See~\cite{Hood-Jones,Chen-Lee} for proofs of related statements. We
shall use below the following simple instance of duality: Consider an
oriented manifold $X$ of dimension $n$ with boundary viewed as an
$S^1$-space with trivial action. Then  
$$
   H_i^{S^1}(X)=\bigoplus_{j\ge 0}H_{i-2j}(X) 
$$
and
\begin{equation} \label{eq:product-sum}
   H^i_{[u^{-1}]}(X,\p X)=\prod_{j\ge 0}H^{i+2j}(X,\p X)=\bigoplus_{j\ge 0}H^{i+2j}(X,\p X),
\end{equation}
so that we indeed have $H_i^{S^1}(X)\cong H^{n-i}_{[u^{-1}]}(X,\p X)$
as a consequence of classical Poincar\'e duality.
\end{example}

In order to define $S^1$-equivariant symplectic homology and
cohomology groups, we use the structure of an $S^1$-complex on each
truncated Floer chain complex $\cC:=FC_*^{(a,b)}(H)$ and
cochain complex $\cC^\vee:=FC^*_{(a,b)}(H)$ constructed
in~\cite{BO3+4,Zhao14}. We set
\begin{gather*}
   FH_*^{(a,b),S^1}(H)=HC_*(\cC),\quad 
   FH_*^{(a,b),[u,u^{-1}]}(H)=HC_*^{[u,u^{-1}]}(\cC),\cr 
   FH_*^{(a,b),[u]}(H)=HC_*^{[u]}(\cC)
\end{gather*}
and
\begin{gather*}
   FH^*_{(a,b),S^1}(H)=HC^*(\cC^\vee),\quad 
   FH^*_{(a,b),[u,u^{-1}]}(H)=HC^*_{[u,u^{-1}]}(\cC^\vee),\cr 
   FH^*_{(a,b),[u^{-1}]}(H)=HC^*_{[u^{-1}]}(\cC^\vee)
\end{gather*}
and use these groups in formulas~\eqref{eq:SH*abW},
\eqref{eq:SH*abWA}, \eqref{eq:SH*abWV}, \eqref{eq:cohSH*abWA},
and~\eqref{eq:cohSH*abWV}, as well as in Definitions~\ref{defi:SH(W)},
\ref{defi:SHWA}, \ref{defi:SHWV}, \ref{defi:cohWA},
and~\ref{defi:cohWV}. The outcome for a pair $(W,V)$ of Liouville
cobordisms with filling are {\em $S^1$-equivariant symplectic homology
groups}  
$$
   SH_*^{S^1,\heartsuit}(W,V),\qquad SH_*^{[u,u^{-1}],\heartsuit}(W,V),\qquad SH_*^{[u],\heartsuit}(W,V),
$$
and {\em $S^1$-equivariant symplectic cohomology groups} 
$$
   SH^*_{S^1,\heartsuit}(W,V),\qquad SH^*_{[[u,u^{-1}]],\heartsuit}(W,V),\qquad S^*_{[[u^{-1}]],\heartsuit}(W,V),
$$
with $\heartsuit\in \{\varnothing,>0,\ge 0, =0, \le 0, <0\}$ as usual. 

\begin{remark} \label{rmk:localization-chain-level}
The notation $[[u]]$ and $[[u,u^{-1}]]$ in the equivariant symplectic cohomology
groups is a reminder that, in the case of a Liouville domain, the inverse
limit in the definition leads in general to formal power series rather
than polynomials. It also indicates the analogy to the
$S^1$-equivariant cohomology groups defined by Jones and
Petrack~\cite{Jones-Petrack}. Indeed, it is proved
in~\cite{Zhao14,ACF14} that for a Liouville domain $W$ and with
rational coefficients the second group satisfies fixed point localization
\begin{equation}\label{eq:fixed}
   SH^*_{[[u,u^{-1}]]}(W;\Q) \cong H_{n+*}(W,\p
   W;\Q)\otimes_\Q\Q[u,u^{-1}].  
\end{equation}
One can define several other potentially interesting versions of 
$S^1$-equivariant symplectic homology by applying the direct/inverse
limit over the bounds of the action window $(a,b)$, the homology
functor, and the completions with respect to $u,u^{-1}$ in different
orders~\cite{ACF14}. In particular, this gives rise to a version of 
periodic/Tate symplectic cohomology of a Liouville domain that equals
the localization of $S^1$-equivariant cohomology and obeys
Goodwillie's theorem~\cite{Goodwillie}. This can also serve as a motivation to phrase the theory of symplectic homology at chain level, see also the discussion of coefficients in the Introduction regarding this point. 
\end{remark}

The equivariant symplectic (co)homology groups are connected to each
other by fundamental exact triangles similar to the one for cyclic homology above, namely 
$$
\xymatrix
@C=10pt
@R=18pt
{
SH_*^{[u],\heartsuit}  \ar[rr] & & 
SH_*^{[u,u^{-1}],\heartsuit} ,  \ar[dl]^{[-2]} \\ & SH_{*}^{S^1,\heartsuit}  \ar[ul]^{[+1]}  
}
\qquad
\xymatrix
@C=10pt
@R=18pt
{
SH^*_{S^1,\heartsuit}  \ar[rr] & & 
SH^*_{[[u,u^{-1}]],\heartsuit} .  \ar[dl]^{[+2]} \\ & SH^{*}_{[[u^{-1}]],\heartsuit}  \ar[ul]^{[-1]}  
}
$$
The non-equivariant and equivariant theories are connected by {\em Gysin exact triangles} 
$$
\xymatrix
@C=18pt
@R=10pt
{
SH_*^{\heartsuit}  \ar[rr] & & 
SH_*^{S^1,\heartsuit} ,  \ar[dl]^{[-2]} \\ & SH_{*}^{S^1,\heartsuit}  \ar[ul]^{[+1]}  
}
\qquad
\xymatrix
@C=18pt
@R=10pt
{
SH^{*}_{S^1,\heartsuit}  \ar[rr]^{[+2]} & & 
SH^*_{S^1,\heartsuit},   \ar[dl] \\ & SH^*_{\heartsuit}  \ar[ul]^{[-1]}  
}
$$
respectively
$$
\xymatrix
@C=18pt
@R=10pt
{
SH_{*}^{[u],\heartsuit}  \ar[rr]^{[-2]} & & 
SH_*^{[u],\heartsuit} ,  \ar[dl] \\ & SH_*^{\heartsuit}  \ar[ul]^{[+1]}  
}
\qquad
\xymatrix
@C=18pt
@R=10pt
{
SH^{*}_{\heartsuit}  \ar[rr] & & 
SH^*_{[[u^{-1}]],\heartsuit}.   \ar[dl]^{[+2]} \\ & SH^{*}_{[[u^{-1}]],\heartsuit}  \ar[ul]^{[-1]}  
}
$$
By construction, all $S^1$-equivariant symplectic homology and cohomology groups
are modules over $\K[u]$.
Moreover, the periodic versions are actually modules over the
larger ring $\K[u,u^{-1}]$. In particular, this module
structure induces periodicity isomorphisms 
$$
SH_*^{[u,u^{-1}],\heartsuit}\cong SH_{*+2}^{[u,u^{-1}],\heartsuit},\qquad
SH^*_{[[u,u^{-1}]],\heartsuit}\cong SH^{*+2}_{[[u,u^{-1}]],\heartsuit}.  
$$
All the exact triangles above are obtained at the level of truncated
Floer homology by writing the complex that computes $HC_*^{[u,u^{-1}]}(\cC)$ as the
product total complex of a multicomplex of the form  
$$
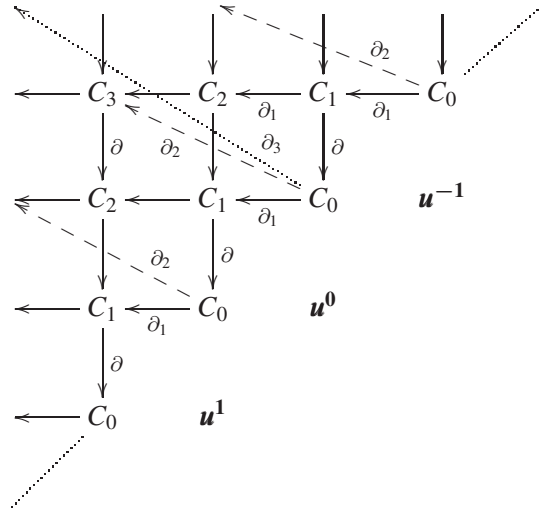
\xymatrix{
& \ar[d] & \ar[d] & \ar[d] & \ar[d] & \ar@{.}[dl] \\
& C_3 \ar[l] \ar[d]^{\p} & C_2 \ar[l] \ar[d] & C_1 \ar[l]^-{\p_1} \ar[d]^{\p} & C_0 \ar[l]^-{\p_1} \ar@{-->}[ull]_(.3){\p_2} & \\
& C_2 \ar[l] \ar[d] & C_1 \ar[l] \ar[d]^\p & C_0 \ar[l]^{\p_1} \ar@{-->}[ull]^(.65){\p_2} \ar@{.>}[uulll]_(.2){\p_3} & \boldsymbol{u^{-1}} & \\
& C_1 \ar[l] \ar[d]^\p & C_0 \ar[l]^-{\p_1} \ar@{-->}[ull]_(.3){\p_2} & \boldsymbol{u^0} & & \\
& C_0 \ar[l] & \boldsymbol{u^{1}} & & & \\
\ar@{.}[ur] & & & & &  
}
$$
and considering natural subcomplexes and quotient
complexes, see~\cite{Jones,BO3+4}. The $[u^{-1}]$-complex sits on the right
half-plane with respect to the $0$-th column, the $[u]$-complex sits
on the left half-plane, and the non-equivariant theory 
sits on the $0$-th column. For cohomology the arrows need to be
reversed. The resulting exact triangles for truncated Floer
(co)homology pass to the limit in symplectic (co)homology due to our
choice of order in the first-inverse-then-direct limit. Note that,
since for a given Hamiltonian $H$ and finite action window $(a,b)$
the complex $FC_*^{(a,b)}(H)$ has finite rank, it actually does not
matter whether we consider the product total complex or the direct sum
total complex to compute $HC_*^{[u,u^{-1}]}(\cC)$.

Here are some further properties of these symplectic (co)homology groups. 

(1) At action level zero we have
$$
   SH_*^{S^1,=0}(W,V)\cong H^{n-*}_{[u^{-1}]}(W,V),\qquad SH_*^{[u],=0}(W,V)\cong H^{n-*}_{S^1}(W,V),
$$
and 
$$
   SH_*^{[u,u^{-1}],=0}(W,V)\cong H^{n-*}_{[u,u^{-1}]}(W,V). 
$$
In particular, for a Liouville domain $W$ of dimension $2n$ we have 
$$
SH_*^{S^1,=0}(W)\cong H^{n-*}_{[u]}(W)\cong H_{*+n}^{S^1}(W,\p W).
$$
This formula appears already in~\cite{Viterbo99}. We interpret in the Introduction this formula as a motivation for viewing the transfer maps as shriek maps.

(2) For a Liouville domain $W$, it is proved in~\cite{BO3+4}
that $SH_*^{S^1,>0}(W)$ is isomorphic over $\Q$ to linearized
contact homology of $\p W$ whenever the latter is defined, see
also~\cite{Gutt15,Gutt-Kang,Kwon-vanKoert} for applications.  

(3) The arguments in~\cite{BO3+4} carry over to the setting of pairs of
Liouville cobordisms with filling in order to show that there is a
spectral sequence converging to $SH_*^{S^1,\heartsuit}(W,V)$ with
second page given by
$E^2=SH_*^\heartsuit(W,V)\otimes\K[u^{-1}]$. In combination
with the Gysin exact triangle 
this yields the fact that the non-equivariant symplectic homology of a
pair $(W,V)$ vanishes if and only if its $S^1$-equivariant symplectic
homology vanishes. The fixed point localization~\eqref{eq:fixed} shows
that this is not true anymore for $SH^{[u,u^{-1}]}_*$. 

(4) The above flavors of $S^1$-equivariant symplectic homology satisfy
Poincar\'e duality in the following general form: given a Liouville
cobordism $W$ and $A\subset \p W$ an admissible union of boundary components, for
any $\heartsuit\in\{\varnothing,>0,\ge 0, =0, \le 0, <0\}$ we have
isomorphisms 
$$
   SH_*^{S^1,\heartsuit}(W,A)\cong SH^{-*}_{[[u^{-1}]],-\heartsuit}(W,A^c),\qquad
   SH_*^{[u],\heartsuit}(W,A)\cong SH^{-*}_{S^1,-\heartsuit}(W,A^c), 
$$
$$ 
   SH_*^{[u,u^{-1}],\heartsuit}(W,A)\cong SH^{-*}_{[[u,u^{-1}]],-\heartsuit}(W,A^c),
$$
where the notation $-\heartsuit$ has the same meaning as
in~\S\ref{sec:Poincare_duality}. There are also algebraic dualities
over the ring $\K[u]$ analogous to those in~\cite{Hood-Jones}  
which pair $SH^*_{S^1,\heartsuit}$ with
$SH_*^{[u],\heartsuit}$, $SH^*_{[[u^{-1}]],\heartsuit}$ with
$SH_*^{S^1,\heartsuit}$, and $SH^*_{[[u,u^{-1}]],\heartsuit}$ with
$SH_*^{[u,u^{-1}],\heartsuit}$. 

Each of the these flavors of $S^1$-equivariant symplectic homology
groups obeys the same set of Eilenberg-Steenrod type axioms as their
nonequivariant counterparts. Transfer maps and invariance for the case of Liouville domains were previously discussed in~\cite{Viterbo99,Zhao14,Gutt15}. Moreover, it follows from the
construction that the Gysin and fundamental exact triangles are
functorial with respect to the tautological exact triangles and also
with respect to the exact triangles of pairs, see
also~\cite{BOGysin,BO3+4} for a basic instance of this phenomenon.

\subsection{Lagrangian symplectic homology, or wrapped Floer homology}\label{sec:LagSH}
Let $W$ be a Liouville cobordism. 
An \emph{exact Lagrangian cobordism in $W$} or, for short, a \emph{Lagrangian cobordism}, is an exact Lagrangian $L\subset W$ which intersects the boundary $\p W$ transversally along a Legendrian submanifold $\p L=L\cap \p W$. This means that $\lambda|_L$ is an exact $1$-form which vanishes when restricted to $\p L$. We denote $\p^\pm L=L\cap \p^\pm W$. Up to applying a Hamiltonian isotopy that fixes $\p W$ one can assume without loss of generality that $L$ is invariant under the Liouville flow near the boundary~\cite[\S3a]{Abouzaid-Seidel}. This means that near its negative or positive boundary we can identify $L$ via the Liouville flow with $[1,1+\epsilon]\times\p^-L$, respectively with $[1-\epsilon,1]\times \p^+L$. We interpret $L$ as a cobordism from $\p^+L$ to $\p^-L$. We refer to $\p^-L$ and $\p^+L$ as being the \emph{positive}, respectively \emph{negative (Legendrian) boundary of $L$}. 

Let $F$ be a Liouville filling of $\p^-W$. An \emph{exact Lagrangian
  filling of $\p^-L$} or, for short, a \emph{filling of
  $\p^-L$}, is a Lagrangian cobordism $F_L\subset F$ whose positive
Legendrian boundary is $\p^-L$ (and which has empty negative boundary). 

One can associate to a Lagrangian cobordism $L$ with filling $F_L$ \emph{Lagrangian symplectic homology groups} 
$$
SH_*^{\heartsuit}(L) , \qquad \heartsuit\in \{\varnothing,>0,\ge 0, =0, \le 0, <0\}.
$$
Similarly, given a pair of Lagrangian cobordisms $K\subset L$ inside a pair of Liouville cobordisms $V\subset W$, with Lagrangian filling $F_L$ inside a Liouville filling $F$, we define Lagrangian symplectic homology groups of the pair $(L,K)$:\footnote{\color{black} Not to be confused with the (wrapped) Lagrangian intersection Floer homology of a pair of Lagrangians.}
$$
SH_*^{\heartsuit}(L,K) , \qquad \heartsuit\in \{\varnothing,>0,\ge 0, =0, \le 0, <0\}.
$$
These are ``open string analogues" of the symplectic homology groups defined for the filled Liouville cobordism $W$, respectively for the pair of Liouville cobordisms $(W,V)$ with filling. They are defined using exactly the same shape of Hamiltonian as in the ``closed string" case. Given such a Hamiltonian, the generators of the corresponding chain complexes are Hamiltonian chords with endpoints on $L$ 
$$
\gamma:[0,1]\to W,\qquad \gamma(\{0,1\})\subset L,\qquad \dot\gamma = X_H\circ\gamma,
$$
and the Floer differential counts strips with Lagrangian boundary
condition on $L$ which are finite energy solutions of the Floer
equation   
$$
u:\R\times[0,1]\to W,\qquad u(\R\times\{0,1\})\subset L, \qquad 
\p_su+J(u)(\p_t u - X_H\circ u)=0.
$$
The theory is naturally defined over $\Z/2$, and an additional assumption on the Lagrangian is needed (e.g. relatively spin) in order to define the theory with more general coefficients. 

\begin{example}
Let $L$ be a Lagrangian cobordism inside a Liouville domain $W$, so
that $L$ has empty negative boundary and empty filling. The Lagrangian
symplectic homology group  
$SH_*(L)$ coincides with the wrapped Floer homology group of $L$
introduced in~\cite{Abouzaid-Seidel, FSS}. 
The Lagrangian symplectic homology group $SH_*^{>0}(L)$ is isomorphic
to the linearized Legendrian contact homology group of
$\p^+L$~\cite{Ekholm, EHK}. The Lagrangian symplectic homology group
$SH_*^{=0}(L)$ is isomorphic to the singular cohomology group $H^{n-*}(L)$ of
$L$. The Lagrangian symplectic homology group of the trivial cobordism
$I\times \p^+L\subset I\times \p^+W$, with $I$ a closed interval in
$(0,\infty[$, is isomorphic to the Lagrangian Rabinowitz-Floer
homology group of $\p^+W$~\cite{Merry,Bounya}. 
\end{example}

The Lagrangian symplectic homology groups obey the same formal
properties as their closed counterparts, reminiscent of the
Eilenberg-Steenrod axioms: functoriality, homotopy invariance, exact
triangle of a pair, excision. Also, the various flavors
$SH_*^\heartsuit(L,K)$ fit into tautological exact triangles, which
are compatible with the exact triangles of pairs. The proofs of all
these properties are word for word the same as for Liouville
cobordisms, using Lagrangian analogues of our confinement
lemmas~\ref{lem:no-escape}, \ref{lem:asy}, \ref{lem:neck},
see also~\cite{Ekholm-Oancea}. 

\noindent \emph{Open-closed theory. } Let $(W,V)$ be a pair of
Liouville cobordisms with filling $F$, and $(L,K)\subset (W,V)$ be a
pair of Lagrangian cobordisms with filling $F_L$. One can define
\emph{open-closed symplectic homology groups}  
$$
SH_*^{\heartsuit}((W,V),(L,K)) , \qquad \heartsuit\in \{\varnothing,>0,\ge 0, =0, \le 0, <0\}
$$
by simultaneously taking into account closed Hamiltonian orbits in $W$
and Hamiltonian chords with endpoints on $L$, using the same shape of
Hamiltonians as in the closed or open setting (see
also~\cite{Ekholm-Oancea}). These homology groups fit into exact
triangles  
\begin{equation*} 
\xymatrix
@C=10pt
@R=18pt
{
SH_*^{\heartsuit}(W,V) \ar[rr] & & 
SH_*^{\heartsuit}((W,V),(L,K)) \ar[dl] \\ & SH_*^{\heartsuit}(L,K) \ar[ul]^{[-1]}  
}
\end{equation*}
and can be thought of as the homology groups of the cone of the
open-closed map,  defined by the count of solutions of a Hamiltonian
Floer equation on a disk with one interior negative puncture and one
boundary positive puncture. The Eilenberg-Steenrod package holds in
this extended setup as well.

\section{Applications} \label{sec:applications}

\subsection{Ubiquity of the exact triangle of a pair} 
A certain number of previous computations in the literature can be
reinterpreted from a unified point of view and generalized from our
perspective.  

(1) One of our original motivations for the definition of the
symplectic homology groups of a Liouville cobordism was the
\emph{exact triangle relating symplectic homology and Rabinowitz-Floer
  homology}~\cite{Cieliebak-Frauenfelder-Oancea} 
\begin{equation*} \label{eq:SHRFH}
\xymatrix
{
SH^{-*}(V) \ar[rr] & & SH_*(V)
\ar[dl] \\ & RFH_*(\p V) \ar[ul]^{[-1]}  
}
\end{equation*}
In view of Poincar\'e duality $SH^{-*}(V)\cong SH_*(V,\p V)$ and the
isomorphism~\eqref{eq:RFH}, this is just the exact triangle of the
pair $(V,\p V)$. See Theorem~\ref{thm:duality-sequence} below for a
more detailed discussion of this triangle. 

(2) The \emph{subcritical and critical handle attaching exact
  triangles} from~\cite{Ci02} and~\cite{Bourgeois-Ekholm-Eliashberg-1}
are special instances of the exact triangle of a pair, see
Sections~\ref{sec:subcritical-handle} and~\ref{sec:critical-handle}
below. Moreover, the surgery exact triangles for linearized contact
homology appear as formal consequences of the corresponding triangles
for symplectic homology, via the relations between equivariant and
non-equivariant symplectic homologies; see
Section~\ref{sec:handle-equivariant} below. 

(3) Let $L\subset V$ be an exact Lagrangian in a Liouville domain $V$
satisfying $SH_*(L)=0$. For example, by a straightforward adaptation
of the vanishing results in~\cite{Cieliebak-Frauenfelder, Kang14} this
is the case if the completion $\wh L$ is displaceable from $V$ in the
completion $\wh V$. Then the tautological sequence yields the
isomorphism 
\begin{equation*}
   SH_*^{>0}(L)\cong SH_{*-1}^{\leq 0}(L)\cong H^{n-*+1}(L),
\end{equation*}
which was previously conjectured by Seidel,
see~\cite[Conjecture~1.2]{Ekholm}, and proved from a Legendrian
contact homology perspective by Dimitroglou
Rizell~\cite[Theorem~2.5]{Rizell-lifting}. This isomorphism implies the 
refinement of Arnold's chord conjecture given
in~\cite{Ekholm-Etnyre-Sabloff}, see Corollary~\ref{cor:chord}
below. A combination of the tautological sequence with the exact
sequence of the pair $(L,\p L)$ and Poincar\'e duality yields the 
\emph{Poincar\'e duality long exact sequence for Legendrian contact
  homology} in~\cite{Ekholm-Etnyre-Sabloff} 
\begin{eqnarray*}\label{eq:EESduality}
\xymatrix{
   H^{n-*}(\p L) \ar[rr] & & SH_{>0}^{-*+2}(\p L) \ar[dl] \\ 
   & SH_*^{>0}(\p L) \ar[ul]^{[-1]}
}
\end{eqnarray*}
as well as its refinement in~\cite[Corollary~1.3]{Ekholm}
and~\cite[Corollary~2.6]{Rizell-lifting}; see
Proposition~\ref{prop:duality-pos} below. 

(4) The results of Chantraine, Dimitroglou Rizell, Ghiggini, and Golovko from~\cite{CDRGG15,Golovko13} can also be reinterpreted from the perspective of the exact triangle of a pair. As an example, consider the following setup: $L$ is an exact Lagrangian cobordism, $\p^-L$ has an exact Lagrangian filling $F_L$, and we assume that $\wh{F_L\circ L}$ is displaceable from the Liouville domain which contains $F_L\circ L$ in the symplectic completion of the ambient exact symplectic manifold. Then $SH_*(F_L\circ L)=0$ and $SH_*(F_L)=0$ (cf. Theorems~\ref{thm:disp-van} and~\ref{thm:van-inherited}), hence also $SH_*(L,\p^-L)=0$. The second long exact sequence in~\cite[Theorem~1.2]{Golovko13} is the exact triangle of the pair $(F_L\circ L,F_L)$ for $SH_{>0}^*$. The setup considered in~\cite{CDRGG15} is that in which $L$ is a Lagrangian concordance, so that the transfer map $SH_{*}^{=0}(F_L\circ L)\stackrel \cong \longrightarrow SH_{*}^{=0}(F_L)$ is an isomorphism. In view of the commutative diagram given by the compatibility of tautological exact triangles with the exact triangle of the pair $(F_L\circ L,F_L)$, 
$$
\xymatrix{
SH_*^{>0}(F_L\circ L)\ar[r] \ar[d] & SH_*^{>0}(F_L) \ar[d] \\
SH_{*-1}^{=0}(F_L\circ L)\ar[r]^\cong & SH_{*-1}^{\color{black}=0}(F_L) 
}
$$
the vertical arrows being isomorphisms since $SH_*(F_L\circ L)$ and $SH_*(F_L)$ vanish, we obtain that the top transfer map is an isomorphism. This is the content of the main result of~\cite{CDRGG15} in the case of linearized Legendrian contact homology, see also~\cite{Golovko13}. The more general bilinearized setup in~\cite{CDRGG15} can be reinterpreted in a similar way.

This circle of ideas should be compared with the results of Biran and Cornea~\cite{BC13}, and also with the results of Dimitroglou Rizell and Golovko~\cite{DRG14}.

\subsection{Duality results}

The following consequence of the long exact sequence of a pair and
Poincar\'e duality is proved in
~\cite{Cieliebak-Frauenfelder-Oancea}. For convenience, we provide 
the short proof in our framework. 

\begin{theorem}[duality sequence~\cite{Cieliebak-Frauenfelder-Oancea}]\label{thm:duality-sequence}
For a Liouville domain $V$ there is a commuting diagram with exact
upper row
\begin{equation}\label{eq:duality}
\xymatrix
{
\cdots SH^{-*}(V) \ar[d] \ar[r]^\phi & SH_*(V) \ar[r]^\psi & SH_*(\p V) \ar[r] &
SH^{1-*}(V) \cdots \\
H_{n+*}(V) \ar[r] & H^{n-*}(V) \ar[u].
}
\end{equation}
Here the horizontal maps come from the long exact sequences of
the pair $(V,\p V)$ in view of Poincar\'e duality $SH_*(V,\p V)\cong
SH^{-*}(V)$ and $H_{n+*}(V)\cong H^{n-*}(V,\p V)$, and the vertical
maps are given by the compositions 
\begin{gather*}
   SH^{-*}(V)\to SH^{-*}_{\leq 0}(V) = SH^{-*}_{=0}(V)\cong H_{n+*}(V), \cr 
   H^{n-*}(V) \cong SH_*^{=0}(V) = SH_*^{\leq 0}(V) \to SH_*(V).
\end{gather*} 
\end{theorem}

\begin{proof}
Commutativity of the diagram~\eqref{eq:duality} follows from
commutativity of the diagram
\begin{equation*}
{\small 
\xymatrix
@C=34pt
{
SH^{-*}(V) \ar[d] \ar[r]^-\cong & SH_*(V,\p V) \ar[d] \ar[r] & SH_*^{\ge 0}(V)=SH_*(V) \\
SH^{-*}_{\le 0}(V)=SH^{-*}_{=0}(V) \ar[d]^\cong \ar[r]^-\cong & SH_*^{=0}(V,\p V)=SH_*^{\ge 0}(V,\p V) \ar[d]^\cong \ar[r] \ar[ur] & SH_*^{=0}(V) \ar[u] \\
H_{n+*}(V) \ar[r]^-\cong & H^{n-*}(V,\p V) \ar[r] & H^{n-*}(V). \ar[u]^\cong
}
}
\end{equation*}
Here the left horizontal maps are Poincar\'e duality isomorphisms and 
the lower right square commutes by
Proposition~\ref{prop:compatibility-exact-triangles}. 
The commutativity of the upper right square can be interpreted as follows: 
by definition of the symplectic homology groups, the composition of
the three maps around the upper square is obtained by considering a
Hamiltonian vanishing on $V$ and increasing its slope near $\p V$ from
large negative to small negative to small positive to large positive,
which yields the upper horizontal map. 
\end{proof}

Here is a computational application of the Poincar\'e Duality
Theorem~\ref{thm:Poincare}, which will be needed for the discussion of
products in Section~\ref{sec:products}.  

\begin{proposition} \label{prop:cob-to-filling}
Let $W$ be a Liouville cobordism with Liouville filling $F$. Then we
have a canonical isomorphism
$$
SH_*^{<0}(W)\cong SH_{>0}^{-*+1}(F).
$$
\end{proposition}

\begin{proof}
We successively have 
$$
SH_*^{<0}(W)\cong SH_{*-1}^{<0}(F\cup W,W)\cong SH_{*-1}^{<0}(F,\p F) \cong SH^{-*+1}_{>0}(F).
$$
The first isomorphism follows from the exact triangle of the pair
$(F\cup W,W)$ for $SH_*^{<0}$ (cf.~\S\ref{sec:exact-triangle-pair})
taking into account that $SH_*^{<0}(F\cup W)=0$ because $F\cup W$ has
empty negative boundary. The second isomorphism is the Excision
Theorem~\ref{thm:excision}. The third isomorphism is Poincar\'e
duality. 
\end{proof}

For further duality results we will need the following vanishing result.

\begin{proposition}\label{prop:triv-cob}
Let $V$ be a Liouville domain. Then 
$$
   SH_*^\heartsuit([0,1]\times\p V,0\times \p V)=0.
$$
{\color{black} for $\heartsuit\in\{\varnothing, >0,\ge 0, =0, \le 0, <0\}$.} 
\end{proposition}

\begin{proof} We are computing the symplectic homology group of a cobordism relative to the concave part of the boundary and therefore the relevant Floer complexes do not involve orbits with negative action. Thus $SH_*^{(a,b)}([0,1]\times\p V,0\times \p V)=SH_*^{(-\epsilon,b)}([0,1]\times\p V,0\times \p V)$ for all $a<0$, $b>0$ and $\epsilon>0$ smaller than the period of a closed Reeb orbit on $\p V$. In the definition of symplectic homology  the inverse limit over $a\to-\infty$ therefore stabilizes and we have 
$SH_*([0,1]\times\p V,0\times \p V)=\lim\limits^{\longrightarrow}_{b\to\infty} SH_*^{(-\epsilon,b)}([0,1]\times\p V,0\times \p V)$. 

The point now is that $SH_*^{(-\epsilon,b)}([0,1]\times\p V,0\times \p V)=0$ for all $b>0$. Indeed, for $b>0$ not lying in the action spectrum of $\p V$, this homology group is computed using the Floer complex generated by closed orbits near $[0,1]\times\p V$ for a Hamiltonian which vanishes on $[0,1]\times \p V$, which has positive slope $b$ near $\{0,1\}\times \p V$, and which is constant in $V$ away from $[0,1]\times \p V$. But such a Hamiltonian can be deformed to one which has constant slope equal to $b$ all over $[0,1]\times \p V$ and for which the corresponding chain complex is zero. See Figure~\ref{fig:SHtrivial-cobordism}, in which the deformed Hamiltonian is drawn with a dashed line. The conclusion follows using the homotopy invariance of the homology under compactly supported deformations. 

{\color{black} This proves $SH_*^{\geq 0}([0,1]\times\p V,0\times \p V)=0$. 
Vanishing of $SH_*^{=0}([0,1]\times\p V,0\times \p V)$ follows from
vanishing of relative singular cohomology, and vanishing of
$SH_*^{>0}([0,1]\times\p V,0\times \p V)$ then follows from the
truncation exact triangle. Since there are no other versions to
consider, this proves the proposition.} 
\end{proof}

\begin{figure} [ht]
\centering
\input{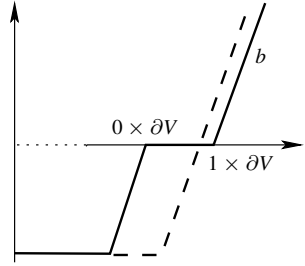}
\caption{Symplectic homology relative to the negative boundary for a trivial cobordism}
\label{fig:SHtrivial-cobordism}
\end{figure}

{\color{black}
\begin{theorem}[Poincar\'e duality for a trivial cobordism]\label{thm:PD-triv-cob}
For every Liouville domain $V$ there exist canonical isomorphisms 
between the symplectic homology and cohomology groups of the trivial
cobordism over $\p V$,
$$
   PD: SH_*^\heartsuit(\p V)\stackrel{\cong}\longrightarrow
   SH^{1-*}_{-\heartsuit}(\p V)
$$
for $\heartsuit\in\{\varnothing, >0,\ge 0, =0, \le 0, <0\}$. 
\end{theorem}

\begin{proof}
We consider the trivial cobordism $W=I\times\p V$ and apply
Proposition~\ref{prop:PD-triple} to the triple $(W,\p W,\p_+ W)$ to obtain
the commuting diagram
\begin{equation*}
\xymatrix
@C=17pt
{
   SH_*^\heartsuit(W,\p_+W) \ar[r] \ar@{=}[d] &
   SH_*^\heartsuit(\p W,\p_+W) \ar[r]^\cong \ar[d]^\cong_{exc} &
   SH_{*-1}^\heartsuit(W,\p W) \ar[r] \ar[d]^\cong_{PD} & SH_{*-1}^\heartsuit(W,\p_+W) \ar@{=}[d] \\
   0 \ar[r] \ar@{=}[d] & SH_*^\heartsuit(W) \ar[r]^\cong \ar[d]^\cong_{PD} &
   SH^{1-*}_\heartsuit(W) \ar[r] \ar[d]^\cong_{exc} & 0 \ar@{=}[d] \\
   SH^{-*}_{-\heartsuit}(W,\p_-W) \ar[r] &
   SH^{1-*}_{-\heartsuit}(W,\p W) \ar[r]^\cong &
   SH^{1-*}_{-\heartsuit}(\p W,\p_-W) \ar[r] & SH^{-*}_{-\heartsuit}(W,\p_-W) \\
}
\end{equation*}
where the first and last row are the long exact sequences of the
triples $(W,\p W,\p_+W)$ and $(W,\p W,\p_-W)$,
respectively, and the vertical arrows are the Poincar\'e duality
and excision isomorphisms. The groups
$SH_*^\heartsuit(W,\p_+W)$ and $SH^{-*}_{-\heartsuit}(W,\p_-W)$ vanish
by Proposition~\ref{prop:triv-cob}. 
The middle horizontal map defined by this diagram is the desired Poincar\'e duality
isomorphism from $SH_*^\heartsuit(\p V)=SH_*^\heartsuit(W)$ to
$SH^{1-*}_{-\heartsuit}(W)=SH^{1-*}_{-\heartsuit}(\p V)$. 
\end{proof}

\begin{theorem}[Poincar\'e duality and exact triangle of $(V,\p V)$]\label{thm:PD-LES}
For every Liouville domain $V$ and $\heartsuit\in\{\varnothing, >0,\ge
0, =0, \le 0, <0\}$ there exists a commuting diagram
\begin{equation}\label{eq:PD-LES}
\xymatrix
@C=20pt
{
    SH_*^\heartsuit(V,\p V) \ar[r] \ar[d]^\cong_{PD} & SH_*^\heartsuit(V)
   \ar[r] \ar[d]^\cong_{PD} & SH_*^\heartsuit(\p V) \ar[r] \ar[d]^\cong_{PD} &
   SH_{*-1}^\heartsuit(V,\p V)  \ar[d]^\cong_{PD} \\
    SH^{-*}_{-\heartsuit}(V) \ar[r] & SH^{-*}_{-\heartsuit}(V,\p
   V) \ar[r] & SH^{1-*}_{-\heartsuit}(\p V) \ar[r] & SH^{1-*}_{-\heartsuit}(V)  \\
}
\end{equation}
where the rows are the long exact sequences of the pair $(V,\p V)$
and the vertical arrows are the Poincar\'e duality
isomorphisms from Theorem~\ref{thm:PD-triv-cob} (the third one) and
Theorem~\ref{thm:Poincare} (the other ones). 
Moreover, the Poincar\'e duality isomorphisms are compatible with 
filtration exact sequences. 
\end{theorem}

\begin{proof} 
Denote by $W$ the trivial cobordism given by a collar neighborhood of
the boundary $\p V$ in $V$. Denote $U=\overline{V\setminus W}$, so
that $\p_+W=\p V$ and $\p _-W=\p U\simeq \p V$. Consider the following
diagram.  
$$
\small
\xymatrix
@C=19pt
{
   SH_*^\heartsuit(V,\p V) \ar[r] \ar[d]^\cong_{PD} & SH_*^\heartsuit(V)
   \ar[r] \ar[d]^\cong_{PD} & SH_*^\heartsuit(\p V) \ar[r] \ar[d]^\cong_{PD} &
   SH_{*-1}^\heartsuit(V,\p V)   \ar[d]^\cong_{PD} \\
      SH^{-*}_{-\heartsuit}(V) \ar[r] \ar[d]^\cong_{exc.} & SH^{-*}_{-\heartsuit}(V,\p
   V) \ar[r] \ar@{=}[d] & SH^{-*}_{-\heartsuit}(V,U\cup \p V) \ar[r] \ar@{=}[d] & SH^{1-*}_{-\heartsuit}(V) \ar[d]^\cong_{exc.} \\
      SH^{-*}_{-\heartsuit}(U\cup \p V,\p V) \ar[r] & SH^{-*}_{-\heartsuit}(V,\p
   V) \ar[r] \ar@/_2.5pc/[ddrr] & SH^{-*}_{-\heartsuit}(V,U\cup \p V) \ar[r] & SH^{1-*}_{-\heartsuit}(U\cup \p V,\p V)    \\
 & & SH^{-*}_{-\heartsuit}(W,\p W) \ar[u]_\cong^{exc.} \ar[r]^\cong \ar[dr] & SH^{-*+1}_{-\heartsuit} (\p W,\p_+W) \ar[u] \\
& & & SH^{-*+1}_{-\heartsuit}(\p _-W) \ar[u]_\cong^{exc.}    
}
$$
The diagram is commutative. The first three rows with their vertical maps
correspond to the commutative diagram in Proposition~\ref{prop:PD-triple} applied to
the triple $(V,W,\varnothing)$, so the first and third rows are the long
exact sequences of the triples $(V,W,\varnothing)\cong(V,\p
V,\varnothing)$ and $(V,U\cup\p V,\p V)$, respectively.  
The right bottom most square is commutative because the maps are induced by the inclusion of triples $(W,\p W,\p_+W)\hookrightarrow (V,U\cup \p V,\p V)$. The bottom right triangle is commutative by definition. 

The third column vertical downward composition 
$$
{\small
\xymatrix
@C=15pt
{
SH_*^\heartsuit(\p V) \ar[r] & SH^{-*}_{-\heartsuit}(V,U\cup \p V)\ar[r] & SH^{-*}_{-\heartsuit}(W,\p W) \ar[r] & SH^{-*+1}_{-\heartsuit}(\p_- W)\simeq SH^{-*+1}_{-\heartsuit}(\p V)
}
}
$$
is the Poincar\'e duality isomorphism of
Theorem~\ref{thm:PD-triv-cob} (by inspection of the diagram in its proof). 
The bottom arrow composition 
$$
{\small
\xymatrix
@C=11pt
{SH^{-*}_{-\heartsuit}(V,\p V)\ar[r] & SH^{-*}_{-\heartsuit}(V,U\cup \p V) \ar[r] & SH^{-*}_{-\heartsuit}(W,\p W)\ar[r] & SH^{-*+1}_{-\heartsuit}(\p _-W)\simeq SH^{-*+1}_{-\heartsuit}(\p V)
}
}
$$ 
is the connecting homomorphism in the cohomology long exact sequence of the pair $(V,\p V)$. 
Finally, the fourth column vertical upward composition 
$$
{\small 
\xymatrix
@C=12pt
{SH^{1-*}_{-\heartsuit}(\p V) \simeq SH^{1-*}_{-\heartsuit}(\p_-W) \ar[r] & SH^{1-*}(\p W,\p _+W) \ar[r] & SH^{1-*}(U\cup \p V,\p V) \ar[r] & SH^{1-*}_{-\heartsuit}(V)
}
}
$$
is the cohomology transfer map for the inclusion $\p V\hookrightarrow V$. 
\end{proof}

{\color{black} \begin{remark} Upon considering the triple $(W,\p W,\p_+W)$ in the proof of Theorem~\ref{thm:PD-triv-cob} and the triple $(V,U\cup \p V,\p V)$ in the proof of Theorem~\ref{thm:PD-LES} we formally enter the setup of multilevel cobordisms discussed in~\S\ref{sec:multilevel}. While we have not explicitly provided proofs for the excision theorem and for the existence of the homology long exact sequences of pairs/triples in that setup, the particular situations that we consider in Theorems~\ref{thm:PD-triv-cob} and~\ref{thm:PD-LES} are the simplest possible and the proofs of those results clearly follow from the corresponding theorems for cobordisms with one level. See also the discussion at the end of~\S\ref{sec:multilevel}.
\end{remark} 
}

Recall that at action zero symplectic homology specialises to singular
cohomology, $SH_*^{=0}(V) \cong H^{n-*}(V)$, 
and similarly for the other versions. Therefore, we obtain

\begin{corollary}
The commuting diagram in Theorem~\ref{thm:PD-LES} specialises at
action zero to
\begin{equation}\label{eq:PD-LES-zero}
\xymatrix{
    H^{n-*}(V,\p V) \ar[r] \ar[d]^\cong_{PD} & H^{n-*}(V)
   \ar[r] \ar[d]^\cong_{PD} & H^{n-*}(\p V) \ar[r] \ar[d]^\cong_{PD} &
   H^{n-*+1}(V,\p V)  \ar[d]^\cong_{PD} \\
    H_{n+*}(V) \ar[r] & H_{n+*}(V,\p
   V) \ar[r] & H_{n+*-1}(\p V) \ar[r] & H_{n+*-1}(V)  \\
}
\end{equation}
where the rows are the long exact sequences of the pair $(V,\p V)$
and the vertical arrows are the Poincar\'e duality
isomorphisms for the closed manifold $\p V$ (the third one) and the
manifold-with-boundary $V$ (the other ones). \hfill$\square$
\end{corollary}
}

We conclude this subsection with an example illustrating that full symplectic homology and
cohomology do not obey any kind of algebraic duality for general
Liouville cobordisms.  

\begin{example}\label{ex:Peter}
Let $V$ be the canonical Liouville filling of a Brieskorn manifold
$\{z\in\C^{n+1}\mid \sum_{j=0}^nz^{a_j}=0,\;|z|=1\}$ with $n\geq 3$
and integers $a_j\geq 2$ satisfying
$\sum_{j=0}^n\frac{1}{a_j}=1$. P.~Uebele~\cite{Uebele} has shown that
with $\Z_2$-coefficients its symplectic homology in degrees $n$ and
$1-n$ is an infinite direct sum 
$$
   SH_k(V;\Z_2)\cong\bigoplus_{\N}\Z_2\quad\text{for $k=n$ and $k=1-n$}.
$$
By algebraic duality, it follows that its symplectic cohomology in
these degrees is an infinite direct product
$$
   SH^k(V;\Z_2)\cong SH_k(V;\Z_2)^\vee\cong \prod_{\N}\Z_2\quad\text{for $k=n$ and $k=1-n$}.
$$
In view of the exact sequence~\eqref{eq:duality} with the map $\phi$
of finite rank, $SH_k(\p V;\Z_2)$ agrees with $SH_k(V)\oplus
SH^{1-k}(V)$ up to an error of finite dimension, hence 
$$
   SH_k(\p V;\Z_2) \cong \bigoplus_{\N}\Z_2\oplus \prod_{\N}\Z_2\quad\text{for $k=n$ and $k=1-n$}.
$$
By Theorem~\ref{thm:PD-triv-cob}, the symplectic cohomology groups in
these degrees are the same,
$$
   SH^k(\p V;\Z_2) \cong \bigoplus_{\N}\Z_2\oplus \prod_{\N}\Z_2\quad\text{for $k=n$ and $k=1-n$}.
$$
Since the dual of the infinite direct product is not the infinite
direct sum, this shows that for $k=n,1-n$ neither $SH^k(\p
V;\Z_2)=SH_k(\p V;\Z_2)^\vee$ nor $SH_k(\p V;\Z_2)=SH^k(\p
V;\Z_2)^\vee$. 
\end{example}

\subsection{Vanishing and finite dimensionality} \label{sec:vanishing}
In this subsection we give some conditions under which symplectic
homology groups are zero or finite dimensional. We begin with a simple
consequence of the duality sequence~\eqref{eq:duality}. 

\begin{corollary} \label{cor:vanishing-finitedim}
For a Liouville domain $V$ the following hold using field coefficients:

(a) If one among $SH_n(V)$, $SH^{-n}(V)$, $SH_n(\p V)$, or $SH_n(V,\p V)$ vanishes, then
   all of $SH_*(V)$, $SH^{-*}(V)$, $SH_*(\p V)$, and $SH_*(V,\p V)$ vanish. 

(b) If one among $SH_*(V)$, $SH^*(V)$, $SH_*(\p V)$, or $SH_*(V,\p V)$ is finite
   dimensional, then so are the other three. 
\end{corollary}

\begin{proof}
Part (a) is~\cite[Theorem~13.3]{Ritter}, except for the statement
involving $SH_*(V,\p V)$, which is a consequence of Poincar\'e
duality. For part (b), in view of Poincar\'e duality $SH_*(V,\p
V)\cong SH^{-*}(V)$ we only need to deal with $SH_*(V)$, $SH^*(V)$,
and $SH_*(\p V)$. 
Since $SH^k(V)\cong \Hom\bigr(SH_k(V),\K\bigr)$ in each degree,
$SH^*(V)$ is finite dimensional iff $SH_*(V)$ is. If both are finite
dimensional, then two out of three terms in the exact sequence~\eqref{eq:duality} are
finite dimensional, so the third term $SH_*(\p V)$ is finite
dimensional as well. Conversely, suppose that $\dim SH(\p V)<\infty$. 
Then the map $\psi$ in~\eqref{eq:duality} has finite rank, as does the
map $\phi$ (because it factors through singular homology), and thus $\dim SH_*(V)<\infty$. 
Alternatively, one could argue by contradiction: If $\dim SH(\p
V)<\infty$ and $SH_*(V)$, $SH^*(V)$ were infinite dimensional, then
the long exact sequence~\eqref{eq:duality} would imply $\dim SH_*(V) = \dim SH^*(V)$,
which is impossible by Remark~\ref{rem:univ-coeff} below. 
\end{proof}

\begin{remark}\label{rem:univ-coeff}
A $\K$-vector space is isomorphic to its dual space if and
only if it is finite dimensional (see~\cite{Dubuque} for a nice proof 
-- we thank I.~Blechschmidt for pointing this out). Hence for a pair
of Liouville cobordisms with filling $(W,V)$ and using field
coefficients we obtain that $SH^k_\heartsuit(W,V)$ is isomorphic to
$SH_k^\heartsuit(W,V)$ for $\heartsuit\in\{<0\le 0,=0,\ge 0,>0\}$ if
and only if both vector spaces are finite dimensional. 
\end{remark}

We say that a subset of a symplectic manifold is {\em
displaceable} if it can be displaced from itself by a compactly
supported Hamiltonian isotopy. It has been known for a while that
displaceability implies vanishing of Rabinowitz-Floer
homology~\cite{Cieliebak-Frauenfelder} and symplectic
homology~\cite{Kang14} of a Liouville domain. In the context of this
paper, these appear as special cases of the following general
vanishing result, whose proof is a straightforward adaptation of the
ones in~\cite{Cieliebak-Frauenfelder} and~\cite{Kang14}. 

\begin{theorem}[displaceability implies vanishing]\label{thm:disp-van}

\hfill

(a) Let $(W,V)$ be a Liouville cobordism pair with filling $F$ such that $V$ is
displaceable in the completion of $F\circ W$. Then $SH_*(V)=0$. 

(b) Let $L\subset V$ be an exact Lagrangian in a Liouville domain $V$ whose
completion $\wh L$ is displaceable from $V$ in the completion
$\wh V$. Then $SH_*(L)=0$. 
\hfill$\square$
\end{theorem}

For example, the displaceability hypothesis in (a) is always satisfied if the
completion of $F\circ W$ is a subcritical Stein manifold, or more
generally the product of a Liouville manifold with $\C$. 

\begin{remark}
(i) If in Theorem~\ref{thm:disp-van}(a) the cobordism $V$ as well as its
filling $E=F\cup W^{bottom}$ are connected, then displaceability of $V$
implies displaceability of $E\cup V$ and the vanishing of $SH_*(V)$
follows from the vanishing of symplectic homology of the Liouville
domains $E$ and $E\cup V$. 

(ii) In the situation of Theorem~\ref{thm:disp-van}(a),
displaceability of $V$ implies that of $\p V$, so 
we also have $SH_*(\p_\pm V)=SH_*(\p V)=0$ and (via exact sequences of
triples) $SH_*(V,\p_\pm V)=SH_*(V,\p V)=0$. 
\end{remark}

Another condition that ensures vanishing of $SH_*(V)$ is the vanishing
of $SH_*(W)$ for a pair $(W,V)$. This was observed for Liouville
domains by Ritter~\cite{Ritter} as a consequence of the product
structure: vanishing of $SH_*(W)$ implies that its unit $1_W$
vanishes, hence so does its image $1_V$ under the transfer map
$SH_*(W)\to SH_*(V)$, which implies $SH_*(V)=0$. In view of
Theorem~\ref{thm:product}, the same argument proves

\begin{theorem}[vanishing is inherited]\label{thm:van-inherited}
Let $(W,V)$ be a Liouville cobordism pair. 
Then $SH_*(W)=0$ implies $SH_*(V)=0$. \hfill$\square$ 
\end{theorem}

Again, the hypothesis $SH_*(W)=0$ is satisfied if the
completion of $F\circ W$ is a subcritical Stein manifold, or more
generally the product of a Liouville manifold with $\C$. However,
there exist Liouville domains $W$ that are not of this type and still
have vanishing symplectic homology, e.g.~flexible Stein
domains~\cite{Cieliebak-Eliashberg-book} as well as certain
non-flexible Stein
domains~\cite{Maydanskiy,Abouzaid-Seidel-recombination,Murphy-Siegel,Ritter09}. 
Conversely, there exist many examples of Liouville pairs $(W,V)$ with
$V$ displaceable and $SH_*(W)\neq 0$. So neither of the two Vanishing
Theorems~\ref{thm:disp-van} and~\ref{thm:van-inherited} implies the
other.

\subsection{Consequences of vanishing of symplectic homology} 
Suppose that $V$ is a Liouville domain with $SH_*(V)=0$. Then the
tautological sequence yields
\begin{equation}\label{eq:vanishing-hom}
   SH_*^{>0}(V)\cong SH_{*-1}^{\leq 0}(V)\cong H^{n-*+1}(V)\neq 0. 
\end{equation}
Similarly, if $L\subset V$ is an exact Lagrangian with $SH_*(L)=0$, then 
\begin{equation}\label{eq:vanishing-hom-Lag}
   SH_*^{>0}(L)\cong SH_{*-1}^{\leq 0}(L)\cong H^{n-*+1}(L)\neq 0. 
\end{equation}
This has the following dynamical consequences~\cite{Viterbo99,Ritter}.

\begin{corollary}\label{cor:chord}
(a) Let $V$ be a Liouville domain with $SH_*(V)=0$ (e.g., this is the case
if $\p V$ is displaceable in $\wh V$). Then there exists at least one closed Reeb orbit. 

(b) Let $L$ be an exact Lagrangian $L\subset V$ with
$SH_*(L)=0$ (e.g., this is the case if $\wh L$ is displaceable from
$V$ in $\wh V$). Then there exists at least one Reeb chord with boundary on $\p L$.
If all the Reeb chords are nondegenerate their number is bounded from below
by $\mbox{rk}\,H_*(L)\geq \mbox{rk}\,H_*(\p L)/2$. 
\end{corollary}

\begin{proof} 
{\color{black}The assertion in (a) follows immediately from~\eqref{eq:vanishing-hom} because
$SH_*^{>0}(V)$ is generated by closed Reeb orbits. Similarly, the
first assertion in (b) follows from~\eqref{eq:vanishing-hom-Lag}.}  

The second assertion in (b) also follows
from~\eqref{eq:vanishing-hom} because, if all Reeb chords are
nondegenerate, their number is bounded from below by $\mbox{rk}\,
SH_*^{>0}(V)=\mbox{rk}\, H^*(V)$. The estimate
$\mbox{rk}\,H_*(V)\geq \mbox{rk}\,H_*(\p V)/2$ follows readily 
from the long exact sequence of the pair $(V,\p V)$ in singular
homology and Poincar\'e duality. 
\end{proof}

Vanishing of symplectic homology also implies the following refinement
of the duality sequence~\eqref{eq:duality}.  

\begin{proposition}[duality sequence for positive symplectic homology]
\label{prop:duality-pos}
(a) Let $V$ be a Liouville domain with $SH_*(V)=0$ (e.g., this is the case
if $\p V$ is displaceable in $\wh V$). Then there exists a commuting
diagram with exact rows
\begin{equation*}
\xymatrix{
 H^{n-*}(\p V) \ar[r]^-\sigma \ar[d]^= & SH^{2-*}_{>0}(\p V) \ar[r]^-\tau &
SH_*^{>0}(\p V) \ar[r]^-\rho \ar[d]_g^\cong & H^{n-*+1}(\p V) \ar[d]^= \\
 H^{n-*}(\p V) \ar[r]^-{\sigma_0} & H^{n-*+1}(V,\p V) \ar[r]^-{\tau_0} \ar[u]_f^\cong &
H^{n-*+1}(V) \ar[r]^-{\rho_0} & H^{n-*+1}(\p V) 
}
\end{equation*}
(b) Let $L\subset V$ be an exact Lagrangian in a Liouville domain with
$SH_*(L)=0$ (e.g., this is the case if $\wh L$ is displaceable from
$V$ in $\wh V$). Then there exists a commuting diagram with exact rows
\begin{equation*}
\xymatrix{
 H^{n-*}(\p L) \ar[r]^-\sigma \ar[d]^= & SH^{2-*}_{>0}(\p L) \ar[r]^-\tau &
SH_*^{>0}(\p L) \ar[r]^-\rho \ar[d]_g^\cong & H^{n-*+1}(\p L) \ar[d]^= \\
 H^{n-*}(\p L) \ar[r]^-{\sigma_0} & H^{n-*+1}(L,\p L) \ar[r]^-{\tau_0} \ar[u]_f^\cong &
H^{n-*+1}(L) \ar[r]^-{\rho_0} & H^{n-*+1}(\p L) 
}
\end{equation*}
\end{proposition}

\begin{proof}
For part (a) consider the commuting diagram whose columns are the exact sequences of
the pair $(V,\p V)$ and whose rows are the tautological sequences
$$
\xymatrix
@C=20pt
{
 SH_*^{=0}(V) \ar[r] \ar[d] & SH_*^{\geq 0}(V) \ar[r] \ar[d] &
SH_*^{>0}(V) \ar[r] \ar[d] & SH_{*-1}^{=0}(V) \ar[d] \\
 SH_*^{=0}(\p V) \ar[r] \ar[d] & SH_*^{\geq 0}(\p V) \ar[r] \ar[d] &
SH_*^{>0}(\p V) \ar[r] \ar[d] & SH_{*-1}^{=0}(\p V) \ar[d] \\
 SH_{*-1}^{=0}(V,\p V) \ar[r] \ar[d] & SH_{*-1}^{\geq 0}(V,\p V) \ar[r] &
SH_{*-1}^{>0}(V,\p V) \ar[r] & SH_{*-2}^{=0}(V,\p V) \\
SH_{*-1}^{=0}(V)
}
$$
We replace the groups $SH_*^{=0}$ by the corresponding singular
cohomology groups, and insert $SH_*^{\geq 0}(V)=SH_*(V)=0$ (which holds by
hypothesis) and $SH_{*-1}^{>0}(V,\p V)=0$ (which always
holds). Moreover, we replace $SH_{*-1}^{\geq 0}(V,\p V)$ by the
isomorphic group $SH_{*-1}^{\geq 0}(V,\p V)\cong SH_{*-2}^{<0}(V,\p
V)\cong SH^{2-*}_{>0}(V) = SH^{2-*}_{>0}(\p V)$, where the first
isomorphism comes from the tautological sequence in view of $SH_*(V,\p
V)=0$ (which follows from the hypothesis $SH_*(V)=0$ via
Corollary~\ref{cor:vanishing-finitedim}) and the second one is
Poincar\'e duality. Then the diagram becomes
$$
\xymatrix{
 H^{n-*}(V) \ar[r] \ar[d] & 0 \ar[r] \ar[d] &
SH_*^{>0}(V) \ar[r]^\cong \ar[d]^\cong & H^{n-*+1}(V) \ar[d]^{\rho_0} \\
 H^{n-*}(\p V) \ar[r] \ar[d]^{\sigma_0} \ar[rd]_\sigma & SH_*^{\geq 0}(\p V) \ar[r] \ar[d]^\cong &
SH_*^{>0}(\p V) \ar[r]_-\rho \ar[d] \ar[ru]_g^\cong & H^{n-*+1}(\p V) \ar[d] \\
 H^{n-*+1}(V,\p V) \ar[r]_-f^-\cong \ar[d]^{\tau_0} 
& SH^{2-*}_{>0}(\p V) \ar[r] \ar[ru]_\tau & 0 \ar[r] & H^{n-*+2}(V,\p V) \\
H^{n-*+1}(V)
}
$$
From this we read off the commuting diagram in
Proposition~\ref{prop:duality-pos}(a). Part (b) is proved analogously. 
\end{proof}

Corollary~\ref{cor:chord}(b) and the upper long exact sequence in
Proposition~\ref{prop:duality-pos}(b) were proved
in~\cite{Ekholm-Etnyre-Sabloff} in the context of contact manifolds of
the form $P\times \R$ (compare also with~\cite{Ritter}).
The commuting diagram in Proposition~\ref{prop:duality-pos}(b) appears
in~\cite[Corollary~1.3]{Ekholm}
and~\cite[Corollary~2.6]{Rizell-lifting}.

\subsection{Invariants of contact manifolds}\label{sec:contact} 
We describe in this subsection how to obtain invariants of contact
manifolds from the various symplectic homology groups that we defined
in this paper.  
Recall that a contact manifold with chosen contact form
$(M^{2n-1},\alpha)$ is called {\em hypertight} if it has no
contractible closed Reeb orbits. Following~\cite{Uebele-products} we
call $(M,\alpha)$ {\em index-positive} if $\xi=\ker\alpha$ satisfies
either
\renewcommand{\theenumi}{\roman{enumi}}
\begin{enumerate}
\item {\color{black} $c_1(\xi)|_{\pi_2(M)}=0$ and} the Conley-Zehnder index of every contractible closed Reeb orbit
  $\gamma$ in $M$ satisfies $\CZ(\gamma)+n-3>1$, or
\item {\color{black} $(M,\alpha)$ admits a Liouville filling $F$ with
  $c_1(F)|_{\pi_2(F)}=0$ such that $\CZ(\gamma)+n-3>0$ for every
  closed Reeb orbit $\gamma$ in $M$ which is contractible in $F$.} 
\end{enumerate}
We will call a (as always,
cooriented) contact manifold $(M,\xi)$ {\em hypertight} resp.~{\em
  index-positive} if it admits a defining contact form with this
property.  

{\color{black}
\begin{remark}\label{rem:subcritical}
Condition (ii) is in particular satisfied if $(M,\alpha)$ admits a subcritical Stein filling $F$ of dimension $2n\geq 4$ with $c_1(F)|_{\pi_2(F)}=0$. Indeed, $M=\p F$ then admits a contact form so that all Conley-Zehnder indices of closed Reeb orbits which are contractible in $\p F$ are $>1$~\cite{MLYau}, and therefore $>3-n$ provided that $n\geq 2$. Since $F$ is Stein subcritical, the map $\pi_1(\p F)\to \pi_1(F)$ induced by the inclusion is injective. Indeed, the subcritical skeleton has codimension $\ge n+1\ge 3$ and a generic homotopy of paths will avoid it, so that it can afterwards be pushed by the Liouville flow to the boundary. Thus any loop in $\p F$ which is contractible in $F$ is also contractible in $\p F$ and the condition on the indices therefore holds for all loops which are contractible in $F$. 
\end{remark}
}

{\color{black} The following result follows in the index-positive case
(ii) from the arguments of \cite{BOcont}, as remarked
in~\cite{Cieliebak-Frauenfelder-Oancea,BO3+4}. For the hypertight case or the  
index-positive case (i) see~\cite{Uebele-products,BO3+4}. For another instance in the
$S^1$-equivariant case see~\cite{Gutt15}. We sketch below a short unified proof. 

\begin{proposition}\label{prop:SH-without-filling}
Given a Liouville cobordism $W$ whose negative boundary $\p^-W$ is
hypertight or index-positive, the symplectic homology groups  
$$
   SH_*^{\heartsuit}(W) \mbox{\quad and \quad} 
   SH_*^{S^1,\heartsuit}(W), \qquad \heartsuit\in \{\varnothing,>0,\ge
   0, =0, \le 0, <0\} 
$$
are defined, independent of the contact form $\alpha$ on $\p^-W$ in
the given class, and independent of the filling in case (ii). 
\end{proposition}

\begin{proof}
We will discuss the case $SH_*^{\heartsuit}(W)$, the equivariant case being analogous. 

In case (ii) we define $SH_*^{\heartsuit}(W)$ as the usual symplectic
homology group with respect to a filling $F$ in the given class. To
show independence of the filling, fix a finite action window $(a,b)$
and consider a Hamiltonian $H$ on the
completion $\wh{W}_F$ as in Figure~\ref{fig:H-heuristic}. We perform
neck stretching as described in the proof of Lemma~\ref{lem:neck},
inserting cylindrical pieces $[-R_k,R_k]\times M$ with
$R_k\to\infty$, at the hypersurface $M:=\{\delta\}\times\p^-W$ where
$H\equiv c$ for a constant $c>-a$. We claim that for $k$
sufficiently large, Floer cylinders appearing in the differential between $1$-periodic orbits $x_\pm$ of
$H$ of types $I^-,I^0,I^+$ with action in $(a,b)$ do not enter the region $F\setminus[\delta,1]\times\p F$. 
Then it follows that all these Floer cylinders can be viewed as lying in the
$2$-sided completion $\wh{W}$, so $FH_*^{(a,b)}(H)$ is independent of
the filling. By the same claim applied to continuation morphisms, we
deduce independence of the filling for the filtered symplectic
homology groups $SH_*^{(a,b)}(W)$ and the groups $SH_*^{\heartsuit}(W)$.

To prove the claim, we argue by contradiction and suppose that for all
$k$ there exist Floer cylinders $u_k$ as above entering $F\setminus[\delta,1]\times\p F$.
In the limit $k\to\infty$ they converge {\color{black}in the SFT sense~\cite{BEHWZ,CM}} to a broken
holomorphic curve $C$ with punctures asymptotic to closed Reeb orbits
on $M$. {\color{black}Here it is understood that the almost complex structure is chosen to be cylindrical and time-independent in the neck $[-R_k,R_k]\times M$ that is inserted near the hypersurface $M=\{\delta\}\times \p^-W$.} We first observe that $C$ can have only one component in $\wh{W}$. 
This follows by the argument in the proof of Lemma~\ref{lem:neck}:
Otherwise there would exist for large $k$ a separating loop $\delta_k$
on the domain $\R\times S^1$, winding around in the negative $S^1$-direction, 
such that $u_k(\delta_k)$ is $C^1$-close to a (positively parameterized)
closed Reeb orbit $\gamma$ on $M$, and the resulting estimate
$A_H(x_-)\leq -c<a$ would contradict the condition $A_H(x_-)>a$. 
It follows that $C$ consists of a Floer cylinder $C_+$ in $\wh{W}$ with
$p\geq 1$ negative punctures asymptotic to closed Reeb orbits $\gamma_i$ and
holomorphic planes $C_i$ in $\wh{F}$ asymptotic to $\gamma_i$. 
{\color{black}
In particular, the orbits $\gamma_i$ are contractible, and this already leads to a contradiction in the hypertight case. 
To reach a contradiction in the index-positive case, we remark that the component $C_+$ belongs to a moduli space which is transversely cut out. Indeed, the equation is perturbed by an $S^1$-dependent Hamiltonian term near the punctures where $C_+$ converges to Hamiltonian periodic orbits, and the almost complex structure is chosen to be generic and time-dependent in the region where all the Hamiltonian orbits are located, hence transversality follows as in Hamiltonian Floer theory, see e.g.~\cite{SZ92}. If non-empty, the moduli space to which $C_+$ belongs has dimension at least $1$ (due to $\R$-translations in the domain), so
the Fredholm index of $C_+$ satisfies $\ind(C_+)\geq 1$. 
}
On the other hand, the index of $C_+$ is given by
$$
   \ind(C_+) = \CZ(x_+)-\CZ(x_-)-\sum_{i=1}^p\bigl(\CZ(\gamma_i)+n-3\bigr),
$$
which in view of $\CZ(x_+)-\CZ(x_-)=1$ for contributions to the Floer
differential yields
$$
   \sum_{i=1}^p\bigl(\CZ(\gamma_i)+n-3\bigr) \leq 0.
$$ 
{\color{black}Now the assumption of index-positivity and the fact that the orbits $\gamma_i$ are contractible implies $\CZ(\gamma_i)+n-3>0$. This contradicts the fact that $p\geq 1$, and proves case (ii).}

The proof in case (i) is very similar. We again consider $(a,b)$ and
$H$ as above, where $H$ is now defined on the $2$-sided completion
$\wh{W}$ rather than $\wh{W}_F$. We define the Floer differential for
$H$ by counting Floer cylinders between orbits $x_\pm$ in
$\wh{W}$. This is well-defined because SFT type breaking of Floer cylinders 
at the negative end of $\wh{W}$ is ruled out by exactly the same
argument as in case (ii). In contrast to case (ii) where this was
automatic, we now must also show that the Floer differential squares
to zero. For this, we must rule out SFT type breaking of Floer cylinders
connecting orbits $x_\pm$ of index difference $2$. If such breaking
occurs {\color{black}the argument in case (ii) directly leads to a contradiction in the hypertight case, while in the index-positive case it leads to $p\geq 1$ contractible orbits $\gamma_i$ 
satisfying}
$$
   \sum_{i=1}^p\bigl(\CZ(\gamma_i)+n-3\bigr) \leq 1.
$$ 
Under the stronger hypothesis $\CZ(\gamma_i)+n-3>1$ this is again a contradiction
and case (i) is proved.
\end{proof}

This proposition leads to the definition of homological invariants of
hypertight or index-positive contact manifolds,
$$
   SH_*^{[S^1,] \heartsuit}(M,\xi) = SH_*^{[S^1,] \heartsuit}(I\times
   M),\qquad \heartsuit\in \{\varnothing,>0,\ge 0, =0, \le 0, <0\}, 
$$
where $I=[0,1]$ and $I\times M$ is the trivial Liouville cobordism. 
Here the notation $SH_*^{[S^1,] \heartsuit}$ means that the symbol $S^1$
is optional.}

\begin{example}
In view of \cite{Cieliebak-Frauenfelder-Oancea}, the group $SH_*(M,\xi)$ can be interpreted as the Rabinowitz-Floer homology group of $(M,\xi)$. A construction of Rabinowitz-Floer homology for hypertight contact manifolds has been recently carried out in \cite{Albers-Fuchs-Merry}.
\end{example}

These contact invariants satisfy various functoriality relations, as
dictated by our functoriality relations for Liouville cobordisms. The
general picture is the following: Given a Liouville cobordism $W$
whose negative 
boundary is hypertight or index-positive, we have maps 
$$
SH_*^{[S^1,] \heartsuit}(\p^-W) \longleftarrow  SH_*^{[S^1,] \heartsuit}(W) \longrightarrow SH_*^{[S^1,] \heartsuit}(\p^+W)
$$
determined by the embedding of trivial cobordisms 
$$
I\times \p^-W\subset W \supset I\times \p^+W.
$$
Since $I\times \p^-W$ and $W$ share the same negative boundary we have an isomorphism $SH_*^{[S^1,] <0}(\p^-W) \stackrel \cong\longleftarrow SH_*^{[S^1,]<0}(W)$, and since $W$ and $I\times \p^+W$ share the same positive boundary we have an isomorphism $SH_*^{[S^1,] >0}(W)\stackrel \cong \longrightarrow SH_*^{[S^1,] >0}(\p^+W)$. In particular we obtain maps 
$$
SH_*^{[S^1,] >0}(\p^-W)\longleftarrow SH_*^{[S^1,] >0}(\p^+W) 
$$
and 
$$
SH_*^{[S^1,] <0}(\p^-W)\longrightarrow SH_*^{[S^1,] <0}(\p^+W).
$$
In the equivariant case and under slightly different assumptions the
first of these two maps was previously constructed by Jean Gutt in
\cite{Gutt15}. 
Such direct maps do not exist for the other versions $\heartsuit\in \{\varnothing,\ge 0, =0, \le 0\}$. In general the cobordism $W$ has to be interpreted as providing a \emph{correspondence}, and this holds in particular for the case of Rabinowitz-Floer homology. 

\noindent \emph{Invariants of Legendrian submanifolds. }
Let $(M^{2n-1},\alpha)$ be a manifold with chosen contact form and
$\Lambda^{n-1}\subset M$ a Legendrian submanifold. 
Extending the earlier definitions to the open case, we call $\Lambda$  
{\em hypertight} if $(M,\alpha)$ is hypertight and $\Lambda$ has
no contractible Reeb chords. We call $\Lambda$ {\em index-positive}
if $(M,\alpha)$ is index-positive and in addition 
\renewcommand{\theenumi}{\roman{enumi}}
\begin{enumerate}
\item in case (i) the Maslov class of $\Lambda$ vanishes on $\pi_2(M,\Lambda)$
  and every Reeb chord $c$ that is trivial in $\pi_1(M,\Lambda)$ 
  satisfies $\CZ(c)>1$;
\item {\color{black} in case (ii) $\Lambda$ admits an exact Lagrangian
  filling $L\subset F$ in the filling $F$ whose Maslov class vanishes on
  $\pi_2(F,L)$ such that $\CZ(c)>0$ for every
  Reeb chord $c$ for $\Lambda$ that is trivial in $\pi_1(F,L)$.} 
\end{enumerate}
We call a Legendrian submanifold in a contact manifold $(M,\xi)$ {\em
  hypertight} resp.~{\em index-positive} if it admits a
defining contact form with this property. 
 
The arguments given in the closed case adapt in a straightforward way
in order to define invariants of hypertight or index-positive
Legendrian submanifolds by  
$$
SH_*^{\heartsuit}(\Lambda)=SH_*^{\heartsuit}(I\times \Lambda), \qquad
\heartsuit\in \{\varnothing,>0,\ge 0, =0, \le 0, <0\}. 
$$

\subsection{Subcritical handle attaching} \label{sec:subcritical-handle}
In this subsection we compute the symplectic homology groups
corresponding to a subcritical handle in the sense of~\cite{Ci02}, {\color{black}with coefficients in a field $\K$.} 

\begin{proposition}\label{prop:subcrit}
Let $W^{2n}$ be a filled Liouville cobordism corresponding to a subcritical
handle of index $k<n$. Then  
\begin{align*}
   SH_*(W,\p^-W) =0, & \qquad SH_{*}(W,\p^+W)=0,\cr
   SH_*^{=0}(W,\p^-W) &\cong SH_{-*}^{=0}(W,\p^+W)=\begin{cases}
   \K & *=n-k, \\ 0 & \text{else,} \end{cases}
\end{align*}
and the restriction maps induce isomorphisms
$$
   SH_*(\p^-W)\stackrel{\cong}\longleftarrow
   SH_*(W)\stackrel{\cong}\longrightarrow SH_*(\p^+W). 
$$
\end{proposition}

\begin{proof}
The vanishing of $SH_*(W,\p^-W)$ is proved in~\cite{Ci02} with
arbitrary coefficients as a consequence of the following fact: for
each degree $i$ there exists $b_i>0$ such that
$SH_i^{(a,b)}(W,\p^-W)=0$ for any $a<0$ and $b\ge b_i$.

Since
$SH_*(W,\p^-W)=SH_*^{\ge 0}(W,\p^-W)$, we can apply the algebraic
duality Proposition~\ref{prop:dualityk} to obtain
$SH^*(W,\p^-W)=SH^*_{\ge 0}(W,\p^-W)=0$, which implies by Poincar\'e
duality $SH_{-*}(W,\p^+W)=0$.  

Since $H^*(W,\p^-W)$ equals $\K$ in degree $k$ and vanishes in all the
other degrees, we obtain 
$$
   SH_*^{=0}(W,\p^-W) \cong H^{n-*}(W,\p^-W) = \begin{cases}
   \K & *=n-k, \\ 0 & \text{else}. \end{cases}
$$
The remaining two isomorphisms follow from the long exact sequences
\begin{gather*}
   0=SH_*(W,\p^-W) \to SH_*(W) \to SH_*(\p^-W) \to SH_{*-1}(W,\p^-W)
   =0, \cr
   0=SH_*(W,\p^+W) \to SH_*(W) \to SH_*(\p^+W) \to SH_{*-1}(W,\p^+W)
   =0.
\end{gather*}
\end{proof}

\begin{remark}
(a) From Proposition~\ref{prop:subcrit} and the tautological sequence we can compute the remaining 
relevant symplectic homology groups of the pair $(W,\p_\pm W)$, namely
$$
   SH_*^{>0}(W,\p^-W) \cong SH_{-*}^{<0}(W,\p^+W)=\begin{cases}
   \K & *=n-k+1, \\ 0 & \text{else}. \end{cases}
$$
Note that the symplectic homology groups relative to one boundary
component only depend on the index $k$, whereas the group $SH_*(W)$ 
depends on the whole hypersurface $\p^-W$ and its filling. 

(b) In view of~\eqref{eq:RFH}, the last statement in
Proposition~\ref{prop:subcrit} gives in particular the isomorphism of
Rabinowitz Floer homology groups
$$
   RFH(\p^+W) \cong RFH(\p^-W). 
$$

(c) Suppose that $(W,V,U)$ is a Liouville cobordism triple such that $W\setminus
   V$ is subcritical. Then Proposition~\ref{prop:subcrit} implies
   $SH_*(W,V)=0$, which together with the exact sequence of the triple (Proposition~\ref{prop:triple}) yields the isomorphism
$$
   SH_*(W,U)\stackrel{\cong}\longrightarrow SH_*(V,U). 
$$
In particular, for $U=\emptyset$ we recover by induction the vanishing
of symplectic homology for subcritical Stein domains. 

{\color{black}(d) The computation of Proposition~\ref{prop:subcrit} is valid more generally with coefficients in an abelian group, but the proof uses filtered symplectic homology and a more general universal coefficients theorem. } 
\end{remark}

Together with the exact triangle of a pair, these computations
provide a complete understanding of the behaviour of all the flavors of 
non-equivariant symplectic homology groups under subcritical handle
attachment, as a consequence of the exact triangle of the pair
$(V\circ W,V)$, where $V$ is a Liouville domain.  
The equivariant case is discussed in
Section~\ref{sec:handle-equivariant} below.

\subsection{Critical handle attaching}\label{sec:critical-handle}
{\color{black}Recall that we use coefficients in a field $\K$. }
In the previous section we saw that the key computation was that of $SH_*(W,\p^-W)$, and the key exact triangle was the exact triangle of the pair $(V',V)$, where $V$ is the filling of $\p^-W$ and $V'=V\circ W$ is the Liouville domain obtained after attaching the handle. 
These same objects form the relevant structure in the case of a critical handle attachment. 

Let $V$ be a Liouville domain, let
$\Lambda=\Lambda_1\sqcup\dots\sqcup\Lambda_\ell$ be a collection of
disjoint Legendrian spheres in $\p V$, denote by $W$ the cobordism
obtained by attaching $\ell$ critical handles (of index $n$) along
these spheres, and denote $V'=V\circ W$. 
Bourgeois, Ekholm, and Eliashberg~\cite{Bourgeois-Ekholm-Eliashberg-1}
assert the existence of {\em surgery exact triangles}\footnote{Since at the time of
writing this article the proof of this result is not yet completed, 
we formulate its consequences below as conjectures.}  
\begin{equation}  \label{eq:BEE} 
\xymatrix
@C=10pt
@R=18pt
{
L\H^{\text{Ho}}(\Lambda)_* \ar[rr] & & 
SH_*(V') \ar[dl] \\ & SH_*(V) \ar[ul]^{[-1]}  
}
\qquad
\xymatrix
@C=10pt
@R=18pt
{
L\H^{\text{Ho}+}(\Lambda)_* \ar[rr] & & 
SH_*^{>0}(V') \ar[dl] \\ & SH_*^{>0}(V) \ar[ul]^{[-1]}  
}
\end{equation}
in which $L\H_*^{\text{Ho}}(\Lambda)$ and
$L\H^{\text{Ho}+}(\Lambda)_*$ are homology groups of Legendrian
contact homology flavour, see also~\cite[\S2.8]{EGH}
\cite{Ekholm-Etnyre-Sullivan}. More precisely,
$L\H^{\text{Ho}+}(\Lambda)_*$ is defined as the homology of a complex
$LH^{\text{Ho}+}(\Lambda)_*$ whose generators are words in Reeb chords
on $\p V$ with endpoints on $\Lambda$, and whose differential counts
certain pseudo-holomorphic curves in the symplectization of $\p V$
with boundary on the conical Lagrangian $S\Lambda$ determined by
$\Lambda$, with a certain number of interior and boundary punctures at
which rigid pseudo-holomorphic planes in $\wh V$, respectively rigid
pseudo-holomorphic half-planes in $\wh V$ with boundary on $S\Lambda$
are attached (following the terminology
of~\cite{Bourgeois-Ekholm-Eliashberg-1} we call such curves {\em
  anchored in $V$}). The homology group $L\H_*^{\text{Ho}}(\Lambda)$
is defined as the cone of a map $LC^{\text{Ho}+}(\Lambda)_*\to
C^{n-*+1}$, where $C^{n-*+1}$ is the cohomological Morse complex for
the pair $(W,\p^-W)$, which has rank $\ell$ in degree $n-*+1=n$ and
vanishes otherwise, and with zero differential. This map counts curves
of the type taken into account in $LH^{\text{Ho}+}(\Lambda)_*$,
rigidified by imposing an intersection with an unstable manifold of a
critical point in $W$.  
The exact sequence of the cone of a map reads in this case 
\begin{equation} \label{eq:LHcone}
\xymatrix
{
H^{n-*}(W,\p^-W) \ar[rr] & & L\H^{\text{Ho}}(\Lambda)_*
\ar[dl] \\ & L\H^{\text{Ho}+}(\Lambda)_* \ar[ul]^{[-1]}  
}
\end{equation}
The surgery exact triangles of Bourgeois, Ekholm, and Eliashberg can be
reinterpreted in our language as follows. 

\begin{conjecture}\label{conj:criticalBEE}
Let $W$ be a filled Liouville cobordism corresponding to attaching $\ell\ge 1$ critical
handles of index $k=n$ along a collection $\Lambda$ of disjoint Legendrian spheres. With field coefficients we have isomorphisms
$$
SH_*^{>0}(W,\p^-W)\cong L\H^{\text{Ho}+}(\Lambda)_*,\qquad 
SH_*(W,\p^-W)\cong L\H^{\text{Ho}}(\Lambda)_*
$$
such that the following hold: 

(i) the tautological exact triangle involving $SH_*^{=0}$, $SH_*$, and $SH_*^{>0}$ for the pair $(W,\p^-W)$ is isomorphic to~\eqref{eq:LHcone};

(ii) the exact triangles~\eqref{eq:BEE} are isomorphic to the exact triangles of the pair $(V',V)$ for $SH_*$, respectively $SH_*^{>0}$. 
\end{conjecture}

Let us explain how this conjecture would follow from the surgery exact
triangle in~\cite{Bourgeois-Ekholm-Eliashberg-1}.
To establish the first two isomorphisms, the first step is to give a description of $SH_*(W,\p^-W)$ and $SH_*^{>0}(W,\p^-W)$ in terms of pseudo-holomorphic curves in a symplectization; this is similar to the description of $SH_*^{>0}(V)$ as a non-equivariant linearized contact homology group given in~\cite{BOcont} and used in~\cite{Bourgeois-Ekholm-Eliashberg-1} as a definition of $SH_*^{>0}(V)$. The second step is to apply to this formulation of $SH_*^\heartsuit(W,\p^-W)$ with $\heartsuit=\{\varnothing,>0\}$ the methods of~\cite{Bourgeois-Ekholm-Eliashberg-1}. The proof of (i) is then straightforward, since $SH_*$ can naturally be expressed as the homology of a cone using the action filtration. 

To prove (ii), the main step is to establish an isomorphism between the transfer map $SH_*^\heartsuit(V')\to SH_*^\heartsuit(V)$ and the map with the same source and target that appears in~\eqref{eq:BEE} for $\heartsuit\in\{\varnothing,>0\}$. The latter map is described in terms of anchored pseudo-holomorphic curves in the symplectization of the cobordism $W$, and the proof of the isomorphism between these maps follows the same ideas as those in~\cite{BOcont}, applied to the monotone homotopies which induce in the limit the transfer map. The claim in (ii) then follows from the results of~\cite{Bourgeois-Ekholm-Eliashberg-1} because, up to rotating a triangle, the groups $L\H^{\text{Ho}+}(\Lambda)_*$ and $L\H^{\text{Ho}}(\Lambda)_*$ can be expressed as homology groups of cones of such maps induced by cobordisms. 

\begin{remark} \label{rmk:SFT}
Following~\cite{BOcont,BO3+4}, all the constructions that we describe in the setup of symplectic homology can be replicated in the language of symplectic field theory, or SFT (with the usual caveat regarding the analytical foundations of the latter). One outcome of this parallel is that our six flavors of symplectic homology provide some new linear SFT-type invariants (the group $SH_*(\p V)$ for $V$ a Liouville domain is the most prominent of these). It is a general fact that the Viterbo transfer maps for symplectic homology correspond to the well-known SFT cobordism maps. 
\end{remark}

As in the proof of Proposition~\ref{prop:subcrit},
Conjecture~\ref{conj:criticalBEE} would imply 

\begin{conjecture} \label{cor:criticalrelhom}
With coefficients in a field $\K$ the following isomorphisms hold:

(i) $SH^{-*}(W,\p^+W)\cong L\H^{\text{Ho}}(\Lambda)_*$ and 
$$
SH_{-*}(W,\p^+W)\cong SH^*(W,\p^-W)\cong (L\H^{\text{Ho}}(\Lambda)_*)^\vee.
$$

(ii) $SH^{-*}_{<0}(W,\p^+W)\cong L\H^{\text{Ho}+}(\Lambda)_*$ and 
$$
SH_{-*}^{<0}(W,\p^+W)\cong SH^*_{>0}(W,\p^-W)\cong (L\H^{\text{Ho}+}(\Lambda)_*)^\vee.
$$
\end{conjecture}

We also have the obvious 
$$
   SH_*^{=0}(W,\p^-W) \cong SH_{-*}^{=0}(W,\p^+W)=\begin{cases}
   \K & *=0, \\ 0 & \text{else.} \end{cases}
$$
Together with the long exact sequence of a pair, these computations
provide a theoretically complete understanding of the behaviour of all
the flavors of symplectic homology groups under critical handle
attachment.  

A particular case of interest is that of comparing $SH_*(\p^-W)$ and $SH_*(\p^+W)$. The answer does not take the form of a long exact sequence because these groups do not sit naturally in a long exact sequence of a pair. The best answer that one can give in such a generality is that we have a correspondence 
$$
SH_*(\p^-W)\longleftarrow SH_*(W)\longrightarrow SH_*(\p^+W)
$$
in which the kernel and cokernel of each arrow can be described in terms of the groups $SH_*(W,\p^-W)$, respectively $SH_*(W,\p^+W)$, which in turn are described in terms of the groups $L\H^{\text{Ho}}(\Lambda)$ as above, using the long exact sequences of the pairs $(W,\p^-W)$ and $(W,\p^+W)$. 
This situation parallels the one encountered when comparing the singular cohomology groups of a manifold before and after surgery (in this case $\p^+W$ is obtained by surgery of index $n$ on $\p^-W$).

\subsection{Handle attaching and $S^1$-equivariant symplectic
  homologies}\label{sec:handle-equivariant}
 The discussion in~\S\ref{sec:subcritical-handle} and~\S\ref{sec:critical-handle} has $S^1$-equivariant analogues. We treat here only $S^1$-equivariant symplectic homology, since negative $S^1$-equivariant symplectic homology and also (negative) $S^1$-equivariant symplectic cohomology can be reduced to the former using Poincar\'e and algebraic duality. 

\noindent \emph{Subcritical handle attaching.} 

\begin{proposition}\label{prop:subcritS1}
Let $W$ be a Liouville cobordism corresponding to a subcritical
handle of index $k<n$. Then with $\K$-coefficients we have 
\begin{align*}
   SH_*^{S^1}(W,\p_\pm W) &= 0,\cr
   SH_*^{S^1,=0}(W,\p^-W) &=\begin{cases}
   \K & *=n-k+2\N, \\ 0 & \text{else,} \end{cases}, \cr
   SH_{*}^{S^1,=0}(W,\p^+W) & =\begin{cases}
   \K & *=k-n+2\N, \\ 0 & \text{else,} \end{cases} \cr
   SH_*^{S^1,>0}(W,\p^-W) &=\begin{cases}
   \K & *=n-k+1+2\N, \\ 0 & \text{else,} \end{cases}, \cr
   SH_{*}^{S^1,<0}(W,\p^+W) & =\begin{cases}
   \K & *=k-n-1+2\N, \\ 0 & \text{else,} \end{cases}
\end{align*}
and the restriction maps induce isomorphisms
$$
   SH_*^{S^1}(\p^-W)\stackrel{\cong}\longleftarrow
   SH_*^{S^1}(W)\stackrel{\cong}\longrightarrow SH_*^{S^1}(\p^+W). 
$$
\end{proposition}

\begin{proof}
The vanishing of $SH_*^{S^1}(W,\p_\pm W)$ follows from that
of $SH_*(W,\p_\pm W)$ using the spectral sequence from non-equivariant
to equivariant symplectic homology. The statement concerning
$SH_*^{S^1,=0}(W,\p_\pm W)$ is a direct computation, using the fact
that the Floer complex reduces in low energy to the Morse complex, see
also~\cite{Viterbo99,BO3+4}:
$$
   SH_*^{S^1,=0}(W,\p_\pm W) \cong H_{S^1}^{n-*}(W,\p_\pm W) \cong
   H^{n-*}(W,\p_\pm W)\otimes\K[u^{-1}].
$$ 
The statement concerning $SH_*^{S^1,>0}(W,\p^- W)$ and $SH_*^{S^1,<0}(W,\p^+ W)$ follows from tautological exact triangles in view of the fact that, by definition, $SH_*^{S^1}(W,\p^- W)=SH_*^{S^1,\ge 0}(W,\p^- W)$ and $SH_*^{S^1}(W,\p^+ W)=SH_*^{S^1,\le 0}(W,\p^+ W)$. The last statement follows from the exact triangles of the pairs $(W,\p_\pm W)$. 
\end{proof}

\begin{remark}
Let $D^{2n}$ be the unit ball in $\R^{2n}$. Then $SH_*^{S^1}(D^{2n})=0$ and a direct computation, together with the tautological exact triangle, shows that 
\[
SH_*^{S^1,=0}(W,\p^- W)\cong SH_*^{S^1,=0}(D^{2(n-k)})
\]
and 
\[
SH_*^{S^1,>0}(W,\p^- W)\cong SH_*^{S^1,>0}(D^{2(n-k)}).
\]
These isomorphisms are not just algebraic or formal, but have the following geometric interpretation~\cite{Ci02}: for any given finite action window there exists a Liouville structure on $W$ for which the periodic Reeb orbits on $\p^-W$ in the given action window survive to $\p^+W$, and the new periodic Reeb orbits which are created after handle attachment are in one-to-one bijective correspondence with the periodic Reeb orbits on the boundary of the symplectic reduction of the coisotropic cocore disk in the handle, which is a symplectic ball $D^{2(n-k)}$. 
\end{remark}

\begin{corollary} \label{cor:subcritexseqS1} Let $V$ be a Liouville domain of dimension $2n$ and $V'$ be obtained from $V$ by attaching a subcritical handle of index $k<n$. We then have an exact triangle 
\[
\xymatrix
@C=20pt
{
SH_*^{S^1,>0}(D^{2(n-k)}) \ar[rr] & & 
SH_*^{S^1,>0}(V')  \ar[dl] \\ & SH_*^{S^1,>0}(V) \ar[ul]^{[-1]}  
}
\]
in which the map $SH_*^{S^1,>0}(V') \to SH_*^{S^1,>0}(V)$ is the transfer map. 
\end{corollary}

\begin{proof}
This is simply a reformulation of the exact triangle of the pair $(V',V)$, using excision and the computation of $SH_*^{S^1,>0}(W,\p^-W)$ above, with $W=\overline{V'\setminus V}$.
\end{proof}

This statement can be interpreted as a subcritical surgery exact triangle for linearized contact homology in view of~\cite{BO3+4}. In that formulation, the case $k=1$ of contact connected sums was
proved using different methods by Bourgeois and van
Koert~\cite{BvK}. Also in that formulation, the exact triangle implies
Espina's  formula~\cite[Corollary~6.3.3]{Espina} for the behaviour of
the mean Euler characteristic of linearized contact homology under
subcritical surgery. By induction over the handles, it yields
M.-L.~Yau's formula for the linearized contact homology of subcritical
Stein manifolds~\cite{MLYau}. 

\bigskip

\noindent \emph{Critical handle attaching. } 
We restrict to rational coefficients, and recall the geometric setup
of section~\S\ref{sec:critical-handle}: $V\subset V'$ is a pair of
Liouville domains of dimension $2n$ such that $V'$ is obtained by
attaching $\ell\ge 1$ handles of index $n$ to $\p V$ along a
collection $\Lambda$ of $\ell$ disjoint embedded Legendrian
spheres. Following~\cite{ Bourgeois-Ekholm-Eliashberg-1} we denote
$C\H(V)$ the linearized contact homology of $\p V$. One of the main
statements in~\cite{ Bourgeois-Ekholm-Eliashberg-1} is the existence
of a surgery exact triangle 
\begin{equation} \label{eq:BEES1}
\xymatrix
@C=20pt
{
L\H^{\text{cyc}}(\Lambda)_* \ar[rr] & & 
C\H(V'),  \ar[dl] \\ & C\H(V) \ar[ul]^{[-1]}  
}
\end{equation}
where $L\H^{\text{cyc}}(\Lambda)_*$ is a homology group of Legendrian contact homology flavour. More precisely, $L\H^{\text{cyc}}(\Lambda)_*$ is defined as the homology of a complex $LH^{\text{cyc}}(\Lambda)_*$ whose generators are \emph{cyclic} words in Reeb chords on $\p V$ with endpoints on $\Lambda$, and whose differential counts certain pseudo-holomorphic curves in the symplectization of $\p V$, anchored in $V$, with boundary on the conical Lagrangian $S\Lambda$ determined by $\Lambda$. 
This exact triangle can be reinterpreted in our language as follows.

\begin{conjecture}\label{conj:criticalBEES1}
Let $W$ be a Liouville cobordism corresponding to attaching $\ell\ge 1$ critical
handles of index $k=n$ along a collection $\Lambda$ of disjoint Legendrian spheres. With rational coefficients we have an isomorphism
$$
SH_*^{S^1,>0}(W,\p^-W)\cong L\H^{\text{cyc}}(\Lambda)_*
$$
such that the exact triangle~\eqref{eq:BEES1} is isomorphic to the exact triangle of the pair $(V',V)$ for $SH_*^{S^1,>0}$.
\end{conjecture}

The proof should go along the same lines as the one of Conjecture~\ref{conj:criticalBEE}, adding on top the isomorphism between $SH_*^{S^1,>0}(V)$ and $C\H(V)$ whenever the latter is defined~\cite{BO3+4}. There is also an $S^1$-equivariant counterpart of Conjecture~\ref{cor:criticalrelhom}(ii), which involves duality and hence the groups $SH_*^{[u],>0}$. 

\begin{remark} One can also give a Legendrian interpretation of $SH_*^{S^1}(W,\p^-W)$. This can be obtained either formally algebraically by computing ranks from the $S^1$-equivariant tautological exact triangle of the pair $(W,\p^-W)$ using the fact that $SH_*^{S^1,=0}(W,\p^-W)$ is supported in positive degrees, or geometrically along the lines of~\cite{BO3+4}, where a linearized contact homology counterpart of $SH_*^{S^1}(V)$ is defined. 
\end{remark}

\pagebreak

\section{Product structures}\label{sec:products}

\subsection{TQFT operations on symplectic homology}\label{ss:products}
{\color{black} As before, we use coefficients in a field $\K$}.  
Recall from~\cite{Seidel07,Ritter} the definition of TQFT operations on the
Floer homology of a Hamiltonian $H$ on a completed Liouville domain
$\wh V$. We freely use in this section the terminology therein, namely ``negative punctures", ``positive punctures", ``cylindrical ends", ``weights", see also~\cite{Ekholm-Oancea}. Consider a punctured Riemann surface $S$ with $p$ negative
and $q$ positive punctures. Pick positive weights $A_i,B_j>0$ and a
$1$-form $\beta$ on $S$ with the following properties:
\begin{enumerate}
\item[(i)] $H\,d\beta\leq 0$;
\item[(ii)] $\beta=A_idt$ in cylindrical coordinates $(s,t)\in\R_-\times
  S^1$ near the $i$-th negative puncture;
\item[(iii)] $\beta=B_jdt$ in cylindrical coordinates $(s,t)\in\R_+\times
  S^1$ near the $j$-th positive puncture.
\end{enumerate}
Note that $\beta$ and the weights are related by Stokes' theorem
$$
   \sum_{i=1}^pA_i - \sum_{j=1}^qB_j = -\int_Sd\beta. 
$$
Conversely, if the quantity on the left-hand side is nonnegative (zero,
nonpositive), then we find a $1$-form $\beta$ with properties (ii) and
(iii) such that $d\beta\leq 0$ ($=0$, $\geq0$). Thus we can arrange
conditions (i)--(iii) in the following situations:
\begin{enumerate}
\item[(a)] $H$ arbitrary, $d\beta\equiv 0$, $p,q\geq 1$;
\item[(b)] $H\geq 0$, $d\beta\leq 0$, $p\geq 1$;
\item[(c)] $H\leq 0$, $d\beta\geq 0$, $q\geq 1$.
\end{enumerate}
Note that the condition $H\geq0$ is satisfied for admissible
Hamiltonians on a Liouville cobordism. 

We consider maps $u:S\to\wh V$ that are holomorphic in the sense that
$(du-X_H\otimes\beta)^{0,1}=0$ and have finite energy
$E(u)=\frac{1}{2}\int_S|du-X_H\otimes\beta|^2\mbox{vol}_S$. They
converge at the negative/positive punctures to $1$-periodic orbits
$x_i,y_j$ and satisfy the energy estimate
\begin{equation}\label{eq:energy}
   0 \leq E(u) \leq \sum_{j=1}^qA_{B_jH}(y_j) -
   \sum_{i=1}^pA_{A_iH}(x_i)
\end{equation}
(beware that our action is minus that in~\cite{Ritter}). 
The signed count of such holomorphic maps yields an operation
$$
   \psi_S:\bigotimes_{j=1}^qFH_*(B_jH) \to
   \bigotimes_{i=1}^pFH_*(A_iH).  
$$
of degree $n(2-2g-p-q)$ which does not increase action. These
operations are graded commutative if degrees are shifted by $-n$ and
satisfy the usual TQFT composition rules. Let us pick
real numbers $a_j<b_j$, $j=1,\dots,q$ and $a_i'<b_i'$, $i=1,\dots,p$
satisfying 
\begin{equation}\label{eq:op-actions}
   \sum_ia_i' = \max_j\Bigl(a_j+\sum_{j'\neq j}b_{j'}\Bigr),\qquad 
   b_i'=\sum_jb_j-\sum_{i'\neq i}a_{i'}'.
\end{equation}
Consider a term {\color{black} $x_1\otimes\cdots\otimes x_p$ appearing in
$\psi_S(y_1\otimes\cdots\otimes y_q)$.} If $A_{B_jH}(y_j)\leq a_j$ for some $j$ and
$A_{B_{j'}H}(y_{j'})\leq b_{j'}$ for all $j'\neq j$, then the energy estimate
and the first condition in~\eqref{eq:op-actions} yield 
$$
   \sum_{i=1}^pA_{A_iH}(x_i)\leq a_j+\sum_{j'\neq j}b_{j'} \leq
   \sum_ia_i',
$$
thus $A_{A_iH}(x_i)\leq a_i'$ for at least one $i$. This shows that
$\psi_S$ is well-defined as an operation
$$
   \psi_S:\bigotimes_{j=1}^qFH_*^{(a_j,b_j]}(B_jH) \to
   \bigotimes_{i=1}^pFH_*^{(a_i',\infty)}(A_iH).  
$$
Similarly, if $A_{B_jH}(y_j)\leq b_j$ for all $j$ and
$A_{A_iH}(x_i)>a_i'$ for all $i$ (so that $a_1\otimes\cdots\otimes a_p\neq
0$ in the quotient space), then for each $i$ the energy estimate yields
$$
   A_{A_iH}(x_i)+\sum_{i'\neq i}a_{i'}' \leq
   A_{A_iH}(x_i)+\sum_{i'\neq i}A_{A_iH}(x_{i'}) \leq \sum_jb_j, 
$$
thus $A_{A_iH}(x_i)\leq b_i'$ by the second condition
in~\eqref{eq:op-actions}. It follows that $\psi_S$ induces an
operation on filtered Floer homology 
$$
   \psi_S:\bigotimes_{j=1}^qFH_*^{(a_j,b_j]}(B_jH) \to
   \bigotimes_{i=1}^pFH_*^{(a_i',b_i']}(A_iH).  
$$
To proceed further, let us first assume $p,q\geq 1$, so we are in case
(a) above. We specialise the choice of actions to $a_j=a$, $b_j=b$ 
for all $i$ and $a_i'=a'$, $b_i'=b'$ for all
$i$. Then~\eqref{eq:op-actions} becomes  
\begin{equation}\label{eq:op-actions2}
   pa' = a+(q-1)b,\qquad 
   b' = qb - (p-1)a',
\end{equation}
and under these conditions $\psi_S$ induces an operation 
$$
   \psi_S:\bigotimes_{j=1}^qFH_*^{(a,b]}(B_jH) \to
   \bigotimes_{i=1}^pFH_*^{(a',b']}(A_iH).  
$$
We now apply this to admissible Hamiltonians for a Liouville cobordism
$W$ relative to some admissible union $A$ of boundary components as in~\S\ref{sec:SHWA}. The map $\psi_S$ is compatible with
continuation maps for $H\leq H'$ in the obvious way, and therefore
passes through the inverse and direct limit to define a map on
filtered symplectic homology
$$
   \psi_S:\bigotimes_{j=1}^qSH_*^{(a,b]}(W,A) \to
   \bigotimes_{i=1}^pSH_*^{(a',b']}(W,A).  
$$
Let us first consider the case $p=1$. Then $a'\to-\infty$ and $b'=qb$
remains constant as $a\to-\infty$, so we can pass to the inverse limits 
to obtain an operation
$$
   \psi_S:\bigotimes_{j=1}^qSH_*^{(-\infty,b]}(W,A) \to
   SH_*^{(-\infty,qb]}(W,A). 
$$
In the direct limit as $b\to\infty$ this yields an operation 
$$
   \psi_S:\bigotimes_{j=1}^qSH_*(W,A) \to SH_*(W,A). 
$$
Taking instead limits as $b\searrow 0$ and $b\nearrow 0$,
respectively, we see that this operation restricts to operations
\begin{align*}
   \psi_S:\bigotimes_{j=1}^qSH_*^{\leq 0}(W,A) &\to SH_*^{\leq 0}(W,A), \cr 
   \psi_S:\bigotimes_{j=1}^qSH_*^{<0}(W,A) &\to SH_*^{<0}(W,A). 
\end{align*}
In the case $p>1$ this procedure fails because $b'\to\infty$ as
$a\to-\infty$, so we cannot take the inverse limits $a,a'\to-\infty$
keeping $b,b'$ fixed. If all actions are nonnegative, as in the case
of a Liouville domain or a pair $(W,\p^-W)$, then there is no need to
take the inverse limit $a,a'\to-\infty$, but we can simply fix
$a,a'<0$ and take the direct limits $b,b'\to\infty$ to obtain
operations $\psi_S$ on all symplectic homology groups. 

Next consider the case $q=0$, $p\geq 1$, which is possible for $H\geq
0$ (and thus $A=\emptyset$) according to case (b) above. Pick $a'\leq
0$ and consider the associated map
$$
   \psi_S:{\color{black}\K} \to \bigotimes_{i=1}^pSH_*^{(a',\infty)}{\color{black}(W)},  
$$
with {\color{black}$\K$ the ground field}. For a nonzero term $x_1\otimes\cdots\otimes x_p$ appearing in
$\psi_S(1)$ we have $A_{A_iH}(x_i)>a'$ for all $i$, so the energy
estimate yields
$$
   A_{A_iH}(x_i)+(p-1)a' \leq
   A_{A_iH}(x_i)+\sum_{i'\neq i}A_{A_iH}(x_{i'}) \leq 0, 
$$
thus $A_{A_iH}(x_i)\leq -(p-1)a'$. So we obtain a map
$$
   \psi_S:{\color{black}\K} \to \bigotimes_{i=1}^pSH_*^{(a',-(p-1)a']}{\color{black}(W)}.  
$$
If $p=1$, then we take the inverse limit as $a'\to-\infty$ to
obtain the unit
$$
   \psi_S:{\color{black}\K} \to SH_*^{\leq 0}{\color{black}(W)}.  
$$
If $p>1$, then we set $a'=0$ to obtain the operation
$$
   \psi_S:{\color{black}\K} \to \bigotimes_{i=1}^pSH_*^{=0}{\color{black}(W)}.  
$$
So we have proved

\begin{proposition}\label{prop:ops}
For a filled Liouville cobordism $W$ and an admissible union $A$ of boundary
components, there exist operations 
$$
   \psi_S:\bigotimes_{j=1}^qSH_*^{\heartsuit}(W,A) \to
   \bigotimes_{i=1}^pSH_*^\heartsuit(W,A),\qquad
   \heartsuit\in\{\emptyset,\leq 0,<0\}  
$$
of degree $n(2-2g-p-q)$ associated to punctured Riemann surfaces $S$
with $p$ negative and $q$ positive punctures, graded commutative if
degrees are shifted by $-n$ and satisfying the usual
TQFT composition rules, in each of the following situations:
\begin{enumerate}
\item $\p^-W=A=\emptyset$, $p\geq 1$, $q\geq 0$;
\item $A=\p^-W$, $p\geq 1$, $q\geq 1$;
\item $A=\emptyset$, $p=1$, $q\geq 0$;
\item $A$ arbitrary, $p=1$, $q\geq 1$.\hfill$\square$
\end{enumerate}
\end{proposition}

As a consequence, we have

\begin{theorem}\label{thm:product}
(a) For a filled Liouville cobordism $W$ and an admissible union $A$ of boundary
components, the pair-of-pants product on Floer homology induces a
product on $SH_*(W,A)$. The product has degree $-n$, and it is
associative and graded commutative when degrees are shifted by $-n$. 

(b) The symplectic homology groups $SH_*^{\leq 0}(W,A)$ and
$SH_*^{<0}(W,A)$ also carry induced products which are compatible with
the tautological maps $SH_*^{<0}(W,A)\to SH_*^{\leq 0}(W,A)\to
SH_*(W,A)$. The image of the map $SH_*^{<0}(W,A)\to SH_*^{\leq
  0}(W,A)$ is an ideal in $SH_*^{\leq 0}(W,A)$.  

(c) The symplectic homology group $SH_*^{=0}(W,A)$ carries a product,
which coincides with the cup product in cohomology via the isomorphism
$SH_*^{=0}(W,A)\cong H^{n-*}(W,A)$. The map $SH_*^{\leq 0}(W,A)\to
SH_*^{=0}(W,A)$ is compatible with the product structures.  

(d) In the case $A=\emptyset$, the products on $SH_*^{\le 0}(W)$,
$SH_*(W)$, and $SH_*^{=0}(W)$ have units, and the tautological maps
$SH_*^{\leq 0}(W)\to SH_*(W)$ and $SH_*^{\le 0}(W)\to SH_*^{=0}(W)$
are morphisms of rings with unit.  

(e) For a filled Liouville cobordism pair $(W,V)$, the transfer map
$SH_*^\heartsuit(W)\to SH_*^\heartsuit(V)$ 
is a morphism of rings for $\heartsuit\in \{<0,\le 0,\varnothing\}$,
and a morphism of rings with unit for $\heartsuit\in \{\le 0,\varnothing\}$.  
\end{theorem}

{\color{black}
\begin{proof}
Parts (a)--(d) follow directly from the preceding discussion, so it
remains to prove part (e). For this, fix a finite action interval
$(a,b)$ and consider two Hamiltonians $K\leq
H$ for the Liouville cobordism pair $(W,V)$ as in Figure~\ref{fig:KH-new}. 

Let us first describe more explicitly the transfer map from
Section~\ref{sec:transfer}. For this, let $\chi:\R\to[0,1]$ be a
smooth nondecreasing function with $\chi(s)=0$ for $s\leq 0$ and
$\chi(s)=1$ for $s\geq 1$ and define the $s$-dependent Hamiltonian 
$$
   \wh H := \bigl(1-\chi(s)\bigr) H + \chi(s)K,
$$ 
where $(s,t)$ are coordinates on the cylinder $\R\times S^1$. Then
$\p_s\wh H\leq 0$ and the count of Floer cylinders for $\wh H$ defines
a chain map $f:FC^{(a,b]}(K)\to FC^{(a,b]}(H)$. 

Now we describe the products. Let $S$ be the Riemann sphere with two
positive punctures and one negative puncture. Let $\tau:S\to\R\times S^1$ be a
degree $2$ branched cover with $\tau(s,t)=(s,t)$ in cylindrical
coordinates $(s,t)\in[1,\infty)\times S^1$ near the positive punctures
and $\tau(s,t)=(s,t)$ in cylindrical coordinates
$(s,t)\in(-\infty,-1]\times S^1$ near the negative puncture. 
We use the $1$-form $\beta:=\tau^*dt$ on $S$ (with $d\beta=0$) and
weights $B_1=B_2=1$ and $A_1=2$ at the positive/negative punctures to
define the pair-of-pants product
$$
   \mu_K:FC^{(a,b]}(K)\otimes FC^{(a,b]}(K)\to FC^{(a+b,2b]}(2K),
$$
and similar $\mu_H$. Next, consider for $\sigma\in\R$ the function
$\chi_\sigma(s,t):=\chi(s-\sigma)$ and the Hamiltonian
$$
   \wh H_\sigma := (1-\chi_\sigma\circ\tau) H + \chi_\sigma K
$$
depending on points $z\in S$. Since $H\,d\beta=0$ and
$d_zH\wedge\beta\leq 0$ as $2$-forms on $S$, the maximum principle
holds for the Floer equation of $\wh H_\sigma$ (see
e.g.~\cite{Abouzaid-Seidel,Ekholm-Oancea,Ritter}). It follows that the moduli spaces $\MM_\sigma(y_1,y_2;x_1)$
of pairs-of-pants for $\wh H_\sigma$ are compact modulo breaking,
where $y_1,y_2$ and $x_1$ are $1$-periodic orbits of $K$ and $2H$, respectively.
Considering for index $\CZ(y_1)+\CZ(y_2)-\CZ(x_1)-n=0$ the natural
compactifications of the $1$-dimensional moduli spaces
$\bigcup_{\sigma\in\R}\{\sigma\}\times\MM_\sigma(y_1,y_2;x_1)$, we
obtain the relation
\begin{equation}\label{eq:theta}
   \mu_H(f\otimes f)-f_2\mu_K = \p_{2H}\theta-\theta\p_K.
\end{equation}
Here $\p_K$ and $\p_{2H}$ are the Floer boundary operators for $K$ and
$2H$, respectively, $f_2:FC^{(a,b]}(2K)\to FC^{(a,b]}(2H)$ is the chain map defined
by $2\wh H$, and 
$$
   \theta:FC^{(a,b]}(K)\otimes FC^{(a,b]}(K)\to FC^{(a+b,2b]}(2H)
$$
counts index $-1$ pairs-of-pants for $\wh H_\sigma$ occurring at
isolated values of $\sigma$. 

Let us now choose $K,H$ such that the orbits in group $F$ for $K$ and
in groups $F,I,III^{0+}$ for $H$ have actions below $a$, so that
$FC^{(a,b]}(K)=FH^{(a,b]}_I(K)$ and $FC^{(a,b]}(H)=FH^{(a,b]}_{II,III^{-}}(H)$. 
By Lemma~\ref{lem:no-escape} and Lemma~\ref{lem:asy},
$FH^{(a,b]}_{III^{-}}(H)$ is a $2$-sided ideal for the product $\mu_H$, 
so the latter passes to the quotient as a product
$$
   \mu_H:FC^{(a,b]}_{II}(H)\otimes FC^{(a,b]}_{II}(H)\to FC^{(a+b,2b]}_{II}(2H).
$$
It follows that relation~\eqref{eq:theta} persists when we compose the
maps $f$ and $f_2,\theta$ with their projections to
$FC^{(a,b]}_{II}(H)$ and $FC^{(a,b]}_{II}(2H)$, respectively (keeping
the same notation for the new maps). Passing to homology and the
direct limit over $K,H$ we obtain the commuting diagram on filtered
symplectic homology
\begin{equation*}
\xymatrix{
   SH^{(a,b]}(W)\otimes SH^{(a,b]}(W) \ar[r]^-{\mu_W} \ar[d]_{f\otimes
    f} & SH^{(a+b,2b]}(W) \ar[d]^f\\
   SH^{(a,b]}(V)\otimes SH^{(a,b]}(V) \ar[r]^-{\mu_V} & SH^{(a+b,2b]}(V)\,. 
}
\end{equation*}
Passing to the limits $a\to-\infty$ and $b\nearrow 0$, $b\searrow 0$,
or $b\to\infty$, we conclude that the transfer map
$SH_*^\heartsuit(W)\to SH_*^\heartsuit(V)$ preserves the product for
$\heartsuit\in \{<0,\le 0,\varnothing\}$. A similar argument shows
that the transfer map preserves the unit for $\heartsuit\in \{\le
0,\varnothing\}$ and Theorem~\ref{thm:product} is proved. 
\end{proof}
}

In particular, Theorem~\ref{thm:product} provides a product of
degree $-n$ with unit and a coproduct of degree $-n$ (without counit) on
$SH_*(W)$ for every filled Liouville cobordism $W$. Applied to the
trivial cobordism, this yields via the isomorphism~\eqref{eq:RFH} a
corresponding product and coproduct on Rabinowitz--Floer homology.  We
refer to Uebele~\cite{Uebele-products} {\color{black} and Appendix~\ref{app}} for a discussion of conditions
under which the product is defined in the absence of a filling if the
negative boundary is index-positive.

If $W$ is a Liouville cobordism with filling and $L\subset W$ is an
exact Lagrangian cobordism with filling, then the preceding discussion
shows that Lagrangian symplectic homology $SH_*^\heartsuit(L)$ is a
module over $SH_*^\heartsuit(W)$ for $\heartsuit\in \{<0,\le
0,\varnothing\}$, see also~\cite{Ritter}.  

\subsection{Dual operations} \label{sec:products-dual}

Combining Proposition~\ref{prop:ops} with the Poincar\'e duality
isomorphism $S^*_{\heartsuit}(W,A)\cong SH_{-*}^{-\heartsuit}(W,A^c)$, we obtain

\begin{proposition}\label{prop:ops-dual}
Consider a filled Liouville cobordism $W$ and an admissible union $A$ of boundary components. Then there exist operations 
$$
   \psi_S:\bigotimes_{j=1}^qSH^*_{\heartsuit}(W,A) \to
   \bigotimes_{i=1}^pSH^*_\heartsuit(W,A),\qquad
   \heartsuit\in\{\emptyset,\geq 0,>0\}  
$$
of degree $-n(2-2g-p-q)$, 
graded commutative if degrees are shifted by $n$ and satisfying the usual
TQFT composition rules, in the following situations: 
\begin{enumerate}
\item $\p^-W=\emptyset$, $A=\p^+W$, $p\geq 1$, $q\geq 0$;
\item $A=\p^+W$, $p\geq 1$, $q\geq 1$;
\item $A=\p W$, $p=1$, $q\geq 0$;
\item $A$ arbitrary, $p=1$, $q\geq 1$.\hfill$\square$
\end{enumerate}
\end{proposition}

Note that in Propositions~\ref{prop:ops} and~\ref{prop:ops-dual} the
conditions on $p,q$ are the same, whereas $\heartsuit$ is replaced by
$-\heartsuit$ and $A$ by $A^c$. 

Suppose now that the filled Liouville cobordism $W$ has vanishing first Chern class and that $\p W$ carries only finitely many closed Reeb
orbits of any given Conley-Zehnder index. Using field
coefficients Corollary~\ref{cor:dualityk} yields canonical
isomorphisms $SH_k^\heartsuit(W,A)\cong SH^k_\heartsuit(W,A)^\vee$ 
for all $A$ and all flavors $\heartsuit$. The dualization of the
operations in Proposition~\ref{prop:ops-dual} then yields

\begin{corollary}\label{cor:ops-dual}
Consider a filled Liouville cobordism $W$ with vanishing first Chern
class and an admissible union $A$ of boundary components. 
Suppose that $\p W$ carries only finitely many closed Reeb
orbits of any given Conley-Zehnder index.
Then with field coefficients there exist operations (note the reversed
roles of $p$ and $q$)
$$
   \psi_S^\vee:\bigotimes_{i=1}^pSH_*^{\heartsuit}(W,A) \to
   \bigotimes_{j=1}^qSH_*^\heartsuit(W,A),\qquad
   \heartsuit\in\{\emptyset,\geq 0,>0\}  
$$
of degree $n(2-2g-p-q)$, 
graded commutative if degrees are shifted by $-n$ and satisfying the usual
TQFT composition rules, in the following situations: 
\begin{enumerate}
\item $\p^-W=\emptyset$, $A=\p^+W$, $p\geq 1$, $q\geq 0$;
\item $A=\p^+W$, $p\geq 1$, $q\geq 1$;
\item $A=\p W$, $p=1$, $q\geq 0$;
\item $A$ arbitrary, $p=1$, $q\geq 1$.\hfill$\square$
\end{enumerate}
\end{corollary}

\subsection{A coproduct on positive symplectic homology}
Consider a Liouville cobordism $W$ filled by a Liouville domain
$V$. The choice of $W$ will be irrelevant, so we can take
e.g.~$W=I\times \p V$. Proposition~\ref{prop:ops}(iii) provides a product
of degree $-n$ on $SH_*^{<0}(W)$. In view of the isomorphism
$SH_*^{<0}(W) \cong SH_{>0}^{-*+1}(V)$ from
Proposition~\ref{prop:cob-to-filling}, this gives a product of
degree $n-1$ on the symplectic cohomology group $SH_{>0}^*(V)$. 
Note that this cannot be the product arising from
Proposition~\ref{prop:ops-dual}(iv) (with $V$ in place of $W$ and
$A=\emptyset$) because the latter has degree $n$. Under the
finiteness hypothesis in Corollary~\ref{cor:ops-dual}, this gives a
coproduct of degree $1-n$ on the symplectic homology group
$SH_*^{>0}(V)$.  

{\color{black}\begin{remark} Following Seidel, there is another coproduct of degree $1-n$
on $SH_*^{>0}(V)$ obtained as a secondary operation in view of the
fact that the natural coproduct given by counting pairs of pants with
one input and two outputs vanishes, see also~\cite{Ekholm-Oancea} for a generalization and~\cite{Goresky-Hingston} for a topological version of it. These two coproducts of
degree $1-n$ agree. The isomorphism between them is part of a larger picture related to Poincar\'e duality and will be the topic of another paper. 
\end{remark}
}

\appendix


\section{An obstruction to symplectic cobordisms \\ \ (joint with Peter~Albers)}\label{app}

In this joint appendix we use the results of this paper to define an
obstruction to Liouville cobordisms between contact manifolds. 

Consider a Liouville cobordism $W$ whose negative end $\p_-W$ is
hypertight, index-positive, or Liouville fillable. As explained in
Section~\ref{sec:contact}, in these cases one can define symplectic
homology groups $SH_*^\heartsuit(W)$,
$\heartsuit\in\{\varnothing,\le 0,<0,=0,\ge 0, >0\}$ which will be
independent of a filling in the first two cases but may 
depend on the filling in the Liouville fillable case. 
We would like to show that vanishing of $SH_*(\p_+W)$ implies
vanishing of $SH_*(\p_-W)$. However, it is unclear how to deduce this
from the functoriality under cobordisms, which only gives correspondences
$$
\xymatrix{
   & SH_*^\heartsuit(W) \ar[dl]\ar[dr] & \\
   SH_*^\heartsuit(\p_- W)  & & SH_*^\heartsuit(\p_+ W).
}
$$
Instead, we will consider the following property
{\color{black}(using coefficients in a field $\K$).}

\begin{definition}
A Liouville cobordism $W$ is called {\em SAWC} if $1_W$ is mapped to
zero under the map $H^0(W)\cong SH_n^{=0}(W)\to SH_n^{\geq 0}(W)$,
where $1_W$ is the unit in $H^0(W)$.
\end{definition}

For a connected Liouville domain $W$, this agrees with the ``Strong
Algebraic Weinstein conjecture'' property of Viterbo~\cite{Viterbo99}. 
As usual, we define the SAWC property for $\p_\pm W$ via the trivial
cobordism $[0,1]\times\p_\pm W$, where $SH_*(\p_+W)$ is defined 
with respect to the partial filling $W$. 
Then this property is inherited under cobordisms:

\begin{proposition}\label{prop:cob-vanishing}
Let $W$ be a Liouville cobordism {\color{black}with vanishing first Chern class} whose negative end $\p_-W$ is
hypertight, index-positive, or Liouville fillable. 
If $\p_+W$ is SAWC, then so are $W$ and $\p_-W$. 
\end{proposition}

\begin{proof}
{\color{black}If the first Chern class of $W$ vanishes the symplectic homology groups $SH_*^\heartsuit$ are canonically graded in the component of constant loops.} Consider thus the diagram with commutative squares and exact rows
$$
{\small
\xymatrix{
   SH_{n+1}^{> 0}(\p_-W)\ar[r] & SH_n^{=0}(\p_-W)\simeq H^0(\p_-W) \ar[r]
   & SH_n^{\ge 0}(\p_-W) \ar[r] & SH_n^{>0}(\p_-W)\\
   SH_{n+1}^{> 0}(W)\ar[r]\ar[d]^\simeq\ar[u] & SH_n^{=0}(W)\simeq
   H^0(W) \ar[r]\ar[d]^{\mbox{injective}}_{1_W\mapsto 1_{\p_+W}}\ar[u]^{1_W\mapsto 1_{\p_-W}} & SH_n^{\ge 0}(W)
   \ar[r]\ar[d]^{\mbox{$\Rightarrow\ $injective}}\ar[u] & SH_n^{>0}(W)\ar[d]^\simeq\ar[u] \\
   SH_{n+1}^{> 0}(\p_+W)\ar[r] & SH_n^{=0}(\p_+W)\simeq H^0(\p_+W)
   \ar[r] & SH_n^{\ge 0}(\p_+W) \ar[r] & SH_n^{>0}(\p_+W).
}
}
$$
The lower vertical arrows at the extremities are isomorphisms since
$W$ and $I\times\p_+W$ share the same positive boundary. The map
$H^0(W)\to H^0(\p_+W)$ is injective because every component of $W$ has
a positive boundary component. Injectivity of the vertical map
$SH_n^{\ge 0}(W)\to SH_n^{\ge 0}(\p_+W)$ then follows from the 5-lemma
as in~\cite[Exercise 1.3.3]{Weibel}. 

Suppose now that $1_{\p_+W}$ is sent to
zero by the map $H^0(\p_+W)\to SH_n^{\geq 0}(\p_+W)$. Then
commutativity of the lower middle square implies that $1_{W}$
goes to zero under the map $H^0(W)\to SH_n^{\geq 0}(W)$, and  
commutativity of the upper middle square implies that $1_{\p_-W}$
goes to zero under the map $H^0(\p_-W)\to SH_n^{\geq 0}(\p_-W)$. 
\end{proof}

Note that Proposition~\ref{prop:cob-vanishing} uses the product
structure on singular cohomology but not on symplectic homology. Using
the latter we will now reformulate the SAWC condition. As observed
by Uebele in~\cite{Uebele-products}, the pair-of-pants product $\cdot$ in
Section~\ref{sec:products} makes 
$SH_*(W)$, $SH_*^{\leq 0}(W)$ and $SH_*^{=0}(W)$ unital graded commutative
rings for $W$ as in Proposition~\ref{prop:cob-vanishing},
provided that in the index-positive case we require {\color{black} the following stronger
condition (called ``product index-positivity'' in~\cite{Uebele-products}): 
 
(i) {\color{black} $c_2(W)|_{\pi_2(\p_-W)}=0$ and} $\pi_1(\p_- W)=1$, and 
\begin{equation}\label{eq:index}
\CZ(\gamma)>3 \,\,\,\text{for every closed Reeb orbit }
    \gamma \text{ in } \p_-W, 
\end{equation}

or

(ii) denoting $\xi_-$ the contact distribution on $\p_- W$, there exists a trivialisation of the square of the canonical bundle $\Lambda_\C^{\max}\xi_-^{\otimes 2}$ such that, with respect to that trivialisation, all closed Reeb orbits $\gamma$ in $\p_-W$ satisfy~\eqref{eq:index}. {\color{black}In addition, we require that the trivialization of $\Lambda_\C^{\max}TW|_{\p_- W}^{\otimes 2}$ determined by the trivialization of $\Lambda_\C^{\max}\xi_-^{\otimes 2}$ extends over $W$. 

\noindent {\bf Remark.} Since the homotopy classes of trivializations of a line bundle are classified by the first integral cohomology group, the extension property in (ii) above is automatic if the map $H^1(W;\Z)\to H^1(\p_-W;\Z)$ is surjective.
}
}

{\color{black}
{\bf Remark.} Examples in which (i) is satisfied are unit cotangent bundles of spheres $S^n$ of dimension $n\ge 5$, and more generally Milnor fibers of $A_k$-singularities $\{z_0^k+z_1^2+\dots+z_n^2=0\}$ for $n\ge 5$, see~\cite[Appendix~A]{Kwon-vanKoert} and also~\cite{Uebele-products}. 
}

{\color{black} The proof of this observation is similar to that of
Proposition~\ref{prop:SH-without-filling}. The new feature is that a
pair-of-pants with inputs $x_1,x_2$ and output $x_-$ might break into
a Floer cylinder $C_1$ connecting $x_1$ and $x_-$ with a negative puncture
asymptotic to a closed Reeb orbit $\gamma_1$, a Floer plane $C_2$
with input $x_2$ and a negative puncture at a closed Reeb orbit
$\gamma_2$, and a holomorphic cylinder with two positive punctures
asymptotic to $\gamma_1,\gamma_2$. The first two components are
regular, so their indices satisfy
\begin{align*}
   \ind(C_1)&=\CZ(x_1)-\CZ(x_-)-\bigl(\CZ(\gamma_1)+n-3\bigr)\geq 0,\cr
   \ind(C_2)&=\CZ(x_2)+n-\bigl(\CZ(\gamma_2)+n-3\bigr)\geq 0.
\end{align*}
When showing well-definedness of the product (resp.~commutativity with
the boundary operator) we consider orbits satisfying
$$
   \CZ(x_1)+\CZ(x_2)-\CZ(x_-)-n = 0 \text{ (resp. $1$).}
$$
Adding the two inequalities and inserting this relation yields
$$
   \bigl(\CZ(\gamma_1)-3\bigr)+\bigl(\CZ(\gamma_2)-3\bigr)\leq 0 \text{ (resp. $1$),}
$$
contradicting condition~\eqref{eq:index}.} 

Let us fix a Liouville form $\lambda$ on $W$ and consider for $b\in\R$
the filtered symplectic homology groups $SH_*^{(-\infty,b)}(W)$ defined in
Section~\ref{sec:SH} (which also exist under the above assumptions on
$W$). We define the \emph{spectral value} of a class $\alpha\in
SH_*(W)$ by 
$$
   c(\alpha) := \inf \{b\in\R \mid
   \alpha\in\mbox{im}(SH_*^{(-\infty,b)}(W)\to SH_*(W))\} \in[-\infty,\infty).
$$
Here $c(\alpha)<\infty$ follows from the definition of
$SH_*(W)=\lim\limits^{\longrightarrow}_{b\to\infty}
SH_*^{(-\infty,b)}(W)$. The fundamental inequality satisfied by
spectral values is   
$$
   c(\alpha\cdot\beta)\le c(\alpha)+c(\beta),
$$
as a consequence of the fact that the pair-of-pants product decreases
action (see inequality~\eqref{eq:energy} with {\color{black}$A_1=2$ and $B_1=B_2=1$}). 

The unit $1_W\in SH_n(W)$ plays a particular role. Indeed, we have
$c(1_W)\le 0$ since $SH_*^{\le 0}(W)\to SH_*(W)$ is a map of
rings with unit, but also  
$$
   c(1_W)=c(1_W\cdot 1_W)\le 2c(1_W).
$$
Thus either $c(1_W)= 0$ or $c(1_W)=-\infty$ (note that these
conditions are independent of the Liouville form $\lambda$). The condition $c(1_W)=-\infty$ is equivalent to the fact that the unit belongs to the image of the map $SH_n^{<0}(W)\to SH_n(W)$. In the latter case we also
obtain $c(\alpha)=-\infty$ for all $\alpha\in SH_*(W)$ since
$c(\alpha)\le c(1_W)+c(\alpha)$. This is in particular the case if
$SH_*(W)=0$, {\color{black}and the converse is also true. Indeed, assume $c(1_W)=-\infty$ and represent $1_W$ as the image of an element $\alpha^b\in SH_*^{(-\infty,b)}(W)$ for some $b<0$. By definition of the inverse limit, such an element is the equivalence class of a sequence $\alpha^b_n\in SH_*^{(-n,b)}(W)$ for $n>|b|$. We claim that each such element $\alpha^b_n$ is zero, hence $1_W=0$. Indeed, for any given $n$ we can choose $b'<-n$ and represent by assumption $1_W$ by an element $\beta^{b'}\in SH_*^{(-\infty,b')}(W)$, given by a sequence $\beta^{b'}_{n'}\in SH_*^{(-n',b')}(W)$ for $n'>|b'|$. But then $\alpha^b_n$ must be the image of $\beta^{b'}_{n'}$ under the map $SH_*^{(-n',b')}(W)\to SH_*^{(-n,b)}(W)$, which is zero for $b'<-n$. 
}

{\color{black}We thus obtain:}

\begin{lemma}\label{lem:SAWC}
Let $W$ be a Liouville cobordism whose negative end $\p_-W$ is
hypertight, Liouville fillable, or index-positive with the
stronger index condition~\eqref{eq:index}. Then $W$ is SAWC 
if and only if {\color{black}$SH_*(W)=0$.} 
\end{lemma}

\begin{proof}
Proposition~\ref{prop:taut-triangles} yields the commuting diagram
with exact rows and columns
$$
\xymatrix
@C=30pt
@R=30pt
{
   & SH_{n+1}^{>0}(W) \ar@{=}[r] \ar[d]^f & SH_{n+1}^{>0}(W) \ar[d]^g & \\
   SH_n^{<0}(W) \ar[r]^h \ar@{=}[d] & SH_n^{\leq 0}(W) \ar[r]^i \ar[d]^j &
   SH_n^{=0}(W) \ar[d]^k \\ 
   SH_n^{<0}(W) \ar[r]^\ell & SH_n(W) \ar[r]^m & SH_n^{\geq
    0}(W) \;,
}
$$
where $i$ and $j$ are maps of unital rings. We will denote all units by
$1_W$. {\color{black}We prove that $W$ is SAWC if and only if $c(1_W)=-\infty$, 
which by the discussion above is equivalent to $SH_*(W)=0$.}
Suppose first that $c(1_W)=-\infty$, i.e.~$1_W=\ell\alpha$ for some
$\alpha\in SH_n^{<0}(W)$. Then $1_W-h\alpha=f\beta$ for some $\beta\in
SH_{n+1}^{>0}(W)$, hence $1_W=i(1_W-h\alpha)=g\beta$ is mapped to
zero under $k$, which means that $W$ is SAWC. The converse
implication is proved similarly. 
\end{proof}

\begin{corollary}\label{cor:hypertight}
There is no Liouville cobordism $W$ with $\p_-W$ hypertight and such that
$SH_*(\p_+W)=0$ {\color{black} (where $SH_*(\p_+W)$ is defined 
with respect to the partial filling $W$).} 
\end{corollary}

\begin{proof}
If $\p_-W$ is hypertight then the map $SH_n^{=0}(\p_-W)\to SH_n^{\geq
  0}(\p_-W)$ is an isomorphism, so $\p_-W$ is not SAWC. On the other
hand, $SH_*(\p_+W)=0$ implies by Lemma~\ref{lem:SAWC} that $\p_+W$ is
SAWC. This is impossible by Proposition~\ref{prop:cob-vanishing}.  
\end{proof}

{\color{black}
\begin{corollary}\label{cor:subcritical}
There is no Liouville cobordism $W$ of dimension $2n\geq 4$ with vanishing first Chern class such that $\p_-W$ is hypertight, $\p_+W$ is fillable by a subcritical Stein manifold with vanishing first Chern class, 
and the map $\pi_1(\p_+W)\to \pi_1(W)$ induced by inclusion is injective.
\end{corollary}

\begin{proof} Let $F$ be such a subcritical Stein filling of $\p_+W$. 
Denote ${^F}SH_*(\p_+W)$ the symplectic homology computed with respect to the filling $F$, and ${^W}SH_*(\p_+W)$ the symplectic homology computed with respect to the partial filling $W$. Since $SH_*(F)=0$, we also have ${^F}SH_*(\p_+W)=0$ by Corollary~\ref{cor:vanishing-finitedim}. By Remark~\ref{rem:subcritical}, one can choose on $\p_+W$ a contact form so that all Conley-Zehnder indices of closed Reeb orbits which are contractible in $\p_+W$ are $>3-n$. 
The injectivity of the map $\pi_1(\p_+W)\to \pi_1(W)$ implies that the same condition on the indices holds for all closed Reeb orbits which are contractible in $W$. Hence by Proposition~\ref{prop:SH-without-filling} 
we have ${^W}SH_*(\p_+W)={^F}SH_*(\p_+W)=0$, and the conclusion follows from Corollary~\ref{cor:hypertight}.
\end{proof}

\begin{corollary} \label{cor:subcritical+}
There is no Weinstein cobordism $W$ of dimension $2n\geq 6$ with vanishing first Chern class such that $\p_-W$ is hypertight and $\p_+W$ is fillable by a subcritical Stein manifold with vanishing first Chern class. 
\end{corollary}

\begin{proof}
Indeed, in this situation the skeleton of $W$ has codimension $\ge n\ge 3$ and a generic homotopy of paths will avoid it and can be subsequently pushed by the Liouville flow to $\p_+W$. Thus any loop in $\p_+W$ which is contractible in $W$ is also contractible in $\p_+W$, i.e., the map $\pi_1(\p_+W)\to \pi_1(W)$ induced by the inclusion is injective. We then conclude via Corollary~\ref{cor:subcritical}.   
\end{proof}
}

\noindent {\bf Examples. }

(1) Many examples of contact manifolds $M$ with $SH_*(M)=0$ arise as
boundaries of Liouville domains with vanishing symplectic homology, 
e.g.~subcritical or flexible Stein manifolds~\cite{Cieliebak-Eliashberg-book}. 

(2) Examples of hypertight contact manifolds are the unit
cotangent bundles of Riemannian manifolds of nonpositive
curvature. Other examples are the $3$-torus $T^3$ with a Giroux
contact structure $\xi_k=\ker\bigl(\cos (ks) d\theta+\sin (ks) dt\bigr)$ and its higher-dimensional generalizations
$(T^2\times N,\xi_k)$ by Massot--Niederkr\"uger--Wendl~\cite{Massot-Niederkruger-Wendl}. The
latter are not strongly symplectically fillable (so in particular not
Liouville fillable) for $k\ge 2$. Therefore, it appears that
Corollary~\ref{cor:hypertight} with $\p_-W$ one of these manifolds
cannot be obtained by more classical tools such as symplectic homology
of Liouville domains. 

(3) Let us mention in the same vein the fact that there is no Liouville cobordism $W$ with $\p_-W$ hypertight and $\p_+W$ overtwisted. This is proved in the same way as non-fillability of overtwisted contact manifolds~\cite{BEM,CMP}, using filling by holomorphic discs in the symplectic manifold $(0,1)\times \p_-W \, \cup \, W$. However, this seems to fall outside the scope of our methods, while at the same time the case that we address in Corollary~\ref{cor:hypertight} seems to fall outside the scope of the method of filling by holomorphic discs. 

{\color{black} (4) A contact manifold $(M,\xi)$ fails to satisfy the Weinstein conjecture if there exists a contact form whose Reeb vector field has no periodic orbit. In the simply connected case this is equivalent to the fact that $(M,\xi)$ is cobordant via a trivial cobordism to a hypertight contact manifold. Turning this around, $(M,\xi)$ satisfies the Weinstein conjecture if and only if it is not cobordant by a trivial Liouville cobordism to a hypertight manifold. From this perspective, obstructing the existence of Liouville cobordisms with hypertight negative end can be seen as a geometric generalisation of the Weinstein conjecture.}

%
%
%
%

\bibliographystyle{abbrv}

%

\bibliography{./000_SHpair}

\def\cprime{$'$} \def\cprime{$'$}

\end{document}